\documentclass[a4paper]{amsart}

\usepackage[active]{srcltx}
\usepackage[all]{xy}
\usepackage{subfig}
\usepackage{hyperref}

\usepackage{amssymb,latexsym}
\usepackage{mathrsfs}
\usepackage{verbatim}
\usepackage{epsfig}
\usepackage{color}
\usepackage{pifont,marvosym}
\usepackage{bbm}
\usepackage{enumerate}

\theoremstyle{plain}
\newtheorem{lemma}{Lemma}[section]
\newtheorem{proposition}[lemma]{Proposition}
\newtheorem{theorem}[lemma]{Theorem}
\newtheorem{corollary}[lemma]{Corollary}

\newtheorem*{theorem1}{Theorem \ref{T_1}}
\newtheorem*{theorem2}{Theorem \ref{T_subpart_step}}

\theoremstyle{definition}
\newtheorem{definition}[lemma]{Definition}

\newtheorem{claim}[lemma]{Claim}

\theoremstyle{remark}
\newtheorem{remark}[lemma]{Remark}
\newtheorem{example}[lemma]{Example}

\numberwithin{equation}{section}

\def\to{\rightarrow}
\newcommand{\R}{\mathbb{R}}
\def\S{\mathbb{S}}
\newcommand{\Q}{\mathbb{Q}}
\newcommand{\N}{\mathbb{N}}
\def\LL{\mathcal L^d}
\def\PP{\mathcal P}

\newcommand{\dist}{\mathrm{dist}}

\newcommand{\clos}{\mathrm{clos}}

\newcommand{\Id}{\mathbb{I}}

\newcommand{\ind}{\mathbbm 1}
\def\P{\mathtt{P}}

\newcommand{\A}{\mathfrak{A}}
\newcommand{\B}{\mathfrak{B}}
\def\NN{\mathcal N}

\newcommand{\inter}{\mathrm{int}}

\newcommand{\Graph}{\mathrm{graph}}

\newcommand{\epi}{\mathrm{epi}}
\newcommand{\aff}{\mathrm{aff}}
\newcommand{\conv}{\mathrm{conv}}
\newcommand{\eps}{\varepsilon}

\newcommand{\nfrac}[2]{\genfrac{}{}{0pt}{}{#1}{#2}}
\newcommand{\interr}{\inter_{\mathrm{rel}}}
\def\d#1{|#1|_{D^*}}
\def\a{\mathfrak a}
\def\c{\mathfrak c}
\def\b{\mathfrak b}
\def\q{\mathfrak q}
\def\B{\mathfrak B}
\def\HH{\mathcal H}

\def\ka{^k_\a}

\def\lb{^{',\ell}_\b}
\def\pod{\Pi^{\mathrm{opt}}_{\d{\cdot}}}
\def\p{\mathtt p}

\def\e{\mathtt e}

\def\bD{\mathbf D}
\def\bC{\mathbf C}
\def\tC{\tilde{C}}
\def\C{\tilde{\mathbf C}}

\def\opt{\mathrm{opt}}
\def\O{\mathtt{O}}
\def\dom{\mathrm{dom}}

\def\T{\mathfrak T}
\def\t{\mathfrak t}
\def\tg{\tilde{\Gamma}}

\title[On Sudakov's decomposition of transference plans]{On Sudakov's type decomposition of transference plans with norm costs}

\author{Stefano Bianchini}
\address{SISSA, via Beirut 2, IT-34014 Trieste (ITALY)}
\email{bianchin@sissa.it}

\author{Sara Daneri}
\address{Universit\"at Leipzig, Mathematisches Institut, Augustusplatz 10, D-4109 Leipzig (GERMANY)}
\email{daneri@math.uni-leipzig.de}

\date{\today\text{}}


\begin{document}

\begin{abstract}
We consider the original strategy proposed by Sudakov for solving the Monge transportation problem with norm cost $\d{\cdot}$
\[
\min \bigg\{ \int \d{\mathtt T(x) - x} d\mu(x), \ \mathtt T : \R^d \to \R^d, \ \nu = \mathtt T_\# \mu \bigg\},
\]
with $\mu$, $\nu$ probability measures in $\R^d$ and $\mu$ absolutely continuous w.r.t. $\LL$. The key idea in this approach is to decompose (via disintegration of measures) the Kantorovich optimal transportation problem into a family of transportation problems in $Z_\a\times\R^d$, where $\{Z_\a\}_{\a\in\A} \subset \R^d$ are disjoint regions such that the construction of an optimal map $\mathtt T_\a : Z_\a \to \R^d$ is simpler than in the original problem, and then to obtain $\mathtt T$ by piecing together the maps $\mathtt T_\a$.
When the norm $\d{\cdot}$ is strictly convex \cite{sudak}, the sets $Z_\a$ are a family of $1$-dimensional segments determined by the Kantorovich potential called optimal rays, while the existence of the map $\mathtt T_\a$ is straightforward provided one can show that the disintegration of $\LL$ (and thus of $\mu$) on such segments is absolutely continuous w.r.t. the $1$-dimensional Hausdorff measure \cite{Car:strictly}.
When the norm $\d{\cdot}$ is not strictly convex, the main problems in this kind of approach are two: first, to identify a suitable family of regions $\{Z_\a\}_{\a\in\A}$ on which the transport problem decomposes into simpler ones, and then to prove the existence of optimal maps. 

In this paper we show how these difficulties can be overcome, and that the original idea of Sudakov can be successfully implemented. 

The results yield a complete characterization of the Kantorovich optimal transportation problem, whose straightforward corollary is the solution of the Monge problem in each set $Z_\a$ and then in $\R^d$. The strategy is sufficiently powerful to be applied to other optimal transportation problems.

The analysis requires
\begin{enumerate}
\item the study of the transportation problem on directed locally affine partitions $\{Z\ka,C\ka\}_{k,\a}$ of $\R^d$, i.e. sets $Z\ka \subset \R^d$ which are relatively open in their $k$-dimensional affine hull and on which the transport occurs only along directions belonging to a cone $C\ka$;
\item the proof of the absolute continuity w.r.t. the suitable $k$-dimensional Hausdorff measure of the disintegration of $\LL$ on these directed locally affine partitions;
\item the definition of cyclically connected sets w.r.t. a family of transportation plans with finite cone costs;
\item the proof of the existence of cyclically connected directed locally affine partitions for transport problems with cost functions which are indicator functions of cones and no potentials can be constructed.
\end{enumerate}
\end{abstract}

\thanks{This work is supported by ERC Starting Grant 240385 "ConsLaw"}

\maketitle

{\centerline{Preprint SISSA  51/2013/MATE}}

\tableofcontents

\section{Introduction}
\label{S_intro_Sud}

Let $\mu,\nu \in \mathcal P(\R^d)$ with $\mu \ll \LL$ and consider the Monge optimal transportation problem
\begin{equation}
\label{E_transpo_Sud}
\min \bigg\{ \int \d{\mathtt T(x) - x} d\mu(x), \ \mathtt T : \R^d \to \R^d, \ \nu = \mathtt T_\# \mu \bigg\},
\end{equation}
where $\d{\cdot}$ is a \emph{convex norm} in $\R^d$, namely a positively $1$-homogeneous function whose unit ball $\{x\in\R^d:\,\d{x}\leq1\}$ is a closed $d$-dimensional convex set $D$ with $0\in\inter \,D$. The $\mu$-measurable maps $\mathtt T:\R^d\to\R^d$ satisfying $\mathtt T_\#\mu=\nu$ are called \emph{transport maps}.
Well known examples show that if $\mu$ is not absolutely continuous w.r.t. $\LL$, there may be no optimal transport maps (see Theorem 8.3 of \cite{conf:optcime}). 

Due to the nonlinearity of the constraint $\mathtt T_\#\mu$, the classical approach to solve \eqref{E_transpo_Sud} is first to consider the relaxed problem 
%The main difficulty in this transportation problem is the fact that the function $\d{\cdot}$ is not strictly convex, so that the 
of finding optimal \emph{transference plans} $\pi \in \Pi^{\mathrm{opt}}_{\d{\cdot}}(\mu,\nu)$ defined by
\begin{equation}
\label{E_ptimal_intro}
\int \d{y-x} d\bar\pi(x,y) = \min \bigg\{ \int \d{y-x} d\pi(y,x), \pi \in \Pi(\mu,\nu) \bigg\},
\end{equation}
%are not unique, 
where 
\begin{equation}
 \Pi(\mu,\nu) := \Big\{ \pi \in \mathcal P(\R^d \times \R^d) : (\mathtt p_1)_\# \pi = \mu, (\mathtt p_2)_\# \pi = \nu \Big\}
\end{equation}
and $\mathtt p_i : \underset{j}{\prod} X_j \to X_i$ is the projection on the $i$-coordinate in the product space $\underset{j}{\prod} X_j$. 

Assuming that
\begin{equation}
\label{E_transport_prob_Sud}
\inf_{\pi \in \Pi(\mu,\nu)} \int_{\R^d \times \R^d} \d{y-x} \,d\pi(x,y) < +\infty,
\end{equation}
by standard theorems in optimal transportation there always exists an optimal transference plan, without being in the degenerate situation where every plan $\pi \in \Pi(\mu,\nu)$ is optimal.

Then, if one can show that there exists at least an optimal transport plan $\pi$ which is concentrated on a graph of a $\mu$-measurable map $\mathtt T$, i.e. $\pi:=(\Id\times\mathtt T)_\#\mu$, then $\mathtt T$ is an optimal transport map solving \eqref{E_transpo_Sud}.

The first strategy to show the existence of such a transference plan was proposed by Sudakov in \cite{sudak} and consists in decomposing via disintegration of measures the optimal transportation problem \eqref{E_ptimal_intro} into a family of transportation problems on $Z_\a\times\R^d$, where $\{Z_\a\}_{\a\in\A} \subset \R^d$ are disjoint regions where the construction of an optimal map $\mathtt T_\a : Z_\a \to \R^d$ is simpler than in the original problem, and then to obtain $\mathtt T$ by piecing together the maps $\mathtt T_\a$.\\
 With additional regularity properties on the densities of $\mu$, $\nu$ or on the norm, such as uniform convexity, an approach partially equal to the one proposed by Sudakov was successfully followed in \cite{ambr:lecttrans}, \cite{conf:optcime}, \cite{caffafeldmc} and \cite{trudiwang}. The most general case where up to now this approach has been successfully implemented (see \cite{Car:strictly}) is the case in which $\d{\cdot}$ is strictly convex, namely when the set $D$ is strictly convex. 
 Other approaches have also been used. In \cite{evagangbo}, the problem \eqref{E_transpo_Sud} for strictly convex norms has been solved using PDE methods under the assumption that the marginals $\mu$, $\nu$ have Lipschitz continuous densities w.r.t. $\LL$. The problem \eqref{E_transpo_Sud} was solved for crystalline norms in \cite{ambprat:crist}.

In \cite{champdepasc:MongeS, champdepasc:Monge}, the authors solved the Monge problem first with strictly convex and then with general convex norms using a different method, which does not pass through a geometric/measure theoretic decomposition of the optimal transportation problem \eqref{E_ptimal_intro} into simpler ones, but is based on the selection among the optimal transference plans $\pi \in \Pi^{\mathrm{opt}}_{\d{\cdot}}(\mu,\nu)$ of a transference plan $\check \pi$ which is also minimizing a secondary cost: more precisely, one selects the (unique) transference plan $\check \pi$ such that
\[
\check \pi \ \text{is a minimizer of} \ \inf \bigg\{ \int |x - y|^2 d\pi(x,y):\, \pi \in \Pi^\mathrm{opt}_{\d{\cdot}}(\mu,\nu) \bigg\}.
\]
and the main issue consists in proving that $\check \pi$ is actually induced by a transport map $\mathtt T$, which clearly satisfies \eqref{E_transpo_Sud}.

%It remained then unclear if the original strategy of Sudakov can be successful not only in the case of strictly convex norms, giving a complete geometric characterization of the optimal transport plans via decomposition into lower dimensional transportation problems.

%The aim of this paper is to show how Sudakov's approach can be carried on also in the general convex case. In the next Section we define new concepts, which in the strictly convex case (i.e. when the extremal faces of $\d{\cdot}$ are 1-dimensional) are trivially satisfied by the decomposition in optimal rays $Z^1_\a$, and state our main results, giving an overall idea of the whole construction.

However, the problem of whether Sudakov's strategy could be successfully implemented also in the case of general convex norms has remained open for a long time. The aim of this paper is to show how this problem can be solved. In order to introduce the notation that we need to state our main results and explain the new ideas and concepts in the case of general convex norms, we first resume briefly how Sudakov's strategy works for strictly convex norms.

The first step of Sudakov's approach consists in finding a suitable partition in $\R^d$ on which the transport occurs, namely s.t. the optimal plans move the initial mass inside the elements of the partition.
By duality (see e.g. \cite{villa:Oldnew}), there exists a function $\psi:\R^d\to\R$, called \emph{Kantorovich potential}, which satisfies
\begin{align}
 \psi(y)-\psi(x)&\leq\d{y-x},\quad\forall\,x,y\in\R^d,\label{0_psi1}\\
 \psi(y)-\psi(x)&=\d{y-x},\quad\text{for $\pi$-a.e. $(x,y)$, $\forall\,\pi\in\Pi^{\mathrm{opt}}_{\d{\cdot}}(\mu,\nu)$}.\label{0_psi2}
\end{align}
Observe that, by \eqref{0_psi1}, for all $(x,y)$ as in \eqref{0_psi2} and $\forall\,0\leq s\leq t\leq1$
\begin{equation}
\label{0_ztzs}
 \psi(z_t)-\psi(z_s)=\d{z_t-z_s},\quad z_t:=(1-t)x+ty.
\end{equation}
The open oriented segments $Z^1_\a:=]x,y[\,\subset\R^d$ (where $\a\in\A^1$ is a continuous parameter, $1$ referring to the dimension of the elements) whose extreme points satisfy \eqref{0_psi2} and which are maximal w.r.t. set inclusion are called \emph{optimal rays}. 
By strict convexity, if $(x,y)$ and $(y,z)$, with $x,z \not= y$, satisfy \eqref{0_psi2}, then
\begin{equation}
\label{0_non_branch}
y \in ]x,z[.
\end{equation}
In particular, if $(x,y)$ and $(x',y')$ satisfy \eqref{0_psi2} but $\R^+(y-x)\neq\R^+(y'-x')$, then 
\begin{equation}
\label{0_nointers}
 ]x,y[\,\,\cap\,\,]x',y'[=\emptyset.
\end{equation}
Hence, the optimal rays $\{Z^1_\a\}_{\a\in\A^1}$ form a Borel partition of $\R^d$ into $1$-dimensional open segments, up to the set of their \emph{initial points} $\underset{\a\in\A^1}{\cup}\mathcal I(Z^1_\a)\subset\R^d$ and of their \emph{final points} $\underset{\a\in\A^1}{\cup}\mathcal E(Z^1_\a)\subset\R^d$, defined for every $\a\in\A^1$ by
\begin{equation}
 Z^1_\a=\bigl\{(1-t)\mathcal I(Z^1_\a)+t\mathcal E(Z^1_\a)\bigr\},\qquad \mathcal E(Z^1_\a)-\mathcal I(Z^1_\a)\in C^1_\a,
\end{equation}
being $C^1_\a$ the half-line in $\R^d$ giving the direction on $Z^1_\a$ along which the transport occurs, i.e.
\[
 \psi(y)-\psi(x)=\d{y-x}, \quad x,y\in Z^1_\a\quad\Rightarrow\quad y-x\in C^1_\a.
\]
The partition into optimal rays with directions of transport $\{Z^1_\a, C^1_\a\}_{\a\in\A^1}$ is the simplest example of what we will call \emph{directed locally affine partition}.
Moreover, the set of initial/final points of the optimal rays is $\LL$-negligible (and then also $\mu$-negligible). Indeed, if $(x,y)$ satisfies \eqref{0_psi2} and $\psi$ is differentiable at $x$ --notice that this happens $\LL$-a.e. (and then $\mu$-a.e.) since $\psi$ is Lipschitz--, then 
\begin{equation}
 \label{0_yxpartial}
 y\in x+ \big( \partial\d{\cdot} \big)^{-1}(\nabla\psi(x)),
\end{equation}
where $\partial\d{\cdot}$ is the subdifferential of the convex norm, and, by strict convexity of $\d{\cdot}$, $(\partial\d{\cdot})^{-1}(\nabla\psi(x))$ is an half-line corresponding to a unique $C^1_\a$. 
We recall that the convex cones of the form $(\partial\d{\cdot})^{-1}(\ell)$ for some $\ell\in D^*$ are called \emph{exposed faces} of $\d{\cdot}$, while more generally the \emph{extremal faces} of $\d{\cdot}$ are by definition the projections on $\R^d$ of the extremal faces of the convex cone $\epi\d{\cdot}\subset\R^{d+1}$. In the strictly convex case, both concepts coincide and are given by half-lines.

\begin{figure}%[ht]
\centering
% \psfrag{csqrt}{$c(x,y) = 1 - \sqrt{y-x}$}
	\subfloat[The construction of optimal rays through the potential $\psi$ and the epigraph of $\d{\cdot}$.]
	{
	\label{F:sf:alwaysbad}
	\begin{minipage}[c]{8cm}
	\centering\resizebox{8cm}{5cm}{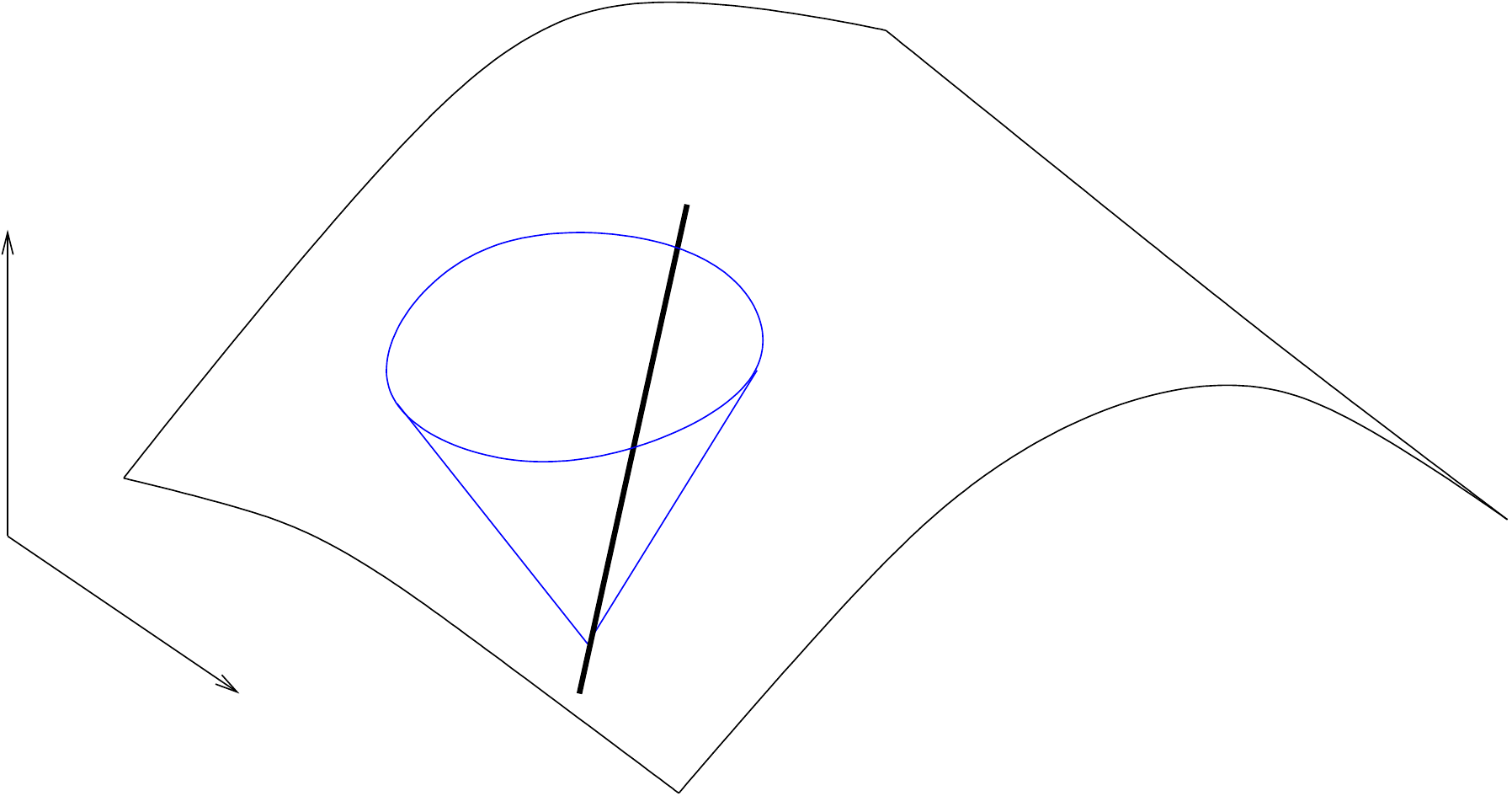}
	\end{minipage}
	}
	\hskip 1cm
	\subfloat[A set of optimal rays and a simple cone vector field.]
	{
	\label{F:sf:intermdep}
	\begin{minipage}[c]{6cm}
	\centering\resizebox{6cm}{5cm}{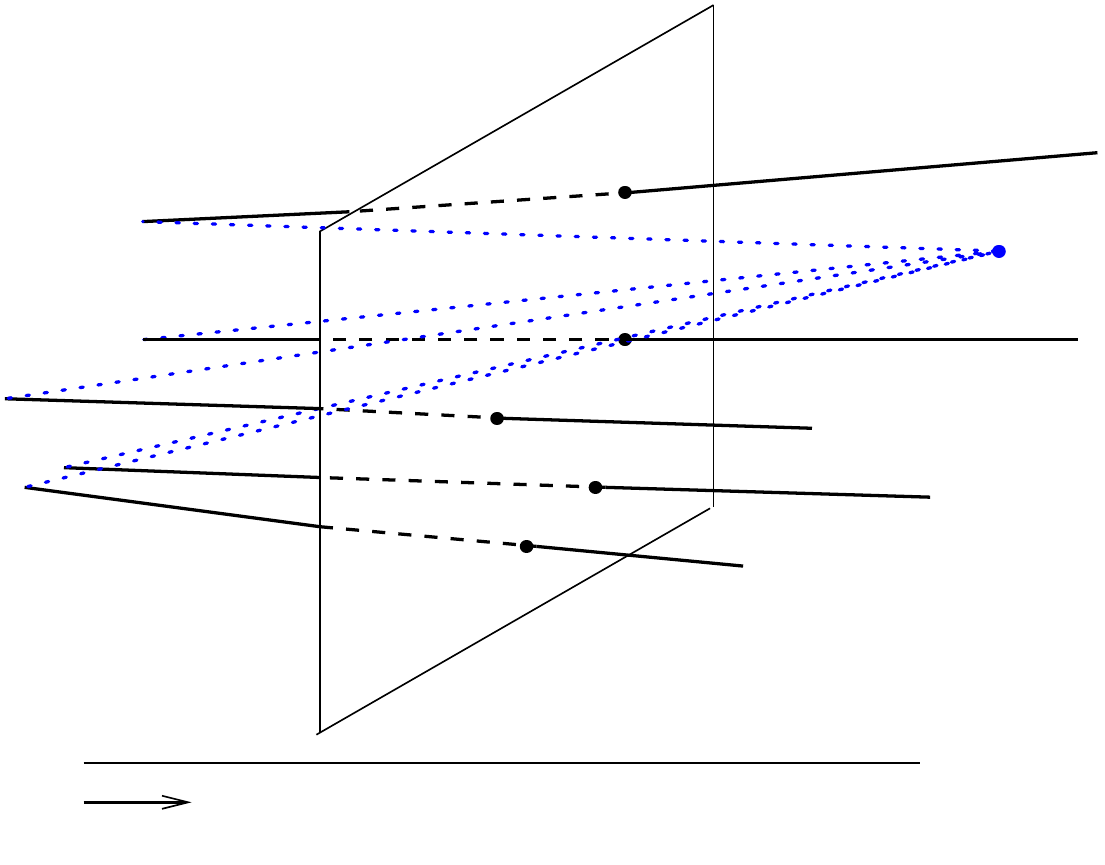}
	\end{minipage}
	}
%  \caption{The dependence on $\Gamma$. In the picture you find, in bold, a set where $\pi$ is concentrated.}
\label{Fi_casostrett_intro}
\end{figure}

Assume w.l.o.g. that $\mu\perp\nu$ --hence $\pi\{(x,y):\,y\neq x\}=1$-- and that, for the moment, also $\nu\ll\LL$.
Then, this first step yields that the optimal rays $\{Z^1_\a\}_{\a\in\A^1}$ on which $\psi$ is differentiable form a partition in $\R^d$ --up to the $\LL$-negligible set (thus also ($\mu+\nu$)-negligible) where $\psi$ is not differentiable-- s.t.
\[
 \pi\Big(\underset{\a\in\A^1}{\bigcup}Z^1_\a\times Z^1_\a\Big)=1,\quad\forall\,\pi\in\Pi^{\mathrm{opt}}_{\d{\cdot}}(\mu,\nu).
\]

The second step of the strategy consists in decomposing the transport problem in the sets $\{Z^1_\a \times \R^d\}_{\a \in \A^1}$. 
More precisely, for any given cone $C\subset\R^d$, let us denote by $\mathtt c_{C}$ the cost function 
\[
 \mathtt c_{C}(x,y):=\ind_{C}(y-x),
\]
where $\ind_A$ is the indicator function of the set $A$ (see \eqref{E_indicator_function_A}).
For notational convenience, if $\mathtt c : \R^d \times \R^d \to [0,\infty]$ is a Borel cost function, we use the notation
\[
\Pi^f_\mathtt c(\mu,\nu) := \bigg\{ \pi \in \Pi(\mu,\nu) : \int \mathtt c(x,y) \,d\pi(x,y) < \infty \bigg\}.
\]
Then, by \eqref{0_yxpartial} and \eqref{0_nointers} it follows that if $\pi=\int\pi^1_\a\,dm(\a)$, $\mu=\int\mu^1_\a\,dm(\a)$ and $\nu=\int\nu^1_\a\,dm(\a)$ denote the strongly consistent disintegrations (see Definition \ref{D_dis}) of $\pi$ w.r.t. $\{Z^1_\a\times\R^d\}_{\a \in \A^1}$ and of $\mu$ and $\nu$ w.r.t. $\{Z^1_\a\}_{\a \in \A^1}$, one has
\begin{equation}
\label{0_opt_char}
 \pi\in\Pi^{\mathrm{opt}}_{\d{\cdot}}(\mu,\nu)\quad\Leftrightarrow\quad \pi=\int\pi^1_\a\,dm(\a),\quad \text{$\pi^1_\a\in\Pi^{f}_{\mathtt c_{C^1_\a}}(\mu^1_\a,\nu^1_\a)$},
\end{equation}
being $\Pi^{f}_{\mathtt c_{C^1_\a}}(\mu_\a,\nu_\a)$ the plans of finite $\mathtt c_{C^1_\a}$-cost between $\mu_\a$ and $\nu_\a$.

%More precisely, this corresponds to the disintegration of the measures $\mu$, $\nu$ and $\pi \in \Pi^\mathrm{opt}_{\d{\cdot}}(\mu,%\nu)$ on the partitions $\{Z^1_\a\}_{\a \in \A^1_n}$, $\{Z^1_\a \times \R^d\}_{\a \in \A^1_n}$, respectively:
%\begin{equation}
%\label{E_disint_str_conv_intro}
%\mu = \int_{\A^1_n} \mu^1_\a dm(\a), \qquad \pi = \int_{\A^1_n} \pi^1_\a dm(\a),
%\end{equation}
%where $m$ is the quotient measure on $\A^1_n$ obtained through the map \eqref{E_A_1_n_def_intro}. By the definition of optimal rays by %means of \eqref{0_psi2}, it is fairly easy to see that the map quotient map of this partition is Borel and thus the disintegrations %\eqref{E_disint_str_conv_intro} are strongly consistent, which means
%\[
%\mu^1_\a(Z^1_\a) = 1, \qquad \nu^1_\a(Z^1_\a) = 1, \qquad \pi^1_\a(Z^1_\a \times \R^d) = 1.
%\]
%One is thus left with the analysis of the reduced transportation problems with marginals $\mu^1_\a$ and $\nu^1_\a$.
%The last problem one faces is the regularity of the conditional probabilities $\mu^1_\a$.

In other words, the transport problem on $\R^d$ reduces to a family of independent $1$-dimensional transport problems with linear cost and prescribed direction. If $\mu^1_\a$ has no atoms, then the unique transference plan concentrated on a monotone graph in $Z^1_\a \times Z^1_\a$ is actually concentrated on a map $\mathtt T^1_\a$. In this setting, monotone means monotone w.r.t. the order induced by $C^1_\a$ on $Z^1_\a$, and the statement is a well known and simple result for 1-dimensional problems, which can be seen as a particular case of a more general structure result for optimal transportation problems with quadratic cost (see for example \cite{Bre:polarrear}). 

%We just observe here that this regularity was a not completely clear point in the original Sudakov paper \cite{sudak}, as pointed out in a counterexample by Alberti, Kirchheim and Preiss in \cite{AlbKirchPre} (see \cite{ambprat:crist}).\cite{ambprat:crist}.

Then, the main problem in \cite{Car:strictly} was to prove that the disintegration of $\LL$ (and thus of $\mu$) on the optimal rays has non-atomic conditional measures. Indeed, for a general Borel partition into segments this might not be true, as discovered in a counterexample to the original Sudakov's proof by Alberti, Kirchheim and Preiss (see personal communication in \cite{ambprat:crist}). The main issue was then to prove that the optimal rays satisfy an additional regularity property which guarantees that the conditional measures of $\LL$ are not atomic. In \cite{ambr:lecttrans}, \cite{conf:optcime}, \cite{caffafeldmc} and \cite{trudiwang}, due to the additional regularity assumptions either on the measures $\mu$, $\nu$ or on the norm, the unit vector field giving at each point of an optimal ray the direction of transport is locally Lipschitz. Then, via changes of variables using the classical Coarea Formula one can reduce to study the disintegration of the Lebesgue measure on families of parallel segments, namely Fubini theorem, which gives the absolute continuity of the conditional measures w.r.t. the $1$-dimensional Hausdorff measure $\mathcal H^1$ on the segments on which they are concentrated. The absolute continuity of the conditional measures w.r.t. the $1$-dimensional Hausdorff measure on the optimal rays --thus implying the solvability of the Monge problem-- for general strictly convex norms was proved in \cite{Car:strictly}. Since in the general case no Lipschitz regularity is available, the author used a technique first introduced for a partition into segments arising from a different  variational problem in \cite{BiaGlo}. Such a technique is based on the validity for the family of segments (in this case, the optimal rays) of an approximation property via sequences of cone vector fields, that we call \emph{cone approximation property} (with the same  terminology used in the first part of \cite{Dan:Phd}).

We point out that, compared to the approach followed in \cite{champdepasc:MongeS, champdepasc:Monge}, Sudakov's approach for the Monge problem gives and relies upon a deeper geometric characterization of the transport via optimal plans, namely the existence of a family of lower dimensional regions (in the strictly convex case, $1$-dimensional) on which the transport occurs and on which the existence of optimal maps becomes easier to prove.

%Other approaches have also been used. In \cite{evagangbo}, the problem \eqref{E_transpo_Sud} for strictly convex norms has been solved using PDE methods under the assumption that the marginals $\mu$, $\nu$ have Lipschitz continuous densities w.r.t. $\LL$. The problem \eqref{E_transpo_Sud} was solved for crystalline norms in \cite{ambprat:crist}.

%In \cite{champdepasc:MongeS, champdepasc:Monge}, the authors solved the Monge problem first with strictly convex and then with general convex norms using a different method, which does not give a characterization of the optimal plans but is based on the selection among the optimal transference plans $\pi \in \Pi^{\mathrm{opt}}_{\d{\cdot}}(\mu,\nu)$ of a transference plan $\check \pi$ which is also minimizing a secondary cost: more precisely, one selects the (unique) transference plan $\check \pi$ such that
%\[
%\check \pi \ \text{is a minimizer of} \ \inf \bigg\{ \int |x - y|^2 d\pi(x,y):\, \pi \in \Pi^\mathrm{opt}_{\d{\cdot}}(\mu,\nu) \bigg\}.
%\]
%The key analysis in this strategy is to prove that $\check \pi$ is actually given by a transport plan, i.e. $\check \pi = (\Id\times\check{\mathtt T})_\sharp \mu$ for an optimal transport map $\check{\mathtt T}$, which clearly satisfies \eqref{E_transpo_Sud}.

It remained unclear if the original strategy of Sudakov can be successful not only in the case of strictly convex norms, thus giving a complete geometric characterization of the optimal transport plans via decomposition into lower dimensional transportation problems.

The aim of this paper is to show how Sudakov's approach can be carried on also in the general convex case. In the next section we define new concepts, which in the strictly convex case (i.e. when the extremal faces of $\d{\cdot}$ are 1-dimensional) are trivially satisfied by the decomposition in optimal rays $Z^1_\a$, and state our main results, giving an overall idea of the whole construction.
%When trying to implement this approach along the same line for the strictly convex norm cost as in \cite{Car:strictly}, we need to define new concepts, which in the strictly convex case (i.e. when the extremal faces of $\d{\cdot}$ are 1-dimensional) are trivially satisfied by the decomposition in optimal rays $Z^1_\a$. To better clarify which points have to be studied more carefully, we first give a short overview the approach of \cite{Car:strictly} to solve the Monge problem for strictly convex norm cost.

\subsection{Sudakov's strategy in the general convex case} 
\label{Ss_main_Sud}

Recall that, for all $(x,y)$ as in \eqref{0_psi2}, \eqref{0_ztzs} holds. In the strictly convex case, we have seen that \eqref{0_ztzs} and \eqref{0_non_branch} imply that whenever
\begin{equation}
 \label{0_xy'y''}
 \exists\,y',y''\neq x \quad \text{s.t.} \qquad \psi(y')-\psi(x)=\d{y'-x} \ \ \text{and} \ \ \psi(x)-\psi(y'')=\d{x-y''},
\end{equation}
%  and $\R^+(y'-x)=\R^+(x-y'')$, 
then $x$ belongs to a segment $Z^1_\a$ called optimal ray, which belongs to a partition on $\R^d$ on which the transport occurs along the direction $C^1_\a=\R^+(y'-x)=\R^+(x-y'')$.
However, for general convex norms, the optimal rays do not satisfy \eqref{0_nointers} and then do not form a partition in $\R^d$. Actually, $\forall\,x\in\R^d$, the sets 
\begin{equation}
\label{0_subsuper}
 \partial^+\psi(x):= \big\{y':\,\psi(y')-\psi(x)=\d{y'-x} \big\},\quad\partial^-\psi(x):= \big\{y'':\,\psi(x)-\psi(y'')=\d{x-y''} \big\},
\end{equation}
called respectively \emph{superdifferential and subdifferential of $\psi$ at $x$}, may be contained in one or even more higher dimensional cones corresponding to extremal faces of $\d{\cdot}$. Now, unlike in the strictly convex case, an extremal face is not in general a $1$-dimensional half line but a $k$-dimensional cone, with $k=1,\dots,d$. Hence Sudakov claimed that the regions on which the transport occurs are relatively open subsets of affine planes whose dimension is equal to $k$. However, even when considering the set of points in the super/subdifferential of $\psi$ at a certain point $x$ which are contained in a single $k$-extremal cone of $\epi\d{\cdot}$, it may not be a convex $k$-dimensional set or more generally a set with a well defined affine dimension (see Figure \ref{Fi_intro_supersub}).

\begin{figure}
\centering{\resizebox{16cm}{6cm}{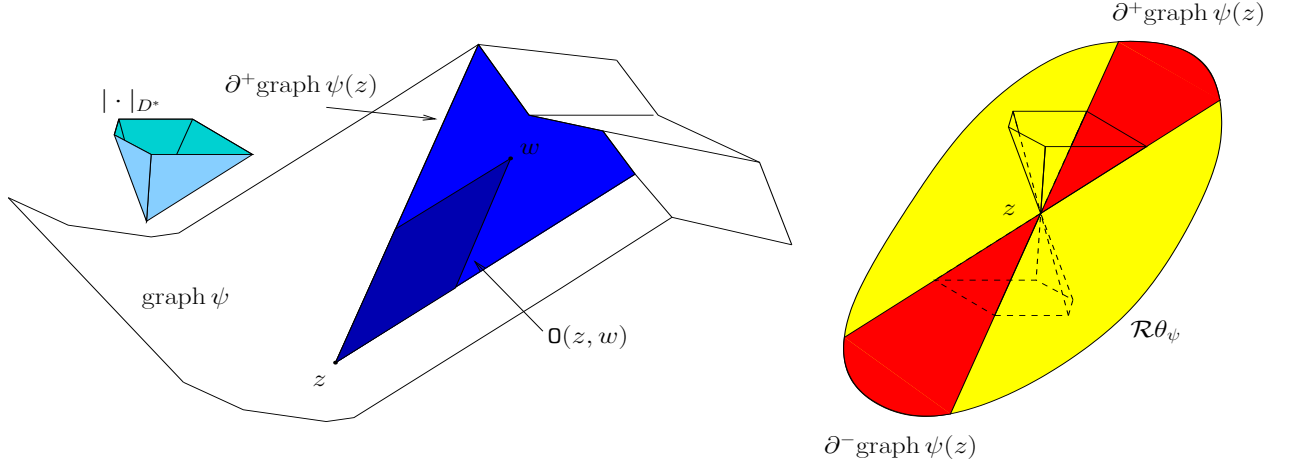}}
 \caption{In the left picture a possible superdifferential $\partial^+\Graph\,\psi$ at a point $z$ of $\Graph\,\psi$ is depicted in black and different blue colors. Notice that it is not convex, not even inside extremal faces of the norm. We also underline in dark blue a set $\O(z,w)$ for some $w\in\partial^+\Graph\,\psi(z)$, in order to show the completeness property. 
  In the right picture we depict in red the super/subdifferential at a point $z$ of the regular set (yellow region) in $\Graph\,\psi$. The cone $z + \epi \d{\cdot}$ is also represented. %union of the two dashed half-lines and of the two continuous half-lines going out from the point $z$ towards the right represent the one dimensional extremal faces of the three-dimensional cone ${\epi\d{\cdot}}$.
  }  
 \label{Fi_intro_supersub}
\end{figure}

When we faced this problem for general convex norms, the first main issue was to find other conditions which determine that a point $x$ belongs to one of the desired $k$-dimensional regions, thus generalizing the property that whenever $y'\in\partial^+\psi(x)\setminus \{x\}$, $y''\in\partial^-\psi(x)\setminus \{x\}$ then $\R^+(y'-x)=\R^+(x-y'')$ and $x$ belongs to the optimal ray containing the segment $]y'',y'[$. 

The natural generalization of the partition into optimal rays for strictly convex norms is to look for a \emph{directed locally affine partition} $\{Z\ka,C\ka\}_{\nfrac{k=1,\dots,d}{\a \in \A^k}}$ of $\R^d$ (see Definition \ref{D_locaffpart}), namely a Borel partition of $\R^d$ into sets $Z\ka$ which are locally affine and $k$-dimensional, i.e. relatively open in their affine hull whose linear dimension is $k$, together with an extremal cone $C^k_\a$ of $\d{\cdot}$ that will correspond to the union of directions of the optimal rays starting from $x \in Z\ka$. 

The first key idea is to observe that Kantorovich duality \eqref{0_psi1}-\eqref{0_psi2} can be rewritten as follows (see Section \ref{Ss_convex_norm_cone}). Let $\hat\mu=(\Id\times\psi)_\#\mu$, $\hat\nu=(\Id\times\psi)_\#\nu$ and $\hat\pi=((\Id\times\psi)\times(\Id\times\psi))_\#\pi$. One has
\begin{align}
 \pi\in\Pi^{\mathrm{opt}}_{\d{\cdot}}(\mu,\nu)\quad&\Leftrightarrow\quad\hat\pi\in\Pi^f_{\mathtt c_{\epi\d{\cdot}}}(\hat\mu,\hat\nu)\label{0_lifting1}\\
 &\Leftrightarrow\quad\hat\pi\bigl(\partial^+\Graph\,\psi\bigr)=1,\label{0_lifting2}
\end{align}
where 
\begin{align}
\partial^+ \Graph\,\psi:=&~\bigl(\Graph\,\psi\times\Graph\,\psi\bigr)\cap\bigl\{\mathtt c_{\epi\d{\cdot}}<+\infty\bigr\}\notag\\
 =&~\Graph\,\psi\times\Graph\,\psi\cap\mathtt p_{\R^d}^{-1}(\partial^+\psi). \label{0_superdiff_graph}
\end{align}
is the \emph{superdifferential of the set $\Graph\,\psi\subset\R^{d+1}$}.
In other words, \eqref{0_lifting1} tells us that studying the optimal transportation problem between $\mu$ and $\nu$ in $\R^d$ is equivalent to study the finite cost transportation problem in $\R^{d+1}$ for a convex cone cost (precisely $\mathtt c_{\epi\d{\cdot}}$) between measures ($\hat\mu$, $\hat\nu$) concentrated on a $\d{\cdot}$-Lipschitz graph (namely $\Graph\,\psi$)
%w.r.t. the (convex) cone cost
%\[
% \mathtt c_{\epi\d{\cdot}}(z,w)=\ind_{\epi\d{\cdot}}(w-z),\quad z,w\in\R^{d+1}
%\]
or, by \eqref{0_lifting2}, to study transport plans which are concentrated on the superdifferential of the graph of the $\d{\cdot}$-Lipschitz function $\psi$.

The advantage of this point of view is that the properties of the super/subdifferential of $\psi$ which permit to generalize \eqref{0_xy'y''}, and then to find a locally affine directed partition, can be more naturally expressed in terms of geometric properties of the super/subdifferential of $\Graph\,\psi$ --where the subdifferential of $\Graph\,\psi$ is the set $\partial^-\Graph\,\psi:=\big(\partial^+\Graph\,\psi\big)^{-1}$. 

First we will find a directed locally affine partition $\{\tilde Z^k_\a,\tilde C^k_\a\}_{\nfrac{k=1,\dots,d}{\a \in \A^k}}$ in $\R^{d+1}$ for this transportation problem, whose direction cones $\tilde C^k_\a$ are extremal faces of $\epi\d{\cdot}$ and on which the disintegration of the $d$-dimensional Hausdorff measure $\mathcal H^d$ on $\Graph\,\psi$ has conditional measures which are absolutely continuous w.r.t. $\mathcal H^k \llcorner \tilde Z^k_\a$, and then we will find the desired locally affine partition $\{Z^k_\a,C^k_\a\}_{\nfrac{k=1,\dots,d}{\a \in \A^k}}$ simply projecting it on $\R^d$. Indeed, the extremal faces of $\d{\cdot}$ are by definition the projections on $\R^d$ of the extremal faces of $\epi\d{\cdot}$ and the ``lifting map'' $\Id\times\psi$ is bi-Lipschitz, thus mapping negligible sets into negligible sets. 

The crucial properties of the super/subdifferential $\partial^\pm\Graph\,\psi$ that we will use to find the partition are the so-called \emph{transitivity property} 
\begin{equation}
 \label{0_trans}
 w'\in\partial^\pm\Graph\,\psi(w)\quad\Rightarrow\quad\partial^\pm\Graph\,\psi(w')\subset\partial^\pm\Graph\,\psi(w)
\end{equation}
and the \emph{completeness property} of the $\d{\cdot}$-Lipschitz graph $\Graph\,\psi$, that we define below.
Let $F$ be an extremal face of the convex cone $\epi\d{\cdot}$ and denote by $\interr F$ its relative interior, namely its interior w.r.t. its affine hull. Moreover, for any $z,w\in\R^{d+1}$ let 
\begin{equation}
 \O(z,w):= z+\epi\d{\cdot} \cap w-\epi\d{\cdot}.
\end{equation}
The completeness property of $\Graph\,\psi$ is the following:
\begin{align}
 &w\in\partial^+\Graph\,\psi(z),\quad w-z\in\interr F\quad\Rightarrow\quad\O(z,w)=z+F\cap w-F\subset\partial^+\Graph\,\psi(z),\label{0_compl1}\\\
 &w\in\partial^-\Graph\,\psi(z),\quad z-w\in\interr F\quad\Rightarrow\quad\O(w,z)=w+F\cap z-F\subset\partial^-\Graph\,\psi(z),\label{0_compl2}
\end{align}
where $z+ F\cap w-F$ is convex and satisfies $\R^+\bigl((z+F\cap w-F)-z\bigr)=\R^+\bigl(w-(z+F\cap w-F)\bigr)=F$ (see Proposition \ref{P_parall}).
 
In the strictly convex case, the extremal faces $F$ of $\epi\d{\cdot}$ are half-lines. Moreover, whenever \eqref{0_compl1} (resp. \eqref{0_compl2}) holds $\mathtt p_{\R^d} F$ is the extremal face of $\d{\cdot}$ giving the direction of an optimal ray starting (resp. arriving) at $x=\mathtt p_{\R^d} z$, and $\mathtt p_{\R^d}\O(z,w)=\bigl[\mathtt p_{\R^d} z,\,\mathtt p_{\R^d} w\bigr]$.
 
In the general convex case, the completeness property \eqref{0_compl1}-\eqref{0_compl2} then implies that whenever the directions of the optimal rays starting/arriving at a point $z$ are contained in a certain face $F$ and there exists a direction in $\interr F$, then they are a cone of directions coinciding with $F$. Moreover, by the transitivity property \eqref{0_trans}, whenever the same thing happens also for two points each belonging to one of the sets $B(x,\delta)\cap (z\pm \interr F)$, then $z$ has a locally affine neighborhood, of the same dimension as $F$ and contained in $z + \aff F$, made of points for which the admissible directions of transport coincide with the directions of $F$. % and inside the same affine plane $z+\aff F$.
Roughly speaking, the relative interior of the extremal face $F$ plays the role of a direction of an optimal ray and the set $\mathtt p_{\R^d}\O(z,w)$ the role of the segment $[x,y]$ inside such optimal ray in the strictly convex case (see Figure \ref{Fi_intro_supersub}).

%We point out that in the general case it is not any more sufficient to consider the \emph{exposed} faces of $\d{\cdot}$, but one needs to consider all the \emph{extremal} faces of the norm. 

The suitable generalization of \eqref{0_xy'y''} and its implications can then be found in the concept of what we call \emph{regular transport set} $\mathcal R \theta_\psi$. The notation will be clear in Section \ref{S_foliations}
%As a side remark, we anticipate that the finite cost transportation problem w.r.t. a (convex) cone cost on a $\d{\cdot}$-Lipschitz graph (i.e., the graph of a $\d{\cdot}$-Lipschitz function) will turn out to be the simplest example of transportation problem on a \emph{$\mathtt c_{\C}$-Lipschitz foliation} --.
when we study the more general transport problem for $\mathtt c_{\C}$-Lipschitz foliations, namely a family of graphs of $|\cdot|_{D(\a)^*}$-Lipschitz functions depending on a continuous parameter $\a$ (Section \ref{S_foliations} and Proposition \ref{P_ex_fol}). The study of $\mathtt c_{\C}$-Lipschitz foliations will be one of the main issues to complete the construction of a suitable directed locally affine partition (Theorem \ref{T_final}) on which to solve \eqref{E_transpo_Sud}. The points in $\mathcal R\theta_\psi$ are the points $z$ such that
\begin{enumerate}
\item \label{Point_intro_loafpr1} the set of directions 
\[
\mathcal D^+ \theta_\psi(z) = \bigg\{ \frac{w-z}{|w-z|}:\, w \in \partial^+\Graph\, \psi(z) \setminus \{z\} \bigg\},
\]
of the optimal rays starting in $z$ is convex in $\mathbb S^{d-1}$, and the same for the set of directions
\[
\mathcal D^- \theta_\psi(z) = \bigg\{ \frac{z-w}{|z-w|}:\,  w\in \partial^- \Graph\,\psi(z) \setminus \{z\} \bigg\},
\]
of the optimal rays arriving in $z$,
\item \label{Point_intro_loafpr2} the two sets $\mathcal D^+ \theta_\psi(z)$, $\mathcal D^- \theta_\psi(z)$ coincide,
\item \label{Point_intro_loafpr3} there are points $w'$, $w''$ such that
\[
\frac{z-w'}{|z-w'|} \in \interr \mathcal D^- \theta_\psi(z), \qquad \frac{w''-z}{|w''-z|} \in \interr \mathcal D^+ \theta_\psi(z)
\]
and Points (\ref{Point_intro_loafpr1}-\ref{Point_intro_loafpr2}) hold for $w'$, $w''$ too.
\end{enumerate}
Then the sets $\tilde Z\ka$, $\tilde C\ka$ are now determined by
\[
z \in \tilde Z\ka \quad \Longrightarrow \quad 
\left\{ 
\begin{aligned} %\begin{array}{rcl}
\tilde Z\ka =~& \mathcal R \theta_\psi \cap \mathrm{aff}\partial^+\Graph\, \psi(z), \crcr
% && \crcr
\tilde C\ka =&~\epi\, \d{\cdot} \cap \Big( \mathrm{aff} \partial^+ \Graph\,\psi(z)-z\Big).
\end{aligned} %\end{array} 
\right.
\]
Such a directed locally affine partition will be called \emph{differential partition}.
One can see that the sets $\tilde Z\ka$ are relatively open in their affine hull, and that $\tilde C\ka$ are extremal faces of $\epi\d{\cdot}$.
%: the latter statement follows from the $\d{\cdot}$-Lipschitz regularity of $\psi$ and Points (\ref{Point_intro_loafpr1}-\ref{Point_intro_loafpr3}) above. 
Recall that the index $k$ denotes the affine dimension of $\tilde Z\ka$, which coincides with the linear dimension of $\tilde C\ka$, while $\a \in \A^k$ is an index of continuum cardinality.

%This property, which we explain in a moment, arises more naturally as a property of the \emph{super/subdifferential of $\Graph\,\psi$}, namely of the subsets of $\R^{d+1}\times\R^{d+1}$ defined by
%\begin{equation}
%\partial^+\Graph\,\psi=\Graph\,\psi\times\Graph\,\psi\cap\{\mathtt c_{\epi\d{\cdot}}<+\infty\},\quad\partial^-\Graph\,\psi=\bigl(\partial^+\Graph\,\psi\bigr)^{-1}.
%\end{equation}
%Notice that $\mathtt p_{\R^d}\bigl(\partial^\pm\Graph\,\psi\bigr)=\partial^\pm\psi$.

The second step in the strategy is then to show that the transport problem $\Pi^f_{\mathtt c_{\epi\d{\cdot}}}(\hat\mu,\hat\nu)$ can be decomposed, via disintegration of measures, into a family of finite cost transport problems on $\{\tilde Z\ka\times\R^{d+1}\}_{\nfrac{k=1,\dots,d}{\a \in \A^k}}$ with first marginals which are absolutely continuous w.r.t. the Hausdorff measure $\mathcal H^k$ on the $k$-dimensional set $\tilde Z\ka$ on which they are concentrated. 
Since the definition of ``good points'', i.e. of the regular transport set $\mathcal R \theta_\psi$, is definitely more complicated than in the strictly convex case, it is perfectly understandable that the proof of the absolute continuity w.r.t. to the Hausdorff measure on the $k$-dimensional sets $\{\tilde Z\ka\}_{k,\a}$ of the conditional probabilities of the disintegration of $\mathcal H^d$ (and then of $\hat\mu$) are considerably more intricate. The main reference for the approach used in this part is \cite{CarDan}, where the so-called \emph{cone approximation property} introduced in \cite{BiaGlo} (with the terminology used in \cite{Dan:Phd}) was first generalized to partitions into higher dimensional sets, showing the absolute continuity property for the conditional probabilities of the disintegration of the surface measure on the graph of a convex function w.r.t. the partition induced by the relative interior of the extremal faces. In particular, in Section \ref{S_disintechnique} it is shown that the differential partition satisfies both the \emph{forward} and the \emph{backward cone approximation property}, namely the cone approximation property holds both for the optimal rays starting at a point $z$ and for the points arriving at $z$, thus giving that the conditional measures of $\mathcal H^d$ are indeed equivalent to the $k$-dimensional Hausdorff measure on the set on which they are concentrated. 

As for the proof of the $\mathcal H^d$-negligibility of the set $\R^{d+1} \setminus \mathcal R\theta_\psi$, since for general convex norms the extremal faces may be more than the exposed ones, it is not possible to use the same reasoning as in the strictly convex case.
However, we will show that the set $\R^{d+1} \setminus \mathcal R\theta_\psi$ is made of \emph{initial/final points} for two other partitions (the \emph{super/subdifferential partitions} introduced in \ref{Ss_regu_resi_set}), which satisfy the (initial/final) \emph{forward/backward cone approximation property}. Hence, the same disintegration technique used in \cite{CarDan} permits to show that they are $\mathcal H^d\llcorner\Graph\,\psi$-negligible (see Theorem \ref{T_FC_no_initial}).

Denoting with $\{Z\ka,C\ka\}_{\nfrac{k=1,\dots,d}{\a \in \A^k}}$ the projection of the differential partition $\{\tilde Z\ka,\tilde C\ka\}_{\nfrac{k=1,\dots,d}{\a \in \A^k}}$ on $\R^d$, in Section \ref{S_theorem_1_proof} we deduce the following theorem. The statement includes also the points which do not belong to any optimal ray, and in that case the dimension $k$ of the elements of the directed locally affine partition they belong to, is $k=0$, as well as $C^0_\a = \{0\}$. 
Since we will often write the graph of a directed locally affine partition $\{Z\ka,C\ka\}_{k,\a}$ as
\[
\bD:=\Big\{(k,\a,z,C\ka):\,k\in\{0,\dots,d\},\,\a\in\A^k,\,z\in Z\ka\Big\},
\]
we will use also the notation
\begin{equation}
\label{E_cost_all_cone_intro}
\mathtt c_\mathbf D(x,y) :=
\begin{cases}
\ind_{C\ka}(y-x) & \text{if} \ \exists\, k,\a \ \text{s.t.} \ x \in Z\ka, \crcr
+\infty & \text{otherwise}.
\end{cases}
\end{equation}
Notice that for costs $\mathtt c$ of the form \eqref{E_cost_all_cone_intro}, one has clearly $\Pi^\mathrm{opt}_\mathtt c(\mu,\nu) = \Pi^f_\mathtt c(\mu,\nu)$, since the only values of $\mathtt c$ are $0$, $\infty$.
\begin{theorem}
\label{T_1}
Let $\mu,\nu \in \PP(\R^d)$ with $\mu \ll \LL$ and let $\d{\cdot}$ be a convex norm in $\R^d$. Then there exists a locally affine directed partition $\{Z^k_\a,C^k_\a\}_{\overset{k=0,\dots,d}{\a\in\A^k}}$ in $\R^d$ with the following properties:
\begin{enumerate}
% \begin{align}
% &(1) \quad \text{
\item \label{Point_1_T_1} for all $\a \in \A^k$ the cone $C^k_\a$ is a $k$-dimensional extremal face of $\d{\cdot}$;
% ;} \notag \\
% &(2) \quad 
\item \label{Point_2_T_1} $\displaystyle{\LL \biggl( \R^d \setminus \underset{k,\a}{\bigcup}\; Z^k_\a \biggr)=0}$;
% ; \notag \\
% &(3) \quad 
\item \label{Point_3_T_1} the disintegration of $\LL$ w.r.t. the partition $\{Z\ka\}_{k,\a}$, $\displaystyle{\LL \llcorner_{\underset{k,\a}{\cup} Z\ka} = \int v^k_\a\,d\eta(k,\a)}$, satisfies
\[
v^k_a \simeq \HH^k \llcorner_{Z^k_\a};
\]
% } \notag \\
% &(4) \quad 
\item \label{Point_4_T_1} for all $\pi \in \pod(\mu,\nu)$, the disintegration $\displaystyle{\pi = \int \pi\ka\, dm(k,\a)}$ w.r.t. the partition $\{Z\ka \times \R^d\}_{k,\a}$ satisfies
\[
\pi^k_\a \in \Pi^ f_{\mathtt c_{C^k_\a}} \big( \mu\ka, (\mathtt p_2)_\# \pi\ka \big),
\]
where $\displaystyle{\mu = \int \mu\ka\,dm(k,\a)}$ is the disintegration w.r.t. the partition $\{Z\ka\}_{k,\a}$, and moreover
\[
(\mathtt p_2)_\#\pi\ka \biggl( Z\ka \cup \biggl( \R^d \setminus \underset{(k',\a') \not= (k,\a)}{\bigcup} Z^{k'}_{\a'} \biggr) \biggr) = 1.
\]
% }. \notag %\label{E_finite_cone_cost}.
% \end{align}
\end{enumerate}

If also $\nu \ll \LL$, then for all $\pi\in\pod(\mu,\nu)$
\[
(\mathtt p_2)_\#\pi\ka = \nu\ka
\]
where $\displaystyle{\nu = \int \nu\ka \,dm(k,\a)}$ is the disintegration w.r.t. the partition $\{Z\ka\}_{k,\a}$, and the converse of Point \eqref{Point_4_T_1} holds:
\begin{equation*}
\pi\ka \in \Pi^ f_{\mathtt c_{C\ka}}(\mu\ka,\nu\ka) \quad \Longrightarrow \quad \pi \in \pod(\mu,\nu).
\end{equation*}
\end{theorem}
A locally affine directed partition satisfying Point \eqref{Point_3_T_1} is called \emph{Lebesgue-regular} (see Definition \ref{D_disint_regular}). This concludes the first part of the paper.

A remark is in order here: in Point \eqref{Point_4_T_1}, the conditional second marginals $(\mathtt p_2)_\# \pi\ka$ are independent on the potential $\psi$ but \emph{depend} on the particular transference plan $\pi$ which we are decomposing. This can be seen with elementary examples (see Example \ref{Ex_2ndmarg} in Section \ref{Ss_partitions_intro}). From now on the analysis will be done in a class of transference plans which have the same conditional second marginals: in fact, we will see in a moment that the partition of Theorem \ref{T_1} needs to be refined and by inspection one sees that such refinement changes when changing the conditional marginals.
We will consider then nonempty subsets of the optimal plans of the form
\[
\Pi^f_{\mathtt c_{\mathbf D}}(\mu,\{\bar\nu_\a\}):=\Big\{\pi\in\Pi^f_{\mathtt c_{\mathbf D}}(\mu,\nu):\, (\mathtt p_2)_\#\pi_\a=\bar\nu_\a\Big\},
\]
that is equivalent to fix a transport plan of finite $\mathtt c_{\mathbf D}$-cost $\check\pi$ and consider all transport plans $\pi\in\Pi^f_{\mathtt c_{\mathbf D}}(\mu,\{(\mathtt p_2)_\#\check \pi_\a\})$.

In the strictly convex case, Theorem \ref{T_1} has been proven in \cite{Car:strictly}. There the dimensions of the sets of the locally affine partition is equal to one, and it is classical and fairly easy to see that the optimal transportation problems
\[
 \Pi^{\mathrm{opt}}_{\mathtt c^1_{\a,2}}(\mu^1_\a,\nu^1_\a),\quad\mu^1_\a(Z^1_\a)=1,
\]
where 
\[
 \mathtt c^1_{\a,2}(x,y)= \begin{cases}
                          |y-x|^2 & \text{if $x\in Z^1_\a$, $y-x\in C^1_\a$}, \\
                          +\infty & \text{otherwise},                         
                          \end{cases}
% \left\{\begin{aligned}
%                           &|y-x|^2 && &\text{if $x\in Z^1_\a$, $y-x\in C^1_\a$}\\
%                           &+\infty && &\text{otherwise}
%                          \end{aligned}\right.
\]
have a solution induced by a map $\mathtt T^1_\a:Z^1_\a\to\R^d$. More precisely, one shows that any $\mathtt c^1_{\a,2}$-cyclically monotone transference plan is induced by a unique transport map $\mathtt T^1_\a$. Since the dependence of the maps $\mathtt T^1_\a$  on $\a$ is $m$-measurable, the map $\mathtt T(x):=\sum_{\a\in\A}\mathtt T^1_\a(x)\chi_{Z_\a^1}(x)$ is an optimal map for \eqref{E_transpo_Sud}.
%and allows to prove easily by classical results in $1$-dimensional transport theory (see \cite{conf:optcime}) the existence of an optimal map by choosing the solutions $\mathtt T\ka$ of the the secondary problems 
%\begin{equation}
% \min\Big\{\int|y-x|^2\,d\pi(x,y):\,\pi\in\Pi^f_{\mathtt c_{C\ka}}(\mu\ka,(\mathtt p_2)_\#\check \pi\ka)\Big\}=\min\Big\{\int \mathtt c^k_{\a,2}(x,y)\,d\pi(x,y):\,\pi\in\Pi(\mu\ka,(\mathtt p_2)_\#\check \pi\ka)\Big\},
%\end{equation}
%where $\mathtt c^k_{\a,2}$ is the secondary cost
%\begin{equation}
%\label{0_c2}
% \mathtt c^k_{\a,2}(x,y):=\left\{\begin{aligned}
% &|y-x|^2 && &\text{if $\mathtt c_{C\ka}(x,y)<+\infty$}\\
% &+\infty && &\text{otherwise}.
%\end{aligned}\right.
%\end{equation}
%and defining 
%\[
% \mathtt T(x):=\sum_{k,\a}\chi_{Z\ka}\mathtt T\ka(x).
%\]
Actually, $\mathtt T$ is the unique optimal transport map relative to the cost
\begin{equation}
\label{0_c2}
 \mathtt c_{2}(x,y):= \begin{cases}
                       |y-x|^2 & \text{if $\mathtt c_{\mathbf D}(x,y)<+\infty$}, \\
		      +\infty & \text{otherwise}.
                      \end{cases}
% \left\{\begin{aligned}
%  &|y-x|^2 && &\text{if $\mathtt c_{\mathbf D}(x,y)<+\infty$}\\
%  &+\infty && &\text{otherwise}.
% \end{aligned}\right.
\end{equation}

In the general convex case, the analogous way to solve \eqref{E_transpo_Sud} would be to prove that the optimal transportation problems on the sets of the partition of Theorem \ref{T_1}
\[
 \Pi^{\mathrm{opt}}_{\mathtt c^k_{\a,2}}(\mu^k_\a,\nu^k_\a),\quad\mu^k_\a(Z^k_\a)=1,
\]
where 
\[
 \mathtt c^k_{\a,2}(x,y)=\begin{cases} %\left\{\begin{aligned}
                          |y-x|^2 & \text{if $x\in Z^k_\a$, $y-x\in C^k_\a$},\\
                          +\infty & \text{otherwise},
                         \end{cases} %\end{aligned}\right.
                        \]
have a solution induced by a map $\mathtt T^k_\a:Z^k_\a\to\R^d$ whose graph is the support of any $\mathtt c^k_{\a,2}$-cyclically monotone transference plan, and then to glue together the maps $\mathtt T\ka$. 

This fact would be true, by classical results in optimal transportation, if there existed a pair of optimal potentials $\phi\ka$, $\psi\ka$ for the cost $\mathtt c_{C\ka}$. Recall that, for a cost $\mathtt c : \R^d \times \R^d \to [0,\infty]$, one calls optimal potentials a pair of functions $\phi$, $\psi$ s.t.
\begin{align*}
&\phi,\,\psi:\R^d\to[-\infty,+\infty), \quad \text{ $\phi$ $\mu$-measurable and $\psi$ $\nu$-measurable}, \\
&\phi(x)+\psi(y)\leq \mathtt c(x,y), \quad \forall\, x,y \in \R^d, \\
&\phi(x)+\psi(y)= \mathtt c(x,y), \quad \text{ $\pi$-a.e. for some $\pi\in \Pi(\mu,\nu)$}.
\end{align*}
Recall also that, if $\Gamma\subset \R^d\times \R^d$ is a carriage for $\pi$ and $(x_0,y_0) \in \Gamma$, then
\begin{align}
\label{E_general_phi_intro}
\phi(x) &:= \inf \bigg\{ \sum_{i=0}^I \mathtt c(x_{i+1},y_i) - \mathtt c(x_i,y_i):\, I \in \N,\, (x_i,y_i) \in \Gamma,\, x_{I+1} = x \bigg\},\\
\psi(x)&:=\mathtt c(x,y)-\phi(x)\
\end{align}
yield a pair $\phi$, $\psi$ of optimal potentials provided $\phi$ is $\mu$-a.e. finite. When $\mathtt c$ is a convex norm, then $\psi=-\phi$ is a Kantorovich potential. 

Indeed, by formula \eqref{E_general_phi_intro}, if $\exists\,\phi\ka,\, \psi\ka$ optimal potentials w.r.t. $\mathtt c_{C\ka}$ then there exist also $\phi^k_{\a,2}$, $\psi^k_{\a,2}$ optimal potentials for $\mathtt c^k_{\a,2}$ and it is then classical to show that any $\mathtt c^k_{\a,2}$-cyclically monotone transference plan is unique and induced by an optimal map $T\ka$. 

However, as shown in \cite{Car1}, in general the transport problem in $\Pi^f_{\mathtt c_{C\ka}}(\mu\ka,\{(\mathtt p_2)_\# \pi\ka\})$ on $Z\ka$ with cost $\mathtt c_{C\ka}$ does not have a potential $\phi\ka$ (see the final example of \cite{Car1}), thus the directed locally affine partition of Theorem \ref{T_1} is not refined enough to give immediately the existence of transport maps in each of the sets $Z\ka \times \R^d$. Another approach that has been used at this point to show the existence of an optimal map assuming the existence of a directed locally affine partition is the one adopted in \cite{JimSan:quadratic}, which though uses techniques similar to \cite{champdepasc:Monge}, and then is not really simplifying the problem in the spirit of Sudakov's strategy. 

What we show in the second part of the paper, more precisely in Section \ref{S_cfibr_cfol}, is that the directed locally affine partition of Theorem \ref{T_1} can be refined into another directed locally affine partition $\{\check Z\lb, \check C\lb\}_{\nfrac{\ell = 0,\dots,d}{\mathfrak b \in \mathfrak B^\ell}}$ such that, given a carriage of any $\mathtt c^k_{\a,2}$-cyclically monotone transference plan, a pair of optimal potentials $\check \phi\lb$, $\check \psi\lb$ can be constructed on each of its elements $\check Z\lb$.

In order to explain what we mean by a ``refinement'' of the partition of Theorem \ref{T_1}, referring to Section \ref{Ss_linear_pre_unique_transf} for wider motivations and more precise statements, let us consider formula \eqref{E_general_phi_intro}. The sequence of points
\[
(x_0,y_0), (x_1,y_0), (x_1,y_1), (x_2,y_1), \dots, (x_i,y_i), ( x_{i+1},y_i), (x_{i+1},y_{i+1}), \dots, (x,y_I), \qquad (x_i,y_i) \in \Gamma,
\]
is an \emph{axial path}, and we say that the axial path is a \emph{$(\Gamma,\mathtt c)$-axial path} if $\mathtt c(x,y) < \infty$ for all couples $(x,y)$ in the axial path: since we can assume that $\Gamma \subset \{\mathtt c < \infty\}$, this condition is equivalent to $\mathtt c(x_{i+1},y_i), \mathtt c(x,y_I) < \infty$. It is a well know fact that if $\mu$-a.a. points belong to an axial path starting from and ending in $(x_0,y_0)$ (which will be called a \emph{$(\Gamma, \mathtt c)$-cycle}), then formula \eqref{E_general_phi_intro} yields a $\mu$-a.e. finite potential $\phi$. Its dual $\psi$ turns out then to be finite and independent on $x$ for $\nu$-a.e. $y \in \R^d$.

It becomes then natural to ask for a directed locally affine partition $\{Z\ka,C\ka\}_{k,\a}$ that, in addition to \eqref{Point_1_T_1}, \eqref{Point_2_T_1}, \eqref{Point_3_T_1} and \eqref{Point_4_T_1} of Theorem \ref{T_1} for all $\pi\in\Pi^f_{\mathtt c_{\mathbf D}}(\mu,\{\bar\nu_\a\})$, it satisfies the following property. For all carriages $\Gamma\subset\{\mathtt c_{\mathbf D}<+\infty\}$ s.t. $\pi(\Gamma)=1$ for some $\pi\in\Pi^f_{\mathtt c_{\mathbf D}}(\mu,\{\bar\nu_\a\})$, the sets $Z\ka$ are contained in a $(\Gamma,\mathtt c_{C\ka})$-cycle up to a $\mu\ka$-negligible set (eventually depending on $\Gamma$): the cost in each $Z\ka$ is the \emph{cone cost} given by $\mathtt c_{C\ka}(x,y) = \ind_{C\ka}(y-x)$. 

This cyclical connectedness condition is called in this paper \emph{$\Pi^f_{\mathtt c_{\mathbf D}}(\mu,\{\bar\nu_\a\})$-cyclical connectedness} (see Definition \ref{D_pimunuconn}) and, as discussed above, when verified it guarantees the existence of optimal potentials.
%then we will prove that one can construct a transport map $\mathtt T\ka$ by standard methods as is \cite{ambprat:crist} and then the %optimal $\mathtt T$ is obtained by piecing together the maps $\mathtt T\ka$. 
%Indeed, any optimal transport plan w.r.t. the secondary cost $\mathtt c_2$ defined in \eqref{0_c2} belongs to $\Pi^f_{\mathtt c_{\mathbf D}}(\mu,\{\bar\nu_\a\})$. Then, considering any of its $\mathtt c_2$-cyclically monotone carriages $\Gamma$, which are by definition of $\mathtt c_2$ contained in $\{\mathtt c_{\mathbf D}<+\infty\}$, the fact that the sets $Z\ka$ are contained in a $(\Gamma,\mathtt c_{\mathbf D})$-cycle up to a $\mu\ka$-negligible set permits to define, via formula \eqref{E_general_phi_intro} with $\mathtt c=\mathtt c_2$, a $\mu\ka$-a.e. finite potential $\phi^k_\a$ on $Z\ka$, for all $k,\,\a$, whose unique monotone transport plan gives the desired map $T\ka$ (see \ref{P_second_cost}).

The second main result of this paper claims the existence of such a partition. The fact that it is a refinement of an already existing locally affine partition, such as the one of Theorem \ref{T_1}, namely that each of its sets is contained in some $Z\ka$
and the corresponding cone of directions is an extremal face of the cone $C\ka$, is expressed by saying that it is a \emph{subpartition} of $\{Z\ka, \,C\ka\}$ (see Definition \ref{D_dir_subpart}). Recall Definition \ref{D_disint_regular} of Lebesgue-regular partition.

\begin{theorem}
\label{T_subpart_final}
Let $\{Z\ka, C\ka\}_{\nfrac{k = 0,\dots,d}{\mathfrak a \in \mathfrak A^k}}$ be a \emph{Lebesgue-regular} directed locally affine partition in $\R^d$ and let $\mu\ll\LL$, $\nu\in\mathcal P(\R^d)$ such that $\Pi^f_{\mathtt c_{\mathbf D}}(\mu,\nu)\neq\emptyset$.

Then, for all $\check \pi\in\Pi^f_{\mathtt c_{\mathbf D}}(\mu,\nu)$ there exists a directed locally affine subpartition $\{\check Z\lb, \check C\lb\}_{\nfrac{\ell = 0,\dots,d}{\mathfrak b \in \mathfrak B^\ell}}$ of $\{Z\ka, C\ka\}_{\nfrac{k = 0,\dots,d}{\mathfrak a \in \mathfrak A^k}}$, up to a $\mu$-negligible set $N'_{\check \pi}$, such that
\begin{equation*}
\{\check Z\lb, \check C\lb\}_{\ell,\b}\quad\text{is Lebesgue-regular},
\end{equation*}
and if $\bar\nu\lb:=(\mathtt p_2)_\#\check \pi\lb$, where $\check \pi\lb$ is the conditional probability on the partition $\{\check Z^{',\ell}_\b \times \R^d\}_{\ell,\b}$, then each set $\check Z\lb$ is $\Pi^f_{\mathtt c_{\mathbf D}}(\mu,\{\bar\nu\lb\})$-cyclically connected, for all $\ell$, $\b$.
\end{theorem}

Applying Theorem \ref{T_subpart_final} to the directed locally affine partition given by Theorem \ref{T_1}, one obtains immediately the following result. As in the case of Theorem \ref{T_1}, the second part of Point \eqref{Point_4_T_final} of the next theorem is a consequence of the precise analysis of the regions where the mass transport occurs.

\begin{theorem}
\label{T_final}
Let $\mu, \nu\in\PP(\R^d)$ with $\mu\ll\LL$ and let $\d{\cdot}$ be a convex norm in $\R^d$. Then, for all $\check \pi\in\Pi^{\mathrm{opt}}_{\d{\cdot}}(\mu,\nu)$ there exists a locally affine directed partition $\{\check Z^k_\a,\check C^k_\a\}_{\overset{k=0,\dots,d}{\a\in\A^k}}$ in $\R^d$ with the following properties:
\begin{enumerate}
\item for all $\a \in \A^k$ the cone $\check C^k_\a$ is a $k$-dimensional extremal face of $\d{\cdot}$; %, $\forall\,\a\in \A^k$;}\notag\\
% &(2)\quad 
\item $\displaystyle{\mu \biggl(\R^d\setminus\underset{k,\a}{\bigcup}\; \check Z^k_\a \biggr) = 0}$; %\notag\\
% &(3) \quad
\item the partition is \emph{Lebesgue-regular};
%the disintegration of $\LL$ w.r.t. the partition $\{\check Z\ka\}_{k,\a}$, $\displaystyle{\LL \llcorner_{\underset{k,\a}{\cup} \check Z\ka} = \int \check v^k_\a\,d\check \eta(k,\a)}$, satisfies
%\[
%\check v^k_a \simeq \HH^k \llcorner_{\check Z^k_\a};
%\]
\item \label{Point_4_T_final} the disintegration $\displaystyle{\check \pi = \int \check \pi\ka\, dm(k,\a)}$ w.r.t. the partition $\{\check Z\ka \times \R^d\}_{k,\a}$ satisfies
\[
\check \pi^k_\a \in \Pi^f_{\mathtt c_{C^k_\a}} \big( \check \mu\ka, (\mathtt p_2)_\# \check \pi\ka \big),
\]
where $\displaystyle{\mu = \int \check \mu\ka\,dm(k,\a)}$ is the disintegration w.r.t. the partition $\{\check Z\ka\}_{k,\a}$, and moreover
\[
(\mathtt p_2)_\#\check \pi\ka \biggl( \check Z\ka \cup \biggl( \R^d \setminus \underset{(k',\a') \not= (k,\a)}{\bigcup} \check Z^{k'}_{\a'} \biggr) \biggr) = 1;
\]
\item the partition 
% &(5)\quad 
$\{\check Z\ka\}_{k,\a}$ is $\Pi^f_{\mathtt c_{\bar{\mathbf D}}}(\mu,\{(\mathtt p_2)_\#\check \pi\ka\})$-cyclically connected. %\notag
% \end{align}
\end{enumerate}
\end{theorem}

\begin{remark}
\label{R_remark_str_conv}
We note that the elements of the locally affine partition $\{\check Z^k_\a,\check C^k_\a\}_{\nfrac{k=1,\dots,d}{\a \in \A^k}}$ given by the above theorem have maximal linear dimension
\[
\max \big\{ k : \check Z^k_\a \not= \emptyset \big\} \leq \max \big\{ \dim C : C \text{ extremal face of } \epi \d{\cdot} \big\}.
\]
In particular, if $D$ is strictly convex, the locally affine decomposition is made only of directed rays, and one recovers the results of \cite{Car:strictly} for strictly convex norms.
\end{remark}

In the case $\nu \ll \LL$, the decomposition does not depends on the transference plan, as in the strictly convex case. In particular, we can say that it is universal, i.e. it is independent on the particular transference plan $\pi \in \Pi^\mathrm{opt}_{\d{\cdot}}(\mu,\nu)$  used.

\begin{theorem}
\label{T_final_nu}
Assume that $\nu\ll\LL$. Then the directed locally affine partition of Theorem \ref{T_final} satisfies the following properties:
\begin{enumerate}
\item for all $\a \in \A^k$ the cone $\check C^k_\a$ is a $k$-dimensional extremal face of $\d{\cdot}$;
% \begin{align}
% &(1)\quad \text{$\check C^k_\a$ is a $k$-dimensional extremal face of $\d{\cdot}$, $\forall\,\a\in \A^k$;}\notag\\
\item[(2')] $\displaystyle{\mu \biggl(\R^d\setminus\underset{k,\a}{\bigcup}\; \check Z^k_\a \biggr) = \nu \biggl(\R^d\setminus\underset{k,\a}{\bigcup}\; \check Z^k_\a \biggr) = 0}$; 
\item the partition is \emph{Lebesgue-regular};
%the disintegration of $\LL$ w.r.t. the partition $\{\check Z\ka\}_{k,\a}$, $\displaystyle{\LL \llcorner_{\underset{k,\a}{\cup} \check Z\ka} = \int \check v^k_\a\,d\check \eta(k,\a)}$, satisfies
%\[
%\check v^k_a \simeq \HH^k \llcorner_{\check Z^k_\a};
%\]
% &(3) \quad\LL\llcorner\underset{k,\a}{\cup}\check Z^k_\a =\int v^k_\a\,d\eta(k,\a), \text{ with $v^k_a\simeq\HH^k\llcorner \check Z^k_\a$;}\notag\\
\item[(4')] for all $\check \pi \in \pod(\mu,\nu)$, the disintegration $\displaystyle{\check \pi = \int \check \pi\ka\, dm(k,\a)}$ w.r.t. the partition $\{\check Z\ka \times \R^d\}_{k,\a}$ satisfies
\[
\check \pi^k_\a \in \Pi^ f_{\mathtt c_{C^k_\a}}(\check \mu\ka, \check \nu\ka),
\]
where $\displaystyle{\mu = \int \check \mu\ka\,dm(k,\a)}$, $\displaystyle{\nu = \int \check \nu\ka\,dm(k,\a)}$ are the disintegration of $\mu$, $\nu$ w.r.t. the partition $\{\check Z\ka\}_{k,\a}$; % and moreover

\item[(5')] $\{\check Z\ka\}_{k,\a}$ is $\Pi^ f_{\mathtt c_{\bar{\mathbf D}}}(\mu,\nu)$-cyclically connected. %}\notag
% \end{align}
\end{enumerate}
\end{theorem}

In particular $\displaystyle{\check \pi \biggl( \underset{k,\a}{\bigcup} \check Z\ka \times \check Z\ka \bigg)=1}$.

The main step in the proof of Theorem \ref{T_subpart_final} is the following 

\begin{theorem}
\label{T_subpart_step}
Let $\{Z\ka, C\ka\}_{\nfrac{k = 0,\dots,d}{\mathfrak a \in \mathfrak A^k}}$ be a Lebesgue-regular directed locally affine partition in $\R^d$ and let $\mu$, $\nu$ be probability measures in $\mathcal P(\R^d)$ such that $\mu\ll\LL$ and $\Pi^f_{\mathtt c_{\mathbf D}}(\mu,\nu)\neq\emptyset$.
Then, for all fixed $\check \pi \in \Pi^ f_{\mathtt c_{\mathbf D}}(\mu,\nu)$, there exists a directed locally affine subpartition $\{\check Z^\ell_{\b}, \check C^\ell_{\b}\}_{\nfrac{\ell = 0,\dots,d}{\mathfrak b \in \mathfrak B^\ell}}$ of $\{Z\ka, C\ka\}_{\nfrac{k = 0,\dots,d}{\mathfrak a \in \mathfrak A^k}}$, up to a $\mu$-negligible set $N_{\check \pi}$, such that %with the following properties
\begin{equation*}
\big\{ \check Z^\ell_\b,\check C^\ell_\b \big\}_{\ell,\b} \quad \text{is Lebesgue-regular},
\end{equation*}
and setting $\check \nu^\ell_\b := (\mathtt p_2)_\#\check \pi^\ell_\b$, where $\check \pi^\ell_\b$ is the conditional probability on the partition $\{\check Z^{\ell}_\b \times \R^d\}_{\ell,\b}$, then the sets
\begin{equation}
\label{E_subp_ell_cycl_intro}
\Bigl\{ \check Z^\ell_\b:\,\check Z^\ell_\b\subset Z^\ell_\a\text{ for some $\a\in\A^\ell$}, \ell=1,\dots,d \Bigr\} 
\end{equation}
form a $\Pi^ f_{\mathtt c_{\mathbf D}}(\mu,\{\check \nu^\ell_\b\})$-cyclically connected partition.
\end{theorem}

Theorem \ref{T_subpart_step} allows to construct a locally directed affine subpartition $\{\check Z^\ell_{\b}, \check C^\ell_{\b}\}_{\nfrac{\ell = 0,\dots,d}{\mathfrak b \in \mathfrak B^\ell}}$ to a directed locally affine partition $\{Z\ka, C\ka\}_{\nfrac{k = 0,\dots,d}{\mathfrak a \in \mathfrak A^k}}$ such that the sets which do not lower their affine dimensions (i.e. for which $\check Z^\ell_\b \subset Z\ka$ and $\ell = k$) are $\Pi^f_{\mathtt c_{\mathbf D}}(\mu,\{\check \nu^\ell_\b\})$-cyclically connected.

Since the subpartition
\[
\big\{\check Z^\ell_{\b}, \check C^\ell_{\b} \big\}_{\nfrac{\ell = 0,\dots,d-1}{\mathfrak b \in \bar{\mathfrak B}^\ell}} \quad \text{such that if} \ \check Z^\ell_\b \subset Z^k_\a \ \text{then} \ \ell < k \ \text{(equivalently neglecting the sets of \eqref{E_subp_ell_cycl_intro})}
\]
is a Lebesgue-regular directed locally affine partition, and as a subpartition of $\{Z\ka,C\ka\}_{k,\a}$ the index $\ell$ is decreasing of at least $1$ in each $Z\ka$, by a finite iterative argument one immediately obtains Theorem \ref{T_subpart_final}.

The proof of Theorem \ref{T_subpart_step} relies on nonstandard tools in measure theory --namely, the sufficient condition for uniqueness/optimality of transference plans based on the existence of suitable Borel linear preorders given in \cite{BiaCar}-- and on the existence of Lebesgue-regular directed locally affine partitions for one parameter families of graphs of Lipschitz functions w.r.t. convex norms (called \emph{$\mathtt c_{\C}$-Lipschitz foliations}), whose construction generalizes the one of the differential partition of Theorem \ref{T_1}.

We give now a brief scheme of the main steps of the proof of Theorem \ref{T_subpart_step}. Then, we will go on stating its consequences, ending with the solution of Monge's problem in Theorem \ref{T_Monge_final}.

First of all, one can reduce to study the finite cost transportation problem on directed locally affine partitions with fixed dimension $k$ and whose cones of directions are close to a given one, called \emph{$k$-directed sheaf sets} (see Section \ref{Ss_mapping_sheaf_to_fibration}, Definition \ref{D_sheaf_set}). Moreover, by a change of variables which preserves the characteristics of the optimal transportation problem $\Pi^f_{\mathtt c}(\mu,\{\bar\nu\ka\})$, one can assume that the sets $Z\ka$ of the sheaf set are contained in distinct parallel planes, thus studying the so called \emph{$k$-directed fibrations} with cones of directions $\tilde{\mathbf C}\subset\A\times\mathcal C(k,\R^k)$ (see Definition \ref{D_fibration}). $\tilde{\mathbf C}(\a)$ is the cone of directions of the region $Z\ka$. 

To give an idea of how the subpartition is constructed, in this introduction we assume that $\A^k=\{\a\}$, so that the finite cost transportation problem on such $\tilde{\mathbf C}$-directed fibration is a finite transportation problem for a single $k$-dimensional cone cost in $\R^k$. In the paper, the variable $\a\in\A$ plays the role of a parameter and is kept in all the constructions and definitions in order to show that the sufficient measurability conditions w.r.t. $\a$, which are needed in order to define global objects, are satisfied.

By the discussion made before Theorem \ref{T_subpart_final}, it is natural to fix a carriage $\Gamma\subset\bigl\{ \mathtt c_{\mathbf C\ka}<+\infty\bigr\}$ of some $\pi\ka\in\Pi^f_{\mathtt c_{C\ka}}(\mu\ka,\nu\ka)$ and to see whether the partition of $\R^k$ into $(\Gamma,\mathtt c_{C\ka})$-cycles satisfies our requirements. It turns out that in general this is not true, the first main reason being that not all the other transport plans are necessarily concentrated on its sets. 

However --as proven for general cost functions $\mathtt c$ in \cite{BiaCar} in order to give very general sufficient conditions for uniqueness/optimality of transport plans-- a partition on which all the transport plans $\pi \in \Pi^\mathrm{opt}_{\mathtt c_{C\ka}}(\mu\ka,\nu\ka)$ are concentrated exists provided one can find a Borel linear preorder (i.e., a transitive relation such that every two points can be compared) which contains the set $\bigl\{\mathtt c_{C\ka}<+\infty\bigr\}$ (i.e, it is \emph{$\mathtt c_{C\ka}$-compatible} according to Definition \ref{D_compatible}) and extends the linear preorder $\preccurlyeq_{(\Gamma, \mathtt c_{C\ka})}$ whose equivalence classes are the $(\Gamma, \mathtt c_{C\ka})$-cycles (i.e., $x \preccurlyeq_{(\Gamma, \mathtt c_{C\ka})} y$ if there is a $(\Gamma,\mathtt c_{C\ka})$-axial path of finite cost connecting $y$ to $x$). 

In Section \ref{Ss_gamma_order}, Theorem \ref{T_order_gamma}, we show how to construct such a preorder. The preorder will be denoted in the following by $\preccurlyeq_{\Gamma,\mathtt W^\Gamma}$, where $\mathtt W^\Gamma$ is related to the countable procedure to construct the Borel preorder (see Section \ref{S_cfibr_cfol}). Its equivalence classes turn out to be either families of graphs of $|\cdot|_{\mathbf D(\a)^*}$-Lipschitz functions $\{\mathtt h_\t^\pm(\a)\}_{\t\in\T}$ --being $\epi\,|\cdot|_{\mathbf D(\a)^*}=C\ka$-- or $k$-dimensional equivalence classes. Such families of sets are called $\mathtt c_{\tilde{\mathbf C}}$-Lipschitz foliations and are studied in Section \ref{S_foliations}. In particular, the finite cost transportation problem w.r.t. $\mathtt c_{C\ka}$ in $\R^k$ reduces to a family of finite cost transportation problems w.r.t. $\mathtt c_{C\ka}$ on the sets of this $\mathtt c_{\tilde{\mathbf C}(\a)}$-Lipschitz foliation.

The equivalence classes $\{\check Z^k_\b\}_{\b}$ which do not lower the affine dimension of the $Z\ka$ are connected by $(\Gamma, \mathtt c_{C\ka})$-cycles, up to a $\mu$-negligible set, and then in principle they are candidate to be the $k$-dimensional sets of Theorem \ref{T_subpart_step}. However, they are not necessarily connected by $(\Gamma', \mathtt c_{C\ka})$-cycles for the other carriages $\Gamma'$ of plans in $\Pi^f_{\mathtt c_{C\ka}}(\mu\ka,\nu\ka)$. In fact, changing $\Gamma$, the Borel preorder $\preccurlyeq_{\Gamma,\mathtt W^\Gamma}$ also varies. Hence we need to use an abstract result on measure theory \cite{BiaCar} (recalled here in Appendix \ref{A_minimal_equivalence}), assuring that there is a minimal Borel linear preorder among the ones of the type $\preccurlyeq_{\Gamma,\mathtt W^\Gamma}$: for this one, the sets which do not lower the dimension of the $Z\ka$ and are of positive $\mu\ka$-measure are $\Pi^f_{\mathtt c_{C\ka}}(\mu\ka,\nu\ka)$-cyclically connected (see Theorem \ref{T_cfibrcfol}). Notice that this $\Pi^f_{\mathtt c_{C\ka}}(\mu\ka,\nu\ka)$-cyclically connectedness property can be indeed interpreted as a minimality or ``indecomposability'' property of the new $k$-dimensional sets.

As for the finite cost transportation problem for the cost $\mathtt c_{C\ka}$ on the classes of the minimal equivalence relation which are graphs of $|\cdot|_{\mathbf D(\a)^*}$-Lipschitz functions $\{\mathtt h_\t(\a)\}_{\t\in\T}$, one uses the same tools as in the proof of Theorem \ref{T_1} to show the existence of a \emph{differential} locally affine partition $\{\check Z^\ell_\b,\check C^\ell_\b\}_{\ell<k,\b}$ on which the transportation problem decomposes (see Section \ref{S_foliations}). Indeed, the sets of such partition which are contained in $\Graph\,\mathtt h_\t(\a)$ are constructed as the sets of the differential partition for the Kantorovich potential $\psi$: the only difference now is that one has to take care of the measurability of these sets w.r.t. the parameters $\a,\t$.

The only missing point in the proof that the union of the differential partition and of the $k$-dimensional equivalence classes $\{\check Z^\ell_\b,\check C^\ell_\b\}_{\ell,\b}$ satisfies the conclusions of Theorem \ref{T_subpart_step} is then the Lebesgue-regularity.
Notice that now the directed locally affine partition is obtained applying the same reasoning as in Theorem \ref{T_1} but for a family of norm-Lipschitz graphs depending on a continuous parameter. Hence one would be tempted to deduce the Lebesgue-regularity property first disintegrating the Lebesgue measure $\mathcal L^k$ on such graphs and then using the cone approximation property (as for the Kantorovich potential) for each conditional measure of $\mathcal L^k$ on a single graph. However, as we will show in a counterexample (see Remark \ref{R_not_gener_potential}), this is not possible because in general the conditional measures of $\mathcal L^k$ on a family of Lipschitz graphs might to be absolutely continuous w.r.t. the $(k-1)$-dimensional Hausdorff measure on the graphs on which they are concentrated.

In fact, the Lebesgue-regularity property for the sets $\{\check Z^\ell_\b,\check C^\ell_\b\}_{\ell,\b}$ follows by the fact that the Lipschitz graphs of $\{\mathtt h_\t(\a)\}$ are the equivalence classes of a $\mathtt c_{C\ka}$-compatible Borel linear preorder on which all the transport plans in $\Pi^f_{\mathtt c_{C\ka}}(\mu\ka,\nu\ka)$ are concentrated. Indeed, by the uniqueness theorem stated in \cite{BiaCar} one can prove the cone approximation property for the subpartition. The procedure is similar to the procedure followed in the case a single potential $\psi$ is present: however, the convergence of the cone approximating vector fields is now due to the uniqueness of a suitable transference plan (see Section \ref{S_cone_approx_folia}).

As discussed before, Theorem \ref{T_final} gives as an application the possibility to construct optimal potentials w.r.t. secondary cost functions such as $\mathtt c_2$ \eqref{0_c2} on each set of the partition $\check Z\ka$. In the case in which the secondary cost function is obtained by minimizing the original transport problem w.r.t. another convex norm $|\cdot|_{(D')^*}$, one obtains a refinement of the directed locally affine partition of Theorem \ref{T_final} with cones of directions given by intersections of extremal faces of $|\cdot|_{(D')^*}$ and $\d{\cdot}$. 

More precisely, let $|\cdot|_{(D')^*}$ be a convex norm with unit ball $D'$, and consider the secondary minimization problem
\begin{equation}
\label{E_seconry_min_prob}
\min \bigg\{ \int |y - x|_{(D')^*}\, d\pi(x,y):\, \pi \in \Pi^{\mathrm{opt}}_{|\cdot|_{D^*}}(\mu,\nu) \bigg\}.
\end{equation}
If $\check \pi$ is a minimizer of the above problem, by the fact that $\check \pi$ is also a minimizer of
\[
\int \mathtt c_{D,D'}(x,y)\,d \pi(x,y), \qquad \mathtt c_{D,D'}(x,y) :=
\begin{cases}
|y-x|_{(D')^*} & x \in \check Z^k_\a, y-x \in \check C^k_\a, \crcr
+\infty & \text{otherwise},
\end{cases}
\]
and that each $\check Z\ka$ is $\Pi^f_{\check C\ka}(\mu\ka,\nu\ka)$-cyclically connected, Proposition \ref{P_second_cost} yields that in each $\check Z^k_\a$ there exists a potential pair $\phi\ka$, $\psi\ka$, and since $\mathtt c_{D,D'}$ satisfies the triangle inequality we can take $\phi\ka = - \psi\ka$. By the existence of such a potential, restricting then to a single set $\check Z^k_\a$ one can prove as in the proof of Theorem \ref{T_1} the existence of a directed locally affine partition $\{\check Z^{k,\ell}_{\a,\b},\check C^{k,\ell}_{\a,\b}\}_{\nfrac{\ell = 0,\dots,k}{\b \in \B^{k,\ell}_\a}}$. In the resulting statement one has to replace the ambient space $\R^d$ with $\check Z\ka$, the measure $\LL$ with $\mathcal H^k\llcorner_{\check Z\ka}$, the marginals $\mu$, $\nu$ with $\mu\ka$, $\nu\ka$ and the cost $\d{\cdot}$ with
\begin{equation}
\label{E_cka_str_intro}
\mathtt c^k_\a := \mathtt c_{D,D'} \llcorner_{\check Z^k_\a \times \R^d} =
\begin{cases}
|y-x|_{(D')^*} & y-x \in \check C^k_\a, \crcr
+\infty & \text{otherwise}.
\end{cases}
\end{equation}

More precisely, we obtain the following theorem.

\begin{theorem}
\label{T_final_seconry}
Let $\mu, \nu\in\PP(\R^d)$ with $\mu\ll\LL$ and let $\check \pi$ be an optimal transport plan for the problem \eqref{E_seconry_min_prob}. Then there exists a locally affine directed partition $\{\check Z^{k,\ell}_{\a,\b},\check C^{k,\ell}_{\a,\b}\}_{\nfrac{k=0,\dots,d,\a\in\A^k}{\ell = 0,\dots,k,\b \in \B^{k,\ell}_\a}}$ in $\R^d$ with the following properties:
\begin{enumerate}
% \begin{align*}
\item for all $k$, $\a \in \A^k$, the cone $\check C^{k,\ell}_{\a,\b}$ is an $\ell$-dimensional extremal face of the cost $\mathtt c\ka$ given by \eqref{E_cka_str_intro}, i.e. the intersection of a $k$-dimensional face of $\d{\cdot}$ with an extremal face of $|\cdot|_{(D')^*}$; %}\notag\\
% &(2)\quad 
\item $\displaystyle{\mu \biggl( \R^d \setminus \underset{k,\a,\ell,\b}{\bigcup} \check Z^{k,\ell}_{\a,\b} \biggr)=0}$; %\notag\\
% &(3) \quad
\item the partition is \emph{Lebesgue-regular};
%the disintegration $\displaystyle{\LL\llcorner_{\underset{k,\a,\ell,\b}{\cup} \check Z^{k,\ell}_{\a,\b}} = \int \check \upsilon^{k,\ell}_{\a,\b}\,d\eta(k,\a,\ell,\b)}$ w.r.t. the partition $\{\check Z^{k,\ell}_{\a,\b}\}_{k,\a,\ell,\b}$ satisfies
%\[
%\upsilon^{k,\ell}_{\a,\b} \simeq \HH^\ell \llcorner_{\check Z^{k,\ell}_{\a,\b}};
%\]
\item the disintegration $\displaystyle{\check \pi = \int \check \pi^{k,\ell}_{\a,\b}\, dm(k,\a,\ell,\b)}$ w.r.t. the partition $\{\check Z^{k,\ell}_{\a,\b} \times \R^d\}_{k,\a,\ell,\b}$ satisfies
\[
\check \pi^{k,\ell}_{\a,\b} \in \Pi^f_{\mathtt c_{C^{k,\ell}_{\a,\b}}} \big( \check \mu^{k,\ell}_{\a,\b}, \big\{ (\mathtt p_2)_\# \check \pi^{k,\ell}_{\a,\b} \big\} \big),
\]
where $\displaystyle{\mu = \int \check \mu^{k,\ell}_{\a,\b}\,dm(k,\a,\ell,\b)}$ is the disintegration w.r.t. the partition $\{\check Z^{k,\ell}_{\a,\b}\}_{k,\a,\ell,\b}$, and moreover
\[
(\mathtt p_2)_\#\check \pi^{k,\ell}_{\a,\b} \biggl( \check Z^{k,\ell}_{\a,\b} \cup \biggl( \R^d \setminus \underset{(k',\a',\ell',\b') \not=(k,\a,\ell,\b)}{\bigcup} \check Z^{k',\ell'}_{\a',\b'} \biggr) \biggr) = 1;
\]
\item the partition 
% &(5)\quad 
$\{\check Z^{k,\ell}_{\a,\b}\}_{k,\a,\ell,\b}$ is $\Pi^f_{\mathtt c_{\bar{\mathbf D}}}(\mu,\{(\mathtt p_2)_\#\check \pi^{k,\ell}_{\a,\b}\})$-cyclically connected. %\notag

\end{enumerate}
\end{theorem}
A completely similar extension can be given to Theorem \ref{T_final_nu}.

A particular case is when each extremal face of $|\cdot|_{(D')^*}$ is contained in an extremal face of $|\cdot|_{D^*}$: in this case condition (1) becomes
\begin{enumerate}[{\it (1')}]
\itemindent .55cm
\item \label{Point_1_prime} {\it for all $k$, $\a \in \A^k$, the cone $\check C^{k,\ell}_{\a,\b}$ is an $\ell$-dimensional extremal face of $|\cdot|_{(D')^*}$.}
\end{enumerate}
The only difference w.r.t. Theorem \ref{T_final} is that now $\check \pi$ is a minimum for the secondary minimization problem \eqref{E_seconry_min_prob}, not a transference plan in $\Pi^{\mathrm{opt}}_{|\cdot|_{(D')^*}}(\mu,\nu)$.

The case (1') above happens if for example $|\cdot|_{(D')^*}$ is strictly convex, so that the $\check Z^{k,\ell}_{\a,\b}$ are now directed segments, i.e. $\ell = 1$. By the standard analysis on transportation problems in 1-d, and the measurable dependence on $k$, $\a$, $\b$, the existence of an optimal transport map $\mathtt T$ for the Monge problem \eqref{E_transpo_Sud} follows as a simple corollary. In particular, the restriction $\mathtt T \llcorner_{\check Z^{k,1}_{\a,\b}}$ is a monotone increasing map in the direction of $\check C^{k,1}_{\a,\b}$ on $\aff\, \check Z^{k,1}_{\a,\b}$, for all $k$, $\a$, $\b$.

\begin{theorem}
\label{T_Monge_final}
Let $\mu, \nu\in\mathcal P(\R^d)$, $\mu\ll\LL$. Then, there exists an optimal transport map $\mathtt T$ for the Monge problem \eqref{E_transpo_Sud}. 
\end{theorem}

% 
%Finally, from the proof of the above theorems and the precise characterization of the points which are contained in a $\Gamma$-cycle by Theorem \ref{T_order_gamma}, we can deduce the following proposition.

%\begin{proposition}
%\label{P_explic_connect}
%Let $W$ denote any of the sets $\check Z^\ell_\b$ in \eqref{E_subp_ell_cycl_intro}, the sets $Z^{',\ell}_\b$ in Corollary \ref{C_subpart_final} and the sets $\check Z\ka$ of Theorems \ref{T_final} and \ref{T_final_nu}. Then it holds %satisfy the following condition
%\begin{equation*}
%\mathrm{Leb} \big( \mathtt p_1(\Gamma) \cap W \big) \ \text{ is contained in a $(\Gamma \cap Z^\ell_\a\times\R^d)$-cycle,}
% \mathcal H^\ell \llcorner_{Z^\ell_\a\Big(\mathtt p_1((\Gamma\llcorner Z^\ell_\a\times\R^d))\cap \check Z^\ell_\b)\Big)\quad\text{ is contained in a $(\Gamma\llcorner Z^\ell_\a\times\R^d)$cycle.}
%\end{equation*}
%for all carriages $\Gamma$ of a finite cost/optimal transference plan, and the Lebesgue points are computed w.r.t. the natural Hausdorff measure restricted to $W$.
%\end{proposition}

\subsection{Structure of the paper}
\label{Ss_struct_Sud}

The paper is organized as follows.

In Section \ref{S_setting} we collect the main notations, definitions and the basic tools we will need in the paper. After recalling some standard definitions of commonly used sets and $\sigma$-algebras, we introduce some notations for functions and multifunctions in Section \ref{Ss_souslin_multifunction}. The basic tools in convex analysis as well as the definitions of the Polish spaces $\mathcal A(k,V)$ made of $k$-dimensional affine subspaces of the affine space $V \subset \R^d$ and the Polish space $\mathcal C(k,V)$ made of non degenerate $k$-dimensional cones of a vector space $V \subset \R^d$ are listed in Section \ref{Ss_intro_affine_subspaces_cones}. \\
The fundamental tools on measure theory and the disintegration theorem are recalled in Section \ref{Ss_measure_disintegration}, while the definition of the optimal transportation problem with the classical sufficient conditions for optimality of transference plans are listed in Section \ref{Ss_transference_plans}. The key analysis on transportation problems for which no potential is available and the role of Borel linear preorders is presented in Section \ref{Ss_linear_pre_unique_transf}. In particular, we define $\Pi^f_{\mathtt c}(\mu,\nu)$-cyclically connected partitions (see Definition \ref{D_pimunuconn}), we state Proposition \ref{P_second_cost} yielding the existence of a family of potentials on the elements of a partition for secondary costs under the assumption of $\Pi^f_{\mathtt c}(\mu,\nu)$-cyclical connectedness, and we show in Theorem \ref{T_A2} \cite{BiaCar} that whenever a Borel linear preorder is $\mathtt c$-compatible and extends a $\preccurlyeq_{(\Gamma,\mathtt c)}$-preorder, then all the transport plans are concentrated on its equivalence classes.

In Section \ref{S_conetransport} we analyze the optimal transportation for cone costs of the form \eqref{E_cost_all_cone_intro}. In Section \ref{Ss_convex_norm_cone} we show the equivalence between optimal transportation problems in $\R^d$ with cost $\d{\cdot}$ and marginals $\mu,\nu \in \mathcal P(\R^d)$ and optimal transportation problems in $\R^{d+1}$ with cost $\mathtt c_{\epi\,\d{\cdot}}$ and marginals $\hat \mu, \hat \nu \in \mathcal P(\R^{d+1})$, where the measures $\hat \mu$, $\hat \nu$ are supported on a $\d{\cdot}$-Lipschitz graph $\Graph\,\psi$ and $\mu = (\mathtt p_{\R^d})_\# \hat \mu$, $\nu = (\mathtt p_{\R^d})_\# \hat \mu$. \\
In Section \ref{Ss_partitions_intro} we generalize the single cone cost transportation problem to the transportation problem on a directed locally affine partition. Here we introduce also the notion of initial and final points of a directed locally affine partition and the notion of conditional second marginals, as well as an example of their dependence w.r.t. the transference plan (Example \ref{Ex_2ndmarg}) and a special case where the conditional second marginals correspond to the disintegration of $\nu$ (Proposition \ref{P_dispiani_2}). \\
A standard covering argument allows to decompose a directed locally affine partition into $k$-directed \emph{sheaf sets}, i.e. directed locally affine partitions whose components $Z\ka$ and cones $C\ka$ are close to a given reference plane $V^k$ and cone $C^k$, and their projection on $V^k$ contains a given open $k$-dimensional cube, Proposition \ref{P_countable_partition_in_reference_directed_planes} and Definition \ref{D_sheaf_set}. This allows to map these sets into \emph{directed fibrations}, where the $Z\ka$ are contained in the planes $\{\a\} \times \R^k$, $\a \in \R^{d-k}$ (Proposition \ref{P_map_sheaf_set_into_fibration}). In this case the transport problem splits into a family of transport problems, each one moving mass on a plane of the form $\{\a\} \times \R^k$ and with cost
\begin{equation}
\label{E_cone_C_a_intro}
\mathtt c_{\C(\a)}(w,w') := \ind_{\C(\a)}(w'-w),
\end{equation}
where $\a \mapsto \C(\a)$ a $\sigma$-compact map with values in $\mathcal C(k,\R^k)$: since $k$ is fixed on a fibration, we can skip it in order to simplify the notation. \\
The final part of the section shows that the mapping of a sheaf set into a fibration preserves the key structures of the optimal transportation problem needed in the proofs of Theorems \ref{T_1} and \ref{T_subpart_final}, and thus allows us to work from now onwards on fibrations.

%As stated before Theorem \ref{T_subpart_step}, the key steps for the technique we develop in this paper are the following: first we show how to find a directed locally affine subpartition of a given directed locally affine partition, and then we show the negligibility of the points outside the subpartition as well as the regularity of the disintegration of the Lebesgue measure on the subpartition.

In Section \ref{S_foliations} we present a technique in order to find the so called \emph{differential} directed locally affine subpartition of a given \emph{$\mathtt c_{\C}$-Lipschitz foliation} of a $\C$-directed fibration. The reason why we introduce and study {$\mathtt c_{\C}$-Lipschitz foliations} is that they are the natural generalization of the notions of graphs of cone-Lipschitz functions --as the Kantorovich potential $\psi$-- and equivalence classes of a $\mathtt c_{C\ka}$-compatible Borel linear preorder (see Proposition \ref{P_ex_fol}). For the terminology used to briefly list the content of this section we refer also to the discussion made in this introduction at the beginning of Section \ref{Ss_main_Sud}.\\
Due to the results of the previous section, when the differential subpartition is mapped back from the $\mathtt c_{\C}$-Lipschitz foliation to the $k$-directed sheaf set, one obtains subpartitions of the sheaf sets covering a given directed locally affine partition, and thus we have a method which yields a subpartition of a given directed locally affine partition. \\
In Section \ref{Ss_cone_lipschitz_graph} we first analyze the simplest example of $\mathtt c_{\C}$-Lipschitz foliation: a \emph{$\mathtt c_{\C(\a)}$-Lipschitz graph}, namely a graph of a $|\cdot|_{\mathbf D(\a)^*}$-Lipschitz function (with $\C(\a)=\epi\,|\cdot|_{\mathbf D(\a)^*}$) in a fibration consisting of a single fiber $\{\a\} \times \R^k$, whose \emph{super/subdifferential} satisfy the \emph{completeness property} \eqref{0_compl1}-\eqref{0_compl2} (as the graph of the Kantorovich potential for the cost $\mathtt c_{\epi\,\d{\cdot}}$).\\
In Section \ref{Ss_cone_lipschitz} we consider general $\mathtt c_{\C}$-Lipschitz foliations, namely partitions of $\A\times\R^k$ whose sets are contained in $\{\a\}\times\R^k$ as $\a$ varies in $\A$ and are given by collections of complete $\mathtt c_{\C(\a)}$-Lipschitz graphs (see Proposition \ref{P_fol_char}). We extend to these families of sets the notion of super/subdifferential (see Definition \ref{D_partial_theta}) and in Section \ref{Ss_regu_resi_set} we show that its completeness and transitivity properties permit to select regions called \emph{forward/backward regular set} and \emph{regular set}.\\
These regions are respectively partitioned in the so called \emph{super/subdifferential partition} and \emph{differential partition} in Section \ref{Ss_partition_transport_set} (see Theorem \ref{T_partition_E+-} and Corollary \ref{C_v}).\\
In Section \ref{Ss_analysis_residual_set} we analyze the \emph{residual set}, namely the complementary of the regular set, and characterize it as the union of the initial/final points of the {super/subdifferential partition} (Theorem \ref{T_partition_E}).\\
In Section \ref{Ss_optim_folia} we give a descriptive characterization of the support of the optimal transportation problem on a $\mathtt c_{\C}$-Lipschitz foliation in terms of the forward/backward regular set and initial/final points.

Section \ref{S_disintechnique} concerns the problem of Lebesgue-regularity of the disintegration on directed locally affine partitions. The main property which allows to deduce the regularity is the \emph{cone approximation property}.\\
First we consider $1$-directed sheaf sets made of segments whose projection on the line generated by the reference cone is a given interval. These particular sheaf sets are called \emph{model sets of directed segments} (see Section \ref{Ss_model_dir_segm}). In the case of strictly convex norms it is sufficient to analyze this special case.\\
The analysis is then extended to \emph{$k$-dimensional model sets}, namely $k$-directed sheaf sets whose projection on the $k$-dimensional plane generated by the reference cone is a given rhomboid (Section \ref{Ss_k_dim_model_set}). In this case, the \emph{cone approximation property} refers to the cone approximation property of any of its $1$-dimensional slices, the latter being sections of a $k$-dimensional model set with $(d-k+1)$-dimensional planes transversal to the reference plane (see Definition \ref{D_1_dim_slice_sheaf}): by transversality, on each of these planes the $k$-dimensional model set becomes a model set of directed segments. \\
Next the analysis is extended to $k$-dimensional sheaf sets (Section \ref{Ss_k_dim_sheaf}). The main observation is that one can partition the sheaf set into countably many $k$-dimensional model sets (Theorem \ref{T_one_d_slicing_FC}). \\
Finally, the property of approximation by cone vector fields also for initial/final points yields the Lebesgue-negligibility of the initial/final points by means of a technique developed first in \cite{celper:euler}, and then extended in \cite{BiaGlo,Car:strictly,CarDan} (see Section \ref{Ss_neglig_init_fin}).

At this point all the techniques needed in order to prove Theorem \ref{T_1} are presented, and its proof is done in Section \ref{S_theorem_1_proof}. Indeed, in Section \ref{S_foliations} we develop a technique to find directed locally affine subpartitions by means of $\mathtt c_{\C}$-Lipschitz foliations, and the graph of the potential $\psi$ is in particular a $\mathtt c_{\C(\a)}$-Lipschitz graph. The only point which remains to be proved is that the disintegration of the Lebesgue measure is regular, which is a consequence of the cone approximation property. The section is thus devoted to the proof of the cone approximation property for cone-Lipschitz graphs (Theorem \ref{T_cone_graph}). 

Let $\tilde {\mathbf D}=\{Z\ka,C\ka\}$ be a directed fibration with the associated transportation problem; as said before, we assume that $\Pi^f_{\mathtt c_{\tilde{\mathbf{D}}}}(\mu,\nu) \not= \emptyset$. In Section \ref{S_cfibr_cfol} we show how to %prove the existence of a subpartition of 
further partition $\{Z\ka,C\ka\}$ into a ${\mathtt c_{\tilde{\mathbf{C}}}}$-Lipschitz foliation, whose $k$-dimensional sets satisfy the assumptions of Theorem \ref{T_subpart_step} (see Theorem \ref{T_cfibrcfol}). The key results are stated in Theorem \ref{T_order_gamma} and Proposition \ref{P_equiv_coun}. The sets of this ${\mathtt c_{\tilde{\mathbf{C}}}}$-Lipschitz foliation are given by the equivalence classes of a ${\mathtt c_{\tilde{\mathbf{C}}}}$-compatible linear preorder on which all the transport plans in $\Pi^f_{\mathtt c_{\tilde{\mathbf{D}}}}(\mu,\nu)$ are concentrated (called \emph{$(\mathtt c_{\tilde{\mathbf{C}}},\mu,\nu)$-compatible} linear preorder in Definition \ref{D_cmunucomp}), as anticipated after the statement of Theorem \ref{T_subpart_step}. 

In Section \ref{S_cone_approx_folia} we prove the cone approximation property for the differential partition of a $\mathtt c_{\C}$-Lipschitz foliation whose sets are given by equivalence classes of a Borel $({\mathtt c_{\tilde{\mathbf{C}}}}, \mu,\nu)$-compatible linear preorder. Since we do not have a potential, we need to use the uniqueness theorem of the linear preorder (Theorem \ref{T_coneappr_fol}): as a corollary, one immediately obtains the Lebesgue-regularity of the disintegration (Corollary \ref{C_infinnegl}).\\
The section is concluded which an example (Remark \ref{R_not_gener_potential}) which shows that this result cannot be deduced as a consequence of the analysis of Section \ref{S_theorem_1_proof}, even if the level sets of $\theta$ are Lipschitz graphs. In fact, the disintegration of the Lebesgue measure on a Lipschitz foliation is in general not absolutely continuous w.r.t. the natural Hausdorff measures on the level sets: we show an example where the level sets of the function $\theta : [0,1]^2 \to [0,1]$ generating the foliation are $C^\infty$, nevertheless the disintegration of $\mathcal L^2 \llcorner_{[0,1]^2}$ on the level sets of $\theta$ have Dirac masses.

The proof of Theorems \ref{T_subpart_step} and \ref{T_final_nu} are done in Section \ref{S_proff_main_Th}, and they are obtained as direct consequences of the results proved so far.

Finally, in Appendix \ref{A_minimal_equivalence} we present one of the main result of \cite{BiaCar} about the minimality of equivalence relations and prove a key consequence used in our proof, Corollary \ref{C_constant_for_minimal_equivalence}. 

In Appendix \ref{A_appendix_nota} we collect the notations used in this paper.

\section{General notations and definitions}
\label{S_setting}

As standard notation, we will write $\N$ for the natural numbers, $\N_0 = \N \cup \{0\}$, $\Q$ for the rational numbers and $\R$ for the real numbers. The sets of positive rational and real numbers will be denoted by $\Q^+$ and $\R^+$ respectively.

The $d$-dimensional real vector space will be denoted by $\R^d$. $B^d(x,r)$ is the open unit ball in $\R^d$ with center $x$ and radius $r>0$ and $\mathbb S^{d-1}$ is the $(d-1)$-dimensional unit sphere. The scalar product of two vectors $x,y \in \R^d$ will be denoted by $x \cdot y$, and the Euclidean norm of $x \in \R^d$ by $|x| = \sqrt{x \cdot x}$. To avoid the analysis of different cases when parameters are in $\R^k$ for $k=1,\dots,d$ or in $\N$, we set $\R^0 := \N$.

We denote the first infinite ordinal number by $\omega$.

Given a set $X$, $\P(X)$ is the power set of $X$. The closure of a set $A$ in a topological space $X$ will be denoted by $\clos\,A$, and its interior by $\inter\,A$. If $A \subset Y \subset X$, then the relative interior of $A$ in $Y$ is $\interr A$: in general the space $Y$ will be clear from the context. The topological boundary of a set $A$ will be denoted by $\partial A$, and the relative boundary is $\partial_\mathrm{rel} A$.

If $A$, $A'$ are subsets of a real vector space, we will write
\begin{equation}
\label{E_vector_sum}
A + A' := \big\{ x + x':\, x \in A, x' \in A' \big\}.
\end{equation}
If $T \subset \R$, then we will write
\begin{equation}
\label{E_product_scalar_vector_def}
T A := \big\{ t x:\, t \in T, x \in A \big\}.
\end{equation}
In particular $A - A' = A + (-A')$.

If $\prod_i X_i$ is the product space of the spaces $X_i$, we will denote the projection on the $\bar i$-component either as $\mathtt p_{\bar i}$ or $\mathtt p_{x_{\bar i}}$ or $\mathtt p_{X_{\bar i}}$: in general no ambiguity will occur.

\subsection{Functions and multifunctions}
\label{Ss_souslin_multifunction}

A multifunction $\mathbf f$ will be considered either as a map $\mathbf f:\,X \subset \dom\,\mathbf f \to \P(Y)$ or as a set $\mathbf f \subset X \times Y$. The set $\dom\, \mathbf f$ is called the \emph{domain} of $\mathbf f$. For every $x\in\dom\,\mathbf f$ we will write
\[
\mathbf f(x) = \big\{ y \in Y: (x,y) \in \mathbf f \big\}.
\]
The inverse of $\mathbf f$ will be denoted by
\begin{equation}
\label{E_inverse_multi_function}
\mathbf f^{-1} = \big\{ (y,x) \in Y \times X : (x,y) \in \mathbf f \big\}.
\end{equation}
Similarly, if $A\subset X\times Y$, then $A^{-1}:=\{(y,x):\,(x,y)\in A\}$.

In the same spirit, we will often not distinguish between a single valued function $\mathtt f$ and its graph, denoted by $\Graph\, \mathtt f$. We say that the function $\mathtt f$ (or the multifunction $\mathbf f$) is \emph{$\sigma$-continuous} if the set $\Graph\,\mathtt f$ (or $\mathbf f\subset X\times Y$) is $\sigma$-compact. Note that we do not require its domain to be the entire space.

If $\mathtt f$, $\mathtt g$ are two functions, their composition will be denoted by $\mathtt g \circ \mathtt f$, with domain $\mathtt f^{-1}(\dom\,\mathtt g)$.
If $\mathtt f:X\to Y$, $\mathtt g:X\to Z$, then the product map is denoted by $\mathtt f\times \mathtt g:X\to Y\times Z$.

The epigraph of a function $\mathtt f : X \to \R$ is the set
\begin{equation}
\label{E_epigraph_function}
\epi\,\mathtt f := \big\{ (x,t)\in X\times\R : \mathtt f(x) \leq t \big\}.
\end{equation}
The identity map will be written as $\mathbb I$, the characteristic function of a set $A$ will be denoted by
\begin{equation}
\label{E_char_funct_set_A}
\chi_A(x) :=
\begin{cases}
1 & x \in A, \crcr
0 & x \notin A,
\end{cases}
\end{equation}
and the indicator function of a set $A$ is defined by
\begin{equation}
\label{E_indicator_function_A}
\ind_A(x) :=
\begin{cases}
0 & x \in A, \crcr
+\infty & x \notin A.
\end{cases}
\end{equation}

\subsection{Affine subspaces, convex sets and norms}
\label{Ss_intro_affine_subspaces_cones}

For $k,k',d \in \N$, $k' \leq k \leq d$, define $\mathcal G(k,\R^d)$ to be the set of $k$-dimensional subspaces of $\R^d$ and let $\mathcal A(k,\R^d)$ be the set of $k$-dimensional affine subspaces of $\R^d$. If $V \in \mathcal A(k,\R^d)$, we define $\mathcal A(k',V) \subset \mathcal A(k',\R^d)$ to be the set of $k'$-dimensional affine subspaces of $V$. We also denote by $\mathtt p_V:\R^d\to V$ the projection map.

If $A \subset \R^d$, then define its \emph{affine span} as
\begin{equation}
\label{E_affine_span_set}
\aff\,A := \bigcap_{\nfrac{V\in\mathcal A(k,\R^d)}{\,k\in\N,\,A\subset V}} V,
\end{equation}
and its convex hull $\conv\,A$ as the set given by the intersection of all convex sets of $\R^d$ containing $A$. Given $k$ vectors $\{\e_1,\dots,\e_k\}\subset\R^d$, their linear span $\mathrm{aff}\,\{\e_1,\dots,\e_k,0\}$ is denoted by $\langle\e_1,\dots,\e_k\rangle$, and the orthogonal space to $V\in\mathcal G(k,\R^d)$ is denoted by $V^\perp$. Notice that
\[
\aff\,A=\conv\,A+\R(\conv\,A-\conv\,A) \qquad \text{and} \qquad \aff\,A\in\underset{k\in\N}{\bigcup}\mathcal A(k,\R^d),
\]
unless $A$ consists of a single point. If we set by convention $\mathcal A(0,\R^d)=\R^d$, then the above formula holds also when $A$ is a singleton. 
%The orthogonal space to $\mathrm{span}\, A := \aff\,A -A$ will be denoted by $A^\perp$.

The \emph{linear dimension} of an affine subspace $V$ is denoted by $\dim\,V$, and we set accordingly $\dim\{x\}=0$ for all $x\in\mathcal A(0,\R^d)$. 

If $A$ is convex, then its relative interior in $\aff\, A$, denoted by $\interr A$, is nonempty and $A\subset\clos\,\interr A$. Hence we define $\dim\,A:=\dim \aff\,A$. 

An \emph{extremal face} of a convex set $A\subset\R^d$ is by definition any convex set $A'\subset A$ such that 
\[
\interr\,[x,y]\cap A'\neq\emptyset\quad\Rightarrow\quad[x,y]\in A',\quad \forall\,x,y\in A. %,\quad.
\]

A convex set $C\subset\R^d$ is a \emph{convex cone} if 
\[
C = \{0\} \cup \R^+C.
\]
In particular, $0$ is called the \emph{vertex} of the convex cone. For all $k\in\{1,\dots,d\}$, we let $\mathcal C(k,\R^d)$ to be the set of closed $k$-dimensional convex cones in $\R^d$ which are \emph{nondegenerate}, meaning that
\[
C\in\mathcal C(k,\R^d)\quad\Rightarrow\quad C\setminus\{0\}\subset \interr\, H_C
\]
for some $k$-dimensional half-plane $H_C \subset \R^d$.

The extremal faces of a convex cone are still convex cones called \emph{extremal cones}.

If $D$ is a $d$-dimensional compact convex set in $\R^d$ and $0\in\inter\,D$, then one can define the \emph{(convex) norm} (or \emph{Minkowski functional}) $\d{\cdot}$ generated by $D$ as 
\begin{equation}
\label{E_norm}
\d{x} = \min \big\{ t \in \R^+:\,x \in t D \big\} = \max \big\{ x' \cdot x:\,x' \in D^* \big\},
\end{equation}
where 
\[
D^* := \Big\{ x' \in \R^d:\,x' \cdot x \leq 1 \text{ for all $x\in D$} \Big\}
\]
is the \emph{convex dual} to $D$. Equivalently, $\d{\cdot}:\R^d\to\R$ is identified by the following properties:
\begin{align*}
& \{x\in\R^d:\,\d{x}\leq1\}=D;\\
& \d{t x} = t \d{x},\quad\forall\,t \geq 0 \qquad \text{positively $1$-homogeneous};\\
& \d{x+y}\leq\d{x}+\d{y}\qquad\text{subadditive}.
\end{align*}
In particular, since $\d{\cdot}$ is positively $1$-homogeneous, subadditivity can be equivalently replaced by convexity. 

\begin{remark}
 \label{R_cone_epi}
Notice that the set $\epi\,\d{\cdot}$ belongs to $\mathcal C(d+1;\R^{d+1})$, with the identification $\R^{d}\times\R\overset{\sim}{=}\R^{d+1}$. Viceversa, given a convex cone $C\in\mathcal C(k,\R^d)$, if we fix a system of coordinates $(x_1,\dots,x_k,x_{k+1},\dots,x_d)\in\R^d$ such that $H_C=\{x_k\geq0,\,x_{k+1}=\dots=x_d=0\}$, then $C$ is the epigraph  of a convex norm on $\R^{k-1}\overset{\sim}{=}\{x_{k}=\dots=x_d=0\}$.
\end{remark}

We will call \emph{extremal cones} of a convex norm $\d{\cdot}:\R^d\to\R$ either the extremal cones of $\epi\,\d{\cdot}$ in $\R^{d+1}$ or their projections on $\R^d$, being the distinctions between the two cases clear from the context.

For $C \in \mathcal C(k,\R^d)$ we call $C\cap\mathbb S^{d-1}$ the \emph{set of directions} of $C$. For $V \in \mathcal G(k,\R^d)$ we will also write
\begin{equation}
\label{E_cones_and_directions_in_subspace}
\mathcal C(k',V) := \big\{ C \in \mathcal C(k',\R^d): C \subset V \big\}.
\end{equation}

If $C \in \mathcal C(k,\R^d)$, $r>0$, we also define the cone
\begin{equation}
\label{E_epsilon_neigh_of_Cka}
\mathring C(r) := \{0\} \cup \R^+ \Big( \big( C\cap\mathbb S^{d-1} + B^d(0,r) \big) \cap \aff\,C \Big).
\end{equation}
Clearly $C(r) := \clos\,\mathring C(r) \in \mathcal C(k,\R^d)$ for $0 < r \ll 1$ and $C = \underset{n}{\cap}\, \mathring C(2^{-n})$. For $r > 0$ we also define the cone
\begin{equation}
\label{E_inverse_neigh_Cone}
%\Sara
\mathring C(-r) := \{0\} \cup \R^+ \Big\{ x \in \mathbb S^{d-1} \cap \aff\,C :\, B^d(x,r) \cap \aff\,C \subset C \Big\},
\end{equation}
so that $\{0\} \cup \interr C = \underset{n}{\cup}\, \mathring C(-2^{-n})$: as before $C(-r) := \clos\,\mathring C(-r) \in \mathcal C(k,R^d)$ for $0 < r \ll 1$. More generally, for $C \in \mathcal C(k,\R^d)$ we will use the notation
\begin{equation}
\label{E_open_cone_iterior_def}
\mathring C := \{0\} \cup \interr C = \bigcup_{n \in \N} \mathring C(-2^{-n}).
\end{equation}

By convention we set $\mathcal C(0,\R^d)=\R^d$ and we will often denote a convex cone $C$ as $C^{\dim\,C}$ to emphasize its dimension.

On $\mathcal G(k,\R^d)$, $\mathcal A(k,\R^d)$ and $\mathcal C(k,\R^d)$ we consider the topology given by the Hausdorff distance $\mathtt d_{\mathrm H}$ in every closed ball $\clos\, B^d(0,r)$ of $\R^d$, i.e.
\begin{equation}
\label{E_distance_cones}
\mathtt d(A,A') := \sum_n 2^{-n} \mathtt d_{\mathrm H} \big( A \cap B^d(0,n), A' \cap B^d(0,n) \big).
\end{equation}
for two generic elements $A$, $A'$.

If $A \subset \mathbb S^{d-1}$, its spherical convex envelope is defined as
\begin{equation}
\label{E_convex_envelope_on_sphere}
\conv_{\mathbb S^{d-1}} A := \mathbb S^{d-1} \cap \big( \R^+  \conv\,A \big).
\end{equation}

\subsection{Measures and disintegration} 
\label{Ss_measure_disintegration}

We will denote the Lebesgue measure on $\R^d$ by $\mathcal L^d$, and the $k$-dimensional Hausdorff measure on $V \in \mathcal A(k,\R^d)$ as $\mathcal H^k \llcorner_V$. In general, the restriction of a function/measure to a set $A \in \R^d$ will be denoted by the symbol $\llcorner_A$ (or sometimes $\llcorner A$) following the function/measure. The product of two locally finite Borel measures $\varpi_0$, $\varpi_1$ will be denoted by $\varpi_0 \otimes \varpi_1$.

The Lebesgue points $\mathrm{Leb}(A)$ of a set $A \subset \R^d$ are the points $z \in A$ such that
\begin{equation}
\label{E_lebesgue_point_definition}
\lim_{r \to 0} \frac{\mathcal L^d(A \cap B^d(z,r))}{\mathcal L^d(B^d(z,r))} = 1.
\end{equation}
If $\varpi$ is a locally bounded Borel measure on $\R^d$, we will write $\varpi \ll \mathcal L^d$ if $\varpi$ is \emph{absolutely continuous} (a.c. for brevity) w.r.t. $\mathcal L^d$.

For a generic \emph{Polish space} $X$ (i.e., a separable and complete metric space), the Borel sets and the set of Borel probability measures will be respectively denoted by $\mathcal B(X)$ and $\mathcal P(X)$.
The \emph{Souslin sets $\Sigma^1_1$} of a Polish space $X$ are the projections on $X$ of the Borel sets of $X \times X$. The $\sigma$-algebra generated by the Souslin sets will be denoted by $\varTheta$.

Two Radon measures $\varpi_0$, $\varpi_1$ on $X$ are \emph{equivalent} if for all $B \in \mathcal B(X)$
\begin{equation}
\label{E_equiv_varpi}
\varpi_0(B) = 0 \quad \Longleftrightarrow \quad \varpi_1(B)=0,
\end{equation}
and we denote this property by $\varpi_0 \simeq \varpi_1$.

If $\varpi$ is a measure on a measurable space $X$ and $\mathtt f : X \to Y$ is an $\varpi$-measurable map, then the \emph{push-forward of $\varpi$} by $\mathtt f$ is the measure $\mathtt f_\# \varpi$ on $Y$ defined by
\begin{equation}
\label{E_push_forward}
\mathtt f_\#\varpi(B)=\varpi(\mathtt f^{-1}(B)),\quad\text{for all $B$ in the $\sigma$-algebra of $Y$.}
\end{equation}

Finally we briefly recall the concept of disintegration of a measure over a partition.

\begin{definition}[Partitions]
\label{D_part}
A \emph{partition} in $\R^d$ is a family $\{Z_{\a}\}_{\a\in\A}$ of disjoint subsets of $\R^d$. We say that $\{Z_{\a}\}_{\a\in\A}$ is a \emph{Borel partition} if $\A$ is a Polish space, $\underset{\mathfrak a \in \mathfrak A}{\cup} Z_\mathfrak a$ is Borel and the \emph{quotient map} $\mathtt h : \underset{\a\in\A} {\cup} Z_\a \to \A$, $\mathtt h : z \mapsto \mathtt h(z) = \a$ such that $z \in Z_{\mathfrak a}$, is Borel-measurable. We say that $\{Z_{\a}\}_{\a\in\A}$ is \emph{$\sigma$-compact} if $\A\subset\R^k$ for some $k\in\N$, $\underset{\a \in \A}{\cup}\, Z_\mathfrak a$ is $\sigma$-compact and $\mathtt h$ is $\sigma$-continuous.
\end{definition}

The sets in the $\sigma$-algebra $\{\mathtt h^{-1}(F):\, F \in \mathcal B(\A)\}$ are also called in the literature \emph{saturated sets}. Notice that we do not require $\{Z_{\a}\}_{\a\in\A}$ to be a covering of $\R^d$.

\begin{definition}[Disintegration]
\label{D_dis}
Given a Borel partition in $\R^d$ into sets $\{Z_\a\}_{\a\in\A}$ with quotient map $\mathtt h : \underset{\a\in\A}{\cup} Z_\a \to \A$ and a probability measure $\varpi \in \mathcal P(\R^d)$ s.t. $\varpi\bigl(\underset{\a\in\A}{\cup}Z_\a\bigr)=1$, a \emph{disintegration} of $\varpi$ w.r.t. $\{Z_\a\}_{\a\in\A}$ is a family of probability measures $\{\varpi_\a\}_{\a\in\A}\subset\mathcal P(\R^d)$ such that
\begin{align}
& \A \ni \a \mapsto \varpi_\a(B) \quad \text{is an $\mathtt h_{\#} \varpi$-measurable map $\forall\,B\in\mathcal B(\R^d)$}, \label{E_disint1}\\
& \varpi \bigl( B \cap \mathtt h^{-1}(F) \bigr) = \int_F \varpi_{\a}(B)\,d\mathtt h_{\#}\varpi(\a), \quad \forall\,B\in\mathcal B(\R^d),\,F\in\mathcal B(\A). \label{E_disint2}
\end{align}
\end{definition}

As proven in Appendix A of \cite{BiaCar} (for a more comprehensive analysis see \cite{Fre}), we have the following theorem.

\begin{theorem}
\label{T_disint}
Under the assumptions of Definition \ref{D_dis}, the disintegration $\{\varpi_\a\}_{\a\in\A}$ is \emph{unique} and \emph{strongly consistent}, namely
\begin{align}
&\text{if }\a\mapsto\varpi^1_{\a},\:\a\mapsto\varpi^2_{\a} \text{ satisfy \eqref{E_disint1}-\eqref{E_disint2}} \quad \Longrightarrow \quad \varpi^1_\a=\varpi^2_\a \ \text{for $\mathtt h_{\#}\varpi$-a.e. $\a\in\A$}; \notag \\
&\varpi_{\a}(Z_\a)=1 \quad \text{ for $\mathtt h_{\#}\varpi$-a.e. $\a\in\A$}.\label{E_disintsc}
\end{align}
\end{theorem}

The measures $\{\varpi_\a\}_{\a\in\A}$ are also called \emph{conditional probabilities}.

To denote the (strongly consistent) disintegration $\{\varpi_\a\}_{\a\in\A}$ of a probability measure $\varpi\in\mathcal P(\R^d)$ on a Borel partition $\{Z_\a\}_{\a\in\A}$ we will often use the formal notation
\begin{equation}
\label{E_dis_not}
\varpi=\int_\mathfrak A \varpi_\mathfrak a \,dm(\a),\quad \varpi_\mathfrak a(Z_\a)=1,
\end{equation}
with $m = \mathtt h_\# \varpi$, $\mathtt h$ being the quotient map.

Since the conditional probabilities $\varpi_\a$ are defined $m$-a.e., many properties (such as $\varpi_\a(Z_\a)=1$) should be considered as valid only for $m$-a.e. $\a \in \A$: for shortness, we will often consider the $\varpi_{\mathfrak a}$ redefined on $m$-negligible sets in order to have statements valid $\forall \a \in \A$.

We also point out the fact that, according to Definition \ref{D_dis}, in order that a disintegration of $\varpi$ over a partition can be defined, $\varpi$ has to be concentrated on the union of the sets of the partition (which do not necessarily cover the whole $\R^d$). In general, if we remove this assumption, since the formulas \eqref{E_disint1}-\eqref{E_disint2} make sense nonetheless for $B \subset \underset{\a\in\A}{\cup}Z_\a$, by means of formula \eqref{E_dis_not} we ``reconstruct'' only $\varpi\llcorner_{\underset{\a\in\A}{\cup}Z_\a}$ .

% \begin{definition}
% \label{D_integr}
Let $m'\in\mathcal P(\A)$, $\{\varpi'_\a\}_{\a\in\A}\subset\mathcal P(\R^d)$ such that
\begin{equation*}
\A \ni \a \mapsto \varpi'_\a(B) \quad \text{ is $m'$-measurable, $\forall\,B\in\mathcal B(\R^d)$}.
\end{equation*}
Then, one can define the probability measure $\varpi'$ on $\R^d$ by
\begin{equation}
\label{E_int_measure}
\varpi'(B)=\int_\mathfrak A \varpi'_\a(B)\,dm'(\a),\quad\text{ $\forall\,B\in\mathcal B(\R^d)$}.
\end{equation}
% \end{definition}
%
The measure defined in \eqref{E_int_measure} will be denoted as
\[
\varpi' = \int_\mathfrak A \varpi'_\a\,dm'.
\]
Notice that, despite the notation is the same as in \eqref{E_dis_not}, the family $\{\varpi'_\a\}_{\a \in \mathfrak A}$ in the above definition is not necessarily a disintegration of $\varpi'$, both because the measure $m'$ is not necessarily a quotient measure of a Borel partition and because the measures $\varpi'_\a$ are not necessarily concentrated on the sets of a partition. In the rest of the paper, such an ambiguity will not occur, since we will always point out whether a measurable family of probability measures is generated by a disintegration or not.  

\begin{remark}
\label{R_disint_lebesgue}
If instead of $\varpi \in \mathcal P(\R^d)$ we consider the Lebesgue measure $\mathcal L^d$ (more generally, a Radon measure) a disintegration $\{\upsilon_\a\}_{\a\in\A}$ is to be considered in the following sense. First choose a partition $\{A_i\}_{i\in\N}$ of $\R^d$ into sets with unit Lebesgue measure, then let
\[
\mathcal L^d \llcorner_{A_i} = \int \upsilon_{\mathfrak a,i} d\eta_i(\mathfrak a), \quad \eta_i := \mathtt h_\# \mathcal L^d \llcorner_{A_i},
\]
be the standard disintegration of the probability measure $\mathcal L^d \llcorner_{A_i}$, and finally
\[
\upsilon_\mathfrak a := \sum_i 2^i \upsilon_{\mathfrak a,i}, \quad \eta := \sum_i 2^{-i} \eta_i.
\]
Clearly, in this definition the ``conditional probabilities'' $\upsilon_\mathfrak a$ and the ``image measure'' $\eta$ depend on the choice of the sets $\{A_i\}_{i\in\N}$. 
\end{remark}

\subsection{Optimal transportation problems}
% \label{Ss_opt_trasp}
\label{Ss_transference_plans}

For a generic Polish space $X$, measures $\mu,\nu \in \mathcal P(X)$ and Borel \emph{cost function} $\mathtt c : X \times X \to [0,\infty]$, we will consider the sets of probability measures
\begin{equation}
\label{E_Pi_mu_nu}
\Pi(\mu,\nu) := \Big\{ \pi \in \mathcal P(X \times X) : (\mathtt p_1)_\# \pi = \mu, (\mathtt p_2)_\# \pi = \nu \Big\},
\end{equation}
\begin{equation}
\label{E_Pi_mu_nu_fin}
\Pi^f_{\mathtt c}(\mu,\nu) := \bigg\{ \pi \in \Pi(\mu,\nu) : \int_{X \times X} \mathtt c \,d\pi < +\infty \bigg\},
\end{equation}
\begin{equation}
\label{E_Pi_mu_nu_optimal}
\Pi^{\mathrm{opt}}_{\mathtt c}(\mu,\nu) := \bigg\{ \pi \in \Pi(\mu,\nu) : \int_{X \times X} \mathtt c \,d\pi = \inf_{\pi' \in \Pi(\mu,\nu)} \int_{X \times X} \mathtt c \,d\pi' \bigg\}.
\end{equation}
The elements of the set defined in \eqref{E_Pi_mu_nu} are called \emph{transference} or \emph{transport plans} between $\mu$ and $\nu$, those in \eqref{E_Pi_mu_nu_fin} \emph{transference} or \emph{transport plans with finite cost} and the set defined in \eqref{E_Pi_mu_nu_optimal} is the set of \emph{optimal plans}. The quantity
\begin{equation}
\label{E_transport_prob}
\mathtt C(\mu,\nu) := \inf_{\pi \in \Pi(\mu,\nu)} \int_{X \times X} \mathtt c \,d\pi
\end{equation}
is the \emph{transportation cost}.

In the following we will always consider costs and measures s.t. $\mathtt C(\mu,\nu)<+\infty$, thus $\Pi^f_{\mathtt c}(\mu, \nu)\neq\emptyset$.

The problem of showing that $\Pi^{\mathrm{opt}}_{\mathtt c}(\mu,\nu)\neq\emptyset$ is called Monge-Kantorovich problem.

We recall (see e.g. \cite{BiaCar,kel:duality}) that any optimal plan $\pi\in\Pi^{\mathrm{opt}}_{\mathtt c}(\mu,\nu)$ is \emph{$\mathtt c$-cyclically monotone}, i.e. there exists a $\sigma$-compact \emph{carriage} $\Gamma\subset X\times X$ such that $\pi(\Gamma)=1$ and for all $I\in \N$, $\{(x_i,y_i)\}_{i=1}^I\subset\Gamma$,
\[
\sum_{i=1}^I\mathtt c(x_i,y_i)\leq\sum_{i=1}^I\mathtt c(x_{i+1},y_i),
\]
where we set $x_{I+1}:=x_1$. Any such $\Gamma$ is called \emph{$\mathtt c$-cyclically monotone carriage}. However, in order to deduce the optimality of a transference plan the $\mathtt c$-cyclical monotonicity condition itself is not sufficient and one has to impose additional conditions. Most of the conditions in the literature exploit the dual formulation of Monge-Kantorovich problem (see \cite{villa:Oldnew}), namely
%Sara
\[
 \mathtt C(\mu,\nu)=\sup_{\underset{\phi\:\mu\text{-meas. and }\psi\:\nu\text{-meas.}}{\phi,\,\psi:X\rightarrow[-\infty,+\infty)}}\Big\{\int \phi(x)\,d\mu(x)+\int\psi(y)\,d\nu(y):\,{\phi(x)+\psi(y)\leq \mathtt c(x,y)}\Big\}.
\]
For example (see Lemma 5.3 of \cite{BiaCar}) if there exists a pair of functions
\begin{align}
&\phi,\,\psi:X\to[-\infty,+\infty), \quad \text{ $\phi$ $\mu$-measurable and $\psi$ $\nu$-measurable}, \label{E_phipsi1} \\
&\phi(x)+\psi(y)\leq \mathtt c(x,y), \quad \forall\, x,y \in X, \label{E_phipsi2} \\
&\phi(x)+\psi(y)= \mathtt c(x,y), \quad \text{ $\pi$-a.e. for some $\pi\in \Pi(\mu,\nu)$}, \label{E_phipsi3}
\end{align}
then $\phi,\,\psi$ are optimizers for the dual problem and $\pi\in\Pi^{\mathrm{opt}}_{\mathtt c}(\mu,\nu)$. Conditions on the cost guaranteeing the existence of such potentials (and indeed of more regular ones) are e.g. the following ones:
\begin{enumerate}
\item \label{Point_boun_fg} $\mathtt c$ is l.s.c. and satisfies $\mathtt c(x,y) \leq \mathtt f(x) + \mathtt g(y)$ for some $\mathtt f \in L^1(\mu)$, $\mathtt g \in L^1(\nu)$ (\cite{rachrusch});
\item $\mathtt c$ is real-valued and satisfies the following assumption (\cite{conf:optcime})
\[
\nu\Big(\Big\{y:\,\int\mathtt c(x,y)\,d\mu(x)<+\infty\Big\}\Big)>0,\quad \mu\Big(\Big\{x:\,\int\mathtt c(x,y)\,d\nu(y)<+\infty\Big\}\Big)>0;
\]
\item $\{\mathtt c<+\infty\}$ is an open set $O$ minus a $\mu\otimes\nu$-negligible set $N$ (\cite{beigolmarsch}).
\end{enumerate}

The weakest sufficient condition for optimality, which does not rely on the existence of global potentials and implies the results recalled above, has been given in \cite{BiaCar}. Since such condition will be needed and of fundamental importance for the proofs of our main results (in particular, Theorem \ref{T_subpart_step}), in the next section we give a brief explanation of the approach followed in \cite{BiaCar} and we state it in a form which will be more convenient for our purposes.

\subsection{Linear preorders, uniqueness and optimality}
\label{Ss_linear_pre_unique_transf}

% The following theorems have been proved in Section 4 of \cite{MR2582736}. %For a more comprehensive analysis, see \cite{MR2462372}.

Let $\mathtt c : X \times X \to [0,+\infty]$ be a Borel cost function on a Polish space $X$ such that $\mathtt c(x,x) = 0$ for all $x \in X$, let $\mu$, $\nu \in \mathcal P(X)$ be such that $\Pi^f_{\mathtt c}(\mu,\nu) \neq \emptyset$ and let $\Gamma \subset X \times X$ be a $\mathtt c$-cyclically monotone carriage of some $\pi \in \Pi^f_{\mathtt c}(\mu,\nu)$ satisfying w.l.o.g. $\{(x,x) : x \in X\} \subset \Gamma$.
%Sara
A standard formula for constructing a pair of optimal potentials is the following: for fixed $(x_0,y_0) \in \Gamma$ and $(x,y)\in\Gamma$, define
\begin{align}
\label{E_pot_form}
\phi(x) &:= \inf \bigg\{ \sum_{i=0}^I \mathtt c(x_{i+1},y_i) - \mathtt c(x_i,y_i):\, (x_i,y_i) \in \Gamma, I \in \N, x_{I+1} = x \bigg\},\\
\psi(y) &:=\mathtt c(x,y)-\phi(x). \notag
\end{align}
If one of the assumptions $(1)$-$(3)$ holds, then this $\phi$, $\psi$ satisfy \eqref{E_phipsi1}-\eqref{E_phipsi3}. However, for general Borel costs $\mathtt c$, the assumptions $(1)$-$(3)$ are not satisfied. In particular, for any choice of $(x_0,y_0)$, there may be a set of positive $\mu$-measure on which $\phi$ is not well defined (namely, the infimum in \eqref{E_pot_form} is taken over an empty set) or takes the value $-\infty$ (see the examples in \cite{BiaCar}).

To explain why this can happen and briefly recall the strategy adopted in \cite{BiaCar} to overcome this problem in a more general setting, we need the following definition.

\begin{definition}[Axial paths and cycles]
\label{D_axpath}
An \emph{axial path with base points $\{(x_i,y_i)\}_{i=1}^I \subset \Gamma$}, $I\in\N$, starting at $x = x_1$ and ending at $x'$ is the sequence of points
\[
(x,y_1) = (x_1,y_1),(x_2,y_1),\dots,(x_i,y_{i-1}),(x_i,y_i),(x_{i+1},y_i),\dots,(x_I,y_I),(x',y_I).
\]
We will say that the axial path \emph{goes from} $x$ to $x'$: note that $x \in \mathtt p_1 \Gamma$. A \emph{closed axial path} or \emph{cycle} is an axial path with base points in $\Gamma$ such that $x = x'$. A \emph{$(\Gamma, \mathtt c)$-axial path} is an axial path with base points in $\Gamma$ whose points are contained in $\{\mathtt c < \infty\}$ and a \emph{$(\Gamma,\mathtt c)$-cycle} is a closed $(\Gamma, \mathtt c)$-axial path.
\end{definition}

Notice that, in order that \eqref{E_pot_form} is well defined, for $\mu$-a.e. point $x\in\mathtt p_1\Gamma$ there must be a $(\Gamma,\mathtt c)$-axial path going from $x_0$ to $x$. Moreover, being $\Gamma$ $\mathtt c$-cyclically monotone, $\phi$ is surely finite valued in the case in which for $\mu$-a.e. point $x\in\mathtt p_1 \Gamma$ there exists also a $(\Gamma,\mathtt c)$-axial path going from $x$ to $x_0$ (and thus to a.a. any other point in $\Gamma$). In particular, $x$ and $x_0$ are connected by a $(\Gamma,\mathtt c)$-cycle.

The first idea in \cite{BiaCar} is then to partition $X$ into the equivalence classes $\{Z_\a\}_{\a\in \A}$ induced by the $(\Gamma,\mathtt c)$-cycle equivalence relation and disintegrate $\mu$, $\nu$ over $\{Z_\a\}_{\a\in \A}$ and $\pi$ over $\{Z_\a\times Z_\b\}_{\a,\b\in\A}$. 

Since $\mathtt c(x,x)=0$ $\forall\,x\in X$ and $\Gamma\supset \Graph\,\Id$, then $(x,y)\in\Gamma$ implies that $x$ and $y$ belong to the same $(\Gamma,\mathtt c)$-cycle (consider the path $(x,y)$, $(y,y)$, $(y,y)$, $(x,y)$) and in particular that 
\begin{equation}
\label{E_piconc}
\pi\bigg(\underset{\a\in\A}{\bigcup}Z_\a\times Z_\a\bigg)=1.
\end{equation}

If the disintegration is strongly consistent (see Theorem \ref{T_disint}), we get
\begin{align}
\mu&=\int\mu_\a\,dm(\a),\quad \mu_\a(Z_\a)=1, \label{E_mudisi}\\
\nu&=\int\nu_\a\,dm(\a),\quad\nu_\a(Z_\a)=1, \label{E_nudisi}\\
\pi&=\int\pi_{\a\a}\,d(\Id\times\Id)_\#m(\a),\quad\pi_\a(Z_\a\times Z_\a)=1, \label{E_pidisint}
\end{align}
where $m=\mathtt h_\#\mu=\mathtt h_\#\nu$ because there exists at least a plan in $\Pi^f_\mathtt c(\mu,\nu)$ --in this case $\pi$-- such that \eqref{E_piconc} is satisfied.

Notice that the fact that $\pi$ is concentrated on the diagonal equivalence classes $\{Z_\a \times Z_\a\}_{\a \in \A}$, i.e. formula \eqref{E_piconc}, is equivalent to say that the quotient measure $(\mathtt h \times \mathtt h)_\#\pi$ %belonging to $\Pi^f_{(\mathtt h \times \mathtt h)_\# \mathtt c}(m,m)$ 
satisfies
\[
(\mathtt h \times \mathtt h)_\#\pi=(\Id\times\Id)_\#m,
\]
i.e. it is concentrated on the diagonal of $\A\times \A$ (see \eqref{E_pidisint}).

%Sara
Now, as a consequence of the fact that $\mu_\mathfrak a$-a.a. points in $Z_\mathfrak a$ can be connected to $\mu_\mathfrak a$-a.a. other points in $Z_\mathfrak a$ by a $(\Gamma\cap Z_\mathfrak a\times Z_\mathfrak a, \mathtt c)$-cycle and $\exists\,\pi_{\mathfrak a\a}\in \Pi^f_\mathtt c(\mu_\mathfrak a,\nu_\mathfrak a)$ $\mathtt c$-cyclically monotone which is concentrated on $\Gamma\cap Z_\mathfrak a\times Z_\mathfrak a$, using \eqref{E_pot_form} we are able to construct optimal potentials $\phi_\a$, $\psi_\a:Z_\a\to[-\infty,+\infty)$ for the transportation problem in $\Pi(\mu_\a,\nu_a)$ and conclude that
\[
\pi_{\a\a} \in \Pi^{\mathrm{opt}}_\mathtt c(\mu_\a,\nu_\a), \quad\text{ for $m$-a.e. $\a$.}
\]

Let us then consider another $\pi'\in \Pi^f_\mathtt c(\mu,\nu)$. After the disintegration w.r.t. $\{Z_\a\times Z_\b\}_{\a,\b \in \A}$ we get
\[
\pi' = \int\pi'_{\a\b}dm'(\a,\b), \quad \pi'_{\a\b}(Z_\a\times Z_\b)=1,
\]
with
\begin{equation}
\label{E_quotient_cost}
m'\in\Pi^f_{(\mathtt h \times \mathtt h)_\# \mathtt c}(m,m), \quad \text{where} \quad (\mathtt h \times \mathtt h)_\# \mathtt c(\mathfrak a,\mathfrak b) = \inf_{Z_\a \times Z_\mathfrak b} \mathtt c(x,y).
\end{equation}

Hence one has the following theorem, which gives a sufficient condition for optimality based on behavior of optimal transport plans w.r.t. disintegration on $(\Gamma,\mathtt c)$-cycle equivalence relations.

\begin{theorem}
\label{T_A1}
Let $\Gamma$ be a $\mathtt c$-cyclically monotone carriage of a transference plan $\pi \in \Pi^f_\mathtt c(\mu,\nu)$. If the partition $\{Z_\a\}_{\a\in\A}$ w.r.t. the $(\Gamma,\mathtt c)$-cycle equivalence relation satisfies
\begin{align}
&\text{the disintegration on $\{Z_\a\}_{\a\in\A}$ is strongly consistent,} \label{E_teoA1} \\
&\pi'\bigg(\underset{\a}{\bigcup}\; Z_\a\times Z_\a\bigg)=1,\quad\forall\,\pi'\in\Pi^f_\mathtt c(\mu,\nu), \label{E_teoA2}
\end{align}
then $\pi$ is an optimal transference plan.
\end{theorem}

Indeed, if \eqref{E_teoA1} and \eqref{E_teoA2} are satisfied, then $\pi'=\int\pi'_{\a\a}\,d(\Id\times\Id)_\#m(\a)$ with $\pi'_{\a\a} \in \Pi^f_\mathtt c(\mu_\a,\nu_\a)$ and one obtains the conclusion by integrating w.r.t. $m$ the optimality of the conditional plans $\pi_{\a\a}$, namely
\[
\int \mathtt c(x,y)\,d\pi_{\a\a}(x,y)\leq\int\mathtt c(x,y)\,d\pi'_{\a\a}(x,y).
\]

The second crucial point in \cite{BiaCar} is then to find weak sufficient conditions such that the assumptions of Theorem \ref{T_A1} are satisfied. 

Before introducing them, we show how the request that the sets of a Borel partition satisfying \eqref{E_teoA2} coincide with the equivalence classes of the $(\Gamma,\mathtt c)$-cycle relation can be weakened, yet yielding the possibility of constructing optimal potentials on each class --and then, as a corollary, to prove the optimality of a $\mathtt c$-cyclically monotone plan $\pi$. 
First, we need the following
\begin{definition}
\label{def:gccyclconn}
 A set $E\subset \mathtt p_1\Gamma$ is \emph{$(\Gamma,\mathtt c)$-cyclically connected} if $\forall\, x,y\in E$ there exists a $(\Gamma,\mathtt c)$-cycle connecting $x$ to $y$.
\end{definition}
According to the above definition, the equivalence classes of $\preccurlyeq_{(\Gamma,\mathtt c)}$ are maximal $(\Gamma,\mathtt c)$-cyclically connected sets, namely $(\Gamma,\mathtt c)$-cyclically connected sets which are maximal w.r.t. set inclusion.

Then notice that, given a Borel partition $\{Z_{\b}'\}_{\b\in\mathfrak B}\subset\R^d$ such that 
\[
 \pi\bigg(\underset{\b}{\bigcup}Z_\b'\times Z_\b'\bigg)=1,\quad\forall\,\pi\in\Pi^f_{\mathtt c}(\mu,\nu)
\]
and whose sets are $(\Gamma,\mathtt c)$-cyclically connected but not necessarily maximal, then it is still possible to define on each of them a pair of optimal potentials and prove the optimality of $\pi$ such that $\pi(\Gamma)=1$.

Moreover, one can weaken this condition by removing a $\mu$-negligible set in the following way.
Let $\mu=\int\mu'_{\b}\,dm'(\b)$, $\mu_{\b}'(Z_\b)=1$.

\begin{definition}
\label{D_gammaconn}
The partition $\{Z_\b'\}_{\b\in\mathfrak B}$ is \emph{$(\mu,\Gamma, \mathtt c)$-cyclically connected} if $\exists\,F\subset X$ $\mu$-conegligible s.t. $Z_\b'\cap F$ is $(\Gamma, \mathtt c)$-cyclically connected $\forall\, \b\in\mathfrak B$. Equivalently, $\exists$ an $m'$-conegligible set $\mathfrak B'\subset \mathfrak B$ s.t. $\forall\,\b'\in\mathfrak B'$ $\exists\,N_\b'\subset Z_\b'$, with` $\mu_{\b}'(N_\b')=0$, s.t. $Z_\b'\setminus N_\b'$ is $(\Gamma, \mathtt c)$-cyclically connected.
\end{definition}

When the $(\mu,\Gamma, \mathtt c)$-cyclically connectedness property holds for all $\mathtt c$-cyclically monotone carriages of all transport plans of finite cost --hence it is possible to construct optimal potentials starting from any $\mathtt c$-cyclically monotone $\Gamma$-- we have the following

\begin{definition}
\label{D_pimunuconn}
 We say that $\{Z_\b'\}$ is \emph{$\Pi^f_\mathtt c(\mu,\nu)$-cyclically connected} if it is $(\mu,\Gamma, \mathtt c)$-cyclically connected $\forall\,\Gamma$ $\mathtt c$-cyclically monotone s.t. $\pi(\Gamma)=1$ for some $\pi\in\Pi^f_\mathtt c(\mu,\nu)$.
\end{definition}

Notice that the $\mu$-conegligible set $F$ in the definition of $(\mu,\Gamma,\mathtt c)$-cyclically connected partition depends on the set $\Gamma$. 

In this paper, in particular for the proof of Theorems \ref{T_final_seconry} and \ref{T_Monge_final}, the importance of $\Pi^f_\mathtt c(\mu,\nu)$-cyclically connected partitions is given by the following proposition.

\begin{proposition}
\label{P_second_cost}
Let $\{Z'_\b\}_{\b\in \B}$ be a $\Pi^f_\mathtt c(\mu,\nu)$-cyclically connected Borel partition satisfying 
\begin{equation}
\label{epseccost}
 \pi\bigg(\underset{\b}{\bigcup}\; Z_\b'\times Z_\b'\bigg)=1,\quad\forall\,\pi\in\Pi^f_\mathtt c(\mu,\nu)
\end{equation}
for a cost function of the form
\begin{equation}
\label{E_cost_ind}
\mathtt c(x,y)=\ind_M(x,y), \qquad M \supset \big\{ (x,x) : x \in X \big\}.
\end{equation}
Let $\mathtt c_\mathtt m : X \times X \to [0,+\infty]$ be any \emph{secondary cost} of the form
\begin{equation}
\label{E_cs}
\mathtt c_\mathtt m(x,y) =
\begin{cases}
\mathtt m(x,y) & \mathtt c(x,y)<+\infty, \\
+\infty & \text{otherwise,}  
\end{cases}
\end{equation}
where $\mathtt m$ is l.s.c. and there exist $\mathtt f \in L^1(\mu),\, \mathtt g \in L^1(\nu)$ s.t. $\mathtt m(x,y) \leq \mathtt f(x) + \mathtt g(y)$. Then, any $\mathtt c_\mathtt m$-cyclically monotone plan $\pi_\mathtt m\in\Pi^f_{\mathtt c_\mathtt m}(\mu,\nu)$ is optimal for $\mathtt c_\mathtt m$. More precisely, for any $\mathtt c_\mathtt m$-cyclically monotone set $\Gamma_\mathtt m$ with $\pi_\mathtt m(\Gamma_\mathtt m)=1$, there exist Borel functions $\phi^{\mathtt m}$, $\psi^\mathtt m$ such that the restrictions
\begin{equation}
\label{E_phipsi_m_restr}
\phi^\mathtt m_\b := \phi^\mathtt m \llcorner_{Z_\b'}, \qquad \psi^\mathtt m_\b := \psi^\mathtt m \llcorner_{Z_\b'}
\end{equation}
are Borel optimal potentials for $\Pi^{\mathrm{opt}}_{\mathtt c_\mathtt m}(\mu_\b',\nu_\b')$, for all $\b$ in an $m'$-conegligible set $\B'\subset \B$.
\end{proposition}

\begin{proof}
Notice that $\Pi^f_{\mathtt c_{\mathtt m}}(\mu,\nu)\subset\Pi^f_\mathtt c(\mu,\nu)$.
Let $\Gamma_\mathtt m\subset \underset{\b}{\cup}\; Z_\b'\times Z_\b'$ be a $\mathtt c_\mathtt m$-cyclically monotone carriage for $\pi_\mathtt m\in \Pi^f_{\mathtt c_{\mathtt m}}(\mu,\nu)$. Then, there exists a conegligible set $F \subset X$ such that $Z_\b' \cap F$ is $(\Gamma_\mathtt m,\mathtt c)$-cyclically connected for all $\b\in \B$. Hence, formula \eqref{E_pot_form}, together with the validity of the Point \eqref{Point_boun_fg} at page \pageref{Point_boun_fg}, yields potentials $\phi^\mathtt m_\b$, $\psi^\mathtt m_\b$ for the transport problem in $\Pi^f_{\mathtt c_{\mathtt m}}(\mu_\b,\nu_\b)$ with cost $\mathtt c_\mathtt m$. In particular, the conditional probability $\pi_{\mathtt m,\b\b}$ is optimal in $\Pi^f_{\mathtt c_{\mathtt m}}(\mu_\b,\nu_\b)$, and thus by \eqref{epseccost} it follows as in Theorem \ref{T_A1} that $\pi_\mathtt m$ is optimal in $\Pi^f_{\mathtt c_{\mathtt m}}(\mu,\nu)$.

The fact that one can find Borel functions $\phi^\mathtt m$, $\psi^\mathtt m$ such that \eqref{E_phipsi_m_restr} holds is an application of standard selection principles, and it can be found in \cite{BiaCar}.
\end{proof}

In order to state the main result of \cite{BiaCar} which is at the core of their sufficient condition concerning optimality, we need the concept of (linear) preorder.

\begin{definition}[(Linear) Preorder]
\label{D_preorder}
A \emph{preorder} on $X$ is a set $A\subset X\times X$ s.t.
\begin{align*}
&(x,x) \in A, \quad \forall\, x \in X\\
&(x,y) \in A \quad \wedge \quad (y,z) \in A \quad \Longrightarrow \quad (x,z) \in A.
\end{align*}
A preorder $A\subset X\times X$ is \emph{linear} if 
\[
X\times X = A \cup A^{-1}.
\]
\end{definition}

The statement $(x,y)\in A$ will also be denoted by $x\preccurlyeq_A y$ and $A$ is also called the \emph{graph of the (linear) preorder $\preccurlyeq_A$}. Any preorder $\preccurlyeq_A$ induces the equivalence relation $\simeq_A$ on $X$
\[
x \simeq_A y \qquad \Longleftrightarrow \qquad x\preccurlyeq_Ay\quad\text{and}\quad y\preccurlyeq_A x.
\]
We also denote the graph of the equivalence relation $\simeq_A$ by 
\[
A \cap A^{-1} \quad \text{ or } \quad \preccurlyeq_A \cap \, (\preccurlyeq_A)^{-1}.
\]

Going back to our problem, one can see that the $(\Gamma,\mathtt c)$-axial relation gives a Borel preorder on $X$, namely
\begin{equation}
\label{E_axpreorder}
x \preccurlyeq_{(\Gamma,\mathtt c)} y \quad\text{ if there exists a $(\Gamma,\mathtt c)$-axial path going from $y$ to $x$.}
\end{equation}

\begin{figure}
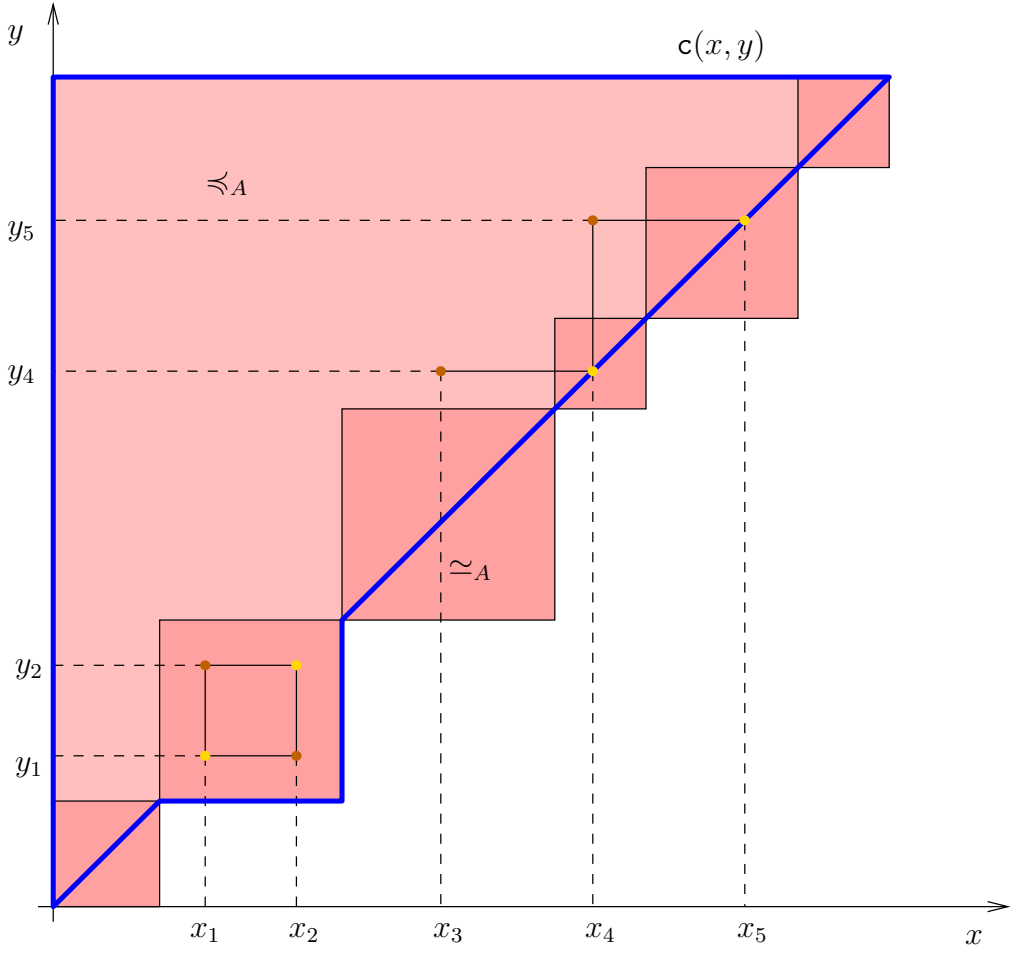
\caption{The graph of the cost $\mathtt c$ is given by the indicator function of the region inside the blue curve. The graph of a $\mathtt c$-compatible linear preorder $\preccurlyeq_A$ is given by the union of the pink and of the red region. The red region corresponds to the graph of the induced equivalence relation $\simeq_A$. We draw also an axial path connecting $x_5$ to $x_3$ with base points $(x_5,y_5)$, $(x_4,y_4)$, and a $(\Gamma, \mathtt c)$-cycle connecting $(x_1,y_1)$ to $(x_2,y_2)$.}
\label{Fi_orderaxial.notat}
\end{figure}

The reason for introducing (linear) preorders in this context is given by the following theorem \cite{BiaCar}.

\begin{theorem}
\label{T_A2}
Let $A \subset X \times X$ be a Borel graph of a linear preorder on $X$ with equivalence classes $\{Z^A_\c\}_{\c \in \mathfrak C}$ 
satisfying
\begin{align}
& \{\mathtt c < +\infty\} \subset A, \label{E_tA22} \\
& \preccurlyeq_{(\Gamma,\mathtt c)} \subset A,\text{ for some $\mathtt c$-cyclically monotone set $\Gamma$ s.t. $\pi(\Gamma)=1$, $\pi\in\Pi^f_{\mathtt c}(\mu,\nu)$.} \label{E_tA21}
\end{align}
Then, the disintegration w.r.t. the partition $\{Z^A_\c\}_{\c \in \mathfrak C}$ is strongly consistent and 
\begin{equation}
\label{E_eqca}
\pi'\biggl( \bigcup_\c Z^A_\c \times Z^A_\c \biggr) = 1, \qquad \forall\,\pi'\in\Pi^f_\mathtt c(\mu,\nu).
\end{equation}
%In particular, if the sets $Z^A_\c$ coincide with the $\preccurlyeq_{(\Gamma,\mathtt c)}$-equivalence classes $\{Z_\a\}$, then \eqref{E_teoA2} holds.
\end{theorem}

For future convenience we give the following definition.

\begin{definition}
\label{D_compatible}
A preorder $\preccurlyeq_A$ on $X$ is \emph{$\mathtt c$-compatible} if \eqref{E_tA22} holds.
\end{definition}

\begin{remark}
\label{rem_pre}
Let $A$ be a $\mathtt c$-compatible linear preorder. Whenever a carriage $\Gamma$ satisfies \eqref{E_tA21} the $\preccurlyeq_{(\Gamma,\mathtt c)}$-equivalence classes are contained in the equivalence classes of $\simeq_A$ and then, as noticed before, since $\Gamma \supset \Graph\,\Id$ and $\mathtt c(x,x)=0$ for all $x$,
\[
\Gamma \subset \underset{\c}{\bigcup} \,Z^A_\c \times Z^A_\c, \quad\pi\bigg(\underset{\c}{\bigcup}Z^A_\c\times Z^A_\c\bigg)=1.
\]
Viceversa, if $\pi'\bigl(\underset{\c}{\cup}Z^A_\c\times Z^A_\c\bigl)=1$ for some $\pi'\in\Pi^f_{\mathtt c}(\mu,\nu)$ and $\pi'(\Gamma')=1$, then by the $\mathtt c$-compatibility of $A$
\[
 \preccurlyeq_{(\Gamma'\cap\underset{\c}{\cup}Z^A_\c\times Z^A_\c,\mathtt c)}\subset A
\]
and then also its equivalence classes are contained in the equivalence classes of $\simeq_A$. 
In particular, \eqref{E_tA21} could also be rewritten as $\pi\bigl(\underset{\c}{\cup}Z^A_\c\times Z^A_\c\bigr)=1$.

We point out that, while a $\mathtt c$-compatible linear preorder satisfying \eqref{E_tA21} for some $\Gamma$ can always be constructed using the axiom of choice, \eqref{E_eqca} may not hold if the linear preorder is not Borel (see \cite{BiaCar}): hence, the main assumption of the theorem is the Borel regularity.
 Finally, notice that the partition into equivalence classes of $\preccurlyeq_{(\Gamma'\cap\underset{\c}{\cup}Z^A_\c\times Z^A_\c,\mathtt c)}$ with $\Gamma'$ as above is $(\mu,\Gamma',\mathtt c)$-cyclically connected in the sense of Definition \ref{D_gammaconn}.
\end{remark}

In order to prove Theorem \ref{T_subpart_step}, in Section \ref{S_cfibr_cfol} we will look --for a particular class of cost functions of the form \eqref{E_cost_ind} called \emph{cone-Lipschitz costs associated to a directed fibration}-- for $\Pi^f_{\mathtt c}(\mu,\nu)$-cyclically connected partitions satisfying \eqref{epseccost}. Therefore, by Theorem \ref{T_A2} and Remark \ref{rem_pre}, we will construct a Borel $\mathtt c$-compatible linear preorder $A$ such that, for any carriage of finite cost $\Gamma'$, the equivalence classes of $\preccurlyeq_{(\Gamma'\cap \underset{\c}{\cup}Z^A_\c\times Z^A_\c,\mathtt c)}$ coincide up to a $\mu$-negligible set with those of $\simeq_A$.

For convenience we give also the following 
\begin{definition}
 \label{D_cmunucomp}
 If $\preccurlyeq_A$ is $\mathtt c$-compatible and \eqref{E_tA21} holds for every $\pi\in\Pi^f_\mathtt c(\mu,\nu)$, then $A$ is called \emph{$(\mathtt c,\mu,\nu)$-compatible}.
\end{definition}
Hence, Theorem \ref{T_A2} can also be restated saying that whenever $A$ is a Borel $\mathtt c$-compatible linear preorder satisfying \eqref{E_tA21} for some $\Gamma$ of finite cost, then it is $(\mathtt c,\mu,\nu)$-compatible.

According to the terminology used in \cite{BiaCar}, $(\mathtt c,\mu,\nu)$-compatibility can also be restated saying that the diagonal in the quotient space
\begin{equation}
\label{E_push_f_A}
(\Id\times\Id)\circ \mathtt h\circ \mathtt p_1 (A)
\end{equation}
is a \emph{set of uniqueness} for $\Pi^f_{(\mathtt h \times \mathtt h)_\# \mathtt c}(m,m)$, where $\mathtt h$ is the quotient map associated to the partition $\simeq_A$: this means that there exists a unique transference plan in $\Pi^f_{(\mathtt h \times \mathtt h)_\# \mathtt c}(m,m)$, namely $(\Id \times \Id)_\# m$.

\section{Optimal transportation problems with convex norm and cone costs}
\label{S_conetransport}

Let $\d{\cdot}:\R^d\to\R$ be a convex norm as defined in \eqref{E_norm} and $\mu,\nu\in\mathcal P(\R^d)$. The transport plans with finite $\d{\cdot}$-cost $\Pi^f_{\d{\cdot}}(\mu,\nu)$ and the optimal plans w.r.t. $\d{\cdot}$ $\Pi^{\mathrm{opt}}_{\d{\cdot}}(\mu,\nu)$ are respectively given by the transference plans with finite cost and the optimal plans w.r.t. the cost function
\begin{equation}
\label{E_norm_cost_1}
\mathtt c(x,y) = \d{y - x}.
\end{equation}
Since the cost is a norm, we have the following well known results \cite{ambrgigli:userguide}: if $\Pi^f_{\d{\cdot}}(\mu,\nu) \not= \emptyset$, then
\begin{enumerate}
\item there exists at least one optimal transference plan $\bar \pi$;
\item if $\Gamma$ is a $|\cdot|_{D^*}$-cyclically monotone carriage of $\bar \pi$, then for $(x_0,y_0) \in \Gamma$ the function given by \eqref{E_pot_form},
\[
\phi(x) := \inf \bigg\{ \sum_{i=0}^I \big| y_i - x_{i+1} \big|_{D^*} - \big| y_i - x_i \big|_{D^*}: I \in \N, (x_i,y_i) \in \Gamma, x_{I+1} = x \bigg\},
\]
is Lipschitz continuous on $\R^d$ and
\begin{equation}
\label{E_potential_norm_leq}
\phi(x) - \phi(y) \leq |y - x|_{D^*},\quad\forall x,y \in \R^d,
\end{equation}
\begin{equation}
\label{E_potential_norm_equal}
\int_{\R^d \times \R^d} |y-x|_{D^*}\, d\bar \pi(x,y) = \int_{\R^d} \phi(x)\, d\mu(x) - \int_{\R^d} \phi(y)\, d\nu(x).
\end{equation}
\end{enumerate}
In particular, $\pi$ is an optimal plan if and only if 
\[
\pi \Big( \Big\{ (x,y) : \phi(x) - \phi(y) = |y - x|_{D^*} \Big\} \Big) = 1.
\]

In the following we will denote by $\psi$ the dual potential
\begin{equation}
\label{E_dual_pot_psi}
\psi(x) := - \phi(x),
\end{equation}
which will be called \emph{Kantorovich potential}.
Clearly
\begin{equation}
\label{E_psi_forward_norm}
\psi(y) - \psi(x) \leq |y-x|_{D^*}, \qquad \pi \in \Pi^{\mathrm{opt}}_{\d{\cdot}}(\mu,\nu) \ \Leftrightarrow \ \pi \Big( \Big\{ (x,y) : \psi(y) - \psi(x) = |y - x|_{D^*} \Big\} \Big) = 1.
\end{equation}

\begin{definition}
\label{D_super_sub_diff}
A function $\varphi : \dom\, \varphi\subset\R^d\to\R$ is \emph{$\d{\cdot}$-Lipschitz} if it satisfies
\[
\varphi(y) - \varphi(x) \leq \d{y - x}, \qquad \forall\,x,y\in\dom\,\varphi.
\]
The \emph{superdifferential} of $\varphi$ is the set
\[
\partial^+\varphi := \Big\{ (x,y) : \varphi(y) - \varphi(x) = |y - x|_{D^*} \Big\},
\]
while its \emph{subdifferential} is the set
\[
\partial^-\varphi := \big(\partial^+\varphi\big)^{-1}.
\]
\end{definition}

Hence, \eqref{E_psi_forward_norm} can be rephrased as
\begin{equation}
\label{E_psi_opt}
\exists\,\psi : \R^d \to \R \text{ $\d{\cdot}$-Lipschitz s.t. }\pi \in \Pi^{\mathrm{\opt}}_{\d{\cdot}}(\mu,\nu) \ \Leftrightarrow \ \pi \big( \partial^+\psi \big) = 1. 
\end{equation}

% \vskip 0.2 cm

Let now $C^k\in\mathcal C(k;\R^k)$.

\begin{definition}
\label{D_cone_cost}
We define the \emph{convex cone cost associated to $C^k$} as the function $\mathtt c_{C^k}:\R^k\times\R^k\to[0,+\infty]$ given by
\begin{equation}
\label{E_cone_cost}
\mathtt c_{C^k}(x,y) =
\begin{cases}
0 & y-x \in C^k, \\
+\infty & \text{otherwise.}
\end{cases}
\end{equation}
\end{definition}

Given $\mu,\nu \in \mathcal P(\R^k)$, let $\Pi^f_{\mathtt c_{C^k}}(\mu,\nu)$ be the set of transport plans of finite cone cost. Notice that 
\[
\Pi^f_{\mathtt c_{C^k}}(\mu,\nu)=\Pi^{\mathrm{opt}}_{\mathtt c_{C^k}}(\mu,\nu)= \Big\{ \pi\in\Pi^f_{\mathtt c_{C^k}}(\mu,\nu):\,\pi\text{ is $\mathtt c_{C^k}$-cyclically monotone} \Big\}.
\]
 
\subsection{Transportation problems with convex norms and cone costs on Lipschitz graphs}
\label{Ss_convex_norm_cone}

The optimal transport problem w.r.t. $\d{\cdot}$ can be casted as a convex cone optimal transportation problem on $\R^{d+1}\simeq\R^d\times\R$ w.r.t. the convex cone cost $\mathtt c_{\mathrm{epi}\,{\d{\cdot}}}$ associated to
\[
\epi\,\d{\cdot}\in\mathcal C(d+1;\R^{d+1})
\]
(see Definition \ref{D_cone_cost} and Remark \ref{R_cone_epi}).
Define in fact the measures in $\mathcal P(\Graph\,\psi)$
\begin{equation}
\label{E_measure_on_graph}
\hat \mu := (\Id\times\psi)_\# \mu, \qquad \hat \nu := (\Id\times\psi)_\# \nu,
\end{equation}
where $\psi$ is the Kantorovich potential of $\Pi^{\mathrm{opt}}_{\d{\cdot}}(\mu,\nu)$, formula \eqref{E_dual_pot_psi}, and for $\pi \in \Pi(\mu,\nu)$ consider the plan in $\mathcal P(\Graph\,\psi \times \Graph\,\psi)$
\begin{equation}
\label{E_graph_plan}
\hat \pi := \big( (\Id\times\psi) \times (\Id\times\psi) \big)_\# \pi.
\end{equation}

The fundamental observations are \eqref{E_psi_opt} and the following: if $\varphi$ is $\d{\cdot}$-Lipschitz, then 
\begin{equation}
\label{E_partvarphi_form}
\partial^+\varphi = \mathtt p_{\R^d\times\R^d} \Big( \Graph\,\varphi \times \Graph\,\varphi \cap \bigl\{ \mathtt c_{\epi\,\d{\cdot}} < +\infty \bigr\} \Big).
\end{equation}

\begin{definition}
\label{D_subsuperdiff}
If $\Graph\,\varphi\subset\R^{d+1}$ is the graph of a $\d{\cdot}$-Lipschitz function, define its \emph{superdifferential} and \emph{subdifferential} respectively as
\begin{equation}
\label{E_sdiff_varphi}
\partial^+\Graph\,\varphi=\Graph\,\varphi\times\Graph\,\varphi\cap\big\{\mathtt c_{\epi\,\d{\cdot}}<+\infty\big\},\qquad\partial^-\Graph\,\varphi=\big(\partial^+ \Graph\,\varphi\big)^{-1}.
\end{equation}
\end{definition}

Then \eqref{E_partvarphi_form} can be rewritten as
\[
\partial^\pm\varphi=\mathtt p_{\R^d\times\R^d} \big( \partial^\pm\Graph\,\varphi \big).
\]

Hence the following proposition holds true.

\begin{figure}
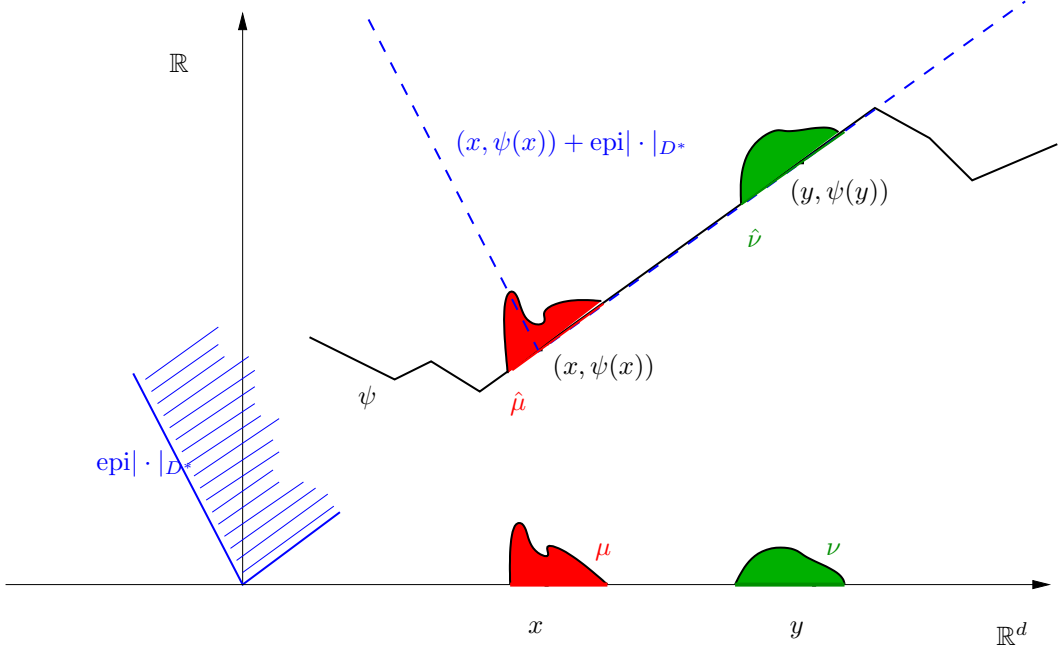
\caption{The equivalence of Proposition \ref{P_equivalence_lifting}.}
\label{Fi_sudakov.lifting}
\end{figure}

\begin{proposition}
\label{P_equivalence_lifting}
The following statements are equiveridical:
\begin{enumerate}
\item \label{PCond_1_setting} \hskip 1cm $\pi \in \Pi^{\mathrm{opt}}_{\d{\cdot}}(\mu,\nu)$;
\item \label{PCond_2_setting} \hskip 1cm $\hat \pi := \big( (\Id\times\psi) \times (\Id\times\psi) \big)_\# \pi \in \Pi^f_{\mathtt c_{\epi\,\d{\cdot}}}(\hat \mu,\hat \nu)$, with $\hat \mu$, $\hat \nu$ given by \eqref{E_measure_on_graph};
\item \label{PCond_3_setting} \hskip 1cm $\pi = \big( \mathtt p_{\R^d \times \R^d} \big)_\# \hat \pi$ for some $\hat \pi \in \Pi^f_{\mathtt c_{\epi\,\d{\cdot}}}(\hat \mu,\hat \nu)$, with $\hat \mu$, $\hat \nu$ given by \eqref{E_measure_on_graph};
\item \label{E_hat_pi_graph} \hskip 1cm $\hat \pi := \big( (\Id\times\psi) \times (\Id\times\psi) \big)_\# \pi$ satisfies $\hat{\pi}(\partial^+\Graph\,\psi) = 1$.
\end{enumerate}
\end{proposition}

Observe that, since $(\Id\times\psi):\R^d\to\Graph\,\psi$ is bi-Lipschitz, then if $\varpi \in \mathcal P(X)$ and $\hat \varpi := (\Id \times \psi)_\# \varpi$,
\begin{equation}
\label{E_graphpsimu}
\varpi(B) = 0 \quad \Longleftrightarrow \quad \hat \varpi \bigl( (\Id\times\psi)(B) \bigr) = 0, \quad \forall\, B \in \mathcal B(\R^d).
\end{equation}

\subsection{Optimal transportation problems on directed locally affine partitions}
\label{Ss_partitions_intro}

We first give the definition of directed locally affine partition.

\begin{definition}
\label{D_locaff}
We say that a nonempty subset $Z \subset \R^d$ is \emph{locally affine} if there exist $k\in\{0,\dots,d\}$ and $V \in \mathcal A(k,\R^d)$ such that $Z \subset V$ and $Z$ is relatively open in $V$, i.e. $Z=\interr Z \not= \emptyset$.
\end{definition}

Notice that, in the above definition, $V=\aff\,Z$. Whenever $Z$ is a locally affine set of dimension $k$ we will often denote it as $Z^{k}$ to emphasize its dimension.

\begin{definition}
\label{D_locaffpart}
A \emph{directed locally affine partition} in $\R^d$ is a partition into locally affine sets $\{Z^k_\mathfrak a\}_{\nfrac{k = 0,\dots,d}{\mathfrak a \in \mathfrak A^k}}$, endowed with a family of closed nondegenerate convex cones $\{C^k_\mathfrak a\}_{\nfrac{k = 0,\dots,d}{\mathfrak a \in \mathfrak A^k}}$ such that 
\begin{enumerate}
\item the set
\begin{equation}
\label{E_bD}
\bD:=\Big\{ (k,\a,z,C\ka) : k\in\{0,\dots,d\},\, \a \in \A^k,\, z\in Z\ka \Big\}\subset \bigcup_{k=0}^{d} \{k\} \times \A^k\times\R^d\times\mathcal C(k,\R^d)
\end{equation}
is $\sigma$-compact;
\item $\aff\,(z+C\ka)=\aff(Z\ka)$ for all $z\in Z\ka$.
\end{enumerate}

\end{definition}

For shortness we will use the notation
\begin{align}
\label{E_mathbf_Z_base_partition}
\mathbf Z^k &:= \mathtt p_z \mathbf D(k) = \bigcup_{\mathfrak a \in \mathfrak A^k} Z^k_\mathfrak a, \qquad \mathbf Z := \mathtt p_z \mathbf D = \bigcup_k \mathbf Z^k = \bigcup_{k=0}^d \bigcup_{\mathfrak a \in \mathfrak A^k} Z^k_\mathfrak a,\qquad
\bar{\mathbf Z}^k:=\underset{\a\in\A^k}{\bigcup}\clos\,Z\ka.
\end{align}
For the conditional probabilities of a measure $\mu$ over a locally affine partition we will use the notation $\{\mu^k_\a\}_{\nfrac{k=0,\dots,d}{\a\in\A^k}}$, with $\mu^k_\a(Z^k_\a)=1$: the fact that the disintegration is strongly consistent is a consequence of the fact that the function $\mathbf Z\ni z \mapsto (k,\a)$ has $\sigma$-compact graph $\mathtt p_{(z,k,\a)} \mathbf D$.
% 
% \noindent 
Notice that the quotient space of the partition is given by
\begin{equation}
\label{E_disj_unio_A}
\A := \bigsqcup_{k} \A^k,
\end{equation}
where $\sqcup$ denotes the disjoint union of sets.

Given a locally affine directed partition $\{Z\ka, C\ka\}_{k,\a}$ one can define the sets of \emph{initial} and \emph{final points} as follows.

\begin{definition}
\label{D_initial_final}
Define for $k = 1,\dots,d$, $\a \in \A^k$ the \emph{initial points of $Z\ka$} as
\begin{equation*}
% \label{E_initial0}
\mathcal I(Z\ka) := \Bigl\{ z \in \R^d \setminus \mathbf Z : \exists\,r>0 \text{ s.t. } z + \interr C\ka \cap B^d(z,r) \subset Z\ka \Bigr\},
\end{equation*}
and the \emph{final points of $Z\ka$} as
\begin{equation*}
% \label{E_final0}
\mathcal E(Z\ka) := \Bigl\{ z \in \R^d \setminus \mathbf Z : \exists\,r>0 \text{ s.t. } z - \interr C\ka \cap B^d(z,r) \subset Z\ka \Bigr\}.
\end{equation*}
Finally, we call \emph{sets of initial points} and \emph{sets of final points} of the locally affine directed partition $\{Z\ka, C\ka\}_{k,\a}$ the sets given respectively by
\begin{equation}
\label{E_initial_final}
\mathcal I := \bigcup_{k,\a} \mathcal I(Z\ka), \qquad \mathcal E := \bigcup_{k,\a} \mathcal E(Z\ka).
\end{equation}
\end{definition}

\begin{figure}
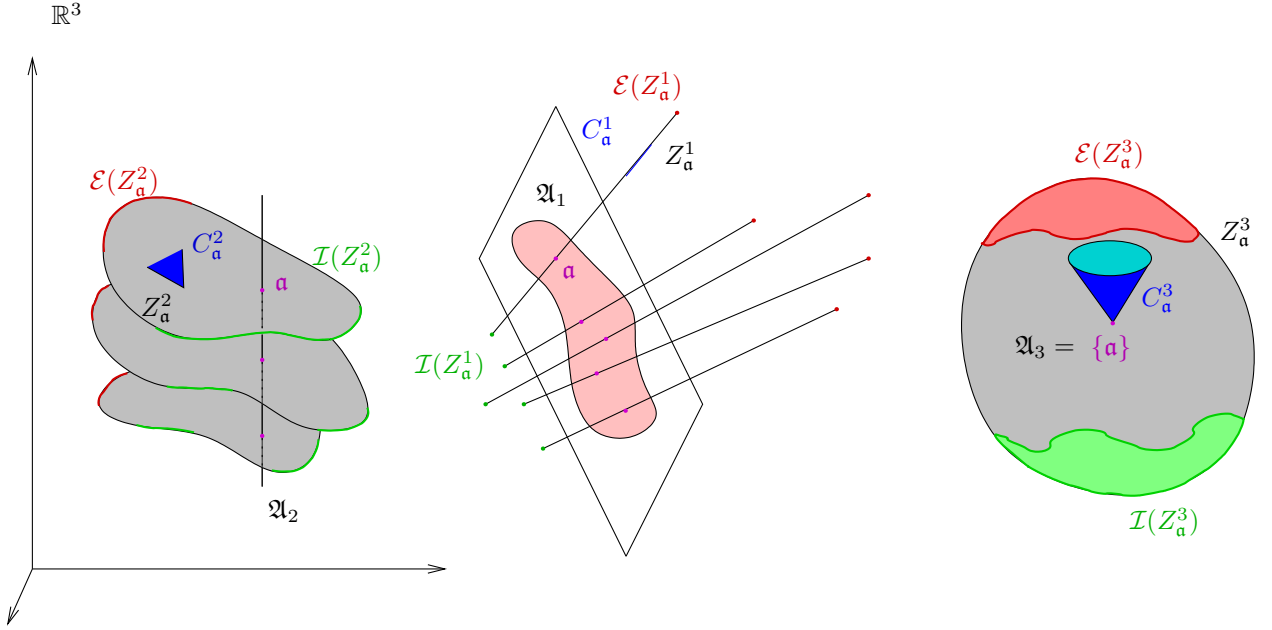
\caption{A directed locally affine partition in $\R^3$ into three $2$-dimensional sets with quotient space $\A^2$, five $1$-dimensional sets with quotient space $\A^1$ and a $3$-dimensional set with quotient space $\A^3$. In each $k$-dimensional subpartition, for $k=1,\dots,3$, we denote a locally affine set as $Z^k_\a$ with cone of directions $C^k_\a$ (colored in blue), initial points $\mathcal I(Z\ka)$ (colored in green), final points $\mathcal E(Z\ka)$ (colored in red) and quotient point $\a$ (colored in purple), chosen as in \eqref{E_mathfrak_A_k_def}.}
\label{Fi_direct_fibra}
\end{figure}

Notice that the sets $\mathcal I(Z\ka)$, $\mathcal I(Z^{k'}_{\a'})$ do not need to be disjoint even if $k\neq k'$ and $\a\neq \a'$, and the same for $\mathcal E(Z\ka)$, $\mathcal E(Z^{k'}_{\a'})$.
Moreover,
\begin{equation}
\label{E_infin_incl}
\mathcal I(Z\ka) \cup \mathcal E(Z\ka)\subset \partial_{\mathrm{rel}}Z\ka,
\end{equation}
but the inclusion \eqref{E_infin_incl} may be strict (see Figure \ref{Fi_direct_fibra} and Figure \ref{Fi_infin}). 
The measurability of the sets of initial/final points is proven in the Lemma \ref{L_regularity_initial_final_points}. In the proof we use the concept of \emph{completeness} of a directed locally affine partition, whose meaning will be clear in Section \ref{S_foliations} and whose definition is given below. Since up to that section, when it will become crucial for our analysis of the super/subdifferential partitions, such a property will be used only in order to prove measurability issues, more precisely in the proofs of Lemma \ref{L_regularity_initial_final_points} and of Proposition \ref{P_countable_partition_in_reference_directed_planes}, a deeper understanding of its meaning is up to then not necessary and can be for the moment neglected.

\begin{definition}
\label{D_part_compl}
 A directed locally affine partition $\{Z\ka,C\ka\}_{k,\a}$ is \emph{complete} if
 \begin{equation}
 \label{E_part_compl}
  x+C\ka\,\cap\, y-C\ka\subset Z\ka,\qquad\forall\,x,y\in Z\ka.
 \end{equation}
\end{definition}
In the proof of Proposition \ref{P_parall}, we will see that the set defined in \eqref{E_part_compl} is a convex set satisfying
\[
 \R^+\bigl((x+C\ka\,\cap\, y-C\ka)-x\bigr)=\R^+\bigl(y-(x+C\ka\,\cap\, y-C\ka)\bigr)=C\ka.
\]

\begin{definition}
\label{D_dir_subpart}
We will say that a directed locally affine partition $\{Z^{',\ell}_\mathfrak b,\,C^{',\ell}_\mathfrak b\}_{\nfrac{\ell = 0,\dots,d}{\mathfrak b \in \mathfrak B^\ell}}$ in $\R^d$ is a \emph{directed locally affine subpartition} of $\{Z^k_\mathfrak a,\,C^k_\mathfrak a\}_{\nfrac{k = 0,\dots,d}{\mathfrak a \in \mathfrak A^k}}$ if the following holds:
\begin{enumerate}
\item \label{E_same_base} $\mathbf Z=\mathbf Z'$, where $\mathbf Z'$ is the set given by \eqref{E_mathbf_Z_base_partition} for $\{Z^{',\ell}_\mathfrak b,\,C^{',\ell}_\mathfrak b\}_{\ell,\mathfrak b}$;
\item $\forall\,\ell,\b$ there exists $k,\a$ s.t. $Z^{',\ell}_\b\subset Z^k_\a$ and $C^{',\ell}_\b$ is an extremal face of $C^k_\a$.
\end{enumerate}
\end{definition}

\begin{definition}
\label{D_disint_regular} 
We say that a locally affine ($\sigma$-compact) partition $\{Z^k_{\a}\}_{\nfrac{k = 0,\dots,d}{\mathfrak a \in \mathfrak A^k}}$ is \emph{Lebesgue-regular} if the conditional probabilities $\{\upsilon^k_\a\}_{k,\a}$ of the disintegration of $\LL$ on the partition $\{Z^k_{\a}\}_{k,\a}$ (see Remark \ref{R_disint_lebesgue}) satisfy
\begin{equation}
\label{E_disint_regular}
\upsilon^k_\a \simeq \HH^k \llcorner_{Z^k_\a}, \quad \text{for $\eta$-a.e. $(k,\a) \in \A$}.
\end{equation}
\end{definition}

From the definition of disintegration of a Radon measure given in Remark \ref{R_disint_lebesgue}, it is not difficult to check that the validity of \eqref{E_disint_regular} is independent on the partition into unit measure sets $\{A_i\}$, hence Definition \ref{D_disint_regular} is consistent.

%Our main issue is the analysis of the optimal transportation problem which models the displacement of a probability measure inside the sets of a directed locally affine partition, along the directions of the associated cones. 

% \vskip 0.2 cm

To a directed locally affine partition $\{Z^k_\mathfrak a,\,C^k_\mathfrak a\}_{\nfrac{k = 0,\dots,d}{\mathfrak a \in \mathfrak A^k}}$ in $\R^d$, we associate the cost function 
\begin{equation}
\label{E_c_bD}
\mathtt c_{\bD}(x,y) :=
\begin{cases} %\left\{\begin{aligned}
0 & x\in Z\ka, \mathtt c_{C\ka}(x,y)<+\infty \ \text{for some} \ (k,\a)\in\A, \\
+\infty & \text{otherwise.}
%\end{aligned}\right.
\end{cases}
\end{equation}
Notice that, since $\bD$ is $\sigma$-compact, $\mathtt c_{\mathbf D}$ is $\sigma$-continuous. Indeed, 
\[
\{\mathtt c_{\mathbf D}<+\infty\} = \mathtt p_{x,y} \bigl( \big\{ (k,\a,x,y), (k,\a,x,y-x)\in\mathbf D \big\} \bigr).
\]

Let us consider $\mu,\nu\in\mathcal P(\R^d)$ satisfying
\[
\quad\Pi^f_{\mathbf c_{\bD}}(\mu,\nu)\neq\emptyset.
\]
By definition of $\mathtt c_{\bD}$, one can easily see that $\mu(\mathbf Z)=1$ and 
\[
\Pi^{\mathrm{opt}}_{\mathtt c_{\bD}}(\mu,\nu) = \Pi^f_{\mathtt c_{\bD}}(\mu,\nu)= \Big\{ \pi\in\Pi^f_{\mathtt c_{\bD}}(\mu,\nu):\,\pi\text{ is $\mathtt c_{\bD}$-cyclically monotone} \Big\}.
\] Let
\[
\mu=\int\mu\ka\,dm(k,\a), \qquad \mu\ka(Z\ka)=1,
\]
be the disintegration of $\mu$ w.r.t. the partition $\{Z\ka\}_{k,\a}$.

We have the following characterization.

\begin{proposition}
\label{P_dispiani}
$\pi\in\Pi^f_{\mathtt c_{\mathbf{D}}}(\mu,\nu)$ if and only if the strongly consistent disintegration $\{\pi^k_\a\}_{k,\a}\subset\mathcal P(\R^d\times\R^d)$ of $\pi$ w.r.t. the partition $\{Z^k_\a\times\R^d\}_{k,\a}$ satisfies the following properties:
\begin{subequations}
\label{E_L_dispiani}
\begin{equation}
\label{E_L_dispiani2}
\pi^k_\a \in \Pi^f_{\mathtt c_{C^k_\a}}(\mu^k_\a,(\p_2)_\#\pi^k_\a) \quad \text{for $m$-a.e. $(k,\a)$},
\end{equation}
\begin{equation}
\label{E_L_dispiani3}
\int (\mathtt p_2)_\# \pi\ka\,dm(k,\a) = \nu,
\end{equation}
where the measure on the l.h.s. of \eqref{E_L_dispiani3} is defined as in \eqref{E_int_measure}. 
\end{subequations}

\end{proposition}

\begin{proof}
If $\pi \in \Pi^f_{\mathtt c_{\bD}}(\mu,\nu)$, then up to an $m$-negligible set one has $\pi\ka \in \Pi^f_{\mathtt c_{\bD}}(\mu\ka,(\p_2)_\# \pi\ka)$, and since $\mathtt c_{\bD} \llcorner_{Z\ka\times\R^d} = \mathtt c_{C\ka}$ one deduces \eqref{E_L_dispiani2}. The equality \eqref{E_L_dispiani3} is a fairly easy consequence of $(\p_2)_\# \pi = \nu$.

Conversely, if $\pi$ satisfies \eqref{E_L_dispiani}, then the two formulas
\[
(\p_2)_\# \pi = \int (\p_2)_\# \pi\ka dm(k,\a), \quad \int \mathtt c_{\bD} d\pi = \int \bigg( \int \mathtt c_{\bD} d\pi\ka \bigg) dm(k,\a) = \int \bigg( \int \mathtt c_{C\ka} d\pi\ka \bigg) dm(k,\a)
\]
yield $\pi \in \Pi^f_{\mathtt c_{\bD}}(\mu,\nu)$.
% 
% If $\nu$ is concentrated in $\mathbf Z$, then from the 
\end{proof}

In other words any optimal transference plan w.r.t. the cost associated to a directed locally affine partition can be decomposed as a family of transference plans on the $k$-dimensional affine hulls of the $k$-dimensional sets of the partition, moving the mass along the cones of directions, and viceversa it can be reconstructed given a family $\{\pi^k_\a\}_{k,\a}$ satisfying \eqref{E_L_dispiani2}-\eqref{E_L_dispiani3}.

In general \eqref{E_L_dispiani3} is not a disintegration (see Section \ref{Ss_measure_disintegration} before Remark \ref{R_disint_lebesgue}), as the following example shows.

\begin{example}
\label{Ex_2ndmarg}

\begin{figure}
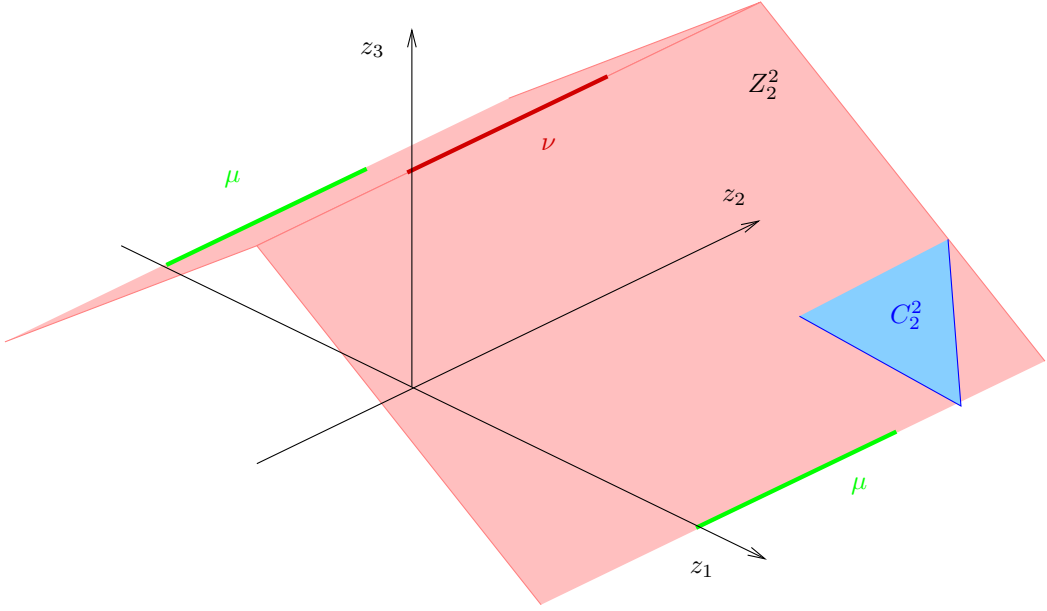
\caption{The transport problem on the directed locally affine partition described in Example \ref{Ex_2ndmarg}.}
\label{Fi_notuniqe}
\end{figure}

For $d=3$ let
\[
\mu := \mathcal H^1 \llcorner_{\{-1\} \times [0,1/2] \times \{0\} \cup \{1\} \times [0,1/2] \times \{0\}}, \quad \nu := 2 \mathcal H^1 \llcorner_{\{0\} \times [0,1/2] \times \{1\}},
\]
and consider the directed locally affine partition
\[
Z^2_1 := \big\{ (z_1,z_2,z_3), z_1 < 0, z_1 = z_3 - 1 \big\}, \quad C^2_1 := \big\{ (z_1,z_2,z_3) : |z_2| \leq z_1, z_1 = z_3 \big\},
\]
\[
Z^2_2 := \big\{ (z_1,z_2,z_3), z_1 > 0, z_1 = - z_3 + 1 \big\}, \quad C^2_2 := \big\{ (z_1,z_2,z_3) : |z_2| \leq - z_1, z_1 = - z_3 \big\}.
\]
Then, for every decomposition $2\nu = \nu_1 + \nu_2$ with $\nu_1, \nu_2 \in \mathcal P(\R^2)$, 
\[
\Pi(\mu,\{\nu_1,\nu_2\}) := \Bigl\{ \pi \in \Pi^f_{\mathtt c_{\mathbf D}}(\mu,\nu) : (\p_2)_{\#}\pi^2_1 = \nu_1,\,(\p_2)_{\#}\pi^2_2 = \nu_2 \Bigr\} \neq \emptyset,
\]
and clearly $\Pi(\mu,\{\nu_1,\nu_2\}) \subset \Pi^f_{\mathtt c_\bD}(\mu,\nu)$.
\end{example}

Example \ref{Ex_2ndmarg} motivates the following definition.

\begin{definition}
\label{D_pinu}
Given a transference plan $\bar\pi\in\Pi^f_{\mathtt c_{\mathbf D}}(\mu,\nu)$, we define the \emph{conditional second marginals} of $\bar\pi$ w.r.t. $\{Z\ka,C\ka\}_{k,\a}$ as
\[
\bar\nu^k_\a := (\p_2)_\#\bar\pi\ka, \quad \text{ for } (k,\a) \in \A. % $k=0,\dots,d$, $\a\in\A^k$}.
\]
We also set
\begin{equation}
\label{E_optplan_dirpar}
\Pi^f_{\mathtt c_\bD}(\mu,\{\bar\nu\ka\}) = \Bigl\{ \pi\in\Pi^f_{\mathtt c_{\mathbf D}}(\mu,\nu):\, (\p_2)_\#\pi\ka=\bar\nu^k_\a \text{ for $m$-a.e. }(k,\a) \Bigr\},
\end{equation}
and we call \eqref{E_optplan_dirpar} the set of \emph{optimal transport plans on the directed locally affine partition $\bD$} w.r.t. $\mu$ and $\{\bar\nu\ka\}$.
\end{definition}

Therefore, in the following by \emph{optimal transportation problem on a directed locally affine partition $\mathbf D$} we mean an optimal transportation problem w.r.t. the cost $\mathtt c_{\bD}$ between measures $\mu$ and $\{\bar\nu\ka\}_{k,\a}$, being the latter \emph{admissible second marginals}, namely  conditional second marginals of at least one transference plan $\pi\in \Pi^f_{\mathtt c_\bD}(\mu,\nu)$.

Notice that in Example \ref{Ex_2ndmarg} the existence of more than one family of admissible second marginals for the given optimal transportation problem would be avoided provided 
\begin{equation}
\label{E_nuz1}
\nu(\mathbf Z)=1. 
\end{equation}
What \eqref{E_nuz1} implies in general is that any family of admissible second marginals $\{\bar\nu\ka\}$ is given by a relabeling of the disintegration of $\nu$ on $Z\ka$, but it may not necessarily occur that 
\begin{equation}
 \label{E_barnuza1}
 \bar\nu\ka(Z\ka)=1
\end{equation}
(see Figure \ref{Fi_notuniqe}). In the next proposition we give a criterion --namely, condition \eqref{E_more_than_complet}-- in order that condition \eqref{E_barnuza1} is satisfied and then there exists just one family of admissible second marginals. Condition \eqref{E_more_than_complet} will be indeed satisfied by directed locally affine partitions called fibrations and $\mathtt c_{\tilde{\mathbf C}}$-foliations given by single Lipschitz graphs (see Corollary \ref{C_plan_fibr} and Proposition \ref{P_hat_bf_D_graph}).

\begin{proposition}
\label{P_dispiani_2}
Assume that %$C\ka$ are extreme faces of a non degenerate cone $C \in \mathcal C(d,\R^d)$, and that
\begin{equation}
\label{E_more_than_complet}
z \in Z\ka, z' \in Z^{k'}_{\a'} \ \text{for} \ (k,\a) \not= (k',\a') \quad \Longrightarrow \quad \mathbf Z \cap ( z + C\ka ) \cap ( z' + C^{k'}_{\a'} ) = \emptyset.
\end{equation}
Hence,
\begin{equation}
 (\p_2)_\#\pi\ka\Big(Z\ka\cup\R^d\setminus\underset{(\a',k')\neq(\a,k)}{\bigcup}Z^{k'}_{\a'}\Big)=1, \quad\text{for all $\pi\in\Pi^f_{\mathtt c_{\mathbf{D}}}(\mu,\nu)$, $m$-a.e. $(k,\a)\in\A$.}
\end{equation}

Moreover, if $\nu(\mathbf Z)=1$, one has that
\begin{equation*}
% \label{E_vac}
(\mathtt p_2)_\#\pi\ka = \nu\ka,\quad \text{being $\nu = \int \nu\ka\,dm(k,\a)$, $\nu\ka(Z\ka)=1$ the disintegration of $\nu$ w.r.t. $\mathbf D$.}
\end{equation*}
\end{proposition}

Hence the conditional second marginals of $\pi \in \Pi^f_{\mathtt c_{\mathbf{D}}}(\mu,\nu)$ are equal to the conditional probabilities of $\nu$, computed via disintegration on $Z\ka$. Notice that part of the statement is that the quotient measure $m$ of $\nu$ is the same as for $\mu$.

\begin{proof}
It is fairly easy to see that \eqref{E_more_than_complet} implies that
\[
\mathbf Z \times \R^d\cap \big\{ \mathtt c_{\mathbf D} < +\infty \big\} \subset \bigcup_{\a,k} Z\ka \times \Big(Z\ka\cup\R^d\setminus \mathbf Z\Big),
\]
so that each $\pi \in \Pi^f_{\mathtt c_{\bD}}(\mu,\nu)$ is concentrated on
\[
\bigcup_{\a,k} Z\ka \times \Big(Z\ka\cup\R^d\setminus \mathbf Z\Big),
\]
and this concludes the proof.
\end{proof}

\subsection{From directed partitions to directed fibrations}
\label{Ss_mapping_sheaf_to_fibration}

In the first part of this section we show that a directed locally affine partition is a countable union of directed locally affine partitions whose elements are locally affine sets having the same dimension and whose direction cones are ``close'' to a fixed reference cone. This kind of partitions will be called \emph{sheaf sets}. Then, we will see that the optimal transportation problem on a \emph{$k$-directed sheaf set} --with $k$ denoting the dimension of its locally affine sets-- can be equivalently reformulated as an optimal transportation problem on a $k$-directed sheaf set whose sets are contained in distinct parallel $k$-dimensional planes, called \emph{$k$-directed fibration}. The advantage of this reformulation is that on a $k$-directed fibration all the supports of the second marginals are disjoint, condition \eqref{E_more_than_complet} holds and then (by Corollary \ref{C_plan_fibr}) one can consider the quotient variables of the partition as parameters for a family of independent convex-cone optimal transportation problems in $\R^k$. 
% Moreover, the first and second conditional marginals of the problem on the sheaf set will be equivalent to the corresponding first and second marginals on the directed fibration. 

Since $0$-dimensional sets, i.e. single points in $\R^d$, are obviously not further partitionable, from now on we will consider partitions into sets of dimension $k\geq1$.

Let $\{Z\ka, C\ka\}_{\nfrac{k = 1,\dots,d}{\mathfrak a \in \mathfrak A^k}}$ be a locally affine directed partition in $\R^d$. If $\{\e^k_i\}_{i=1}^k$ are vectors in $\R^d$, define the sets 
\begin{equation}
\label{E_cone_e_k_i}
C(\{\e^k_i\}) := \bigg\{ \sum_{i=1}^k t_i \e^k_i:\, t_i \in \R^+ \cup \{0\} \bigg\}
\end{equation}
and 
\begin{equation}
\label{E_U_square_mathrm_e}
U(\{\e^k_i\}) := \bigg\{ \sum_{i=1}^k t_i \e^k_i:\, t_i \in [0,1]\bigg\}.
\end{equation}

Recalling the definitions given in Sections \ref{Ss_intro_affine_subspaces_cones} and \ref{Ss_partitions_intro}, and the completeness property of Definition \ref{D_part_compl}, we have the following result.
%Sara{ho rimosso i coni di riferimento: alla fine servono solo per dire che C(\{\e^k_i\}) e' uniformemente contenuto nelle proiezioni dei coni e nelle tecniche di disintegrazione si tagliano gli sheaf sets lungo direzioni di C(\{\e^k_i\}).}

\begin{proposition}
\label{P_countable_partition_in_reference_directed_planes}
There exists a countable covering of $\mathbf D$ into disjoint directed locally affine partitions $\{\mathbf D^k_n\}_{\nfrac{k = 1,\dots,d}{n \in \N}}$, with the following properties: $\forall\,n\in\N$., set
\[
\mathfrak A^k_n := \mathtt p_\mathfrak a \mathbf D_n^k,  \quad C\ka = \mathtt p_{\mathcal C(k,\R^d)} \bD^k_n(\a),\quad Z\ka=\p_{\R^d}\mathbf D_n^k(\a).
\]
Then, $\mathtt p_{\{1,\dots,k\}}\mathbf D^k_n=\{k\}$ for all $n \in \N$, i.e. the elements $Z\ka$, $C\ka$ have linear dimension $k$, for $\a \in \A^k_n$, and there exist 
\begin{itemize}
\item linearly independent vectors $\{\e^k_i(n)\}_{i=1}^k\subset \mathbb S^{d-1}$, with linear span
\[
V^k_n=\langle \e_1^k(n),\dots,\e_k^k(n)\rangle,
\]
\item a given point $z^k_n \in V^k_n$,
\item constants $r^k_n,\lambda^k_n \in \R^+$,
\item a non degenerate cone $C^k \in \mathcal C(k,\R^d)$, with $C(\{\e_i^k\}) \subset \mathring C^k$,
\end{itemize}
such that it holds:
\begin{enumerate}
\item \label{Point_1_Proposition_P_countable_partition} the enlargement of the cone $C^k$ by a factor $2r^k_n$ is non-degenerate
\[
C^k(2r^k_n) \in \mathcal C(k,V^k_n);
\]
\item \label{Point_2_Proposition_P_countable_partition} the projections on $V^k_n$ of the cones $C\ka$, $\a \in \A^k_n$, have a uniform opening containing $C^k$
\[
C^k(r^k_n)\subset \mathtt p_{V^k_n} \mathring C^k_\mathfrak a;
\]
\item \label{Point_3_Proposition_P_countable_partition} the projections on $V^k_n$ of the cones $C\ka$, $\a \in \A^k_n$, are strictly contained in the given cone $C^k(2r^k_n)$
\[
\mathtt p_{V^k_n} C^k_\mathfrak a \subset C^k(2r^k_n); %$ for all $\mathfrak a \in \mathfrak A^k_n$;
\]
\item \label{Point_4_Proposition_P_countable_partition} the projection map on $V^k_n$ is nondegenerate
\[
| \mathtt p_{V^k_n} z | \geq 1/\sqrt{2} \quad \text{ for all $z \in C^k_\mathfrak a\cap \mathbb S^{d-1}$, $\mathfrak a \in \mathfrak A^k_n$};
\]
\item the projection of $Z\ka$ on $V^k_n$ contains a given cube
\[
z^k_n + \lambda^k_n \, U(\{\e^k_i(n)\}) \subset \mathtt p_{V^k_n} Z^k_\mathfrak a.
\]
\end{enumerate}
Moreover, if $\mathbf D$ is complete, then the sets $\{\mathbf D^k_n\}$ are Borel.
\end{proposition}

Observe that from Point \eqref{Point_2_Proposition_P_countable_partition} and Point \eqref{Point_4_Proposition_P_countable_partition} above it follows that there exists $\rho > 0$ such that
\begin{equation*}
% \label{E_non_dege_projection}
|\mathtt p_{V^k_n} (z-z')| \geq \rho |z-z'|, \qquad \forall\, z ,z'\in \aff\, Z^k_\mathfrak a,\,\forall \mathfrak a \in \mathfrak A^k_n.
\end{equation*}
%Sara
\begin{proof}%[Proof of Proposition \ref{P_countable_partition_in_reference_directed_planes}]
If $V \in \mathcal G(k,\R^d)$, $C \in \mathcal C(k,V)$ and given two real numbers $0<\delta, r<1$ such that $C(2r)\in\mathcal C(k,\R^d)$, consider the subset $L(k,C,r,\delta)$ of $\mathcal C(k,\R^d)$ defined by
\begin{align}
\label{E_L_V_r_def}
L(k,C,r,\delta) := \Big\{ C' \in \mathcal C(k,\R^d) : \ (i)&~ C(r) \subset \mathtt p_V \mathring C', \crcr
(ii)&~ \mathtt p_V C' \subset C(2r), \crcr
(iii)&~ \inf \big\{ |\mathtt p_V z|: z \in C'\cap \mathbb S^{d-1} \big\} > 1 - \delta \Big\}.
\end{align}
If is fairly easy to see that for all $0<\delta<1$ as above the family
\begin{equation}
\label{E_base_of_C_k_R_ell}
\mathfrak L(k,\delta) := \Big\{ L(k,C,r',\delta): C \in \mathcal C(k,\R^d), 0 < r' < 1 \text{ s.t. }C(2r')\in\mathcal C(k,\R^d) \Big\}
\end{equation}
generates a base of neighborhoods of $\mathcal C(k,\R^d)$. In particular, we can find a countable family of sets $\{L(k,C^k_n,r^k_n,1\slash\sqrt 2)\}_{n \in \N}$, covering $\mathcal C(k,\R^d)$ --being the latter separable. %Since the cones of the form \eqref{E_cone_e_k_i} are dense in $\mathcal{C}(K,\R^d)$, we can assume that
%\begin{equation*}
%C(\{\e^k_i(n)\})=C^k_n
%\end{equation*}
%for a family of $k$ linearly independent unit vectors $\{\e^k_i(n)\}_{i =1}^k$ in $\R^d$.
Notice that
\[
\mathrm{clos}\,L(k,C,r,\delta) = \bigg\{ C' \in \mathcal C(k,\R^d) :\,C(r) \subset \mathtt p_V C',\,\mathtt p_V C' \subset C(2r),\,\inf \big\{ |\mathtt p_V z|: z \in C'\cap \mathbb S^{d-1} \big\} \geq 1 - \delta \bigg\}
\]
is compact. 

Then, define
\begin{equation*}
\mathbf D^k_n := \bigg\{ \big( \mathfrak a,z,C^k_\mathfrak a \big) \in \mathbf D(k): C^k_\mathfrak a \in \mathrm{clos}\,L \big( k,C^k_n,r^k_n,1/\sqrt{2} \big) \setminus \bigcup_{n' < n} \mathrm{clos}\,L \big( k,C^k_{n'},r^k_{n'},1/\sqrt{2} \big) \bigg\}.
\end{equation*}
%Sara
Clearly $\{\mathbf D^k_n\}_{n \in \N}$ is a covering of $\mathbf D(k)$ into disjoint sets, and it is fairly straightforward to prove that these sets are $\sigma$-compact, because the sets $\mathrm{clos}\,L(k,C,r,\delta)$ are compact.

For each $k$, $C^k_n$, $r^k_n$, consider a family of $k$ linearly independent unit vectors $\{\e^k_i(n)\}_{i =1}^k$ in $\R^d$ such that
\begin{equation*}
C(\{\e^k_i(n)\}) \subset \mathring C^k_n.
\end{equation*}
%Moreover, up to a countable repartition of the $\mathbf D^k_n$, we can assume that $\mathtt p_{\A\times\mathcal C(k,\R^d)}\mathbf D^k_n$. 
%and by completeness $\bigl\{(\a,z)\in\mathtt p_{\A\times\R^d}\mathbf D^k_n:\,\mathtt p_{V^k_n}(z)\in U(\{\e^k_i(n)\}$ is compact. 

Being the family of sets $\bigl\{\{z + \lambda \, \interr U(\{\e^k_i(n)\})\}\bigr\}_{z \in V^k_n,\,\lambda \in \R^+}$ a base of the topology of $V^k_n$, let $\bigl\{\{z_m + \lambda_m \, \interr U(\{\e^k_i(n)\})\}\bigr\}_{m \in \N}$ be a countable base. Define thus
\begin{equation*}
\mathbf D^k_{n,m} := \Big\{ (\mathfrak a,z,C^k_\mathfrak a) \in \mathbf D^k_n : z_m + \lambda_m U(\{\e^k_i(n)\}) \subset \mathtt p_{V^k_n} Z^k_\mathfrak a \Big\} \setminus \bigcup_{m' < m} \mathbf D_{n,m'}^k.
\end{equation*}
Since the directed partition is complete (see Definition \ref{D_part_compl}) and Point \eqref{Point_2_Proposition_P_countable_partition} holds then 
\[
 z_m + \lambda_m U(\{\e^k_i(n)\}) \subset \mathtt p_{V^k_n} Z^k_\mathfrak a\quad\Leftrightarrow\quad \Big\{z_m,\,z_m+\lambda_m\sum_{i=1}^k\e_i^k(n)\Big\}\subset \mathtt p_{V^k_n} Z^k_\mathfrak a.
\]
Let $\mathbf f^k_n:\,\mathtt p_{\A\times\R^d}\mathbf D^k_n \to\mathtt p_{V^k_n}\bigl(\mathtt p_{\A\times\R^d}\mathbf D^k_n\bigr)$ be the $\sigma$-continuous map 
\[
 \mathbf f^k_n(\a,z)=\mathtt p_{V^k_n}(Z\ka).
\]
One has that
\begin{equation} 
\label{E_sheafcomp}
\begin{split}
\bigg\{ (\mathfrak a,z) \in\mathtt p_{\A\times\R^d}\mathbf D^k_n :\, \bigg\{z_m,\,z_m+\lambda_m& \sum_{i=1}^k\e_i^k(n) \bigg\}\subset \mathtt p_{V^k_n} Z^k_\mathfrak a \bigg\} \\
=&~ \mathtt p_{12} \Big(\Graph\, \mathbf f^k_n\cap \big\{(\a,z,z_m):\,\a\in\A,z\in\R^d \big\} \Big) \\
&\cap\mathtt p_{12} \bigg( \Graph\, \mathbf f^k_n \cap \bigg\{ \bigg( \a,z, z_m+\lambda_m \sum_{i=1}^k\e_i^k(n) \bigg) :\, \a\in\A, z\in\R^d \bigg\} \bigg),
\end{split}
\end{equation}
is a $\sigma$-compact set, thus $\mathbf D^k_{n,m}$ is Borel. Relabeling the sets $\mathbf D^k_{n,m}$ as $\mathbf D^k_n$, the proof is completed.
\end{proof}

\begin{remark}
\label{R_sheaf_negl}
In the rest of this section, without further comments, we will assume that the directed locally affine partitions are complete, according to Definition \ref{D_part_compl}. Indeed, this will be always the case for the partitions we analyze in the paper. Since we will be interested into directed locally affine partitions up to sets which are $\varpi$-negligible w.r.t. some fixed measure $\varpi$, we will also consider the sets of the countable partition $\{\mathbf D^k_n\}$ as $\sigma$-compact, which is always the case provided we remove an $\varpi$-negligible set. 
\end{remark}

\begin{figure}
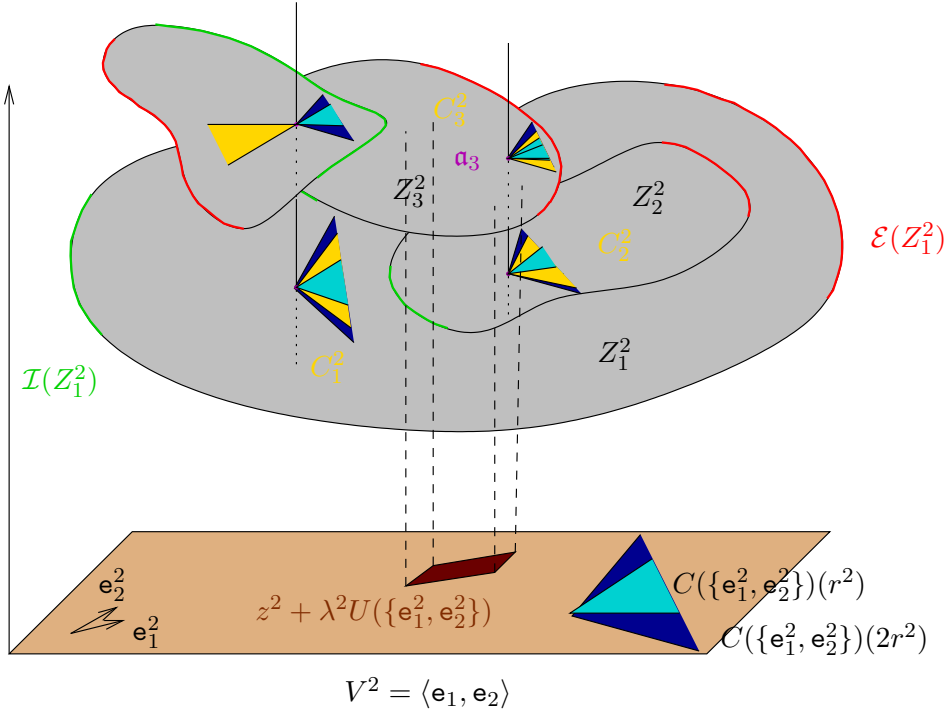
\caption{The locally affine sets $Z^2_1$, $Z^2_2$, $Z^2_3$ with cones of directions $C^2_1$, $C^2_2$, $C^2_3$ (given by the union of the cyan and the yellow triangles) form a $2$-dimensional directed sheaf set with reference plane $V^2$, base vectors $\e^2_1,\,\e^2_2$, reference rectangle $z^2+\lambda^2 U(\{\e^2_1,\,\e^2_2\})$ and base cones of directions $C(\{\e^2_1,\e^2_2\})(r^2)\subset V^2$ (colored in cyan) and $C(\{\e^2_1,\e^2_2\})(2r^2)\subset V^2$ (colored in blue). In the picture, we underline with the corresponding color the counterimages of the reference cones on the affine spans of the locally affine sets. The remaining locally affine set of the partition does not belong to the sheaf set since the projection of its cone of directions (colored in yellow) on the reference plane does not contain/is contained in the reference cone.}
\label{Fi_sheafdef}
\end{figure}

\begin{definition}
\label{D_sheaf_set}
For $k=1,\dots,d$, we call \emph{$k$-(dimensional) directed sheaf set} a $\sigma$-compact directed locally affine partition into $k$-dimensional sets $\bD^k$ which satisfy the same properties of the sets $\bD^k_n$ in Proposition \ref{P_countable_partition_in_reference_directed_planes}: there exist
\begin{itemize}
\item linearly independent vectors $\{\e^k_i\}_{i=1}^k\subset \mathbb S^{d-1}$, with linear span
\[
V^k = \langle \e_1^k,\dots,\e_k^k \rangle,
\]
\item a given point $z^k \in V^k$,
\item constants $r^k,\lambda^k \in \R^+$,
\item a non degenerate cone $C^k \in \mathcal C(k,\R^d)$, with $C(\{\e_i^k\})(r_k) \subset \mathring C^k$, 
\end{itemize}
such that, denoting
% \[
$\mathfrak A^k := \mathtt p_\mathfrak a \mathbf D^k$, $C\ka = \mathtt p_{\mathcal C(k,\R^d)} \bD^k(\a)$, $Z\ka=\p_{\R^d}\mathbf D^k(\a)$,
% \]
it holds:
\begin{enumerate}
\item %\label{Point_1_Proposition_P_countable_partition} the cone generated by $\{\mathrm e_i^k(n)\}$ is not degenerated and strictly contained in $C^k_n$,
% \[
$C^k(r^k) \in \mathcal C(k,V^k)$;
% \]
\item %\label{Point_2_Proposition_P_countable_partition} the cones $C\ka$, $a\in \A^k_n$ have a uniform opening,
% \[
$C^k \subset \mathtt p_{V^k} \mathring C^k_\mathfrak a$;
% \]
\item %\label{Point_3_Proposition_P_countable_partition} the cones $C\ka$, $a\in\A^k_n$, are contained strictly contained in $C^k_n$,
% \[
$\mathtt p_{V^k} C^k_\mathfrak a \subset C^k(r^k)$ for all $\mathfrak a \in \mathfrak A^k$
% \]
\item %\label{Point_4_Proposition_P_countable_partition} the projection on $V^k_n$ is not degenerated,
% \[
$|\mathtt p_{V^k} z| \geq 1/\sqrt{2}$ for all $z \in C^k_\mathfrak a\cap \mathbb S^{d-1}$, $\mathfrak a \in \mathfrak A^k$;
% \]
\item %the projection of $Z\ka$ on $V^k_n$ contains a given cube,
% \[
$z^k + \lambda^k \, U(\{\e^k_i\}) \subset \mathtt p_{V^k} Z^k_\mathfrak a$.
% \]
\end{enumerate}
The $k$-dimensional plane $V^k=\langle \e_1^k,\dots \e_k^k\rangle$ will be called \emph{reference plane}, the cones $C^k \subset C^k(r^k) =:C^{',k}$ \emph{base cones of directions}, $z^k$ \emph{base point} and $z^k + \lambda^k \, U(\{\e^k_i\})$ \emph{base rectangle} of the sheaf set.
\end{definition}

Moreover, we can choose
\begin{equation}
\label{E_mathfrak_A_k_def}
\mathfrak A^k := \mathbf Z^k \cap \big( z^k + (V^k)^\perp \big).
\end{equation}
In this way the quotient space $\mathfrak A^k$ is a subset of a $(d-k)$-dimensional affine space. This will be our default choice for the quotient space $\A^k$ of $\bD^k$.

Before going on, we prove the following lemma, announced in Section \ref{Ss_partitions_intro}.
%Sara
\begin{lemma}
\label{L_regularity_initial_final_points}
If the directed locally affine partition is $\sigma$-compact and complete, then the sets $\mathcal I$, $\mathcal E$ are Souslin.
\end{lemma}

\begin{proof}
We prove the statement only for $\mathcal I$, since the proof for $\mathcal E$ is analogous. Moreover, we can consider w.l.o.g. a directed locally affine partition given by a $\sigma$-compact directed sheaf set as in Definition \ref{D_sheaf_set}. 

Let $\A = \underset{l\in\N}{\cup}\, \A_l$ such that the sets $\mathtt p_{\A_l\times\mathcal C(k,\R^d)}\mathbf D(k)$ are compact,
and for $n\in\N$ define
\[
\mathcal I^k_{n,l} := \Bigl\{(\a,z)\in\A_l\times\R^d :\, z\in\clos\,Z^k_\a,\quad z + \interr C^k_{\a,l} \cap B^d(z,2^{-n}) \subset Z^k_{\a} \Bigr\}.
\]
By the completeness property of the sheaf set, there exists $n'$ such that
\begin{equation}
\label{E_iknl}
\mathcal I^k_{n,l}:= \Bigl\{(\a,z)\in\A_l\times\R^d :\, z\in\clos\,Z^k_\a,\quad z + \interr C^k_{\a,l} \cap \partial B^d(z,2^{-n'})\cap Z\ka\neq\emptyset \Bigr\}.
\end{equation}

Since the set $\{(\a,z):\,z\in\clos\,Z^k_\a\}$ is Borel, then by \eqref{E_iknl}, reasoning as in \eqref{E_sheafcomp} the sets $\mathcal I^k_{n,l}$ are Borel too and finally
\[
 \mathcal I=\mathtt p_{\R^d}\Big(\underset{k,n,l}{\bigcup}\,\mathcal I^k_{n,l}\Big)\setminus \mathbf Z
\]
is Souslin.
\end{proof}

Now we show that the graph $\bD^k$ of a $k$-directed sheaf set in $\R^d$ can be mapped injectively into a subset  of $\R^{d-k}\times\R^k$ called \emph{fibration}, consisting of a family of parallel $k$-dimensional locally affine sets. In this section, points in $\R^{d-k}\times\R^k$ will be denoted as $(\q,w)$.

\begin{figure}
\centerline{\resizebox{14cm}{12cm}{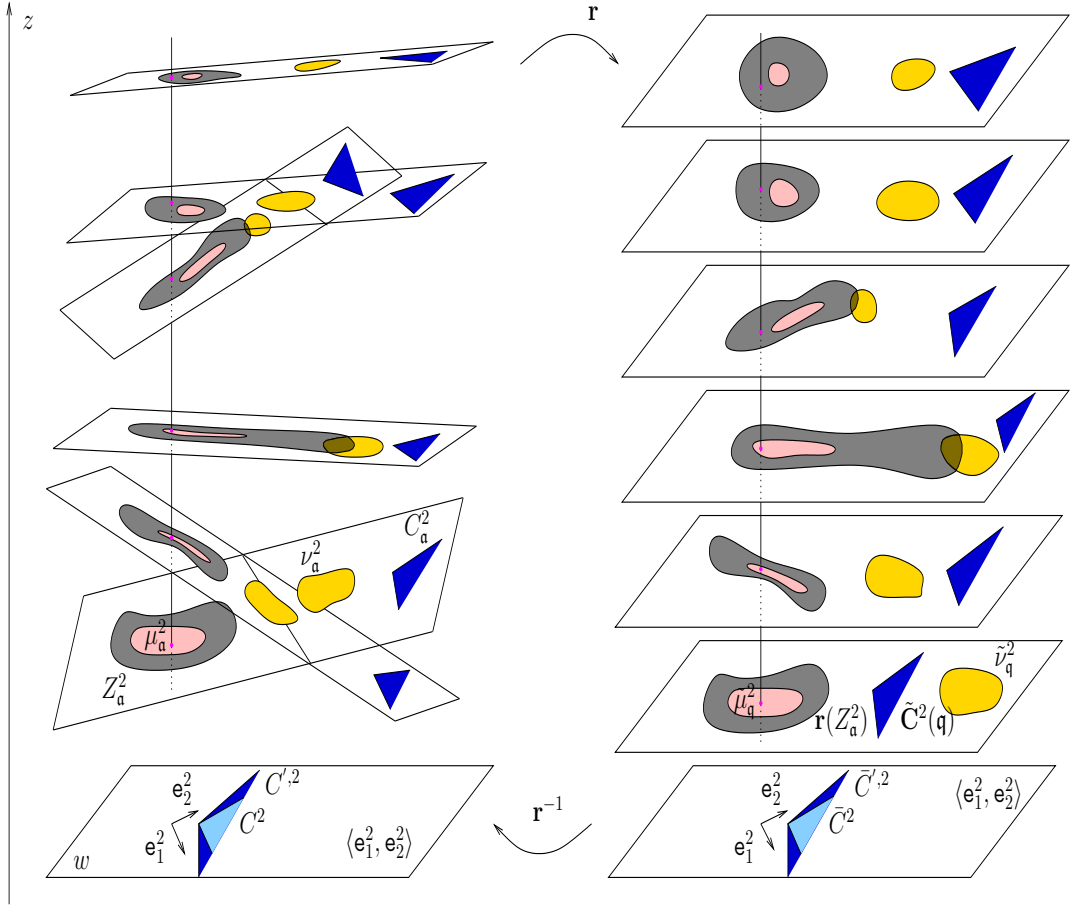}}
\caption{The mapping of a $2$-dimensional sheaf set into a fibration, Proposition \ref{P_map_sheaf_set_into_fibration}. The pink region denotes the support of the conditional measures $\mu^2_\a$ (resp. $\tilde\mu^2_\q$), the yellow one the support of the conditional measures $\nu^2_\a$ (resp. $\tilde\nu^2_\q$), and the blue cones $C^2_\a$ (resp. $\tilde {\mathbf C}^2(\q)$) are the cones of directions of the locally affine sets $Z^2_\a$ (in gray color). The reference cones $C^2\subset C^{',2}$ and $\bar C^2\subset\bar C^{',2}$ are also depicted.}
\label{Fi_mappingfibra}
\end{figure}

\begin{definition}
\label{D_fibration}
A \emph{$k$-(dimensional) directed fibration} is a $\sigma$-compact set $\tilde{\bD}^k\subset\R^{d-k}\times\R^k$ endowed with a $\sigma$-continuous map $\tilde{\bC}^k:\p_{\R^{d-k}}(\tilde{\bD}^k)\to\mathcal C(k,\R^k)$, $\mathfrak q \mapsto \tilde{\bC}^k(\mathfrak q)$, such that
\begin{align}
&\tilde{\bD}^k(\q)\text{ is open in $\R^k$,} \notag \\
&\exists\, \bar C^k \subset \bar C^{',k}, \ \text{with} \ \bar C^k, \bar C^{',k} \in \mathcal C(k,\R^k), \quad \text{ s.t. } \quad \bar C^k \subset \mathring{\tilde{\bC}}^k(\q) \subset \tilde{\bC}^k(\q) \subset \mathring{\bar{C}}^{',k}, \label{E_cones_fibration}
\end{align}
for all $\q \in \p_{\R^{d-k}}(\tilde{\bD}^k)$.
\end{definition}

To a directed fibration $\tilde\bD^k$ we associate the cost
\begin{equation}
\label{E_cost_fibr}
\mathtt c_{\tilde \bD^k}(\q,w,\q',w') =
 \begin{cases}
0 & \q = \q', w'-w \in \tilde {\mathbf C}^k(\q), \\
+\infty & \text{otherwise.}
\end{cases}
\end{equation}
Recalling the definition given in \eqref{E_cone_cost}, notice that 
\[
 \mathtt c_{\tilde{\mathbf D}^k}(\q,w,\q',w')=
 \begin{cases}
\mathtt c_{\tilde{\mathbf C}^k(\q)}(w,w') & \text{ if $\q = \q'$}, \\
+\infty & \text{otherwise.}
\end{cases}
\]

\begin{proposition}
\label{P_map_sheaf_set_into_fibration}
Let $\bD^k$ be a $k$-dimensional directed sheaf set with base cones $C^k\subset C^{',k}\in\mathcal C(k,\R^d)$, reference plane $V^k$, base point $z^k$ and quotient space $\A^k=Z^k\cap(z^k+(V^k)^{\perp})$ and define the map
\begin{equation}
\label{E_r_map}
\mathbf r:\A^k \times \R^d \ni (\a,z) \mapsto \Big( \mathtt i_{d-k}(\a), \mathtt i_k \circ \p_{V^k}(z) \Big) \in \R^{d-k} \times \R^k,
\end{equation}
where 
\begin{equation}
\label{E_identif_map}
\mathtt i_{d-k}:z^k+(V^k)^\perp\to\R^{d-k},\quad\mathtt i_k:V^k\to\R^k
\end{equation}
are the identification maps.

Then, $\mathbf r_{|_{\underset{\a\in\A^k}{\cup}\{\a\}\times \aff Z\ka}}$ is a bijection onto $\underset{\q\in\mathtt i_{d-k}(\A^k)}{\cup}\{\q\}\times\R^k$ and $\mathbf r_{|_{\underset{\a\in\A^k}{\cup}\{\a\}\times Z\ka}}=\mathbf r_{|_{\p_{\A^k\times\R^d}(\bD^k)}}$ maps the sheaf set $\mathtt p_{\A^k\times\R^d}(\bD^k)$ into the elements of the fibration $\tilde{\bD}^k:=\mathbf r(\p_{\A^k\times\R^d}(\bD^k))$ endowed with the direction map
\begin{equation*}
\tilde{\bC}^k:\mathtt i_{d-k}(\A^k)\ni\q\mapsto \mathtt i_k\circ\p_{V^k}\circ\p_{\mathcal C(k,\R^d)}\bigl(\bD^k(\mathtt i_{d-k}^{-1}(\q))\bigr).
\end{equation*}

Moreover,
\begin{equation}
\label{E_rxr}
(\mathbf r\times\mathbf r)_\#\Pi^f_{\mathtt c_{\bD^k}}(\mu,\{\nu^k_\mathfrak a\}) =  \Pi^f_{\mathtt c_{\tilde\bD^k}}(\hat \mu,\hat \nu),
\end{equation}
where
\begin{align}
\label{E_hat_mu_nu}
\tilde \mu = \int (\mathtt i_{k} \circ \mathtt p_{V^k})_\#\mu^k_{\mathtt i_{d-k}^{-1}(\q)}\,d\mathtt i_{d-k}\#m(\q), \quad \tilde \nu = \int (\mathtt i_{k} \circ \mathtt p_{V^k})_\#\nu^k_{\mathtt i_{d-k}^{-1}(\q)}\,d\mathtt i_{d-k}\#m(\q).
\end{align}

Finally
\begin{align*}
&(\mathtt i_{k} \circ \mathtt p_{V^k})_\# \mu^k_{\mathtt i_{d-k}^{-1}(\q)} \big( \mathtt i_{k} \circ ( \mathbf r(B) \cap \R^{d-k} \times \{\q\} ) \big) = 0 \quad \Longleftrightarrow \quad \mu^k_{\mathtt i_{d-k}^{-1}(\q)}(B) = 0, \quad \forall\, B \in \mathcal B(\R^d) \cap \aff Z\ka, \\
&(\mathtt i_{k} \circ \mathtt p_{V^k})_\# \nu^k_{\mathtt i_{d-k}^{-1}(\q)} \big( \mathtt i_{k} \circ ( )\mathbf r(B) \cap \R^{d-k} \times \{\q\} ) \big) = 0 \quad \Longleftrightarrow \quad \nu^k_{\mathtt i_{d-k}^{-1}(\q)}(B) = 0, \quad \forall\, B \in \mathcal B(\R^d) \cap \aff Z\ka.
\end{align*}

\end{proposition}

\begin{proof}
% [Proof of Proposition \ref{P_map_sheaf_set_into_fibration}]
The functions $\mathbf r$ and $\tilde{\mathbf C}^k$ are $\sigma$-continuous, because their graphs are projections of $\sigma$-compact sets.

The facts that $\mathbf r$ is a bijection and that $\mathbf r(\p_{\A^k\times\R^d}\bD^k)$ is a fibration are straightforward, observing that \eqref{E_cones_fibration} is satisfied by the cones $\bar{C}^k=\mathtt i_k(C^k)$, $\bar{C}^{',k}=\mathtt i_k(C^{',k})$, thanks to $(2)$, $(3)$ and $(4)$ of Definition \ref{D_sheaf_set}.

As for the last statements, it is sufficient to observe that 
\[
\mathtt c_{\bD^k}((\a,z),(\a',z'))=\mathtt c_{\tilde\bD^k}(\mathbf r(\a,z),\mathbf r(\a',z')) \cdot \ind_{\tilde \bD^k \times \R^{d-k} \times \R^k}(\mathbf r(\a,z),\mathbf r(\a',z')),
\]
and that $\mathbf r_{|_{\{\a\}\times \aff Z\ka}}$ is bi-Lipschitz, $\forall\,\a\in\A^k$.
\end{proof}

In the following, we set
\begin{equation}
\label{E_tilde_mu_nu_k_a}
\tilde\mu^k_\q = (\mathtt i_{k} \circ \mathtt p_{V^k})_\# \mu^k_{\mathtt i_{d-k}^{-1}(\q)}, \quad \tilde\nu^k_\q = (\mathtt i_{k} \circ \mathtt p_{V^k})_\# \nu^k_{\mathtt i_{d-k}^{-1}(\q)}, \quad \tilde\pi^k_\q = (\mathtt i_{k} \circ \mathtt p_{V^k} \times \mathtt i_{k} \circ \mathtt p_{V^k})_\# \pi^k_{\mathtt i_{d-k}^{-1}(\q)}.
\end{equation}

%Sara
\begin{remark}
\label{rem_sheaffibr}
 Notice that, once we fix an orthogonal basis $\{\e_i\}_{i=1}^d\subset\R^d$ and identify $\R^{d-k}\times\R^k\simeq \langle \e_1,\dots,\e_{d-k}\rangle\times\langle \e_{d-k+1},\dots,\e_{d}\rangle=\R^d$, a $k$-directed fibration is the image through the map $\mathtt i_{d-k}\times\mathtt i_k$ of a $k$-directed sheaf set whose reference $k$-plane is $\{0\}\times\langle \e_{d-k+1},\dots,\e_{d}\rangle$. Therefore, we can think of a $k$-directed fibration as a $k$-directed sheaf set whose sets are contained in disjoint parallel $k$-planes. In particular, when we speak about directed locally affine subpartitions of a fibration (as e.g. in Proposition \ref{P_sub_sheaf_fol}) we mean the image through the map $\mathtt i_{d-k}\times\mathtt i_k$ of directed locally affine subpartitions of the corresponding $k$-directed sheaf set. 
 %In particular, each set of a directed locally affine partition on a fibration is contained in $\q\times \R^k$ for some $\q\in\R^{d-k}$. For this reason, in the following sections (see e.g. the superdifferential/subdifferential directed partitions of Section \ref{Ss_partition_transport_set}) we 
%will either consider the sets of such partitions either as subsets of $\{\q\}\times \R^k\subset\R^{d-k}\times\R^{k}$ or of $\R^k$.
 \end{remark}

\begin{proposition}
\label{P_sub_sheaf_fol}
Let $\{Z^\ell_{\a,\b},C^\ell_{\a,\b}\}_{\nfrac{\ell=0,\dots,k}{\a \in \A^k,\b \in \B^\ell}}$ be a locally affine directed subpartition of the sheaf set $\{Z\ka,C\ka\}_{k,\a}$. Then the sets
\[
\tilde Z^\ell_{\q,\b} = \mathbf r(\a,Z^\ell_{\a,\b}),\quad \tilde C^\ell_{\q,\b}=\mathtt i_{k}\circ\mathtt p_{V^k}(C^\ell_{\a,\b})
\]
form a directed locally affine subpartition of the fibration $\tilde{\mathbf D}^k=\mathbf r(\mathtt p_{\A^k\times\R^d}(\mathbf D^k))$, and viceversa.

If moreover  the subpartition $\{Z^\ell_{\a,\b},C^\ell_{\a,\b}\}_{\ell,\a,\b}$ is regular, then also the subpartition $\{\tilde Z^\ell_{\q,\b},\tilde C^\ell_{\q,\b}\}_{\ell,\q,\b}$ is regular, and viceversa.
\end{proposition}

\begin{proof}
% [Proof of Proposition \ref{P_sub_sheaf_fol}:]
Since $\mathtt r_{|_{\{\a\} \times \aff \, Z\ka}}$ is an invertible projection, the first part of the statement is obvious.
% bi-Lipschitz bijection and it is $\sigma$-continuous, it is enough to prove that it is a subpartition for a fixed $\a\in\A$. This is obvious because $\mathtt r_{|_{\{\a\}\times\R^d}}$ is an invertible projection.
 
The same reasoning holds for the regularity of the measures.
\end{proof}
Recalling Propositions \ref{P_dispiani} and \ref{P_dispiani_2}, it is fairly easy to prove the following
\begin{corollary}
\label{C_plan_fibr}
Let $\tilde\bD^k$ be a $k$-directed fibration and $\tilde\mu$, $\tilde\nu\in\mathcal P(\R^{d-k}\times\R^k)$ s.t. $\Pi^f_{\mathtt c_{\tilde\bD^k}}(\tilde\mu,\tilde\nu)\neq\emptyset$. Then,
\begin{equation}
 \tilde\pi\in\Pi^f_{\mathtt c_{\tilde\bD^k}}(\tilde\mu,\tilde\nu)\quad\Leftrightarrow\quad\tilde\pi^k_\q\in\Pi^f_{\mathtt c_{\tilde{\mathbf C}^k(\q)}}(\tilde\mu^k_\q,\tilde\nu^k_\q),
\end{equation}
being $\tilde\mu=\int\tilde\mu^k_\q \, d\tilde m(\q)$, $\tilde\nu=\int\tilde\nu^k_\q \, d\tilde m(\q)$ and $\tilde\pi=\int\tilde\pi^k_\q \, d\tilde m(\q)$ be respectively the disintegrations of $\tilde\mu$, $\tilde\nu$ w.r.t. the partition $\{\{\q\}\times \R^k\}_{\q\in\mathtt i_{d-k}(\A^k)}$ and the disintegration of $\tilde\pi$ w.r.t. $\{\{\q\}\times \R^k\times \{\q\}\times \R^k\}_{\q\in\mathtt i_{d-k}(\A^k)}$.
\end{corollary}

\section{Directed locally affine partitions on cone-Lipschitz foliations}
\label{S_foliations}

In the first part of this section we generalize the notion of graph of a $\d{\cdot}$-Lipschitz function up to the definition of a $\mathtt c_{\tilde{\mathbf C}}$-Lipschitz foliation, where $\C$ is the family of cones of directions associated to a given $k$-dimensional fibration $\tilde\bD\subset\R^{d-k}\times\R^k$. From now on we fix $k$ and we drop the superscript $k$ in the notation for a $k$-directed fibration or $k$-dimensional cone. Moreover we will replace the variable $\q$ with $\a$, since it is clear from Proposition \ref{P_map_sheaf_set_into_fibration} and Remark \ref{rem_sheaffibr} that the quotient spaces of a sheaf set and of the corresponding fibration can be identified.

\noindent In particular, for any fixed $\a\in\A=\p_{\R^{d-k}}(\tilde\bD)$, the intersection of a $\mathtt c_{\C}$-Lipschitz foliation with $\{\a\}\times\R^k$ will be a suitable collection of disjoint \emph{(complete) $\mathtt c_{\tilde{\bC}(\a)}$-Lipschitz graphs} --namely, graphs of $|\cdot|_{\tilde D(\a)^*}$-Lipschitz functions where $\tilde D(\a)$ is convex set s.t. $\tilde{ \mathbf C}(\a)=\epi\,|\cdot|_{\tilde D(\a)^*}$ (see Definition \ref{D_complete_G})-- and at most countably many sets with nonempty interior.

Next, we generalize the notion of super/subdifferentials given in Definition \ref{D_subsuperdiff} for single graphs of $\d{\cdot}$-Lipschitz functions  to this new class of objects: at each point $w\in\tilde\bD(\a)$, the superdifferential will be the intersection of the cone $w+\C(\a)$ with the $\mathtt c_{\C(\a)}$-Lipschitz graph to which $w$ belongs.

Our main result is Theorem \ref{T_partition_E+-}, in which we prove that, up to a residual set, a $\mathtt c_{\C}$-Lipschitz foliation can be decomposed into a directed locally affine partition whose cone of directions at each point is given by the super/subdifferential.

Moreover, in Theorem \ref{T_partition_E} we characterize the residual set as the set of \emph{initial}/\emph{final points} of the super/subdifferential partitions (see Definition \ref{D_initial_final}).

\subsection{Convex cone-Lipschitz graphs}
\label{Ss_cone_lipschitz_graph}

Let $\tC$ be the epigraph of a convex norm $|\cdot|_{\tilde{D}^*}:\R^{k-1}\to\R$: by Remark \ref{R_cone_epi}, as a subset of $\R^k=\R^{k-1}\times\R$, $\tC\in\mathcal C(k,\R^k)$. We
denote variables in $\R^k$ as $w=(x,y)\in\R^{k-1}\times\R$ and we let $\mathtt c_{\tC} : \R^k \times \R^k \to [0,+\infty]$ be the related \emph{convex cone cost} (see Definition \eqref{D_cone_cost}).

Now we introduce a class of subsets of $\R^k$ which includes the graphs of $|\cdot|_{\tilde{D}^*}$-Lipschitz functions $\varphi:\R^{k-1}\to\R$. In particular, when $k=d+1$ and $\tC=\epi\,\d{\cdot}$, this class contains the graphs of the Kantorovich potentials $\psi$ for the transport problem with cost \eqref{E_norm_cost_1} (see Section \ref{Ss_convex_norm_cone}).

\begin{definition}
\label{D_complete_G}
A set $G\subset \R^k$ is a \emph{$\mathtt c_{\tC}$-Lipschitz graph} if
\begin{subequations}
\label{E_G}
\begin{equation}
\label{E_G1}
G \times G \cap \big\{ \mathtt c_{\tC} < +\infty \big\} \subset \big\{ (w,w') : w'-w \in \partial \tC \big\}.
\end{equation}
Moreover, a $\mathtt c_{\tC}$-Lipschitz graph $G\subset\R^k$ is \emph{complete} if 
\begin{equation}
\label{E_G2}
\O(w,w') : =\big\{ \mathtt c_{\tC}(w,\cdot) < +\infty \big\} \cap \big\{ \mathtt c_{\tC}(\cdot,w') < +\infty \big\} \subset G, \quad \forall\, w,w' \in G.
\end{equation}
\end{subequations}
\end{definition}

Notice that \eqref{E_G1} is equivalent to
\[
w'\notin w\,\pm\,\inter\, \tC, \qquad \forall\,w,w'\in G,
\]
which can be rephrased as
\begin{equation}
\label{E_phi_graph_G}
G = \Graph\,\varphi_G \quad \text{ for some $\varphi_G:\dom\,\varphi_G\subset\R^{k-1}\to\R$, $\varphi_G$ $\d{\cdot}$-Lipschitz}.
\end{equation}
The second condition \eqref{E_G2} yields
\[
G \supset \O(G\times G) := \bigcup_{w,w'\in G} \O(w,w'). 
\]

\begin{remark}
\label{R_lipgraph}
If $G=\Graph\,\varphi_G$ with $\varphi_G$ $|\cdot|_{\tilde{D}^*}$-Lipschitz and $\dom\,\varphi_G=\R^{k-1}$, then $G$ is a complete $\mathtt c_{\tC}$-Lipschitz graph. The analysis of this particular case will be sufficient for the proof of Theorem \ref{T_1}. If $\dom\,\varphi_G\neq\R^{k-1}$, we anticipate that the ``completeness'' property \eqref{E_G2} is what we need to construct sets which preserve the properties of the $|\cdot|_{\tilde{D}^*}$-super/subdifferentials of Lipschitz functions on $\R^{k-1}$: these properties are fundamental for our later purposes, culminating with the proof of Theorem \ref{T_partition_E+-}.
\end{remark}

% \begin{remark}
%  \textcolor{blue}{In effetti mi sembra che nell'articolo, anche nelle $\mathtt c_{\C}$-Lipschitz foliations, non usiamo mai grafici Lipschitz di funzioni con dominio diverso da $\R^{k-1}$. Magari si potrebbe cambiare il remark dicendo che la proprietà \eqref{E_G2} è valida nel caso $\dom\varphi_G=\R^{k-1}$ ma che la mettiamo in evidenza perchè è proprio la proprietà fondamentale per avere le proprietà del sopradifferenziale che ci servono...}
% \end{remark}

Recalling Definition \ref{D_subsuperdiff} of super/subdifferential of the graph of a $\d{\cdot}$-Lipschitz function $\varphi$ we give the following definition.

\begin{definition}
\label{D_tilde_C_diff}
Given a set $G\subset\R^k$, define the \emph{$\mathtt c_{\tC}$-superdifferential} and the \emph{$\mathtt c_{\tC}$-subdifferential} of $G$ respectively as
\begin{equation*}
% \label{E_super_sub_G}
\partial^+G := G \times G \cap \bigl\{ \mathtt c_{\tC} < +\infty \bigr\}, \qquad \partial^-G: = G \times G \cap \bigl( \bigl\{ \mathtt c_{\tC} < +\infty \bigr\} \bigr)^{-1}.
\end{equation*}
\end{definition}

Notice that $\partial^-G=(\partial^+G)^{-1}$ and since
\begin{equation}
\label{E_sum_of_cones_inside_cone}
(w+C)+C\subset w+C, \qquad \forall\,w\in\R^k,\ C \in \bigcup_{\ell=1}^k \mathcal C(\ell,\R^k),
\end{equation}
one deduces the \emph{transitivity property}
\begin{equation*}
% \label{E_transitivity_subdifferential}
w' \in \partial^\pm G(w) \quad \Longrightarrow \quad \partial^\pm G(w') \subset \partial^\pm G(w).
\end{equation*}

The property of a set of being a complete $\mathtt c_{\tC}$-Lipschitz graph can be equivalently restated in terms of its $\mathtt c_{\tC}$-super/subdifferentials as follows.

\begin{proposition}
\label{P_compl_para_sd}
$G\subset\R^k$ is a complete $\mathtt c_{\tC}$-Lipschitz graph if and only if
\begin{subequations}
\label{E_gprop_para_sd}
\begin{equation}
\label{E_Gproper}
\partial^\pm G \subset \mathrm{graph}\,\bigr(\Id\pm \partial \tC\bigl),
\end{equation}
\begin{equation}
\label{E_parallelogram_subdiff}
\O(w,w') \subset \partial^+G(w) \cap \partial^-G(w'), \quad \forall(w,w') \in G\times G.
\end{equation}
\end{subequations}
\end{proposition}

\begin{proof}
By Definition \ref{D_tilde_C_diff}, property \eqref{E_Gproper} is a rephrasing of \eqref{E_G1}, and \eqref{E_parallelogram_subdiff} is a rephrasing of \eqref{E_G2}.
\end{proof}
Recalling \eqref{E_partvarphi_form} and \eqref{E_sdiff_varphi}, we notice that if $G$ is a $\mathtt c_{\tC}$-Lipschitz graph and $\varphi_G$ is the function satisfying \eqref{E_phi_graph_G}, then
\[
\partial^\pm G = \partial^\pm\Graph\,\varphi_G=\Id\times\varphi_G \bigl( \partial^\pm\varphi_G \bigr).
\]

Now we state a simple geometric characterization of the set $\O(w,w')$ which will be fundamental in our study of the $\mathtt c_{\tC}$-super/subdifferentials of complete $\mathtt c_{\tC}$-Lipschitz graphs. First we give the following definition.

\begin{definition}
\label{D_cww'}
For $w,\,w'\in\R^k$, we let $C(w,w')$ be the extremal cone of $\tC$ satisfying
\begin{equation}
\label{E_cww'}
w' - w \in \interr C(w,w').
\end{equation}
Equivalently, $C(w,w')$ is the minimal cone w.r.t. set inclusion among the extremal cones of $\tC$ containing $w'-w$.
\end{definition}

Notice that, by \eqref{E_G1}, if $w,w'\in G$ and $\mathtt c_{\tilde C}(w,w') < \infty$ then $C(w,w')\subset\partial\tC$, i.e. $C(w,w')$ is a proper extremal cone of $\tC$.

%Sara
\begin{proposition}\label{P_parall}
$\O(w,w')$ is the convex set given by
\[
 \O(w,w')=w+C(w,w')\cap w'-C(w,w')
\]
and there exists $\delta>0$ such that
\begin{equation}
\label{E_O_ball}
B^k(w,\delta) \cap \big( w + C(w,w') \big) \subset \O(w,w') \quad \text{and} \quad B^k(w',\delta) \cap \big( w'-C(w,w') \big) \subset \O(w,w').
\end{equation}
In particular,
\begin{equation*}
% \label{E_minimal_extremal_face_containing_w_w_prime}
\R^+ (\O(w,w') - w) = \R^+ (w' - \O(w,w')) = C(w,w').
\end{equation*}
\end{proposition}

\begin{proof}
 By Definition, $\O(w,w')\supset w+C(w,w')\cap w'-C(w,w')$ and since $w'-w\in\interr C(w,w')$, \eqref{E_O_ball} follows. 
 
 Let us assume that $\exists\,z\in \O(w,w')\setminus w+ C(w,w')$ and let $F$ be the smallest face of $\tilde C$ such that $z-w$, $w'-w\in F$. Notice that $w'-w\in\partial F$. W.l.o.g. we set $w=0$ and $z=(z_1,z_2)\in \R^{\ell-1}\times\R$, $w'=(w'_1,w'_2)\in \R^{\ell-1}\times\R$
w.r.t. coordinates s.t. $F$ is the epigraph of a convex norm $|\cdot|_{E^*}:\R^{\ell-1}\to\R$. Hence, $w'_2=|w_1'|_{E^*}$ and either $z_2>|z_1|_{E^*}$, or $z_2=|z_1|_{E^*}$ and $w'_2-z_2>|w_1'-z_1|_{E^*}$. 
In the first case, $w'_2-z_2<|w'_1|_{E^*}-|z_1|_{E^*}\leq|w'_1-z_1|_{E^*}$, which implies $w'-z\in\aff F\setminus F\subset\R^k\setminus \tilde C$, contradicting the fact that $z\in\O(w,w')$.
In the second case, we get $w'_2-z_2>|w'_1-z_1|_{E^*}\geq|w'_1|_{E^*}-|z_1|_{E^*}=w'_2-z_2$, thus leading again to a contradiction.
\end{proof}

Observe that, with the characterization of Proposition \ref{P_parall} and by projecting on $\R^{k-1}$, \eqref{E_parallelogram_subdiff} can be equivalently restated in terms of the super/subdifferentials of $\varphi_G$ saying that, for all $x,x'\in\dom\,\varphi_G$, the set $\partial^+\varphi_G(x) \cap \partial^-\varphi_G(x')$ contains the convex set $\O_{\p_{\R^{k-1}}}(x,x')$ such that
\[
 \O_{\p_{\R^{k-1}}}(x,x')=x+C_{\p_{\R^{k-1}}}(x,x')\cap x'-C_{\p_{\R^{k-1}}}(x,x')
\]
and
\begin{equation}
\label{E_proj_Oww'}
\R^+ \bigl( \O_{\p_{\R^{k-1}}}(x,x')- x \bigr) = \R^+ \bigl( x' - \O_{\p_{\R^{k-1}}}(x,x') \bigr) = C_{\p_{\R^{k-1}}}(x,x'),
\end{equation}
being $C_{\p_{\R^{k-1}}}(x,x')$ the minimal extremal cone of $|\cdot|_{\tilde D^*}$ containing $x'-x$ (see Section \ref{Ss_intro_affine_subspaces_cones} for the definition of extremal cone).

\subsection{Convex cone-Lipschitz foliations}
\label{Ss_cone_lipschitz}

Let $\C:\A\times\R^k\to\mathcal C(k;\R^k)$, $\A\subset\R^{d-k}$, be the convex cone direction map of a $k$-directed fibration satisfying \eqref{E_cones_fibration} and let $\mathtt c_{\C}$ be the cost function defined in \eqref{E_cost_fibr}. Recall that, by \eqref{E_cost_fibr}
\begin{equation}
\label{E_caa'}
\mathtt c_{\C}(\a,w,\a',w') < +\infty \quad \Longrightarrow \quad \a=\a',
\end{equation}
and 
\begin{equation}
 \mathtt c_{\C}(\a,w,\a,w)=\mathtt c_{\C(\a)}(w,w').
\end{equation}
Moreover, set
\begin{equation}
\label{E_D_star_q}
\C(\a)=\epi\,|\cdot|_{D(\a)^*},
\end{equation}
where $D(\a) \subset \R^{k-1}$ for some suitable orthonormal coordinates independent of $\a$.

\begin{definition}
\label{D_Lip_fol}
A $\mathtt c_{\C}$-Lipschitz foliation is a $\sigma$-compact partition in $\A \times \R^k$ with quotient map $\theta : \dom\,\theta \subset \A \times \R^k \to \T$ such that
%\begin{equation}
%\label{E_adeptheta}
%\mathtt p_\A\theta=\Id_\A 
%\end{equation}
%and
\begin{subequations}
\label{E_compl_fol0}
\begin{equation}
\label{E_compl_fol}
(\a,w),(\a',w') \in \{\theta=\t\} \quad \Longrightarrow \quad \big\{ \mathtt c_{\C} (\a,w,\cdot,\cdot) < +\infty \big\} \cap \big\{ \mathtt c_{\C}(\cdot,\cdot, \a',w') < +\infty \big\} \subset \{ \theta = \t \},
\end{equation}
\begin{equation}
\label{E_compl_fol1}
\theta(\a,w) = \theta(\a',w') \quad \Longrightarrow \quad \a=\a'.
\end{equation}
\end{subequations}
\end{definition}

%By \eqref{E_adeptheta}, each set of $\mathtt c_{\C}$-Lipschitz foliation is contained in a section $\{\a\}\times\R^k$.
By \eqref{E_caa'}, \eqref{E_compl_fol} and recalling \eqref{E_G2}, we set 
\begin{equation}
\label{E_O_q_ww'}
\O(\a)(w,w') := \bigl\{ \mathtt c_{\C(\a)}(w,\cdot) < +\infty \bigr\} \cap \bigl\{ \mathtt c_{\C(\a)}(\cdot,w') < +\infty \bigr\} = (w+\C(\a)) \cap (w'-\C(\a)).
\end{equation}
We note moreover that from \eqref{E_compl_fol1} one has $\mathtt p_\A \{\theta=\mathfrak t\}= \{\a\}$ for some $\a\in\A$, thus in general $\T=\A\times\mathfrak S$ for some Polish space $\mathfrak S$ and $\p_\A\circ\theta\circ\p_\A^{-1}=\Id_\A$.
Hence, for simplicity of notation, from now onwards we will write --when not leading to confusion-- $\a=\p_\A\{\theta=\t\}$.

The following definition is given to simplify the notation in Proposition \ref{P_fol_char} (see Remark \ref{R_nondeg}).

\begin{definition}
\label{D_non_dege}
We call \emph{non-degeneracy set} of a $\mathtt c_{\C}$-Lipschitz foliation the set
\begin{equation}
\label{E_nondeg}
\begin{split}
\biggl\{ \t \in \T :&~ \exists \, w, w' \in \mathtt p_{\R^k} \{\theta=\t\} \text{ s.t. } \crcr
&~ \inter \big\{ \mathtt c_{\C(\a)}(\cdot,w') < +\infty \big\} \cup \inter \big\{\mathtt c_{\C(\a)}(w,\cdot) < +\infty \big\} \subset \bigl( \R^k \setminus \p_{\R^k}\{\theta=\t\} \bigr) \biggr\}.
\end{split}
\end{equation}
We say that the partition $\{\theta^{-1}(\t)\}_{\t\in\T}$ is \emph{non-degenerate} if the set $\{\t:\,\{\theta=\t\}\neq\emptyset\}$ coincides with its non-degeneracy set.
\end{definition}

In other words, the set $\{\theta = \t\}$ is non-degenerate if there are two points $w$, $w'$ in $\mathtt p_{\R^k} \{\theta = \t\}$ such that
\[
( w' - \C(\a) ) \cap \{\theta = \t\} \subset w' - \partial \C(\a), \qquad ( w + \C(\a) ) \cap \{\theta = \t\} \subset w + \partial \C(\a).
\]

\begin{proposition}
\label{P_fol_char}
Let $\{\theta^{-1}(\t)\}_{\t\in\T}$ be a non-degenerate $\mathtt c_{\C}$-Lipschitz foliation. Then, there exist two Borel functions
\[
\mathtt h^-,\mathtt h^+ : \Bigl\{ (\t,x)\in \T \times \R^{k-1} : x \in \mathtt p_{\R^{k-1}}(\{\theta=\t\}) \Bigr\} \to \R
\]
such that
\begin{enumerate}
\item \label{Prop_1_fol_char} $x \mapsto \mathtt h^+(\t,x),\mathtt h^-(\t,x)$ are $|\cdot|_{D(\a)^*}$-Lipschitz functions for all $\t$, where $\{\a\} = \mathtt p_{\A} \{\theta=\t\}$ and $D(\a)^*$ is given by \eqref{E_D_star_q};
\item \label{Prop_2_fol_char} $\clos\, \{\theta=\t\}(\a) \subset \Big\{ (x,y) \in \R^{k-1}\times \R : \mathtt h^-(\t,x) \leq y \leq \mathtt h^+(\t,x) \Big\}$;
\item \label{Prop_3_fol_char} $\interr \{\theta=\t\}(\a) = \Big\{ (x,y)\in \times \R^{k-1}  \times \R : \mathtt h^-(\t,x) < y < \mathtt h^+(\t,x) \Big\}$.
\end{enumerate}
% 
% \begin{subequations}
% \begin{equation*}
% \mathtt h^+(\a,\t),\,\mathtt h^-(\a,t)\quad\text{ are $|\cdot|_{\tilde{\bD(\a)}}$-Lipschitz functions, }\forall\,(\a,\t); %\\
% \end{equation*}
% \begin{equation*}
% \clos \{\theta=\t\}(\a)\subset\{(x,x')\in\R^{k-1}\times\R:\,\mathtt h^-(\a,\t,x)\leq x'\leq \mathtt h^+(\a,\t,x)\}; %\\
% \end{equation*}
% \begin{equation*}
% \inter \{\theta=\t\}(\a)=\{(x,x')\in\R^{k-1}\times\R:\,\mathtt h^-(\a,\t,x)< x'< \mathtt h^+(\a,\t,x)\}.
% \end{equation*}
% \end{subequations}
In particular,
\begin{equation}
\label{E_Lipschitzgraph}
\begin{split}
\interr \{\theta=\t\}(\a) = \emptyset \quad & \Longleftrightarrow \quad \mathtt h^-(\t,x) = \mathtt h^+(\t,x) \ \text{for} \ x \in \mathtt p_{\R^{k-1}} \{\theta=\t\}(\a) \\
& \Longleftrightarrow \quad \{\theta=\t\}(\a) \text{ is a complete $\mathtt c_{\C(\a)}$-Lipschitz graph} \\
& \Longleftrightarrow \quad \{\theta=\t\}(\a) = \Graph\,\mathtt h^\pm(\t) \llcorner_{\mathtt p_{\R^{k-1}} \{\theta=\t\}(\a) \times \R}.
\end{split}
\end{equation}
\end{proposition}

\begin{proof}
By \eqref{E_compl_fol}, $\{\theta=\t\}(\a)$ satisfies \eqref{E_G2}. In particular, for all $x \in \mathtt p_{\R^{k-1}} (\{\theta=\t\}(\a))$ the set 
\[
\big\{ y \in \R: (x,y) \in \{\theta=\t\}(\a) \big\}
\]
is a segment. Thus define for $x \in \mathtt p_{\R^{k-1}} (\{\theta=\t\}(\a))$
\begin{subequations}
\label{E_h_def}
\begin{equation}
\label{E_h-def}
\mathtt h^-(\t,x) := \inf \Big\{ y : \mathtt c_{\C(\a)}(x',y',x,y) < +\infty \text{ for some } (x',y') \in \{\theta=t\}(\a)\Big\},
\end{equation}
\begin{equation}
\label{E_h+def}
\mathtt h^+(\t,x) := \sup \Big\{ y : \mathtt c_{\C(\a)}(x,y,x',y') < +\infty \text{ for some } (x',y') \in \{\theta=t\}(\a) \Big\}.
\end{equation}
\end{subequations}
Since $\{\theta^{-1}(\t)\}_{\t \in \T}$ is non degenerate, then for all $x \in \mathtt p_{\R^{k-1}} (\{\theta=\t\}(\a))$ it follows that
\[
(x,\mathtt h^-(\t,x)) \cap \inter \big\{ \mathtt c_{\C(\a)}(\cdot,w') < +\infty \big\} =(x,\mathtt h^+(\t,x)) \cap \inter \big\{ \mathtt c_{\C(\a)}(w,\cdot) < +\infty \big\} = \emptyset,
\]
where $w$, $w'$ are the points of non-degeneracy \eqref{E_nondeg}, so that $\mathtt h^+$, $\mathtt h^-$ are real valued functions.

Using again property \eqref{E_G2}, one has also
\[
\mathtt h^-(\t,x) = \inf \Big\{ y' + |x - x'|_{D(\a)^*}: (x',y') \in \{\theta=\t\}(\a) \} \Big\},
\]
\[
\mathtt h^+(\t,x) = \sup \Big\{ y' - |x' - x|_{D(\a)^*}: (x',y') \in \{\theta=\t\}(\a)\} \Big\},
\]
which show that $\mathtt h^+$, $\mathtt h^-$ are $|\cdot|_{D(\a)^*}$-Lipschitz, proving Point \eqref{Prop_1_fol_char}. Points \eqref{Prop_2_fol_char} and \eqref{Prop_3_fol_char} of the statement are an immediate corollary of the definitions \eqref{E_h_def} and property \eqref{E_G2}.

Finally \eqref{E_Lipschitzgraph} is a straightforward consequence of the first part of the statement.
\end{proof}

\begin{remark}
\label{R_nondeg}
From the proof of Proposition \ref{P_fol_char} it is clear that out of the non-degeneracy set there are three possibilities: either the function defined in \eqref{E_h-def} is identically $-\infty$ or the function in \eqref{E_h+def} is identically $+\infty$ or both things happen. Hence, for all $\a\in\A$ there exist at most countably many $\{\t_{\a_n}\}_{n\in\N}\subset\T$ s.t. 
\[
 \interr\{\theta=\t_{\a_n}\}(\a)\neq\emptyset.
 %In particular, the set of such $(\a,\t)$ is the set for which $\{\theta=\t)\}(\a)$ is unbounded._a=\t
\]
%and there exist $h^-,\,h^+$ as above such that whenever $\t\in\p_{\A}^{-1}(\{a\})\setminus\{\t_{\a_n}\}_{n\in\N}$, 
%Since from Section \ref{Ss_regu_resi_set} onwards we will be concerned only with the sets $\{\theta=\t\}$ which satisfy \eqref{E_nondeg}, to consider 
\end{remark}
In the following proposition we give the two examples of $\mathtt c_{\C}$-Lipschitz foliations we will deal with in the rest of the paper.

\begin{proposition}
\label{P_ex_fol}
(1) Let $G=\Graph\,\varphi\subset\R^k$ be a complete $\mathtt c_{\tilde{C}}$-Lipschitz graph. Then, 
 the trivial equivalence relation on $\{\a_0\}\times G$ given by the constant quotient map $\theta_\varphi(\{\a_0\}\times G)=\a_0$
determines a $\mathtt c_{\C}$-Lipschitz foliation in $\{\a_0\}\times G\subset\{\a_0\}\times\R^k$, being $\C$ the constant cone direction map $\C:\{\a_0\}\times G\ni (\a_0,w)\mapsto\tilde C\in\mathcal C(k, \R^k)$, which satisfies \eqref{E_Lipschitzgraph}. 
 
(2) Let $\{\theta^{-1}(\t)\}_{\t\in\T}$ be the equivalence classes of a $\mathtt c_{\C}$-compatible linear preorder $\preccurlyeq$ on $\A\times\R^k$ with $\sigma$-compact graph (see Definition \ref{D_compatible}). Then they form a $\mathtt c_{\C}$-Lipschitz foliation.
\end{proposition}

%\begin{example}
%\label{Ex_graphfol}
%A trivial example of $\mathtt c_{\C}$-Lipschitz partition for which \eqref{E_Lipschitzgraph} holds for all $\t$ is given by a complete %$\mathtt c_{\tilde C}$-Lipschitz graph $G$: simply take $\T=\{\t_0\}$. In particular, if $G=\Graph\,\varphi$, $\varphi:\dom\,%\varphi\subset\R^{k-1}\to\R$, we will denote the quotient map of the associated (trivial) partition by $\theta_\varphi$.
%\end{example}

%\begin{example}
%\label{Ex_orderfol}
\begin{proof}
To prove the first part of the proposition, it is sufficient to notice that property \eqref{E_compl_fol} corresponds to \eqref{E_G2} and \eqref{E_Lipschitzgraph} is a consequence of \eqref{E_G1}.

As for the second example, by \eqref{E_tA22} observe that if $w,w'\in\{\theta=\t\}(\a)$ then 
\[
\big\{ \mathtt c_{\C(\a)}(w,\cdot) < +\infty \big\} \cap \big\{ \mathtt c_{\C(\a)}(\cdot,w') < +\infty \big\} \subset \big\{ w'' : w \preccurlyeq w'' \text{ and } w'' \preccurlyeq w' \big\} \subset \{\theta=\t\}(\a).
\]
\end{proof}

% Notice moreover that, since $\mathtt c_{\C}(\a,w,\a'w')<+\infty\quad\Rightarrow\quad\a=\a'$, the $\mathtt c_{\C}$-compatibility condition implies
% \begin{equation}
%  \{\bar\theta=(\a,\t)\}\subset\{\a\}\times\R^k.
% \end{equation}
%\end{example}

In view of \eqref{E_Lipschitzgraph}, we extend Definition \ref{D_tilde_C_diff} to $\mathtt c_{\C}$-Lipschitz foliations.

\begin{definition}
\label{D_partial_theta}
We define the \emph{superdifferential} of a $\mathtt c_{\C}$-Lipschitz foliation $\{\theta^{-1}(\t)\}_{\t \in \T}$ as the set $\partial^+ \theta \subset \A \times \R^k \times \A \times \R^k$ defined by
\begin{equation*}
\partial^+\theta(\a,\a') = \Big\{ (w,w') : \theta(\a,w) = \theta(\a',w') \quad \text{and} \quad \mathtt c_{\C}(\a,w,\a',w') < +\infty \Big\}.
\end{equation*}
Analogously, we define its \emph{subdifferential} as the set $\partial^-\theta \subset \A \times \R^k \times \A \times \R^k$ given by
\begin{equation*}
\partial^-\theta(\a,\a') = \Big\{ (w,w') : \theta(\a,w) = \theta(\a',w') \quad \text{and} \quad \mathtt c_{\C}(\a',w',\a,w) < +\infty \Big\}.
\end{equation*}
\end{definition}

Clearly
\[
\partial^-\theta=(\partial^+\theta)^{-1}
\]
and since by \eqref{E_caa'}
\begin{equation}
\label{E_subaa'}
\partial^\pm \theta(\a,\a') \not= \emptyset \quad \Longrightarrow \quad \a = \a',
\end{equation}
then for simplicity we will use the notation
\[
\partial^\pm\theta(\a) = \partial^\pm\theta(\a,\a).
\]

\begin{remark}
Recalling Definition \ref{D_tilde_C_diff}, notice that
\[
\partial^{\pm}\theta(\a) = \bigcup_{\nfrac{\t \in \T}{\mathtt p_{\A} \{\theta=\t\} = \a}} \mathtt p_{\R^k \times \R^k} \partial^\pm (\{\theta=\t\}(\a)),
\]
and, in particular, as for $\mathtt c_{\tilde C}$-Lipschitz graphs, we have the \emph{transitivity property}
\begin{equation}
\label{E_transitivity_subdifferential_a}
w' \in \partial^\pm \theta(\a,w) \quad \Rightarrow \quad \partial^\pm \theta(\a,w') \subset \partial^\pm \theta(\a,w).
\end{equation}
%consequence of \eqref{E_sum_of_cones_inside_cone}. %$C + C \subset C$ for every cone $C \in \mathcal C(\ell,\R^k)$, $\ell = 0,\dots,k$.

%We finally observe that, for all $\a \in \mathtt p_{\A} \theta$, there exist at most countably many $\{\t_{n_\a}\}_{n\in\N} \subset %\T$ such that
%\begin{equation}
%\label{E_theta_inter}
%\interr \{\theta=\t_{n_\a}\} \neq \emptyset,
%\end{equation}
%because of $\sigma$-finiteness of the Lebesgue measure $\mathcal L^k$.
\end{remark}

\subsection{Regular transport sets and residual set}
\label{Ss_regu_resi_set}

In this section we consider only the elements of a $\mathtt c_{\C}$-foliation whose $\{\a\}$-sections are partitions into complete $\mathtt c_{\C(\a)}$-Lipschitz graphs, namely level sets of $\theta$ which satisfy \eqref{E_Lipschitzgraph}.

We go through a careful analysis of the geometric properties of the super/subdifferentials of $\mathtt c_{\C}$-Lipschitz foliations, which finally leads to partition them into \emph{forward/backward regular sets} and a \emph{residual set}. 

By \eqref{E_subaa'}, in the following definitions the variable $\{\a\}$ simply plays the role of a parameter. Then, to understand the geometric structure of the problem one can also think from now onwards that $\A=\{\a_0\}$. The dependence on $\a$ is kept in order to show that all the sets and functions constructed below depend measurably (resp. Borel or $\sigma$-continuously) on the parameter $\a$.
We will define the following sets through their $\{\a\}$-sections.

\begin{description}

\item[Forward/backward transport set] the \emph{forward/backward transport sets} are respectively defined by
\begin{subequations}
\label{E_forward_backward_set}
\begin{equation}
\label{E_forward_transport_set}
\mathcal T^+\theta(\a) := \big\{ w : \partial^+ \theta(\a,w)\not= \{w\} \big\} = \mathtt p_1 \big( \partial^+ \theta(\a)\setminus \Id \big),
\end{equation}
\begin{equation}
\label{E_backward_transport_set}
\mathcal T^-\theta(\a) := \big\{ w : \partial^-\theta(\a,w)\not= \{w\} \big\} = \mathtt p_1 \big( \partial^-\theta(\a) \setminus \Id \big).
\end{equation}
\end{subequations}

\item[Set of fixed points] the \emph{set of fixed points} is given by
\begin{equation}
\label{E_set_of_fixed_points}
\mathcal F\theta(\a):= \R^k \setminus \big( \mathcal T^-\theta(\a) \cup \mathcal T^+\theta(\a)
\big).
\end{equation}

\item[Forward/backward direction multifunction] The \emph{forward/backward direction multifunction} are respectively given by
\begin{subequations}
\label{E_forw_back_dir_theta}
\begin{equation}
\label{E_forward_directions_bar_theta}
\mathcal D^+ \theta(\a):= \bigg\{ \bigg( w, \frac{w'-w}{|w'-w|} \bigg) : w \in \mathcal T^+\theta(\a), w' \in \partial^+\theta(\a,w) \setminus \{w\} \bigg\},
\end{equation}
\begin{equation}
\label{E_backward_directions_bar_theta}
\mathcal D^- \theta(\a):= \bigg\{ \bigg( w, \frac{w-w'}{|w-w'|} \bigg) : w \in \mathcal T^-\theta(\a), w' \in \partial^- \theta(\a,w) \setminus \{w\} \bigg\}.
\end{equation}
\end{subequations}

\end{description}

The following proposition collects the fundamental properties of the super/subdifferentials of $\mathtt c_{\C}$-Lipschitz foliations. The most striking feature is that, due to the completeness property \eqref{E_G2} of its $\mathtt c_{\C(\a)}$-Lipschitz graphs, the forward/backward direction multifunctions at a point of $\{\theta=\t\}(\a)$ contain all the information about the super/subdifferential at that point and also in a ``neighborhood'' of it (see Remark \ref{R_directionfaces}). Whenever $w,w'\in\{\theta=\t\}(\a)$, we will define $C(\a)(w,w')$ and $\O(\a)(w,w')$ as in Definitions \ref{D_cww'} and \eqref{E_G2} for the convex cone $\tilde C=\C(\a)$.

\begin{proposition}
\label{Pdirectionfaces1}
Let $F \subset \C(\a)$ be an extremal cone, $w\in\mathcal T^+\theta(\a)$. The following conditions are equivalent:
\begin{enumerate}
\item \label{item1Pfaces} $F\cap \mathbb S^{k-1}\subset \mathcal D^+\theta(\a,w)$; %\\
\item \label{item2Pfaces} there exists $w' \in \partial^+ \theta(\a,w)$ such that $F=C(\a)(w,w')$; %\\
\item \label{item4Pfaces} $\O(\a)(w,w') \subset \partial^+ \theta(\a,w)$ for some $w'$ such that $F=C(\a)(w,w')$; %\\
\item \label{item3Pfaces} there exists $\delta = \delta(w,F)>0$ such that
\begin{equation*}
 %\label{E_ballF}
B^k(w,2\delta)\cap w+F\subset\partial^+\theta(\a,w).
\end{equation*}
\end{enumerate}
In particular, if $F$ satisfies one of the above conditions, then for all $w'$ as in (\ref{item2Pfaces}-\ref{item4Pfaces}) 
\begin{equation}
\label{E_ballF}
F \cap \mathbb S^{k-1} \subset \mathcal D^+\theta(\a,\bar w) \qquad \forall\, \bar w \in \interr \O(\a)(w,w').
\end{equation}

Finally, if $F$ is maximal w.r.t. set inclusion among the extremal cones of $\C(\a)$ satisfying (\ref{item1Pfaces}-\ref{item3Pfaces}), then
\begin{equation}
\label{E_ballFmax}
F \cap \mathbb S^{k-1} = \mathcal D^+\theta(\a,\bar w) = \conv_{\mathbb S^{k-1}} \mathcal D^+\theta(\a,\bar w) \qquad \forall\, \bar w \in \interr \O(\a)(w,w'),
\end{equation}
where $w'$ is chosen as in (\ref{item2Pfaces}-\ref{item4Pfaces}).
\end{proposition}

\begin{proof}
{\it\eqref{item1Pfaces}$\,\Rightarrow\,$\eqref{item2Pfaces}.} It is sufficient to take $w'\in\partial^+\theta(\a,w)\setminus\{w\}$ s.t.
\[
\frac{w'-w}{|w'-w|}\in\interr F\cap \mathbb S^{k-1}
\]
and recall Definition \ref{D_cww'}.

{\it\eqref{item2Pfaces}$\,\Rightarrow\,$\eqref{item4Pfaces}.} If $w'$ satisfies \eqref{item2Pfaces}, then \eqref{item4Pfaces} follows from the completeness assumption \eqref{E_compl_fol} in the definition of $\mathtt c_{\C}$-Lipschitz foliation. 

{\it\eqref{item4Pfaces}$\,\Rightarrow\,$\eqref{item3Pfaces}.} It follows immediately from \eqref{E_O_ball}.

{\it\eqref{item3Pfaces}$\,\Rightarrow\,$\eqref{item1Pfaces}.} It is a direct consequence of the definition of $\mathcal D^+\theta(\a)$.

{\it Proof of \eqref{E_ballF}.}
Let $\bar w\in\interr\O(\a)(w,w')$. Then, by the geometric properties of $\O(\a)(w,w')$ given in Proposition \ref{P_parall}, it is fairly easy to see that $w'\in\partial^+\theta(\a,\bar w)$, and $C(\a)(\bar w,w')\cap \S^{k-1}\subset\mathcal D^+\theta(\a,\bar w)$. Since $\bar w\in\interr\O(\a)(w,w')$, $C(\a)(\bar w,w')=F$.

{\it Proof of \eqref{E_ballFmax}.}
Let now $F$ be maximal and $\bar w\in \interr\O(\a)(w,w')$. By \eqref{E_ballF} we already know that $F\cap \S^{k-1}\subset\mathcal D^+\theta(\a,\bar w)$. Let us assume that there exists $\hat w\in\partial^+\theta(\a,\bar w)$ such that $\hat w - \bar w \in \R^k \setminus F$: being $F$ an extremal cone, then one has also $\hat w - \bar w \in \R^k \setminus \mathrm{aff}\,F$. By the transitivity property \eqref{E_transitivity_subdifferential_a}, $\hat w \in \partial^+\theta(\a,w) \setminus (w+F)$ and, by simple geometrical considerations similar to those made in the proof of Proposition \ref{P_parall},
\[
F \subsetneq \R^+(\O(\a)(w,\hat w)-w)
\]
with strict inclusion. Hence, by the completeness assumption \eqref{E_compl_fol}, this contradicts the maximality of $F$.
\end{proof}

A completely similar proposition can be proved for $\partial^- \theta$: we state it without proof.

\begin{proposition}
\label{Pdirectionfaces1-}
Let $F \subset \C(\a)$ be an extremal cone, $w\in\mathcal T^-\theta(\a)$. The following conditions are equivalent:
\begin{enumerate}
\item \label{item1Pfaces-} $F \cap \mathbb S^{k-1} \subset \mathcal D^-\theta(\a,w)$;
\item \label{item2Pfaces-} there exists $w'' \in \partial^- \theta(\a,w)$ such that $F=C(\a)(w'',w)$;
\item \label{item4Pfaces-} $\O(\a)(w'',w) \subset \partial^- \theta(\a,w)$ for some $w''$ such that $F=C(\a)(w'',w)$;
\item \label{item3Pfaces-} there exists $\delta = \delta(w,F) > 0$ such that 
\begin{equation*}
%\label{E_ballF}
B^k(w,2\delta) \cap w - F \subset \partial^- \theta(\a,w).
\end{equation*}
\end{enumerate}
In particular, if $F$ satisfies one of the above conditions, then for all $w''$ as in (\ref{item2Pfaces-}-\ref{item4Pfaces-})  
\begin{equation*}
% \label{E_ballF-}
F \cap \mathbb S^{k-1} \subset \mathcal D^- \theta(\a,\bar w) \qquad \forall\,\bar w\in \interr\O(\a)(w'',w).
\end{equation*}

Finally, if $F$ be maximal w.r.t. set inclusion among the extremal cones of $\C(\a)$ satisfying (\ref{item1Pfaces-}-\ref{item3Pfaces-}), then
\begin{equation}
\label{E_ballFmax-}
F \cap \mathbb S^{k-1} = \mathcal D^- \theta(\a,\bar w) = \conv_{\mathbb S^{k-1}} \mathcal D^- \theta(\a,\bar w) \qquad \forall\, \bar w\in\interr\O(\a)(w'',w),\end{equation}
where $w''$ is chosen as in (\ref{item2Pfaces-}-\ref{item4Pfaces-}).
\end{proposition}

\begin{remark}
\label{R_directionfaces}
The radii of the balls $\delta=\delta(F,w)$ satisfying Point \eqref{item3Pfaces} in Propositions \ref{Pdirectionfaces1}-\ref{Pdirectionfaces1-} for a fixed $w \in \mathcal T^\pm \theta(\a)$ might actually change as $F$ varies in the set of extremal cones satisfying Point \eqref{item1Pfaces} and even tend to zero for some sequence of distinct $\{F_n\}_{n\in\N}$.
\end{remark}

%As a consequence of \eqref{E_ballFmax} and \eqref{E_ballFmax-}, at points $w\in\mathcal T^\pm \theta(\a)$ for which 
%$\mathcal D^\pm \theta(\a,w)=\conv_{\mathbb S^{k-1}}\mathcal D^\pm(w)$, we will call sometimes $C^+(w):=\conv_{\mathbb S^{k-1}}\mathcal D^+(w)$,
% $C^-(w):=\conv_{\mathbb S^{k-1}}\mathcal D^-(w)$ the \emph{maximal forward/backward extremal face} at $w$. Notice that, by the remark following 
%Definition \ref{D_cww'}, $C^+(w)$ and $C^-(w)$ have at most dimension $k(D)$ \eqref{}. 

Finally, we define the \emph{($\ell$-dimensional) forward/backward transport sets} and the \emph{residual sets}: they are defined in terms 
of properties of the forward/backward direction multifunctions, i.e. of ``local'' (see Remark \ref{R_directionfaces}) properties of their super/subdifferentials.

\begin{description}
\item[$\ell$-dimensional forward/backward regular transport set] for $\ell = 1,\dots,k-1$, the \emph{$\ell$-dimensional forward/backward regular transport sets} are defined respectively as
\begin{subequations}
\label{E_ell_prime_forward_backward_regular_transport_set}
\begin{equation}
\label{E_forward_regular_transport_set}
\begin{split}
\mathcal R^{+,\ell} \theta(\a) := \bigg\{ w \in \mathcal T^+ \theta(\a) : (i)&~ \mathcal D^+ \theta(\a,w) = \conv_{\mathbb S^{k-1}} \mathcal D^+ \theta(\a,w) \crcr
(ii)&~ \dim \big( \conv_{\mathbb S^{k-1}} \mathcal D^+ \theta(\a,w) \big) = \ell-1 \crcr
(iii)&~ \exists\, w'' \in \mathcal T^+ \theta(\a) \cap \Big( w - \inter_{\mathrm{rel}} \big( \R^+\conv_{\mathbb S^{k-1}} \mathcal D^+\theta(\a,w) \big) \Big) \crcr
&~ \text{such that } \theta(\a,w'') = \theta(\a,w) \text{ and } (i),(ii) \ \text{hold for} \ w'' \bigg\}.
\end{split}
\end{equation}
\begin{equation}
\label{E_backward_regular_transport_set}
\begin{split}
\mathcal R^{-,\ell} \theta(\a) := \bigg\{ w \in \mathcal T^- \theta(\a) : (i)&~ \mathcal D^- \theta(\a,w) = \conv_{\mathbb S^{k-1}} \mathcal D^- \theta(\a,w) \crcr
(ii)&~ \dim \big( \conv_{\mathbb S^{k-1}} \mathcal D^- \theta(\a,w) \big) = \ell-1 \crcr
(iii)&~ \exists\, w'' \in \mathcal T^- \theta(\a) \cap \Big( w + \inter_{\mathrm{rel}} \big( \R^+\conv_{\mathbb S^{k-1}} \mathcal D^-\theta(\a,w) \big) \Big) \crcr
&~ \text{such that } \theta(\a,w'') = \theta(\a,w) \text{ and } (i),(ii) \ \text{hold for} \ w'' \bigg\}.
\end{split}
\end{equation}
\end{subequations}

\item[Forward/backward regular transport sets] the \emph{forward/backward regular transport sets} are defined respectively by
\begin{equation}
\label{E_forward_backward_regular_transport_set}
\mathcal R^+ \theta(\a):= \bigcup_{\ell = 1}^{k-1} \mathcal R^{+,\ell} \theta(\a), \qquad \mathcal R^- \theta(\a):= \bigcup_{\ell = 1}^{k-1} \mathcal R^{-,\ell} \theta(\a).
\end{equation}

\item[Regular transport set] the \emph{regular transport set} is defined by
\begin{equation}
\label{E_regular_transport_set}
\mathcal R \theta(\a) := \mathcal R^- \theta(\a) \cap \mathcal R^+\theta(\a).
\end{equation}

\item[Residual set] the \emph{residual set} is defined by
\begin{equation}
\label{E_residual_set_N}
\NN \theta(\a):= \bigl(\mathcal T^+\theta(\a)\cup\mathcal T^-\theta(\a)\bigr) \setminus \mathcal R \theta(\a).
\end{equation}

\end{description}

%Sara
Property \eqref{E_forward_regular_transport_set} $(iii)$ (\eqref{E_backward_regular_transport_set} $(iii)$) will be crucial in Theorem \ref{T_partition_E+-} in order to prove that the sets of points in $\mathcal R^{+,\ell} \theta(\a)$ ($\mathcal R^{-,\ell} \theta(\a)$) which belong to the same level set of $\theta$ and whose superdifferentials (subdifferentials) have the same affine span are $\ell$-dimensional locally affine sets (see Definition \ref{D_locaff}).

\begin{figure}
\centering{\resizebox{14cm}{6cm}{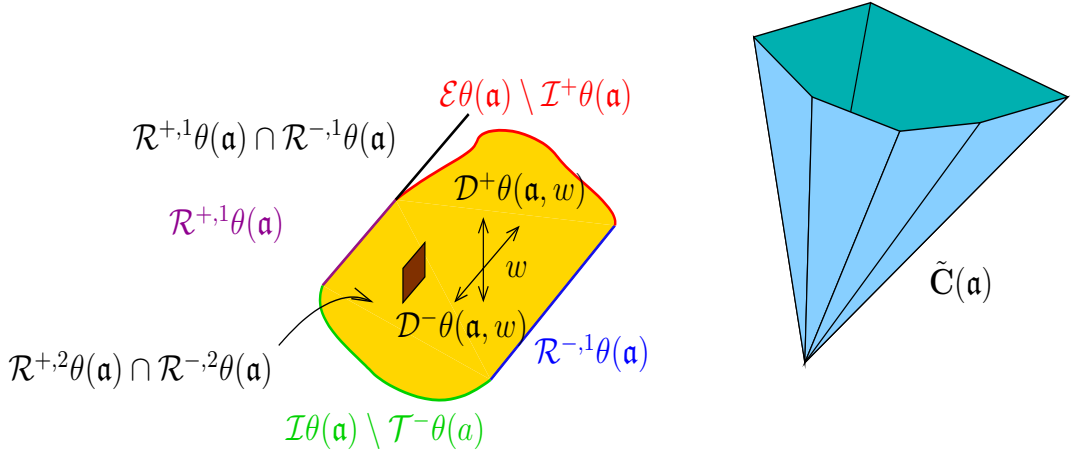}}
\caption{A possible decomposition of a level set of $\theta$ (or equivalently of a complete $\mathtt c_{\tilde{\mathbf{C}}(\a)}$-Lipschitz graph), with the various sets introduced in Section \ref{Ss_regu_resi_set}. More precisely, given the $3$-dimensional cone of directions $\tilde{\mathbf C}(\a)$, the yellow region represents the set of regular points $\mathcal R^{+,2}(\a)\cap\mathcal R^{-,2}(\a)$ with $2$-dimensional forward/backward cones of directions $\mathcal D^{+}\theta(\a)=\mathcal D^-\theta(\a)$, the black line the set of regular points $\mathcal R^{+,1}(\a)\cap\mathcal R^{-,1}(\a)$ with $1$-dimensional forward/backward cones of directions, the blue line the set $\mathcal R^{-,1}\theta(\a)$ of points with $1$-dimensional backward cone $\mathcal D^-\theta(\a)$ and $2$-dimensional forward cone $\mathcal D^{+}\theta(\a)$ and the purple line  the set $\mathcal R^{+,1}\theta(\a)$ of points with $1$-dimensional forward cone $\mathcal D^+\theta(\a)$ and $2$-dimensional backward cone $\mathcal D^{-}\theta(\a)$. The brown region represents a set $\mathtt O(\a)(w',w'')$ as in \eqref{E_O_q_ww'}. As we will see in Section \ref{Ss_analysis_residual_set}, the red curve represents the set of final points $\mathcal E\theta(\a)\setminus \mathcal T^+\theta(\a)$, the green curve the set of initial points $\mathcal I\theta(\a)\setminus\mathcal T^-(\a)$, and also the blue line is contained in $\mathcal I\theta(\a)$ and the purple line in $\mathcal E\theta (\a)$.}   
\label{Fi_decotheta}
\end{figure}

We now prove that all the above sets, except $\mathcal F \theta$ and $\mathcal N \theta$, are $\sigma$-compact.

\begin{proposition}
\label{P_borel_regu_transport_sets}
The sets $\partial^\pm \theta$, $\mathcal T^\pm\theta$, $\mathcal D^\pm \theta$, $\mathcal R^{\pm,\ell}\theta$, $\mathcal R^\pm\theta$, $\mathcal R\theta$ are $\sigma$-compact.

The sets $\mathcal F \theta$, $\mathcal N \theta$ are Borel.
\end{proposition}

\begin{proof}
For every set in the first statement of the proposition, we will construct $\sigma$-compact subsets of a Polish space whose projection corresponds to that particular set.

{\it Subdifferential:} consider the following sets:

\begin{enumerate}

\item $\{(\mathfrak a,w,\mathfrak a',w') : \mathtt c_{\C}(\mathfrak a,w,\mathfrak a',w') < \infty\}$: since the function $\mathtt c_{\C}$ is $\sigma$-continuous, it follows that this set is $\sigma$-compact;

\item $\{(\mathfrak a,w,\mathfrak a',w') : \theta(\mathfrak a,w) = \theta(\mathfrak a',w')\}$: the same reasoning of the previous point applies here, being $\theta$ a $\sigma$-continuous function in Definition \ref{D_Lip_fol} of $\mathtt c_{\tilde{\bC}}$-Lipschitz foliation.
\end{enumerate}
It follows that the set
\begin{align*}
\partial^- \theta =&~ \Big\{ (\a,w,\a',w') \in \A \times \R^k \times\A\times \R^k : {\mathtt c_{\C}}(\mathfrak a',w',\mathfrak a,w) < +\infty, \theta(\mathfrak a,w) = \theta(\mathfrak a,w') \Big\} \crcr
=&~ \big\{ (\mathfrak a,w,\mathfrak a',w') : \mathtt c_{\C}(\mathfrak a',w',\mathfrak a,w) < \infty \big\} \cap \big\{ (\mathfrak a,w,\mathfrak a',w') : \theta(\mathfrak a,w) = \theta(\mathfrak a',w') \big\}
\end{align*}
is $\sigma$-compact. 

{\it Backward transport set:} the set $\mathcal T^-\theta$ is the projection of the $\sigma$-compact set
\[
\bigcup_{n \in \N} \partial^- \theta \cap \big\{ (\mathfrak a,w,\a',w') : |w-w'| \geq 2^{-n} \big\},
\]
and thus it is $\sigma$-compact. %The same reasoning applies to $\mathcal T^+\theta$.

{\it Backward directions:} since
\[
\{w \not= w'\} \ni (w,w') \mapsto \frac{w-w'}{|w-w'|} \in \mathbb S^{k-1}
\]
is continuous, it follows that $\mathcal D^- \theta$ is $\sigma$-compact, being the image of a $\sigma$-compact set by a continuous function.

{\it Backward regular transport sets:} first notice that the map $A \mapsto \conv A$ is continuous w.r.t. the Hausdorff topology and that the sets
\begin{equation}
\label{E_C_m_def}
\mathcal C_m(\ell,\R^k) := \bigg\{ C\in\mathcal C(\ell;\R^k): \mathring C(-1\slash m) \supset \big( B^k(w_m,1\slash(2m)) \cap\, \mathrm{aff}\, C \big), \ \mathrm{dist} \big( w_m, \mathrm{aff}\, C) \leq 1/(4m) \bigg\}
\end{equation}
are closed w.r.t. the Hausdorff topology, for all $\ell=1,\dots,k-1$, $w_m \in \Q^k$, $m\in\N$. Since the function $\mathcal C(\ell,\R^k)\ni C \mapsto \mathrm{dim}\,C$ is constant on these sets, then it is $\sigma$-continuous.

Let us now prove that the set
\[
\Big\{ (w,w',C) \in \R^k \times \R^k \times \mathcal C(\ell,\R^k) : w' \in w - \interr C \Big\}
\]
is $\sigma$-compact. This follows by considering the closed sets $C(-r) \setminus B^k(0,r)$, observing that
\[
C_n \to C \quad \Longrightarrow \quad C_n(-r) \setminus B^k(0,r) \to C(-r) \setminus B^k(0,r),
\]
% 
% \[
% \mathring C(-r) \setminus B^k(0,r) = \R^+ \big\{ w \in \mathbb S^{k-1} :\exists\,\epsilon>0\text{ s.t. }B^k(w,r+\epsilon)\cap \aff\,C \subset C \big\}
% \]
and writing the previous set as the union of countably many $\sigma$-compact sets in the following way:
\[
\bigcup_{n \in \N} \bigg\{ (w,w',C) \in \R^k \times \R^k \times \mathcal C(\ell,\R^k) : w' \in w - \big( C(-2^{-n}) \setminus B^k(0,2^{-n}) \bigr) \bigg\}.
\]

From Proposition \ref{Pdirectionfaces1-}, we have moreover that
\[
 \mathcal D^- \theta(\mathfrak a,w) = \conv_{\mathbb S^{k-1}}\, \mathcal D^- \theta(\mathfrak a,w)\quad\Leftrightarrow\quad \mathcal D^- \theta(\mathfrak a,w) \cap\interr \conv_{\mathbb S^{k-1}}\, \mathcal D^- \theta(\mathfrak a,w)\neq\emptyset.
\]

Hence the set
\begin{align*}
\bigg\{ (\mathfrak a,w,\a', w',C):\ (i)&~ (\mathfrak a,w),(\mathfrak a',w') \in \mathcal T^-\theta \crcr
(ii)&~ \theta(\a,w) = \theta(\a',w')\crcr
(iii)&~ C = \R^+\conv_{\mathbb S^{k-1}}\mathcal D^\pm \theta(\a,w)\crcr
(iv)&~ w' \in w - \inter_{\mathrm{rel}}\, C  \crcr
(v)&~ \dim \big( \conv_{\mathbb S^{k-1}}\, \mathcal D^- \theta(\mathfrak a,w) \big) = \dim \big( \conv_{\mathbb S^{k-1}}\, \mathcal D^- \theta(\mathfrak a',w') \big) = \ell-1 \crcr
(vi)&~ \mathcal D^- \theta(\mathfrak a,w) = \conv_{\mathbb S^{k-1}}\, \mathcal D^- \theta(\mathfrak a,w), \mathcal D^-  \theta(\mathfrak a,w') = \conv_{\mathbb S^{k-1}}\, \mathcal D^- \theta(\mathfrak a,w') \bigg\}
\end{align*}
is $\sigma$-compact, being the finite intersection of $\sigma$-compact sets, and thus $\mathcal R^{-,\ell}\theta$ is $\sigma$-compact too.

The proof for $\partial^+  \theta$, $\mathcal T^+\theta$, $\mathcal D^+ \theta$ and $\mathcal R^{+,\ell}\theta$ is analogous, and hence the $\sigma$-compactness of $\mathcal R^+\theta$, $\mathcal R^-\theta$ and $\mathcal R\theta$ follows.

Being the difference of two $\sigma$-compact sets a Borel set, the Borel measurability of $\mathcal F \theta$ and $\mathcal N \theta$ is proved.
\end{proof}

\subsection{Super/subdifferential directed partitions of regular sets}
\label{Ss_partition_transport_set}

In this section we construct directed locally affine partitions of the forward regular sets and backward regular sets of a $\mathtt c_{\C}$-Lipschitz foliation, which will be respectively called \emph{superdifferential directed partitions} and \emph{subdifferential directed partitions}. As we will see, these partitions coincide on the regular set, thus giving a directed locally  affine partition which will be called \emph{$\mathtt c_{\C}$-differential directed partition}. 

Define the maps
\begin{subequations}
\label{E_partition_on_T}
\begin{equation}
\label{E_partition_on_T+}
\begin{array}{ccccc}
\mathtt v^+ &:& \mathcal R^+ \theta&\to& \T \times \overset{k-1}{\underset{\ell=1}{\cup}} \mathcal A(\ell,\R^k) \crcr
% &&&& \crcr
&& (\a,w) &\mapsto& \mathtt v^+(\a,w) := \big( \theta(\a,w), \aff\, \partial^+ \theta(\a,w) \big)
\end{array}
\end{equation}
% and
\begin{equation}
\label{E_partition_on_T-}
\begin{array}{ccccc}
\mathtt v^- &:& \mathcal R^- \theta&\to& \T \times \overset{k-1}{\underset{\ell=1}{\cup}} \mathcal A(\ell,\R^k) \crcr
% &&&& \crcr
&& (\a,w) &\mapsto& \mathtt v^-(\a,w) := \big( \theta(\a,w), \aff\, \partial^- \theta(\a,w) \big)
\end{array}
\end{equation}
\end{subequations}

% where $\mathfrak a(z)$ is defined by $z \in Z^\ell_{\mathfrak a(z)}$.
In the following, when clear from the context, we identify sets $E_\a\subset\{\a\}\times\R^k$ with $\p_{\R^k}E_\a$.
\begin{theorem}
\label{T_partition_E+-}
The map $\mathtt v^+$ induces a (complete) directed locally affine partition on $\mathcal R^+\theta$ into sets
%Sara
\begin{equation*}
\Big\{ Z^{\ell,+}_{\a,\mathfrak b}\Big\}_{\overset{\ell=1,\dots,k-1}{\a\in\A,\b\in\B_{\ell,+}(\a)}},\quad  Z^{\ell,+}_{\a,\mathfrak b}\subset \{\a\}\times \R^k
\end{equation*}
with direction cones
\[
\Big\{ C^{\ell,+}_{\a,\mathfrak b} \Big\}_{\overset{\ell=1,\dots,k-1}{\a\in\A,\b\in\B_{\ell,+}(\a)}},\quad C^{\ell,+}_{\a,\mathfrak b} \subset \{\a\}\times \mathcal C(\ell,\R^k)
\]
such that the following holds:
\begin{align}
\label{item2TpartE+}
&C^{\ell,+}_{\a,\b}\text{ is the extremal face of $\C(\a)$ s.t. $\mathcal D^+\theta(\a)(w)= C^{\ell,+}_{\a,\b}\cap \mathbb S^{k-1}$, $\forall\,w\in Z^{\ell,+}_{\a,\b}$. }
\end{align}

Analogously, the map $\mathtt v^-$ induces a (complete) directed locally affine partition of $\mathcal R^-\theta$ into sets
\begin{equation*}
% \label{E_v-part}
\Big\{ Z^{\ell,-}_{\a,\mathfrak b}\Big\}_{\overset{\ell=1,\dots,k-1}{\a\in\A,\b\in\B_{\ell,-}(\a)}},\quad  Z^{\ell,-}_{\a,\mathfrak b}\subset \{\a\}\times \R^k
\end{equation*}
with direction cones
\[
\Big\{C^{\ell,-}_{\a,\mathfrak b}\Big\}_{\overset{\ell=1,\dots,k-1}{\a\in\A,\b\in\B_{\ell,-}(\a)}},\quad C^{\ell,-}_{\a,\mathfrak b} \subset \{\a\}\times \mathcal C(\ell,\R^k)
\]
such that the following holds: 
\begin{align}
\label{item2TpartE-}
&C^{\ell,-}_{\a,\b}\text{ is the extremal face of $\C(\a)$ s.t. $\mathcal D^-\theta(\a)(w)= C^{\ell,-}_{\a,\b}\cap \mathbb S^{k-1}$, $\forall\,w\in Z^{\ell,-}_{\a,\b}$.}
\end{align}
\end{theorem}

As a corollary one obtains the following decomposition of $\mathcal R \theta$.

\begin{corollary}
\label{C_v}
The two maps $\mathtt v^+$, $\mathtt v^-$ coincide on $\mathcal R \theta$, i.e.
\begin{equation}
\label{E_v}
\mathtt v:=\mathtt v^+_{|_{\mathcal R\theta}} = \mathtt v^-_{|_{\mathcal R\theta}},
\end{equation}
and the map $\mathtt v : \mathcal R \theta \to \T \times \overset{k-1}{\underset{\ell=1}{\cup}} \mathcal A(\ell,\R^k)$ defined above induces on $\mathcal R \theta$ a (complete) directed locally affine partition $\{Z^\ell_{\a,\b},\,C^\ell_{\a,\b}\}_{\nfrac{\ell=1,\dots,k-1}{\a \in \A, \b \in \B_\ell(\a)}}$, where, for all $\b \in \B_\ell(\a)$ and $\ell=1,\dots,k-1$,
\[
Z^\ell_{\a,\b} = Z^{\ell,+}_{\a,\b} \cap Z^{\ell,-}_{\a,\b}, \qquad C^\ell_{\a,\b} = C^{\ell,+}_{\a,\b} = C^{\ell,-}_{\a,\b}.
\]
% for some $\b^\pm\in\B_{\ell,\pm}(\a)$.
In particular, both \eqref{item2TpartE+} and \eqref{item2TpartE-} are satisfied.
\end{corollary}

In the following we will use the following definitions.

\begin{definition}
\label{D_vpm}
Given a $\mathtt c_{\C}$-Lipschitz foliation $\theta$, the directed locally affine partition induced on $\mathcal R^+\theta$ by $\mathtt v^+$ is called the \emph{superdifferential directed partition}, while the directed locally affine partition induced on $\mathcal R^-$ by $\mathtt v^-$ is called the \emph{subdifferential directed partition}.

The partition induced by $\mathtt v$ on the regular points $\mathcal R\theta$ is called \emph{$\mathtt c_{\C}$-differential directed partition}. 
\end{definition}

We prove the part of Theorem \ref{T_partition_E+-} which regards $\mathcal R^+\theta$ and $\mathtt v^+$, being the one about $\mathcal R^-\theta$ and $\mathtt v^-$ completely symmetric. Notice that in the proof is also shown that the map $\a \mapsto \B_{\ell,\pm}(\a)$ is $\sigma$-continuous.

\begin{proof}[Proof of Theorem \ref{T_partition_E+-}]
Being a single-valued map, $\mathtt v^+$ clearly induces a partition of $\mathcal R^+\theta$. Moreover, $\mathtt v^+$ is $\sigma$-continuous by Proposition \ref{P_borel_regu_transport_sets} and the fact that the affine envelope of compact sets is $\sigma$-continuous w.r.t. Hausdorff topology. Since, by \eqref{E_compl_fol1}
\[
\mathtt v^+(\a,w)=\mathtt v^+(\a',w')\quad\Leftrightarrow\quad\a=\a',
\]
let
\[
\mathfrak B_{\ell,+}(\a) := \mathtt v^+( \{\a\} \times \mathcal R^+\theta(\a))\cap \T\times \mathcal A(\ell,\R^k)
\]
and $\forall\,\b\in \mathfrak B_{\ell,+}(\a)$ let
\[
Z^{\ell,+}_{\a,\b} := \Bigl\{ w \in \mathcal R^+\theta(\a) : \mathtt v^+(\a,w) = \b \Bigr\}.
\]
By \eqref{E_forward_regular_transport_set} (i), for all $w \in Z^{\ell,+}_{\a,\b}$
\[
\R^+\mathcal D^+\theta(\a,w) = \aff\,\partial^+\theta(\a,w) -w\cap \C(\a) = \mathtt p_2 \b-w \cap \C(\a)\in\mathcal C(\ell,\R^k).
\]
Thus, by \eqref{E_forward_regular_transport_set} (iii), $Z^{\ell,+}_{\a,\b}\subset\mathcal R^{+,\ell}\theta$ and
\[
C^{\ell,+}_{\a,\b} := \mathtt p_2 \b-w \cap \C(\a)
\]
satisfies \eqref{item2TpartE+}. 

Let us now show that $Z^{\ell,+}_{\a,\b}$ is relatively open in $\mathcal R^{+,\ell}\theta \cap \mathtt p_2 \b$. More precisely, we prove the following

\begin{claim}
\label{Cl_oww'neigh}
For all $w \in Z^{\ell,+}_{\a,\b}$, there exist
\[
w' \in \partial^+ \theta(\a,w) \cap w + \interr C^{\ell,+}_{\a,\b} \quad \text{and} \quad w'' \in \partial^- \theta(\a,w) \cap w - \interr C^{\ell,+}_{\a,\b}
\]
such that %as in \eqref{E_forward_regular_transport_set} (iii), 
\begin{equation*}
% \label{E_oww'neigh}
\interr \O(\a)(w'',w') \subset Z^{\ell,+}_{\a,\b}.
\end{equation*}
\end{claim}

By \eqref{E_forward_regular_transport_set} $(i)$ and \eqref{E_ballFmax} there exists $w'\in \partial^+ \theta(\a,w) \cap w + \interr C^{\ell,+}_{\a,\b}$ s.t. $w'\in Z^{\ell,+}_{\a,\b}$ and by \eqref{E_forward_regular_transport_set} $(iii)$ there exists $w'' \in \partial^- \theta(\a,w) \cap w - \interr C^{\ell,+}_{\a,\b}$ s.t. $w''\in Z^{\ell,+}_{\a,\b}$. By Proposition \ref{P_parall}, $\interr\O(\a)(w'',w')$ is a relatively open neighborhood of $w$ in $\mathtt p_2 \b$.

Let now $\bar w \in \interr\O(\a)(w'',w')$. By the completeness of the superdifferential \eqref{E_compl_fol}, 
\[
\bar w \in \partial^+ \theta(\a,w'') \quad \text{and} \quad w' \in \partial^+\theta(\a,\bar w).
\]
Hence,
\[
C^{\ell,+}_{\a,\b} = \R^+\mathcal D^+\theta(\a,w')\subset\R^+\mathcal D^+\theta(\a,\bar w)\subset\R^+\mathcal D^+\theta(\a,w'')=C^{\ell,+}_{\a,\b},
\]
thus implying that $\bar w\in Z^{\ell,+}_{\a,\b}$.
\end{proof}

The proof of Corollary \ref{C_v} follows easily from the proof of Theorem \ref{T_partition_E+-}.

\begin{proof}[Proof of Corollary \ref{C_v}]
If $w \in \mathcal R\theta(\a)$, then
\[
w \in Z^{\ell,+}_{\a,\b}\cap Z^{\ell',-}_{\a,\b'}\subset \mathcal R^{+,\ell} \theta(\a) \cap \mathcal R^{-,\ell'} \theta(\a).
\]
By the transitivity property \eqref{E_transitivity_subdifferential_a} one has that 
\[
\ell = \ell', \qquad \mathcal D^+ \theta(\a,w) = \mathcal D^- \theta(\a,w). 
\]
Hence, as in the proof of Claim \ref{Cl_oww'neigh},
\[
w \in \interr \O(\a)(w'',w') \subset Z^{\ell,+}_{\a,\b} \cap Z^{\ell,-}_{\a,\b} , \quad C^{\ell,+}_{\a,\b} = C^{\ell,-}_{\a,\b},
\]
yielding the conclusion.
\end{proof}

In the following we will use the notation $\hat\bD^+$, $\hat\bD^-$ and $\hat\bD$ to denote the $\sigma$-compact graphs of the directed locally affine partitions induced respectively by $\mathtt v^+$, $\mathtt v^-$ and $\mathtt v$, namely

\begin{subequations}
\label{E_sub_directed_partition_all}
\begin{align}
\label{E_sub_directed_partition_minus}
\hat{\mathbf D}^+ := \Big\{ \big( \ell,\mathfrak a,\mathfrak b,w,C \big) : \mathtt v^+(\a,w) = \b, C = \mathtt p_2 \b-w \cap \C(\mathfrak a), w \in Z^{\ell,+}_{\mathfrak a,\mathfrak b} \Big\},
\end{align}
\begin{align}
\label{E_sub_directed_partition_plus}
\hat{\mathbf D}^- := \Big\{ \big( \ell,\mathfrak a,\mathfrak b,w,C \big) : \mathtt v^-(\a,w) = \b, C = \mathtt p_2 \b-w \cap \C(\mathfrak a), w \in Z^{\ell,-}_{\mathfrak a,\mathfrak b} \Big\},
\end{align}
\begin{align}
\label{E_sub_directed_partition}
\hat{\mathbf D} := \Big\{ \big( \ell,\mathfrak a,\mathfrak b,w,C \big) : \mathtt v(\a,w) = \b, C = \mathtt p_2 \b-w\cap \C(\mathfrak a), w \in Z^\ell_{\mathfrak a,\mathfrak b} \Big\}.
\end{align}
\end{subequations}

We will also use the notations $\mathfrak c = (\mathfrak a,\mathfrak b)$, $Z^{\ell,+}_{\mathfrak c}$, $C^{\ell,+}_{\mathfrak c}$. Recalling that $\A\subset\R^{d-k}$ and observing that, after the partition into sheaf sets as in Proposition \ref{P_countable_partition_in_reference_directed_planes} and the injection into a fibration, we can take as in \eqref{E_mathfrak_A_k_def}
\[
\B_{\ell,+} = \bigcup _{\a \in \A} \B_{\ell,+}(\a) \subset \R^{k-\ell}, 
\]
we let $\mathfrak c \in \mathfrak C_{\ell,+} \subset \R^{d-\ell}$, $\ell=1,\dots,k-1$.

\subsection{\texorpdfstring{Analysis of the residual set}{Analysis of the residual set}}
\label{Ss_analysis_residual_set}

Now we give a characterization of the residual set as the union of \emph{initial} and \emph{final points} respectively for the superdifferential partition and the subdifferential partition. Moreover, we fully characterize the super/subdifferentials at each point of the super/subdifferential partitions in terms of the regular and initial/final points.

Recalling Definition \ref{D_initial_final} of initial and final points of a directed locally affine partition, let
\begin{equation}
\label{E_init_Z+}
\mathcal I^+\theta := \bigcup_{\ell,\a,\b} \mathcal I \bigl( Z^{\ell,+}_{\a,\b} \bigr)
\end{equation}
be the \emph{sets of initial points of the superdifferential partition $\{Z^{\ell,+}_{\a,\b}, C^{\ell,+}_{\a,\b}\}_{\ell,\a,\b}$} and let
\begin{equation}
\label{E_final_Z-}
\mathcal E^-\theta := \bigcup_{\ell,\a,\b} \mathcal E \bigl( Z^{\ell,-}_{\a,\b} \bigr)                                                                                                                                                                                                                                         
\end{equation}
be the \emph{set of final points of the subdifferential partition $\{Z^{\ell,-}_{\a,\b}, C^{\ell,-}_{\a,\b}\}_{\ell,\a,\b}$}.

\begin{theorem}
\label{T_partition_E}
The following holds:
\begin{equation}
\label{Point_residual_N_is_initial_final} 
\mathcal N \theta = \mathcal I^+\theta \cup \mathcal E^-\theta,
\end{equation}
and moreover
\begin{subequations}
\label{Point_sudiff_is_partition_in_fin}
\begin{equation}
\label{Point_sudiff_is_partition_in}
\partial^+ \theta(\a,w) = \big( w + C^{\ell,+}_{\a,\b} \big) \cap \big( Z^{\ell,+}_{\a,\b} \cup \mathcal E^- \theta(\a) \big), \qquad \forall\, w \in Z^{\ell,+}_{\a,\b},
\end{equation}
\begin{equation}
\label{Point_sudiff_is_partition_fin}
\partial^- \theta(\a,w) = \big( w - C^{\ell,-}_{\a,\b} \big) \cap \big( Z^{\ell,-}_{\a,\b} \cup \mathcal I^+ \theta(\a) \big), \qquad \forall\, w \in Z^{\ell,-}_{\a,\b}.
\end{equation}
\end{subequations}
\end{theorem}

\begin{remark}
\label{R_not_only_R}
Notice that in general the points of the set $\mathcal N \theta$, i.e. the complement of the set of regular points $\mathcal R\theta$, may not belong to the set $\mathcal I \theta \cup \mathcal E \theta$, i.e. the set of initial and final points for the directed locally affine partition $\{Z^\ell_{\a,\b},C^\ell_{\a,\b}\}_{\ell,\a,\b}$ induced by $\mathtt v$ on $\mathcal R\theta$ (see Figure \ref{Fi_infin}).
\end{remark}

\begin{figure}
\centering{\resizebox{9cm}{7cm}{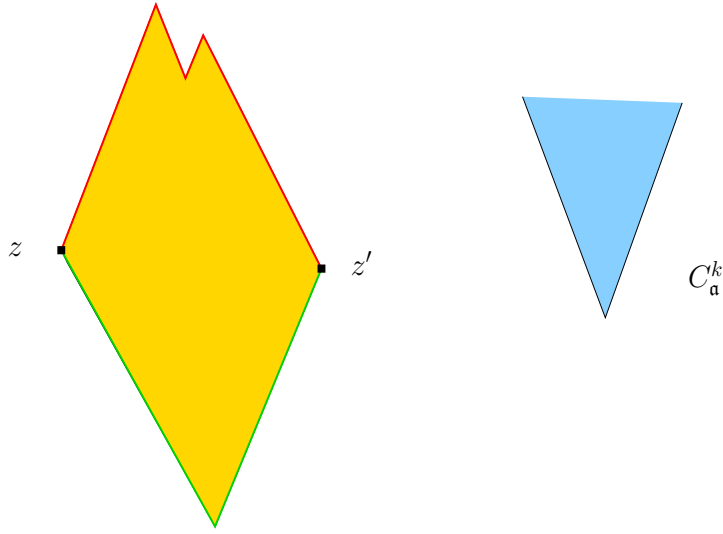}}
\caption{The yellow region is made of points in the regular set $\mathcal R\theta$. The points $z$ and $z'$ in the figure belong to $\mathcal I^+\theta\setminus (\mathcal I\theta\cup\mathcal E\theta)$. The points on the red segments belong instead to $\mathcal E\theta$ and the points on the green segments to $\mathcal I\theta$.}
\label{Fi_infin}
\end{figure}
The main observation in the proof of Theorem \ref{T_partition_E} is contained in the following

\begin{remark}
\label{R_t+t-i+e-}
We observe the following properties of $\mathcal T^+ \theta$, $\mathcal T^- \theta$:
\begin{align*}
& w \in \mathcal T^+\theta(\a) \quad \Longrightarrow \quad \exists\,r>0, \, Z^{\ell,+}_{\a,\b} \text{ such that } B^k(w,r) \cap \big( w + \interr C^{\ell,+}_{\a,\b} \big) \subset Z^{\ell,+}_{\a,\b}; \\
& w \in \mathcal T^-\theta(\a) \quad \Longrightarrow \quad \exists\,r>0, \, Z^{\ell,-}_{\a,\b} \text{ such that } B^k(w,r) \cap \big( w - \interr C^{\ell,-}_{\a,\b} \big) \subset Z^{\ell,-}_{\a,\b}.
\end{align*}
The statements follow respectively from \eqref{E_ballFmax} and \eqref{E_ballFmax-}, taking $C^{\ell,\pm}_{\a,\b}$ equal to any maximal extremal cone of $\R^+\mathcal D^\pm\theta(\a,w)$.
\end{remark}

%\begin{remark}
% \label{R_partition_E}
% From the proof of Theorem \ref{T_partition_E} it will be clear that the subset of $\mathcal E^-\theta(\a)$ in \eqref{Point_sudiff_is_partition_in} is disjoint from %$\underset{(\a',\b')\neq(\a,\b)}{\cup} Z^{\ell,+}_{\a,\b}$ (resp. the subset of $\mathcal I^+\theta(\a)$ in \eqref{Point_sudiff_is_partition_fin} is disjoint from %$\underset{(\a',\b')\neq(\a,\b)}{\cup} Z^{\ell,-}_{\a,\b}$), hence the superdifferential directed partition (resp. subdifferential directed partition) of a $\mathtt %c_{\tilde{\mathbf C}}$-Lipschitz foliation satisfies condition \eqref{E_more_than_complet}.
%\end{remark}

\begin{proof}[Proof of Theorem \ref{T_partition_E}]
Let $w \in \mathcal N\theta(\a)$: then from definition \eqref{E_residual_set_N},
\[
w \in \mathcal T^+\theta(\mathfrak a) \setminus \mathcal R\theta (\mathfrak a) \quad \text{or} \quad w \in \mathcal T^-\theta(\mathfrak a) \setminus \mathcal R\theta(\mathfrak a).
\]
If $w \in \mathcal T^+\theta(\mathfrak a) \setminus \mathcal R^+\theta (\mathfrak a)$, then by Remark \ref{R_t+t-i+e-} and Definition \ref{D_initial_final} it follows that $w\in\mathcal I^+\theta(\a)$. If $w \in \bigl(\mathcal T^+\theta(\mathfrak a)\cap\mathcal R^+\theta (\mathfrak a)\bigr)\setminus \mathcal R^-\theta (\mathfrak a)$,  $w\in\mathcal T^-\theta(\a)$ by \eqref{E_forward_regular_transport_set} (iii) and then again by Remark \ref{R_t+t-i+e-} we conclude that $w\in\mathcal E^-\theta(\a)$. \\
Analogously, $\mathcal T^-\theta(\mathfrak a) \setminus \mathcal R^-\theta (\mathfrak a)\subset\mathcal E^-\theta(\a)$ and $\bigl(\mathcal T^-\theta(\mathfrak a)\cap\mathcal R^-\theta (\mathfrak a)\bigr)\setminus \mathcal R^+\theta (\mathfrak a)\subset\mathcal I^+\theta(\a)$. This concludes the proof of \eqref{Point_residual_N_is_initial_final}.

We prove \eqref{Point_sudiff_is_partition_in}, being the proof of \eqref{Point_sudiff_is_partition_fin} analogous. We already know by \eqref{item2TpartE+} that 
\[
\partial^+\theta(\a,w)\subset w + C^{\ell,+}_{\a,\b} \quad \text{for all } w \in  Z^{\ell,+}_{\a,\b}.
\]
Let now $w\in Z^{\ell,+}_{\a,\b}$ and $w''\in\partial^-\theta(\a,w)\cap w-\interr C^{\ell,+}_{\a,\b}$ as in \eqref{E_forward_backward_regular_transport_set} (iii). Then, as in Claim \ref{Cl_oww'neigh}, for all $w'\in\partial^+\theta(\a,w)$ one has
\[
\interr\O(\a)(w'',w')\subset Z^{\ell,+}_{\a,\b}.
\]
Notice that now we do not need to specify that $w'-w\in\interr C^{\ell,+}_{\a,\b}$ because we already know $Z^{\ell,+}_{\a,\b}$ to be open in $\mathrm{aff}(w+C^{\ell,+}_{\a,\b})$. Then there are two possibilities: either
\begin{equation}
\label{E_firt_T_ele1}
\mathcal D^+\theta(\a,w') =  C^{\ell,+}_{\a,\b} \cap \mathbb S^{k-1},
\end{equation}
or by transitivity
\begin{equation}
\label{E_firt_T_ele2}
\mathcal D^+\theta(\a,w') \subsetneq C^{\ell,+}_{\a,\b} \cap \mathbb S^{k-1}.
\end{equation}
In case \eqref{E_firt_T_ele1} holds, it follows that $w'\in Z^{\ell,+}_{\a,\b}$. Otherwise, one has
\begin{equation}
\label{E_firt_T_ele3}
w' \in \mathcal T^-\theta(\a) \setminus \mathcal R^- \theta(\a),
\end{equation}
and then as seen before $w'\in\mathcal E^-\theta(\a)$. \\
To prove \eqref{E_firt_T_ele3}, it is sufficient to take a maximal backward extremal cone $F$ for $w'$ as in \eqref{E_ballFmax-} containing $C^{\ell,+}_{\a,\b}$: by \eqref{E_firt_T_ele2},
\[
\partial^+\theta(\a,w') \cap \big( w'+\interr F \big) = \emptyset,
\]
and then $w'$ cannot be a backward regular point. 
\end{proof}

\subsection{\texorpdfstring{Optimal transportation on $\mathtt c_{\C}$-Lipschitz foliations}{Optimal transportation on c-Lipschitz foliations}}
\label{Ss_optim_folia}

Let $\{ \theta^{-1}(\t) \}_{\t\in\T}$ be a $\mathtt c_{\C}$-Lipschitz foliation on $\A\times\R^k\subset\R^{d-k}\times\R^k$ and consider two probability measures in $\mathcal P(\A\times\R^k)$ which can be disintegrated as
\[
\tilde \mu = \int \tilde\mu_\a d\tilde m(\a), \quad \tilde\nu = \int \tilde \nu_\a d\tilde m(\a), \qquad \tilde m = (\mathtt p_{\R^{d-k}})_\# \tilde \mu = (\mathtt p_{\R^{d-k}})_\# \tilde \nu,
\]
and satisfying
\begin{subequations}
\label{E_541-2}
\begin{equation}
\label{E_541}
\tilde \mu \big( \theta^{-1}(\T) \big) = \tilde \nu \big( \theta^{-1}(\T) \big) = 1,
\end{equation}
\begin{equation}
\label{E_542}
\emptyset \neq \Pi^f_{\mathtt c_{\C},\theta}(\tilde \mu,\tilde \nu) := \bigg\{ \pi \in \Pi^f_{\mathtt c_{\C}}(\tilde \mu,\tilde \nu) : \pi \bigg( \bigcup_{\t \in \T} \{\theta=\t\} \times \{\theta=\t\} \bigg) = 1 \bigg\}.
\end{equation}
\end{subequations}
Notice that, by Definition \ref{D_partial_theta} of superdifferential of $\theta$,
\begin{equation}
\label{E_pifct}
\Pi^f_{\mathtt c_{\C},\theta}(\tilde \mu,\tilde \nu) = \Pi(\tilde \mu,\tilde \nu) \cap \bigl\{ \pi : \pi(\partial^+ \theta) = 1 \bigr\}.
\end{equation}
By Theorem \ref{T_partition_E},
\[
\bigcup_{\t \in \T} \{\theta=\t\} = \mathcal I^+ \theta \cup \mathcal E^- \theta \cup \mathcal R \theta \cup \mathcal F \theta,
\]
hence $\pi\in\Pi^f_{\mathtt c_{\C},\theta}(\tilde\mu,\tilde\nu)$ if and only if it has $\tilde\mu$ and $\tilde\nu$ as first and second marginal respectively and it is concentrated on the set
\begin{equation}
\label{E_description}
\partial^+ \theta \cap \Bigl[ \big( \mathcal I^+ \theta \cup \mathcal E^- \theta \cup \mathcal R \theta \cup \mathcal F \theta \big) \times \big( \mathcal I^+ \theta \cup \mathcal E^- \theta \cup \mathcal R \theta \cup \mathcal F \theta \big) \Bigr].
\end{equation}

First notice that
\begin{subequations}
\label{E_descr}
\begin{equation}
\label{E_descr1}
\partial^+ \theta \cap \big( \mathcal F \theta \times (\A \times \R^k) \big) = \partial^+ \theta \cap \big( (\A \times \R^k) \times \mathcal F \theta \big) = \Graph\,\Id\, \llcorner_{\mathcal F \theta},
\end{equation}
\begin{equation}
\label{E_descr2}
\partial^+ \theta \cap \big( (\A \times \R^k) \times (\mathcal T^+ \theta \setminus \mathcal T^- \theta) \big) = \Graph\,\Id\, \llcorner_{(\mathcal T^+ \theta \setminus \mathcal T^-\theta)},
\end{equation}
\begin{equation}
\label{E_descr3}
\partial^+ \theta \cap \big( (\mathcal T^- \theta \setminus \mathcal T^+ \theta) \times (\A \times \R^k) \big) = \Graph\,\Id\, \llcorner_{(\mathcal T^- \theta \setminus \mathcal T^+ \theta)},
\end{equation}
\end{subequations}
and thus the restrictions of all the transport plans in \eqref{E_pifct} to the sets \eqref{E_descr} are concentrated on the diagonal $\{(\a,w,\a,w):\,(\a,w)\in\A\times\R^k\}$. Hence, we will assume w.l.o.g. that
\begin{equation}
\label{E_543}
\tilde\mu(\mathcal F\theta)=\tilde\nu(\mathcal F\theta)=0, \quad\tilde\mu(\mathcal T^-\theta \setminus \mathcal T^+\theta)=0,\quad\tilde\nu(\mathcal T^+\theta \setminus \mathcal T^-\theta)=0.
\end{equation}
This is true e.g. if $\tilde\mu\perp\tilde\nu$.

Now we give, in the same spirit of Theorem \ref{T_partition_E}, a more accurate description of the set \eqref{E_description} (we will neglect the subsets \eqref{E_descr} by the above observation), independently of the measures $\tilde\mu$, $\tilde\nu$. As a consequence (see Proposition \ref{P_disint_fol} below) we will get that if $\tilde\mu(\mathcal I^+\theta)=0$, then the class of transport plans \eqref{E_pifct} coincides with the class of transport plans of finite cost on the directed locally affine partition of $\mathcal R^{+}\theta$ with quotient map $\mathtt v^+$ defined in Theorem \ref{T_partition_E+-} and, if $\tilde\nu(\mathcal E^-\theta)=0$, they are moreover of finite cost on the directed partition on $\mathcal R\theta$ induced by the map $\mathtt v$.

In Table \ref{Tab_poss_inter} we put on the horizontal and vertical line the sets of a partition of $\mathcal I^+\theta\cup\mathcal R\theta\cup\mathcal E^-\theta$. If $A$ is a set belonging to the horizontal line and $B$ belongs to the vertical, in the square $(A,B)$ we write $Y$ in case possibly $\partial^+\theta \cap (A\times B)\neq\emptyset$, and $N$ in case always $\partial^+\theta \cap (A\times B)=\emptyset$. Recall that $\mathcal R\theta=\mathcal R^+\theta\cap\mathcal R^-\theta$, $\mathcal R^+\theta\cap\mathcal E^-\theta=\mathcal R^+\theta\setminus \mathcal R^-\theta$, $\mathcal R^-\theta\cap\mathcal I^+\theta=\mathcal R^-\theta\setminus \mathcal R^+\theta$, $\mathcal I^-\theta\setminus\mathcal T^+\theta=\mathcal T^-\theta\setminus\mathcal T^+\theta$, 
\[
\begin{split}
\mathcal I^+ \theta =&~ (\mathcal I^+ \theta \setminus \mathcal T^- \theta) \cup (\mathcal I^+ \theta \cap \mathcal E^- \theta) \cup (\mathcal R^- \theta\cap \mathcal I^+ \theta ), \crcr
\mathcal E^- \theta =&~ (\mathcal E^- \theta \setminus \mathcal T^+ \theta) \cup (\mathcal I^+ \theta \cap \mathcal E^- \theta) \cup (\mathcal R^+ \theta\cap \mathcal E^- \theta ),
\end{split}
\]
are disjoint unions.

%Sara{ho cambiato da N a Y le caselle 2orizz-6vert e 6orizz-3 vert}.
\begin{table}[h]
\caption{The possible intersection of $\partial^+ \theta$ with a partition of $\mathcal T \theta$}
\begin{tabular}{|c|c|c|c|c|c|c|}
\hline
 & $\mathcal R\theta$ & $\mathcal R^+ \theta\cap \mathcal E^-\theta$ & $\mathcal R^- \theta\cap \mathcal I^+\theta$ & $\mathcal I^+ \theta\setminus\mathcal T^-\theta$ & $\mathcal E^- \theta\setminus \mathcal T^+\theta$ & $\mathcal I^+\theta \cap \mathcal E^-\theta$ \\
\hline
$\mathcal R\theta$ & Y & Y & N & N & Y & Y \\
\hline
$\mathcal R^+\theta \cap \mathcal E^-\theta$ & N & Y & N & N & Y & Y\\
\hline
$\mathcal R^-\theta \cap \mathcal I^+\theta$ & Y & Y & Y & N & Y & Y \\
\hline
$\mathcal I^+ \theta\setminus \mathcal T^-\theta$ & Y & Y & Y & Y ($\Graph\, \Id$) & Y & Y \\
\hline
$\mathcal E^-\theta \setminus \mathcal T^+\theta$ & N & N & N & N & Y ($\Graph \,\Id)$ & N \\
\hline
$\mathcal I^+ \theta\cap \mathcal E^-\theta$ & Y & Y & Y & N & Y & Y \\
\hline
\end{tabular}
% \medskip
\label{Tab_poss_inter}
\end{table}
% }
% \vskip 0.3 cm

\begin{proof}[Proof of Table \ref{Tab_poss_inter}]
First of all, \eqref{E_descr2} and \eqref{E_descr3} yield immediately the row of $\mathcal E^- \theta \setminus \mathcal T^+ \theta$ and the column of $\mathcal I^+ \theta \setminus \mathcal T^- \theta$, that we write just for symmetry. Moreover, by Remark \ref{R_t+t-i+e-}, also the row of $\mathcal I^+ \theta \setminus \mathcal T^- \theta$ and the column of $\mathcal E^- \theta \setminus \mathcal T^+ \theta$ easily follow. 

 In order to prove the other squares we will use the following facts, which have already been proven and used in Sections \ref{Ss_partition_transport_set} and \ref{Ss_analysis_residual_set} and are a consequence of \eqref{E_transitivity_subdifferential_a} and definitions \eqref{E_forward_regular_transport_set} and \eqref{E_backward_regular_transport_set}:
 \begin{align}
  &w\in\mathcal R\theta(\a)\quad\Rightarrow \quad\mathcal D^+\theta(\a,w)=\mathcal D^-\theta(\a,w);\label{E_table1}\\
  &w\in\mathcal R^-\theta(\a)\cap \mathcal I^+\theta(\a)\quad\Rightarrow \quad\mathcal D^+\theta(\a,w)\supsetneq\mathcal D^-\theta(\a,w);\label{E_table2}\\
  &w\in\mathcal R^+\theta(\a)\cap \mathcal E^-\theta(\a)\quad\Rightarrow \quad\mathcal D^-\theta(\a,w)\supsetneq\mathcal D^+\theta(\a,w).\label{E_table3}
 \end{align}
 
 Let us first prove the relations given by the squares containing the letter $N$.
 Let $w\in\mathcal R^+\theta(\a)$ and $w'\in\partial^+\theta(\a,w)$. Then,
 \[
  \mathcal D^+\theta(\a,w')\overset{\eqref{E_transitivity_subdifferential_a}}{\subset}\mathcal D^+\theta(\a,w)\overset{\eqref{E_table1},\eqref{E_table3}}{\subset}\mathcal D^-\theta(\a,w)\overset{\eqref{E_transitivity_subdifferential_a}}{\subset}\mathcal D^-\theta(\a,w').
 \]
By \eqref{E_table1} and \eqref{E_table2} we then conclude that $\partial^+\theta\cap (\mathcal R\theta\times \mathcal R^-\theta \cap \mathcal I^+\theta)=\emptyset$ and by \eqref{E_table1}-\eqref{E_table3} that $\partial^+\theta\cap(\mathcal R^+\theta \cap \mathcal E^-\theta\times \mathcal R\theta)=\emptyset$, $\partial^+\theta\cap(\mathcal R^+\theta \cap \mathcal E^-\theta\times\mathcal R^-\theta \cap \mathcal I^+\theta)=\emptyset$.

Let us now deal with the relations given by the squares containing the letter $Y$. We refer to Figure \ref{Fi_structsubdiff}.

\begin{figure}
\centering{\resizebox{12cm}{10cm}{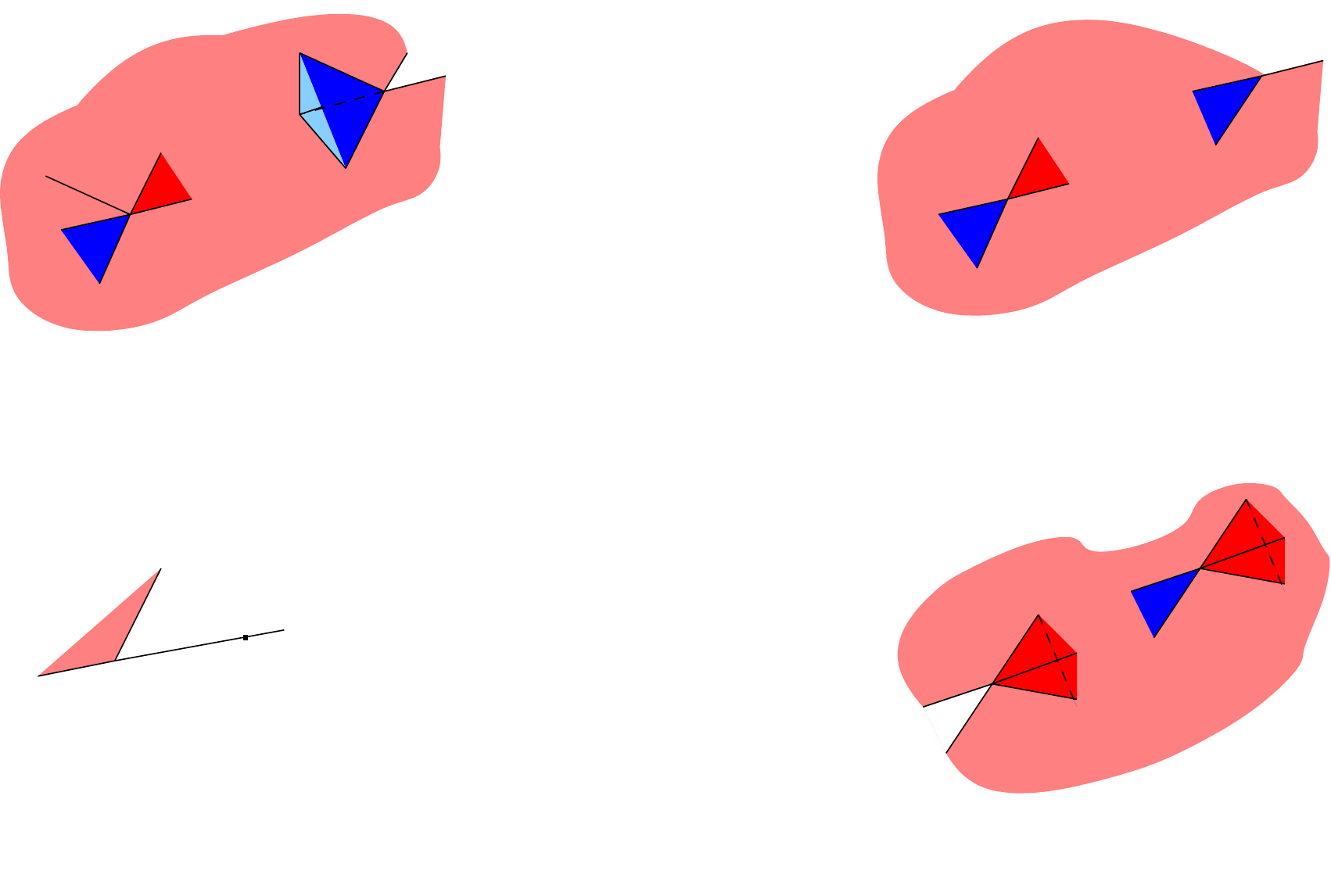}}
\caption{The relations given by the letter $Y$.}
\label{Fi_structsubdiff}
\end{figure}

First we claim that 
\begin{equation}
\label{E_Y1}
 \partial^+\theta\cap (\mathcal R^+\theta\cap \mathcal E^-\theta\times \mathcal I^+\theta\cap\mathcal E^-\theta)\neq\emptyset.
\end{equation}
Indeed, let e.g. $w\in \mathcal R^+\theta(\a)\cap \mathcal E^-\theta(\a)$ such that
\[
\begin{split}
& \dim(\R^+\mathcal D^+\theta(\a,w))=2, \\
& \mathcal D^+\theta(\a,w)\subsetneq \mathcal D^-\theta(\a,w)\subsetneq \conv_{\S^{k-1}} \mathcal  D^-\theta(\a,w), \\
& \dim (\R^+\conv_{\S^{k-1}} \mathcal D^-\theta(\a,w))=3,
\end{split}
\]
and $w'\in\partial^+\theta(\a,w)$ s.t.
\[
\mathcal D^-\theta(\a,w')= \conv_{\S^{k-1}} \mathcal D^-\theta(\a,w) \quad \text{but} \quad \mathcal D^+\theta(\a,w')\subsetneq \conv_{\S^{k-1}}\mathcal D^+\theta(\a,w')=\mathcal D^+\theta(\a,w).
\]

Next we claim that 
\begin{equation}
\label{E_Y2}
 \partial^+\theta\cap (\mathcal R\theta\times \mathcal R^+\theta\cap \mathcal E^-\theta)\neq\emptyset.
\end{equation}
Indeed, it is sufficient to take e.g. $w\in\mathcal R\theta(\a)$ with $\dim (\R^+\mathcal D^+\theta(\a,w))=2$ and $w'\in \partial^+\theta(\a,w)$ s.t.
\[
\mathcal D^-\theta(\a,w')=\mathcal D^+\theta(\a,w) \quad \textrm{but} \quad \mathcal D^+\theta(\a,w')=\conv_{\S^{k-1}}\mathcal D^+\theta(\a,w')\subsetneq \mathcal D^-\theta(\a,w').
\]

Now let us prove that 
\begin{equation}
\label{E_Y3}
 \partial^+\theta\cap (\mathcal R^-\theta\cap \mathcal I^+\theta\times \mathcal R\theta)\neq\emptyset.
\end{equation}
Take for example $w\in\mathcal R^-\theta(\a)\cap \mathcal I^+\theta(\a)$ and $\bar w\in\partial^+\theta(\a,w)\cap w+\interr F$ as in \eqref{E_ballFmax} with $F$ maximal face in $\R^+\mathcal D^+\theta(\a,w)$ containing $\R^+\mathcal D^-\theta(\a,w)$.

Finally we claim that 
\begin{equation}
 \label{E_Y4}
  \partial^+\theta\cap (\mathcal I^+\theta\cap\mathcal E^-\theta\times\mathcal R^-\theta\cap \mathcal I^+\theta)\neq\emptyset.
\end{equation}
Hence, by the transitivity property \eqref{E_transitivity_subdifferential_a}, \eqref{E_Y1}-\eqref{E_Y4} immediately give all the remaining $Y$ squares in the table.

In order to show \eqref{E_Y4}, let $w\in\mathcal I^+\theta(\a)\cap\mathcal E^-\theta(\a)$ with
\[
\begin{split}
& \mathcal D^+\theta(\a,w)=\conv_{\S^{k-1}}\mathcal D^+\theta(\a,w), \\
& \dim\R^+\mathcal D^+\theta(\a,w)=3 \quad \text{and} \\
& \mathcal D^-\theta(\a,w)\subsetneq\conv_{\S^{k-1}}\mathcal D^-\theta(\a,w)\subsetneq\mathcal D^+\theta(\a,w). 
\end{split}
\]
Then, let $w'\in \partial^+\theta(\a,w)\,\cap\, w+\interr (\R^+\conv_{\S^{k-1}}\mathcal D^-\theta(\a,w))$ s.t.
\[
\mathcal D^+\theta(\a,w')=\mathcal D^+\theta(\a,w) \quad \text{and} \quad \mathcal D^-\theta(\a,w')=\conv_{\S^{k-1}}\mathcal D^-\theta(\a,w).
\]
\end{proof}

Neglecting \eqref{E_descr}, we conclude that
\begin{align}
\partial^+ \theta \cap \Big[ \big( \mathcal I^+ \theta \cup \mathcal R \theta \cup \mathcal E^- \theta \big) \times \big( \mathcal I^+ \theta \cup \mathcal R \theta \cup \mathcal E^- \theta \big) \Big] =&~ \Bigl[ \partial^+ \theta \cap \big( \mathcal R \theta \times (\mathcal R \theta \cup \mathcal E^- \theta) \big) \Bigr] \label{E_subdiff_dec}\\
&~ \cup \Bigl[ \partial^+ \theta \cap \big( (\mathcal R^+ \theta \cap \mathcal E^- \theta) \times \mathcal E^- \theta \big) \Bigr] \notag\\
&~ \cup \Bigl[ \partial^+ \theta \cap (\mathcal I^+ \theta \times \mathcal T^- \theta) \Bigr].\notag
\end{align}
Using \eqref{E_subdiff_dec}, Theorem \ref{T_partition_E+-} and Corollary \ref{C_v} we have the following.

\begin{proposition}
\label{P_disint_fol}
Let $\{\theta^{-1}(\t)\}_{\t\in\T} \subset \mathtt P(\A \times\R^k)$ be a $\mathtt c_{\C}$-Lipschitz foliation and $\tilde \mu,\tilde \nu \in \mathcal P(\A \times \R^k)$ such that \eqref{E_541-2} and \eqref{E_543} hold.

If $\tilde\mu(\mathcal I^+\theta)=0$, then
\[
\tilde \pi \in \Pi^f_{\mathtt c_{\C},\theta}(\tilde \mu,\tilde \nu) \quad \Longleftrightarrow \quad \tilde \pi \in \Pi^f_{\mathtt c_{\hat{\bD}^+}}(\tilde \mu,\tilde \nu),
\]
where $\hat{\bD}^+$ is the locally affine partition induced by $\mathtt v^+$ on $\mathcal R^+\theta$ and $\mathtt c_{\hat{\bD}^+}$ the related cost, as defined in \eqref{E_c_bD}.

Moreover, if $\tilde\mu(\mathcal I^+\theta)=0$ and $\tilde\mu(\mathcal R^+\theta\cap\mathcal E^-\theta)=0$ then
\[
\tilde \pi \in \Pi^f_{\mathtt c_{\C},\theta}(\tilde \mu,\tilde \nu) \quad \Longleftrightarrow \quad \tilde \pi \in \Pi^f_{\mathtt c_{\hat{\bD}}}(\tilde\mu,\tilde\nu),
\]
where $\hat{\bD}$ is the locally affine partition induced by $\mathtt v=\mathtt v^+_{|_{\mathcal R\theta}}$ on $\mathcal R\theta$ and $\mathtt c_{\hat{\bD}}$ the related cost. 
%Moreover, in this case 
%\[
%\tilde \pi \in \Pi^f_{\mathtt c_{\hat{\bD}}}(\tilde\mu,\tilde\nu)\quad \Longleftrightarrow \quad \tilde \pi =\int \tilde \pi \ka\,dm(k,\a),\quad\tilde\pi\ka\in
%\]

\end{proposition}

We end this section with the following special case.

\begin{proposition}
\label{P_hat_bf_D_graph}
Let $\t\in\T$, $\a=\p_\A(\{\theta=\t\})$. Then \eqref{E_more_than_complet} holds for the differential directed locally affine partition of the regular set of $\{\theta=\t\}$
\[
 \{Z^\ell_{\a,\b}, C^\ell_{\a,\b}\}_{\overset{\ell=1,\dots,k-1}{\b\in\B_{\ell}(\a)}}.
\]
\end{proposition}

\begin{proof}
For $(\ell,\b) \not= (\ell',\b')$, let
\[
\bar z \in (Z^\ell_{\a,\b} + C^\ell_{\a,\b}) \cap (Z^{\ell'}_{\a,\b'} + C^{\ell'}_{\a,\b'})\cap \p_{\R^k}\hat{\mathbf D}(\t),
\]
where $\{Z^\ell_{\a,\b},C^\ell_{\a,\b}\}_{\ell,\b}$ is the directed locally affine partition of $\hat{\mathbf D}(\t)$. Hence, $\bar z\in w+C^\ell_{\a,\b}\cap w'+C^{\ell'}_{\a,\b'}$ for some $w\in Z^\ell_{\a,\b}$, $w'\in Z^{\ell'}_{\a,\b'}$ and, since by assumption $\theta$ is constant (namely, equal to $\t$) on $\p_{\R^k}\hat{\mathbf D}(\t)$, then by definition of superdifferential
\[
 \bar z\in \partial^+\theta(\a,w)\cap\partial^+\theta(\a,w').
\]
By the transitivity property and the fact that $\bar z$ is regular
\[
 C^\ell_{\a,\b}\cup C^{\ell'}_{\a,\b'}\subset\R^+\mathcal D^-\theta(\a,\bar z)=\R^+\mathcal D^+\theta(\a,\bar z)\subset C^\ell_{\a,\b}\cap C^{\ell'}_{\a,\b'},
\]
which implies that $C^\ell_{\a,\b}=C^{\ell'}_{\a,\b'}$, thus $\ell=\ell'$ and $\aff Z^\ell_{\a,\b}=\aff Z^{\ell'}_{\a,\b'}$.
Hence, by definition of $\mathtt v$, $Z^\ell_{\a,\b}=Z^{\ell'}_{\a,\b'}$ contradicting our initial assumption.
\end{proof}

\begin{remark}
\label{R_4.27}
By Proposition \ref{P_hat_bf_D_graph}, the differential partition of a single complete $\mathtt c_{\tilde C}$-Lipschitz graph satisfies \eqref{E_more_than_complet}.
\end{remark}

From Proposition \ref{P_hat_bf_D_graph} and Proposition \ref{P_dispiani_2}, one obtains immediately the following
\begin{corollary}
\label{C_transp_graph}
 If $\tilde\mu$, $\tilde\nu$ are as in Proposition \ref{P_disint_fol} and $\tilde\mu(\mathcal I^+\theta)=0$, $\tilde\nu(\mathcal E^-\theta)=0$, then 
 \[
\tilde\pi\in\Pi^f_{\mathtt c_{\C}, \theta}(\tilde\mu,\tilde\nu)\quad\Leftrightarrow\quad\tilde\pi=\int\tilde\pi^\ell_{\c} \, d\tilde m(\c),\quad\tilde\pi^\ell_\c\in\Pi^f_{\mathtt c_{C^\ell_\c}}(\tilde\mu^\ell_\c,\tilde\nu^\ell_\c),
 \]
 where $\{\tilde\pi^\ell_\c\}$, $\{\tilde\mu^\ell_\c\}$ and $\{\tilde\nu^\ell_\c\}$ are respectively the disintegrations of $\tilde\pi$, $\tilde\mu$, $\tilde\nu$ w.r.t. the partition induced by $\mathtt v$.
\end{corollary}

\section{Dimensional reduction on directed partitions via cone approximation property}
\label{S_disintechnique}

In this section we recall, in an abstract and more general setting, the main steps of the disintegration technique first introduced in \cite{BianchGlo} for partitions into segments and then extended to locally affine partitions of any dimension in \cite{CarDan}. This technique allows to prove the absolute continuity of the conditional probabilities of the Lebesgue measure and to deduce that the \emph{initial} and \emph{final points} of a directed locally affine partition are Lebesgue negligible, provided the direction map satisfies a suitable regularity assumption that we call (\emph{initial/final}) \emph{forward/backward cone approximation property}. For more details on the proofs of the results contained in this section, we refer to \cite{CarDan}, Section 4.

\subsection{Model sets of directed segments}
\label{Ss_model_dir_segm}

We first deal with \emph{model sets of directed segments}, namely $1$-dimensional sheaf sets whose projection on their reference line is a given segment. At the end of the paragraph, the forward/backward cone approximation property for these model sets will be introduced as a sufficient condition in order to have absolutely continuous disintegrations.

\begin{definition}
\label{D_modeldire_segm}
A \emph{model set of directed segments} or \emph{1-dimensional model set} is a 1-dimensional directed sheaf set $\{Z^1_\a,C^1_\a\}_{\a\in\A^1}$ with $\sigma$-continuous direction vector field
\begin{equation}
\label{E_mathit_v_def}
\mathtt d : \bigcup_{\a \in \A^1} Z^1_\a \to \mathbb \S^{d-1}, \qquad \mathtt d : Z^1_\a \ni z \mapsto C^1_\a \cap \S^{d-1},
\end{equation}
and reference line $\langle\e\rangle$ for which there exist $h^-, h^+ \in \R$, $h^- < h^+$, such that
\[
\mathtt p_{\langle \e \rangle} (Z^1_\a) = (h^-,h^+) \e \qquad \forall\, \a \in \A^1.
\]
\end{definition}

We will also call \emph{model set} the set $\mathbf Z^1= \underset{\a \in \A^1}{\cup}\, Z^1_\a$, and we say that the triple $(\e, h^-,h^+)$ is a \emph{reference configuration}. Moreover we assume that
\[
\mathcal L^{d}(\mathbf Z^1) < +\infty.
\]
For shortness we will sometimes use the notation $\mathbf Z^1(\mathtt d,\A^1,\e,h^-,h^+)$.

We also set $\overline{\mathbf Z}^1=\overline{\mathbf Z^1}$ as in \eqref{E_mathbf_Z_base_partition}
\begin{equation}
\label{E_clos_Z_1a}
\overline{\mathbf Z}^1 := \bigcup_{\a\in\A^1} \clos\, Z^1_\a.
\end{equation}
Notice that
\[
\overline{\mathbf Z}^1 \cap \mathtt p^{-1}_{\langle\e\rangle} \big( (h^-,h^+) \e \big) = \mathbf Z^1.
\]

Given a $1$-dimensional model set $\{Z^1_\a,C^1_\a\}_{\a\in\A^1}$ with reference configuration $(\e,h^-,h^+)$, define the perpendicular \emph{sections} 
\begin{equation}
\label{E_P_t_section}
P_t := \overline{\mathbf Z}^1 \cap \mathtt p^{-1}_{\langle \e \rangle}(t \e), \qquad t \in [h^-,h^+].
\end{equation}
Clearly from Definition \ref{D_initial_final} one has
\begin{equation*}
P_{h^-} = \mathcal I(\mathbf Z^1),\qquad P_{h^+} = \mathcal E(\mathbf Z^1),
\end{equation*}
where $\mathcal I(\mathbf Z^1)$, $\mathcal E(\mathbf Z^1)$ are the initial/final points of $\{Z^1_\a,C^1_\a\}_{\a\in\A^1}$.

For all $t\in[h^-,h^+]$, denote also
\begin{equation}
\label{E_d_t_def}
\mathtt d^t := \mathtt d \llcorner_{P_t},
\end{equation}
where $\mathtt d^{h^-} : P_{h^-} \to \S^{d-1}$, $\mathtt d^{h^+} : P_{h^+} \to \S^{d-1}$ are the multivalued extensions of $\mathtt d$ defined by
\begin{equation}
\label{E_tt_d_ext}
\mathtt d^{h^-}(z) = \bigl\{ C^1_\a \cap \S^{d-1} : z \in \mathcal I(Z^1_\a) \bigr\}, \qquad \mathtt d^{h^+}(z) = \bigl\{ C^1_\a \cap \S^{d-1} : z \in \mathcal E(Z^1_\a) \bigr\}.
\end{equation}

\begin{lemma}
\label{L_reg_tt_d_h_pm}
The sets $\Graph\,\mathtt d^{h^-}$, $\Graph\,\mathtt d^{h^+}$ are $\sigma$-compact, and hence there exist Borel sections
\[
\mathcal I(\mathbf Z^1) \ni z \mapsto \tilde{\mathtt d}^{h^-}_+(z) \in \mathtt d^{h^-}(z), \quad \mathcal E(\mathbf Z^1) \ni z \mapsto \tilde{\mathtt d}^{h^+}_-(z) \in \mathtt d^{h^+}(z).
\]
\end{lemma}

\begin{proof}
W.l.o.g., we prove the result only for $\Graph\,\mathtt d^{h^-}$.

Since $\mathtt d$ is $\sigma$-compact, let $\mathbf Z^1 = \underset{l}{\cup}\, \mathbf Z^1_l$ such that $\mathbf Z^1_l$ are compact and $\mathtt d \llcorner_{\mathbf Z^1_l}$ is continuous. Then it is fairly easy to see that the multivalued maps
\[
\mathtt p^{-1}_{\langle \e \rangle} (h^- \e) \cap \big( z + \R \mathtt d(z) \big) \mapsto \mathtt d(z), \quad z \in \mathbf Z^1_l,
\]
are compact, and thus the regularity of $\Graph\,\mathtt d^{h^-}$ follows.

The existence of sections $\tilde{\mathtt d}^{h^-}_+$, $\tilde{\mathtt d}^{h^+}_-$ with Borel regularity is standard for compact multifunctions (see for example Theorem 5.2.1, page 189 of \cite{Sri:courseborel}) and an easy argument yields the conclusion.
\end{proof}

% Universally measurable sections of $\mathtt d^{h^-}$, $\mathtt d^{h^+}$ will be denoted by $\tilde{\mathtt d}_+^{h^-}$, $\tilde{\mathtt d}_-^{h^-}$: they exist by the $\sigma$-continuous assumption on $\mathtt d$, which implies that . Moreover, w
We will define the vector fields
\begin{subequations}
\label{E_tilde_tt_d_pm}
\begin{equation}
\label{E_tilde_tt_d_+}
\tilde{\mathtt d}_+ (z) := \tilde{\mathtt d}^{h^-}_+(z') \qquad \text{if} \quad z \in (z' +  \R \tilde{\mathtt d}^{h^-}_+(z')) \cap \mathtt p^{-1}_{\langle \e \rangle}([h^-,h^+]), \ z' \in \dom\,\tilde{\mathtt d}^{h^-}_+,
\end{equation}
\begin{equation}
 \label{E_tilde_tt_d_-}
\tilde{\mathtt d}_- (z) := \tilde{\mathtt d}^{h^+}_-(z') \qquad \text{if} \quad z \in (z' +  \R \tilde{\mathtt d}^{h^+}_-(z')) \cap \mathtt p^{-1}_{\langle \e \rangle}([h^-,h^+]), \ z' \in \dom\,\tilde{\mathtt d}^{h^+}_-.
\end{equation}
% \mathtt d\chi_{\{\clos Z^1_\a:\, C^1_\a\cap\mathbb S^{d-1}=\tilde{\mathtt d}_+^{h^-}(\clos Z^1_\a\cap P_{h^-}\}},\\
% \tilde{\mathtt d}_-:=\mathtt d\chi_{\{\clos Z^1_\a:\, C^1_\a\cap\mathbb S^{d-1}=\tilde{\mathtt d}_-^{h^+}(\clos Z^1_\a\cap P_{h^+}\}}.
% \end{align*}
\end{subequations}
In particular,
\[
\tilde{\mathtt d}_+ \llcorner_{P_{h^-}} = \tilde{\mathtt d}_+^{h^-}, \qquad \tilde{\mathtt d}_- \llcorner_{P_{h^+}}=\tilde{\mathtt d}_-^{h^+},
\]
and 
\[
\tilde{\mathtt d}_\pm \llcorner_{\mathtt p^{-1}_{\langle\e\rangle} ((h^-,h^+) \e)} = \mathtt d \llcorner_{\dom\,\tilde{\mathtt d}_\pm \cap \mathtt p^{-1}_{\langle\e\rangle} ((h^-,h^+) \e) }.
\]
% 
%$\mathtt d\llcorner \big(P_{h^-}\cup\mathbf Z^1\big)$ ($\mathtt d\llcorner \big(P_{h^+}\cup\mathbf Z^1\big)$) will be denoted by $\tilde{\mathtt d}_+$ ($\tilde{\mathtt d}_-$): it exists by the $\sigma$-continuous assumption on $\mathtt d$, and in particular $\tilde{\mathtt d}_\pm \llcorner_{\mathbf Z^1\cap\mathrm{dom}\tilde{\mathtt d}_\pm} = \mathtt d$ (see \cite{Sriv}.

Define for $s \in (h^-,h^+)$, $t \in [h^-,h^+]$ the map
\begin{equation}
\label{E_sigma_s_t}
% \begin{array}{ccccc}
\sigma^{s,t} : P_s \to P_t, %\crcr
% &&&& \crcr
% &&
\qquad
z \mapsto \displaystyle{z + (t-s) \frac{\mathtt d(z)}{\mathtt d(z) \cdot \mathrm e}},
% \end{array}
\end{equation}
which sends each point of the section $P_s$ in the unique point of $P_t$ which belongs to the same segment of the model set. It is a bijection for $t \in (h^-,h^+)$, with inverse $\sigma^{t,s}$.

Given Borel measurable selections $\tilde{\mathtt d}_\pm$, as in \eqref{E_tilde_tt_d_+} and \eqref{E_tilde_tt_d_-}, one defines for $s \in [h^-,h^+)$ ($s \in (h^-,h^+]$), $t \in [h^-,h^+]$ the map
%Sara
\begin{equation}
\label{E_tilde_sigma_s_t}
% \begin{array}{ccccc}
\tilde \sigma^{s,t}_\pm : P_s\cap \dom\, \tilde{\mathtt d}_\pm \to P_t, %\crcr
% &&&& \crcr
% && 
\qquad
z \mapsto \displaystyle{z + (t-s) \frac{\tilde{\mathtt d}_\pm(z)}{\tilde{\mathtt d}_\pm(z) \cdot \mathrm e}}.
% \end{array}
\end{equation}
Notice that $\tilde \sigma^{s,t}_\pm $ coincides with $\sigma^{s,t}_\pm\llcorner_{\dom\,\tilde{\mathtt d}_\pm \cap \mathtt p^{-1}_{\langle\e\rangle} ((h^-,h^+) \e) }$ for $s\neq h^\mp$.

\subsubsection{Cone approximation property and absolute continuity}
\label{Sss_cone_approx}

We recall that our problem is the following: if
\[
\mathcal L^{d} \llcorner_{\mathbf Z^1} = \int \upsilon_\a \,d\eta(\a)
\]
is the disintegration of $\mathcal L^{d} \llcorner_{\mathbf Z^1}$ w.r.t. the partition $\{Z^1_\a\}_{\a\in\A^1}$, then we ask whether
\[
\upsilon_\a \ll \mathcal H^1 \llcorner_{Z^1_\a} \quad \text{ and/or } \quad \upsilon_\a \simeq \mathcal H^1 \llcorner_{Z^1_\a}.
\]
By the next lemma, the absolute continuity problem along the segments $\{Z^1_\a\}_\a$ can be reduced to an absolute continuity problem for the push-forward of $\mathcal H^{d-1}$ on the sections through the maps $\sigma^{s,t}$.

\begin{lemma}
\label{L_equivarea}
Let us fix a section $P_t$ of $\mathbf Z^1$, $t \in (h^-,h^+)$. Then, the following two statements are equivalent:
\begin{equation}
\label{E_equivarea1}
\sigma^{s,t}_\# \big( \mathcal H^{d-1} \llcorner_{P_s} \big) \ll \mathcal H^{d-1} \llcorner_{P_t},\quad\text{for $\mathcal L^1$-a.e. $s \in(h^-,h^+)$;}
\end{equation}
\begin{equation}
\label{E_equivarea4}
\eta \ll \mathcal H^{d-1} \llcorner_{P_t} \quad \text{and} \quad \upsilon_\a \ll \mathcal H^1 \llcorner_{Z^1_\a}, \quad\text{for $\eta$-a.e. $\a$.}
\end{equation}

Moreover, also the following two statements are equivalent:
\begin{equation}
\label{E_equivarea2}
\sigma^{t,s}_\# \big( \mathcal H^{d-1} \llcorner_{P_t} \big) \ll \mathcal H^{d-1} \llcorner_{P_s} \quad\text{and formula \eqref{E_equivarea1} hold for $\mathcal L^1$-a.e. $s\in(h^-,h^+)$;}
\end{equation}
\begin{equation}
\label{E_equivarea3}
\eta \simeq \mathcal H^{d-1} \llcorner_{P_t} \quad \text{and} \quad \upsilon_\a \simeq \mathcal H^1 \llcorner_{Z^1_\a}, \quad\text{for $\eta$-a.e. $\a$.}
\end{equation}

\end{lemma}
In particular, whenever \eqref{E_equivarea3} holds, $\mathbf Z^1$ is a regular partition according to Definition \ref{D_disint_regular}.

The proof of Lemma \ref{L_equivarea} is an application of Fubini-Tonelli theorem w.r.t. the projection on $\langle\e\rangle$ and the change of variables formula for the maps $\sigma^{s,t}$. We give a short proof for completeness.

\begin{proof}
By \eqref{E_equivarea1}
\[
\sigma^{s,t}_\# \big( \mathcal H^{d-1} \llcorner_{P_s} \big) = \mathtt f(s,t,\cdot) \mathcal H^{d-1} \llcorner_{P_t}
\]
%Sara
for some Borel non-negative function $\mathtt f$, for a.e. $s\in(h^-,h^+)$ and thus by Fubini-Tonelli theorem we can write for a compactly supported function $\phi : \R^d \to \R$
\[
\begin{split}
\int \phi \mathcal L^d \llcorner_{\mathbf Z^1} =&~ \int_{h^-}^{h^+} \bigg[ \int_{P_s} \phi(z) d\mathcal H^{d-1}(z) \bigg] ds = \int_{h^-}^{h^+} \bigg[ \int_{P_t} \phi(\sigma^{t,s}(z)) d\big( \sigma^{s,t}_\# \mathcal H^{d-1} \llcorner_{P_s} \big)(z) \bigg] ds \crcr
=&~ \int_{h^-}^{h^+} \bigg[ \int_{P_t} \phi(\sigma^{t,s}(z)) \mathtt f(s,t,z) d\mathcal H^{d-1}(z) \bigg] ds = \int_{P_t} \bigg[ \int_{h^-}^{h^+} \phi(\sigma^{t,s}(z)) \mathtt f(s,t,z) ds \bigg] d\mathcal H^{d-1}(z).
\end{split}
\]
By the uniqueness of the disintegration (see Theorem \ref{T_disint}), this shows that \eqref{E_equivarea4} is true. Repeating the argument starting from the end, one can prove the equivalence of \eqref{E_equivarea1} and \eqref{E_equivarea4}.

By using the additional assumption \eqref{E_equivarea2}, $\mathtt f(s,t,\cdot)$ is $\mathcal H^{d-1}$-a.e. strictly positive on $P_t$, for a.e. $s\in (h^-,h^+)$ and the equivalence of \eqref{E_equivarea2} with \eqref{E_equivarea3} follows immediately.
\end{proof}

Now we are ready to introduce the forward/backward cone approximation properties, which will imply the assumptions of Lemma \ref{L_equivarea}. For notational convenience, we will state the definition of cone vector field.

\begin{figure}
\centering{\resizebox{12cm}{10cm}{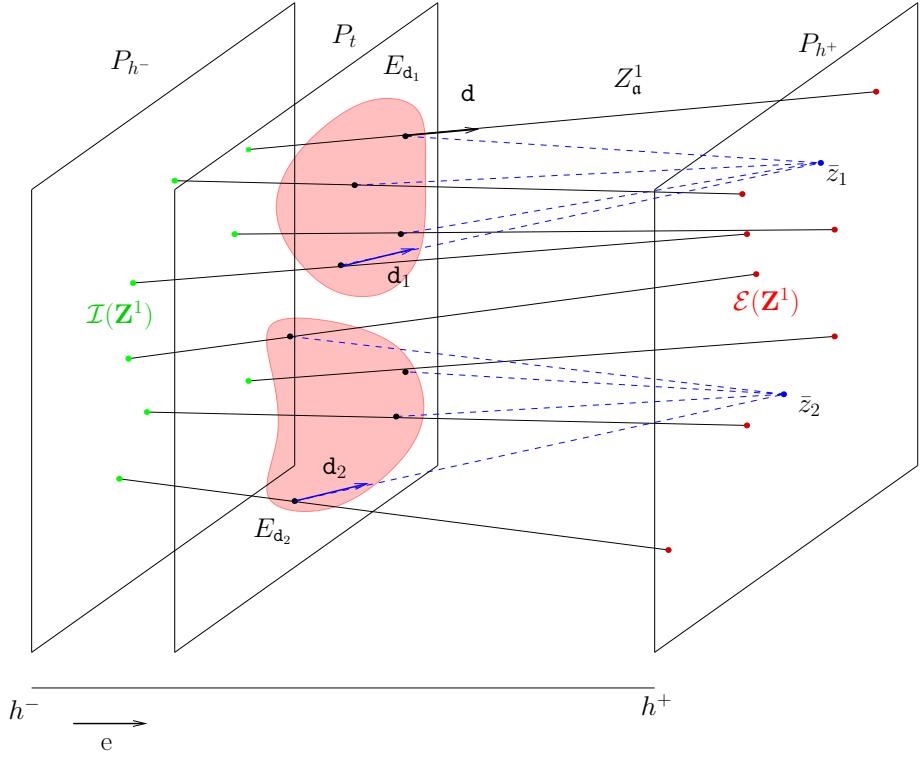}}
\caption{A model set of directed segments and a union of two cone vector fields.}
\label{Fi_modelset}
\end{figure}

\begin{definition}
\label{D_cone_vector}
The \emph{cone vector field} with base in $E_1 \subset \R^{d}$, and vertex $\bar z \in E_2 \subset \R^{d} \setminus E_1$ is defined as
\begin{equation*}
% \label{E_cone_vector}
\begin{array}{ccccc}
\mathtt d &:& E_1 \supset \mathrm{dom}\,\mathtt d &\to& \mathbb S^{d-1} \crcr
&& z &\mapsto& \mathtt d(z) := \frac{\bar z- z}{\|\bar z- z\|}
\end{array}
\end{equation*}
We say that $\mathtt d$ is a \emph{finite union of cone vector fields} with base in $E_1$ and vertices in $E_2$ if there exist finitely many cone vector fields $\{\mathtt d_i\}_{i=1}^I$ with bases in $E_1$ and vertices $\{\bar z_i\}_{i=1}^I$ in $E_2 \subset \R^{d} \setminus E_1$ such that the sets
\begin{equation*}
% \label{E_base_vector_d_i}
E_{\mathtt d_i} := \big\{ (1-t) \bar z_i + t z, t \in [0,1], z \in \mathrm{dom}\, \mathtt d_i \big\}, \quad i = 1,\dots,I,
\end{equation*}
satisfy $E_{\mathtt d_i} \cap E_{\mathtt d_j} = \emptyset$, for all $i \neq j$.
\end{definition}

% The fundamental approximation property of 

\begin{definition}
\label{D_forw_back}
We say that the model set of directed segments $\mathbf Z^1(\mathtt d,\A^1,\e,h^-,h^+)$ has the \emph{forward cone approximation property} if there exists $\epsilon > 0$ such that for all $t \in (h^-,h^+)$ there exists $\{\mathtt d^t_j\}_{j \in \N}$ finite union of cone vector fields with base in $\mathtt p^{-1}_{\langle\e\rangle}(t \e)$ and vertices in $\mathtt p^{-1}_{\langle\e\rangle}((h^+ + \epsilon)\e)$ such that
\[
\mathcal H^{d-1}(P_t\setminus\mathrm{dom}\, \mathtt d^t_j) = 0
\]
and $\mathtt d^t_j \to \mathtt d^t$ $\mathcal H^{d-1} \llcorner_{P_t}$-a.e.. 

Analogously, we say that the model set of directed segments $\mathbf Z^1(\mathtt d,\A^1,\mathrm e,h^-,h^+)$ has the \emph{backward cone approximation property} if there exists $\epsilon > 0$ such that for all $t \in (h^-,h^+)$ there exists $\{\mathtt d^t_j\}_{j \in \N}$ finite union of cone vector fields with base in $\mathtt p^{-1}_{\langle\e\rangle}(t \e)$ and vertices in $\mathtt p^{-1}_{\langle\e\rangle}((h^- - \epsilon) \e)$ such that
\[
\mathcal H^{d-1}(P_t\setminus\mathrm{dom}\, \mathtt d^t_j) = 0
\]
and $\mathtt d^t_j \to -\mathtt d^t$ $\mathcal H^{d-1} \llcorner_{P_t}$-a.e.. 
% 
%Finally we say that $\mathbf Z^1(\mathtt d,\A^1,\mathrm e,h^-,h^+)$ has the \emph{cone approximation property} if both the forward and backward properties hold.
\end{definition}

\begin{lemma}
\label{L_forward_esti_area}
If $\mathbf Z^1(\mathtt d,\A^1,\e,h^-,h^+)$ has the forward cone approximation property, then for $h^- < s \leq t \leq h^+$
\begin{equation}
\label{E_forward_esti_area1}
\sigma^{s,t}_\# \mathcal H^{d-1} \llcorner_{P_s} \leq \bigg( \frac{h^+ + \epsilon - s}{h^+ + \epsilon - t} \bigg)^{d-1} \mathcal H^{d-1} \llcorner_{P_t}.
\end{equation}

Analogously, if $\mathbf Z^1(\mathtt d,\A^1,\e,h^-,h^+)$ has the backward cone approximation property, then for $h^- \leq t < s <h^+$
\begin{equation}
\label{E_forward_esti_area2}
\sigma^{s,t}_\# \mathcal H^{d-1} \llcorner_{P_s} \leq \bigg( \frac{s-h^-+ \epsilon}{t-h^-+ \epsilon} \bigg)^{d-1} \mathcal H^{d-1} \llcorner_{P_t}.
\end{equation}
\end{lemma}

\begin{proof}
We prove only the first estimate, since the proof of the second is completely similar. 

It is fairly easy to see that the estimates hold for the map $\sigma^{s,t}$ associated to a cone vector field with bases in $P_s$ and vertices in $P_{h^++\epsilon}$ as in Definition \ref{D_cone_vector} by similitude criteria for triangles or equivalently by the polar change of coordinates in $\R^d$. Hence, the same estimate holds also for the finite unions of cone vector fields approximating $\mathtt d^s$ as in Definition \ref{D_forw_back}.

Restricting by Egorov's Theorem to continuous and uniformly convergent sequences $\{\mathtt d^s_j\}_{j\in\N}$ on compact subsets of $P_s$, by the u.s.c. of $\mathcal H^{d-1}$ on the hyperplanes perpendicular to $(h^-,h^+)\mathtt e$ w.r.t. the Hausdorff convergence of compact sets, the inequality immediately passes to the limit.
\end{proof}
%Sara{ho tolto il remark}
%\begin{remark}
%\label{E_extens_end_points}
%From the proof it is clear that if both the forward and backward cone approximation properties hold, then \eqref{E_equivarea3} holds also for $t \in \{h^-,h^+\}$. A %standard argument in fact shows that the sets
%\[
%\big\{ z \in \mathcal E(\mathbf Z^1) : \mathtt d^+ \ \text{not single valued} \big\}, \quad \big\{ z \in \mathcal I(\mathbf Z^1) : \mathtt d^- \ \text{not single valued} \big\},
%\]
%are $\mathcal H^{d-1}$-negligible, so that the equivalence \eqref{E_equivarea1}, \eqref{E_equivarea4} of Lemma \ref{L_equivarea} can be extended also to the case $t \in \{h^+,h^-\}$.
%\end{remark}

It is straightforward to observe that \eqref{E_forward_esti_area1} implies \eqref{E_equivarea1} and \eqref{E_forward_esti_area2} implies the first part of \eqref{E_equivarea2}.
Hence we have the following

\begin{corollary}
\label{C_disinte_ac}
If $\mathbf Z^1(\mathtt d,\A^1,\e,h^-,h^+)$ has either the forward cone approximation property or the backward cone approximation property, then
\[
\eta \ll \mathcal H^{d-1} \llcorner_{\A^1} \quad \text{and} \quad \upsilon_\a \ll \mathcal H^1 \llcorner_{Z^1_\a} \ \text{for $\eta$-a.e. $\a \in \A^1$.}
\]

If both the forward cone approximation and the backward cone approximation properties hold, then $\mathbf Z^1$ is a regular partition, i.e.
\[
\eta \simeq \mathcal H^{d-1} \llcorner_{\A^1} \quad \text{and} \quad \upsilon_\a \simeq \mathcal H^1 \llcorner_{Z^1_\a} \ \text{for $\eta$-a.e. $\a \in \A^1$.}
\]
\end{corollary}

% \vskip 0.2 cm

We can extend the forward/backward cone approximation properties to Borel sections of initial/final points. This will be useful later (see Theorem \ref{T_FC_no_initial}), when we will give conditions ensuring that $\LL(\mathcal I(\mathbf Z))=0$/$\LL(\mathcal E(\mathbf Z))=0$ for a directed locally affine partition $\mathbf Z$. %(when $\mathbf Z=\mathbf Z^1(\mathtt d,\A^1,\e,h^-,h^+)$, this is already trivial).

\begin{definition}
\label{D_initial_forward}
We say that $\mathbf Z^1(\mathtt d,\A^1,\e,h^-,h^+)$ satisfies the \emph{initial forward cone approximation property} if there exists a Borel section $\tilde{\mathtt d}_+$ which satisfies the assumptions of the forward cone approximation property of Definition \ref{D_forw_back} for all $t \in [h^-,h^+)$.

Similarly, $\mathbf Z^1(\mathtt d,\A^1,\e,h^-,h^+)$ satisfies the \emph{final backward cone approximation property} if there exists a Borel section $\tilde{\mathtt d}_-$ which satisfies the assumptions of the backward cone approximation property of Definition \ref{D_forw_back} for all $t\in(h^-,h^+]$. 
\end{definition}

The next lemma is the analogue of Lemma \ref{L_forward_esti_area} for the Borel sections $\tilde{\mathtt d}_\pm$.

\begin{lemma}
\label{Cinfin}
If $\mathbf Z^1(\mathtt d,\A^1,\e,h^-,h^+)$ satisfies the initial (final) forward (backward) cone approximation property, then \eqref{E_forward_esti_area1} (resp. \eqref{E_forward_esti_area2}) holds for $\tilde\sigma^{s,t}_+$ ($\tilde\sigma^{s,t}_-$), for all $h^- \leq s \leq t \leq h^+$ (resp. $h^- \leq t \leq s\leq h^+$).
\end{lemma}

\subsection{\texorpdfstring{$k$-dimensional model sets}{k-dimensional model sets}}
\label{Ss_k_dim_model_set}

In this section we extend the results for $1$-dimensional model sets to the $k$-dimensional model sets defined below.

\begin{definition}
\label{D_k_dim_model}
A \emph{$k$-dimensional model set} is a $k$-dimensional directed sheaf set $\{Z\ka,C\ka\}_{\a\in\A^k}$ with reference plane $V^k = \langle \e^k_1,\dots,\e^k_k \rangle$, for which there exist $\mathrm h^- = (h^-_1,\dots,h^-_k)$, $\mathrm h^+ = (h^+_1,\dots,h^+_k) \in \mathbb R^k$, with $h^-_j < h^+_j$ for all $j=1,\dots,k$, such that
\begin{equation}
\label{E_pkdimen}
\mathtt p_{V^k} Z\ka = \interr U \big( \{\e^k_i\},\mathrm h^-,\mathrm h^+ \big) := \bigg\{ \sum_{j=1}^k t_j \mathtt e^k_j : t_j \in (h^-_j,h^+_j) \bigg\}.
\end{equation}
\end{definition}

We will also call \emph{$k$-dimensional model set} the set $\mathbf Z^k=\underset{\a\in\A^k}{\cup}Z\ka$. Setting
\begin{equation*}
% \label{E_D_dir_k_model}
\mathcal D : \mathbf Z^k \to \mathcal C(k,\R^d) \cap \S^{d-1}, \quad \mathcal D : Z\ka \ni z \mapsto \mathcal D(z) := C\ka \cap \S^{d-1}
\end{equation*}
for the direction map, sometimes we will use the more precise notation $\mathbf Z^k(\mathcal D,\A^k,\e^k_1,\dots,\e^k_k,\mathrm h^-,\mathrm h^+)$ and we say that the $3k$-tuple $(\e^k_1,\dots,\e^k_k, \mathrm h^-,\mathrm h^+)$ is a \emph{reference configuration} for the model set.

The absolute continuity problem for the disintegration of the Lebesgue measure on a $k$-dimensional model set $\mathbf Z^k$ can be reduced to the absolute continuity problem for $1$-dimensional model sets obtained cutting $\mathbf Z^k$ with suitable $(d-k+1)$-dimensional planes, called \emph{slices}. 

Indeed, for all $\e \in C(\{\e^k_i\})\cap\S^{d-1}$ and $w\in  \interr U(\{\e^k_i\},\mathrm h^-,\mathrm h^+)$, set
\begin{subequations}
% \label{E_h_-_h_+}
\begin{equation*}
% \label{E_h_-_h_+1}
h^-(w,\e) := \inf \Big\{ t \in \R : w + t \e \in U \big( \{\e^k_i\},\mathrm h^-,\mathrm h^+ \big) \Big\},
\end{equation*}
\begin{equation*}
% \label{E_h_-_h_+2}
h^+(w,\e) := \sup \Big\{ t \in \R : w + t \e \in U \big( \{\e^k_i\},\mathrm h^-,\mathrm h^+ \big) \Big\}.
\end{equation*}
\end{subequations}

\begin{definition}
\label{D_1_sim_slice}
Given a $k$-dimensional model set $\mathbf Z^k(\mathcal D,\A^k,\mathrm e^k_1,\dots,\mathrm e^k_k,\mathrm h^-,\mathrm h^+)$ we define \emph{$1$-dimensional slice} of $\mathbf Z^k$ in the direction $\e\in C(\{\e^k_i\})\cap\S^{d-1}$ any set of the form 
\begin{equation}
\label{E_slicekdim}
\mathbf Z^k \cap \mathtt p^{-1}_{V^k} \Big( w + \big( h^-(w,\e),h^+(w,\e) \big) \e \Big), \quad w \in U \big( \{\e^k_i\},\mathrm h^-,\mathrm h^+ \big).
\end{equation}
\end{definition}

The important observation is the following:

%Sara {changed property into remark}
\begin{remark}
\label{Prop_1dim_slice}
By \eqref{E_pkdimen} and Point $(3)$ of Definition \ref{D_sheaf_set}, the $1$-dimensional slice \eqref{E_slicekdim} is a model set of directed segments in the $(d+1-k)$-dimensional space $\mathtt p^{-1}_{V^k}(w + \langle\e\rangle)$ with direction vector field 
\begin{equation}
\label{E_mathtt_d_e}
\mathtt d_\e:=\mathcal D\cap\mathtt p^{-1}_{V^k}(\langle \e\rangle),
\end{equation}
quotient space $\A^k$ and reference configuration $(\e, h^-(w,\e),h^+(w,\e))$.
\end{remark}

As a consequence we obtain the following $k$-dimensional version of Lemma \ref{L_forward_esti_area}.

\begin{lemma}
\label{L_one_d_slicing_FC}
If the forward (or backward) cone approximation property holds for all the $1$-dimensional slices of a $k$-dimensional model set $\mathbf Z^k$ in the directions $\e^k_1,\dots\e^k_k$, then
\begin{equation*}
% \label{E_abs_Zk}
\LL \llcorner_{\mathbf Z^k} = \int \upsilon\ka\, d\eta(k,a), \quad \text{with} \ \upsilon\ka \ll \mathcal H^{k}\llcorner_{Z\ka} \ \text{and} \ \eta(k) \ll \mathcal H^{d-k} \llcorner_{\A^k}.
\end{equation*}
If both the forward and the backward properties hold for all the above slices, then $\mathbf Z^k$ is regular, namely
\begin{equation}
\label{E_equiv_Zk}
\eta(k) \simeq \mathcal H^{d-k} \llcorner_{\A^k} \quad \text{and} \quad \upsilon\ka \simeq \mathcal H^{k} \llcorner_{Z\ka} \ \text{for $\eta$-a.e. $(k,\a)$}.
\end{equation}
\end{lemma}

\begin{proof}
W.l.o.g. we assume $\e^k_i$ to be the first $k$ unit vectors of a standard orthonormal base in $\R^d$, $h_j^-<0<h_j^+$ $\forall\,j=1,\dots,k$, and $\A^k = \mathbf Z \cap \mathtt p_{V^k}^{-1}(0)$.

For $x \in U(\{\e^k_i\},\mathtt h^-,\mathtt h^+) \cap \mathtt p^{-1}_{\langle \e^k_1 \rangle}(0)$, consider the $1$-dimensional slice
\[
 \mathbf Z^1_x := \mathbf Z^k \cap \mathtt p^{-1}_{V^k} \big( x + (h^-_1,h^+_1) \e^k_1 \big) = \bigcup_{\a \in \A^k} Z^k_\a \cap \mathtt p^{-1}_{V^k} \big( x + (h^-_1,h^+_1) \e^k_1 \big).
\]

By Fubini-Tonelli theorem 
\[
\mathcal L^d \llcorner_{\mathbf Z^k} = \int _{\mathbf Z^k\cap\p^{-1}_{\langle \e_1^k\rangle}(0)}\mathcal L^{d-k+1} \llcorner_{\mathbf Z^1_{x}} \mathcal L^{k-1}(dx)
\]
and applying Lemma \ref{L_forward_esti_area} in the case of forward (or backward) approximation property, we obtain the disintegration
\[
\mathcal L^d \llcorner_{\mathbf Z^k} = \int \upsilon^1_{\a,x} d\eta_1(x,\a), \quad \text{with} \ \upsilon^1_{\a,x} \ll \mathcal H^1 \llcorner_{Z^k_\a \cap \mathtt p^{-1}_{V^k} ( x + (h^-_1,h^+_1) \e^k_1 )}, \ \eta_1 \ll \mathcal H^{d-1} \llcorner_{\mathbf Z^k \cap\, \mathtt p_{\langle \e^k_1 \rangle}^{-1}(0)}.
\]

Now one repeats the procedure starting with the measure $\mathcal L^{d-1}$ restricted to the $(k-1)$-dimensional model set in $\R^{d-1}$ given by
\[
\mathbf Z^{k-1} = \mathbf Z^k \cap \mathtt p_{\langle \e^k_1 \rangle}^{-1}(0)
\]
and considering directions along $\e^k_2$, which allows to write
%Sara
\[
\eta_1 = \int_{\mathbf Z^k\cap\mathtt p_{\langle \e^k_1,\e^k_2 \rangle}^{-1} (0)} \eta_{1,\a,y} d\eta_2(y,\a), \quad \text{with} \ \eta_{1,\a,y} \ll \mathcal H^1 \llcorner_{Z^k_\a \cap \mathtt p^{-1}_{V^k}(y+(h_2^-,h_2^+)\e^k_2)}, \ \eta_2 \ll \mathcal H^{d-2} \llcorner_{\mathbf Z^k \cap \mathtt p^{-1}_{\langle\e^k_1,\e^k_2 \rangle}(0)}
\]
Hence by composing the two disintegrations one obtains
\[
\mathcal L^d \llcorner_{\mathbf Z^k} = \int \upsilon^2_{\a,y} d\eta_2(y,\a), \quad \text{with} \ \upsilon^2_{\a,y} \ll \mathcal H^2 \llcorner_{Z^k_{\a} \cap \mathtt p^{-1}_{V^k}(y+(h^-_1,h^+_1) \e^k_1+(h_2^-,h_2^+)\e^k_2)}.
\]
Iterating the process $k$-times, one obtains the result.

In case both approximation properties holds, the same analysis shows \eqref{E_equiv_Zk}.
\end{proof}

\subsection{\texorpdfstring{$k$-dimensional sheaf sets and $\mathcal D$-cylinders}{k-dimensional sheaf sets and D-cylinders}}
\label{Ss_k_dim_sheaf}

Now we apply the cone approximation property technique also to general $k$-directed sheaf sets.

Let $\{Z\ka,C\ka\}_{\a\in\A^k}$ be a $k$-directed sheaf set with reference plane $V^k = \langle \e^k_1,\cdots,\e^k_k \rangle$, base rectangle $U(\{\e_i^k\},\mathtt h^-,\mathtt h^+)$ and direction map
\begin{equation}
\label{E_k_dim_dire_map}
\mathcal D^k(\a) := C\ka \cap \mathbb S^{d-1}.
\end{equation}
For $\A^{k,'} \subset \A^k$ $\sigma$-compact, set
\begin{equation}
\label{E_bf_Zk'_def}
\mathbf Z^{k,'} := \bigcup_{\a \in \A^{k,'}} Z\ka.
\end{equation}

\begin{definition}
\label{D_mathcal_D_cyl}
Any $k$-dimensional model set $\mathbf Z^k(\mathcal D, \A^{k,'},\e^k_1,\dots,\e^k_k,\mathrm h^-, \mathrm h^+)$ of the form
\begin{equation*}
% \label{Edcil}
\mathbf Z^k \big( \mathcal D, \A^{k,'},\e^k_1,\dots,\e^k_k,\mathrm h^-, \mathrm h^+ \big) = \mathbf Z^{k,'} \cap \mathtt p_{V^k}^{-1} \big(\interr  U(\{\e^k_i\},\mathrm h^-,\mathrm h^+) \big)
\end{equation*}
for which there exists $\epsilon>0$ such that
\[
\begin{split}
&\mathbf Z^k \Big( \mathcal D,\A^{k,'},\e^k_1,\dots,\e^k_k, \mathrm h^- -(\epsilon,\dots,\epsilon), \mathrm h^+ + (\epsilon,\dots,\epsilon) \Big) \\
& \qquad \qquad \qquad = \mathbf Z^{k,'} \cap \mathtt p_{V^k}^{-1} \Big( \interr  U \big( \{\e^k_i\},\mathrm h^--(\epsilon,\dots,\epsilon),\mathrm h^++(\epsilon,\dots,\epsilon) \big) \Big)
\end{split}
\]
is also a $k$-dimensional model set, will be called \emph{$k$-dimensional $\mathcal D$-cylinder}.
\end{definition}

In particular, by the above definition and Definition \ref{D_k_dim_model}
\[
\mathtt p_{V^k} Z \ka \supset \interr U \Big( \{\e^k_i\},\mathrm h^--(\epsilon,\dots\,\epsilon), \mathrm h^++(\epsilon,\dots\,\epsilon) \Big) \qquad \forall\, \a \in \A^{k,'}.
\]
%Sara{ho incluso l'osservazione in un  remark}
\begin{remark}
\label{R_property}
Notice that, since any $Z\ka$ is a relatively open set, then the sheaf set can be covered by a countable disjoint collection of $k$-dimensional $\mathcal D$-cylinders
\begin{equation}
\label{E_bf_Z_k'_n}
\mathbf Z^{k,'}_n = \mathbf Z^k \big( \mathcal D, \A^{k,'}_n,\e^k_1,\dots,\e^k_k,\mathrm h^-_n, \mathrm h^+_n \big), \qquad n \in \N,
\end{equation}
up to the points which belong to the perpendicular sections
\[
\mathbf Z^{k,'}_n \cap \mathtt p^{-1}_{V^k} \big( \partial U(\{\e^k_i\},\mathrm h_n^-, \mathrm h_n^+) \big).
\] 
In particular, the $k$-dimensional $\mathcal D$-cylinders as in \eqref{E_bf_Z_k'_n} define a partition of the sheaf set $\mathbf Z^{k,'}$ up to an $\mathcal L^d$-negligible set. 
\end{remark}
%Sara
\begin{definition}
\label{D_1_dim_slice_sheaf}
Define \emph{$1$-dimensional slices} of a directed locally affine partition $\{Z\ka, C\ka\}_{\underset{\a\in\A^k}{k=1,\dots,d}}$ the $1$-dimensional slices of any of the $k$-dimensional $\mathcal D$-cylinders given by \eqref{E_bf_Z_k'_n}, for any of the countably many $k$-directed sheaf sets $\mathbf Z^k$ given by Proposition \eqref{P_countable_partition_in_reference_directed_planes}.
\end{definition}

\begin{remark}
\label{R_slice0}
Notice that, for any $1$-dimensional slice of a $\mathcal D$-cylinder
\[
\mathbf Z^{k,'}_n \cap \mathtt p^{-1}_{V^k} \Big( w_n + \big( h^-(w_n,\e_n),h^+(w_n,\e_n) \big) \e_n \Big), \quad w_n \in \interr U \big( \{\e^k_i\},\mathrm h^-_n,\mathrm h^+_n \big), \ \e_n \in C(\{\e^k_i\}),
\]
as in \eqref{E_slicekdim}, there exists $\epsilon = \epsilon_n >0$ such that the set
\[
\mathbf Z^{k,'}_n \cap \mathtt p^{-1}_{V^k} \Big( w_n + \big( h^-(w_n,\e_n) - \epsilon_n,h^+(w_n,\e_n) + \epsilon_n \big) \e_n \Big)
\]
is still a $1$-dimensional model set. This assures that the extreme points of the segments of the slice do not contain relative boundary points of the sets of the partition, and in particular the vector field $\mathtt d_{\e_n}$ is single-valued up to the boundary of $Z^{k,'}_{n,\a}$.
\end{remark}

The main result of this section is then the following theorem.

\begin{theorem}
\label{T_one_d_slicing_FC}
If either the forward cone approximation property or the backward cone approximation property holds for the $1$-dimensional slices of a directed locally affine partition $\{Z\ka, C\ka\}_{k,\a\in\A^k}$, then
\[
\mathcal L^d \llcorner_{\mathbf Z^k} = \int \upsilon^k_a d\eta(k,\a) \quad \text{with} \ \eta(k) \ll \mathcal H^{d-k} \llcorner_{\A^k} \ \text{and} \ \upsilon\ka \ll \mathcal H^{k} \llcorner_{Z\ka} \ \text{for $\eta(k)$-a.e.} \ \a \in \A^k.
\]

If both properties hold, then $\mathbf Z^k$ is regular, i.e.
\[
\eta(k) \simeq \mathcal H^{d-k} \llcorner_{\A^k} \ \text{and} \ \upsilon\ka \simeq \mathcal H^{k} \llcorner_{Z\ka} \ \text{for $\eta(k)$-a.e.} \ \a \in \A^k.
\]
\end{theorem}

%Sara {la dimostrazione mi sembra si possa ridurre, visti i remark precedenti, alla seguente}
\begin{proof}
By Proposition \ref{P_countable_partition_in_reference_directed_planes} and Remark \ref{R_property}, using a simple covering argument and $\sigma$-additivity of measures one can reduce to study the absolute continuity of the disintegrations on $\mathcal D$-cylinders. Then, by Definition \ref{D_1_dim_slice_sheaf}, one concludes using Lemma \ref{L_one_d_slicing_FC} .
\end{proof}

\subsection{Negligibility of initial/final points}
\label{Ss_neglig_init_fin}

Now we deal with the other measure-theoretic problem connected to directed locally affine partitions, namely to establish whether
\[
\LL(\mathcal I) = 0 \quad \text{ and/or } \quad \LL(\mathcal E) = 0.
\]
It turns out that these properties are implied by the validity of the initial/final cone approximation properties for \emph{initial/final $1$-dimensional slices}.

\begin{definition}
\label{L_init_fin_1_dim_slice}
An \emph{initial $1$-dimensional slice} of a directed locally affine partition $\{Z\ka, C\ka\}_{\underset{\a\in\A^k}{k=1,\dots,d}}$ is a $1$-dimensional model set of the form
\[
\mathbf Z^k \cap \mathtt p_{V^k}^{-1} \big(w+(h^-,h^+)\e \big)
\]
for a $k$-directed sheaf set $\mathbf Z^k $ with reference plane $V^k=\langle\e_1,\dots,\e_k\rangle$, $\e\in C(\{\e_i^k\})\cap\mathbb S^{d-1}$, for which there exists $\epsilon > 0$ such that the set
\[
\mathbf Z^k \cap \mathtt p_{V^k}^{-1} \big( w+(h^-,h^++\epsilon) \e \big)
\]
is still a $1$-dimensional model set.

Similarly, a \emph{final $1$-dimensional slice} of a directed locally affine partition $\{Z\ka, C\ka\}_{\underset{\a\in\A^k}{k=1,\dots,d}}$ is a $1$-dimensional model set of the form
\[
\mathbf Z^k \cap \mathtt p_{V^k}^{-1} \big( w+(h^-,h^+)\e \big)
\]
for which there exists $\epsilon>0$ such that the set
\[
\mathbf Z^k \cap \mathtt p_{V^k}^{-1} \big( w+(h^--\epsilon,h^+) \e \big)
\]
is still a $1$-dimensional model set.
\end{definition}
%Sara{I added remark \ref{R_infinslices}}
By Remark \ref{R_slice0}, the vector field $\mathtt d_\e$ of an initial slice can be multivalued only at the points of the section $P_{h^-}$, while for a final slice it can be multivalued only on $P_{h^+}$. 

\begin{remark}
 \label{R_infinslices}
Notice that a $1$-dimensional slice of a directed locally affine partition according to Definition \ref{D_1_dim_slice_sheaf} is both an initial/final $1$-dimensional slice. In particular, since the direction vector field of a $1$-dimensional slice is single-valued, its Borel-measurable sections coincide trivially with itself and then the initial/final forward/backward cone approximation properties (see Definition \ref{D_initial_forward}) are simply an extension of the forward/backward cone approximation property (see Definition \ref{D_forw_back}) to the initial/final points of the slice. Hence, saying that the initial/final $1$-dimensional slices of a directed locally affine partition satisfy the initial forward/final backward cone approximation property implies that the $1$-dimensional slices of that directed locally affine partition satisfy the forward/backward cone approximation property. 
\end{remark}

The following theorem follows from Corollary \ref{Cinfin}, as in the density Lemma 4.19 proved in \cite{CarDan} for the relative boundary points of the locally affine partition into the faces of a convex function.

\begin{theorem}
\label{T_FC_no_initial}
If the initial $1$-dimensional slices of a directed locally affine partition satisfy the initial forward cone approximation property, then
\[
\LL(\mathcal I) = 0.
\]

Similarly, if the final $1$-dimensional slices of a directed locally affine partition satisfy the final backward cone approximation property, then
\[
\LL(\mathcal E) = 0.
\]
\end{theorem}

We give only a sketch of the proof, since the details have already been given in Lemma 4.19 of \cite{CarDan}.

\begin{proof}
W.l.o.g. we can restrict to a $k$-directed sheaf set $Z^k$ with reference $k$-plane $V^k=\langle\e^k_1,\dots,\e^k_k\rangle$. Let us then consider the map
\begin{equation}
\label{E_r_init_map}
\mathcal I \ni z \mapsto \mathtt l(z) := \sup \Big\{ r : z + \big( \interr\, C\ka \cap B^d(0,r) \big) \subset Z\ka, \ \text{for some} \ \a \in \A^k \Big\}.
\end{equation}
By Definition \ref{D_initial_final}, $\mathtt l(z) > 0$ for all $z \in \mathcal I$, and then by a countable covering argument we need only to prove the negligibility of the set
\begin{equation}
\label{E_cal_I_r}
\mathcal I^{\bar r} := \mathcal I \cap \mathtt l^{-1}(\bar r),
\end{equation}
with $\bar r > 0$ fixed.

Assume that $\mathcal L^d(\mathcal I^{\bar r})>0$ and that $\bar z \in \bar{\mathbf Z}^k$ is a Lebesgue point of $\mathcal I^{\bar r}$. Then if $\e\in C(\{\e_i^k\})\cap\mathbb S^{d-1}$, at least one of the sets
\[
P_{w + t \e} = \mathcal I^{\bar r} \cap \mathtt p^{-1}_{V^k} ( w + t \e), \quad w \in \langle\e\rangle^\perp \cap V^k,
\]
has $\mathcal H^{d-k}$-positive measure, and we can assume that $\bar z$ is also a Lebesgue point for $\mathcal H^{d-k} \llcorner_{P_{w+t\e}}$. For definiteness, we will assume that $P_{w + t \e}=P_{\bar w + h^-(\bar w, \e) \e}=\overline{\mathbf Z}^{k,'}\cap\p^{-1}_{V^k}(\bar w+h^-(\bar w,\e)\e)$ for some $\bar w$ in the relative interior of the base rectangle of $Z^{k,'}\subset Z^k$, and let
\[
P_{\bar w + h^-(\bar w, \e) \e}\ni z\mapsto\mathtt d^{h^-}(z) := \Big\{ C\ka \cap \mathtt p^{-1}_{V^k}\langle\e\rangle: z \in \mathcal I(Z\ka) \Big\}
\]
be the multivalued maps defined in \eqref{E_tt_d_ext} for the $1$-dimensional slice $\mathbf Z^{k,'}_{\bar w, \e}$ defined by
\[
\mathbf Z^{k,'}_{\bar w,\e}:=\Big\{ Z\ka \cap \mathtt p^{-1}_{V^k} \big( \bar w + (h^-,h^+) \e \big), C\ka \cap \mathtt p^{-1}_{V^k}\langle\e\rangle\Big\}_{\a \in \A^{k,'}}, \quad h^+ := h^- + \frac{\bar r}{2},
\]
where $\A^{k,'}$ is the sets of $\a$ such that $Z\ka$ has an initial point $z$ on $P_{\bar w + h^-(\bar w, \e) \e}$ and
\[
z + \interr\, C\ka \cap B^d(0,\bar r/2) \subset Z\ka.
\]

If $\tilde{\mathtt d}^{h^-}_+$ is a Borel section of $\mathtt d^{h^-}$ chosen accordingly to Definition \ref{D_initial_forward}, then consider the $1$-dimensional slice $\mathbf Z^{k,''}_{\bar w, \e} \subset \mathbf Z^{k,'}_{\bar w, \e}$ defined by
\[
\mathbf Z^{k,''}_{\bar w,\e}:=\Big\{ Z\ka \cap \mathtt p^{-1}_{V^k} \big( \bar w + (h^-,h^+) \e \big), C\ka \cap \mathtt p^{-1}_{V^k}(\langle\e\rangle) \Big\}_{\a \in \A^{k,''}}, \quad h^+ := h^- + \frac{\bar r}{4},
\]
where $\A^{k,''} \subset \A^{k,'}$ is the set of $\a$ satisfying
\[
z \in P_{\bar w + h^- (\bar w,\e)\e} \quad \Longrightarrow \quad \tilde{\mathtt d}^{h^-}_+(z) = C\ka \cap \mathtt p^{-1}_{V^k}\langle\e\rangle \cap \mathbb S^{d-1}.
\]
In other words, $\mathbf Z^{k,''}_{\bar w,\e}$ is the $1$-dimensional slice whose initial points belong to $P_{\bar w + h^-(\bar w,\e) \e}$ and whose segments are given by $\dom\,\tilde{\mathtt d}_+\cap\p_{V^k}^{-1}\big(\bar w+\big(h^-(\bar w,\e),h^-(\bar w,\e)+\frac{\bar r}{4}\big)\big)\e$, where $\tilde{\mathtt d}_+$ was defined in \eqref{E_tilde_tt_d_+}. Clearly, by restricting to a $\sigma$-compact $\mathcal H^{d-k}$-conegligible subset of $P_{\bar w + h^-(\bar w, \e) \e}$ so that $\tilde{\mathtt d}^{h-}_+$ is $\sigma$-continuous, this procedure defines an initial $1$-dimensional slice. 

By the initial forward cone approximation property, Lemma \ref{Cinfin} implies that if $\bar z \in \mathcal I(Z^{k,''}_{\bar w,\e})$ is a Lebesgue point of $\mathcal H^{d-k} \llcorner_{P_{\bar w + h^-(\bar w,\e) \e}}$, then
\[
\lim_{r \to 0} \bigg[ \lim_{t \searrow h^-} \mathcal H^{d-k} \big( \mathbf Z^{k,''}_{\bar w,\e} \cap P_{\bar w + t \e} \cap B^d(\bar z,r) \big) \bigg] = 1.%\longrightarrow_{t \searrow h^-} 
\]
Since $\mathbf Z^{k,''}\cap\mathcal I^{\bar r}=\emptyset$, this clearly contradicts the fact that $\bar z$ is a Lebesgue point of $\mathcal I^{\bar r}$.

\end{proof}
%Sara{ho aggiunto questa definizione per uniformare la notazione}
In view of Theorems \ref{T_one_d_slicing_FC} and \ref{T_FC_no_initial} and recalling Remark \ref{R_infinslices}, for future convenience we give the following 
\begin{definition}
\label{D_coneapprpart}
 A directed locally affine partition satisfies the (initial/final) forward/backward cone approximation property if its (initial/final) $1$-dimensional slices satisfy the 
 (initial/final) forward/backward cone approximation property.
\end{definition}

\section{Proof of Theorem \ref{T_1}}
\label{S_theorem_1_proof}

This section is devoted to the proof of Theorem \ref{T_1}, which we recall below.

\begin{theorem1}
Let $\mu,\nu \in \PP(\R^d)$ with $\mu \ll \LL$ and let $\d{\cdot}$ be a convex norm in $\R^d$. Then there exists a locally affine directed partition $\{Z^k_\a,C^k_\a\}_{\overset{k=0,\dots,d}{\a\in\A^k}}$ in $\R^d$ with the following properties:
\begin{enumerate}
% \begin{align}
% &(1) \quad \text{
\item for all $\a \in \A^k$ the cone $C^k_\a$ is a $k$-dimensional extremal face of $\d{\cdot}$;
% ;} \notag \\
% &(2) \quad 
\item $\LL \Bigl( \R^d \setminus \underset{k,\a}{\bigcup}\, Z^k_\a \Bigr)=0$;
% ; \notag \\
% &(3) \quad 
\item $\{Z\ka\}_{k,\a}$ is regular, namely the disintegration of the measure $\LL$ w.r.t. the partition $\{Z\ka\}_{k,\a}$, $\displaystyle{\LL \llcorner_{\underset{k,\a}{\cup} Z\ka} = \int v^k_\a\,d\eta(k,\a)}$, satisfies
\[
v^k_a \simeq \HH^k \llcorner_{Z^k_\a}\quad\text{for $\eta(k)$-a.e. $\a\in\A^k$.}
\]
% } \notag \\
% &(4) \quad 
\item \label{Point_4_T_1_h} for all $\pi \in \pod(\mu,\nu)$, the disintegration $\displaystyle{\pi = \int \pi\ka\, dm(k,\a)}$ w.r.t. the partition $\{Z\ka \times \R^d\}_{k,\a}$ satisfies
\[
\pi^k_\a \in \Pi^f_{\mathtt c_{C^k_\a}}(\mu\ka, (\mathtt p_2)_\# \pi\ka),
\]
where $\displaystyle{\mu = \int \mu\ka\,dm(k,\a)}$ is the disintegration w.r.t. the partition $\{Z\ka\}_{k,\a}$, and moreover
\[
(\mathtt p_2)_\#\pi\ka \biggl( Z\ka \cup \biggl( \R^d \setminus \underset{(k',\a') \not= (k,\a)}{\bigcup} Z^{k'}_{\a'} \biggr) \biggr) = 1.
\]
% }. \notag %\label{E_finite_cone_cost}.
% \end{align}
\end{enumerate}

If also $\nu \ll \LL$, then for all $\pi\in\pod(\mu,\nu)$
\[
(\mathtt p_2)_\#\pi\ka = \nu\ka
\]
where $\displaystyle{\nu = \int \nu\ka \,dm(k,\a)}$ is the disintegration w.r.t. the partition $\{Z\ka\}_{k,\a}$, and the converse of Point \eqref{Point_4_T_1_h} holds:
\begin{equation*}
\pi\ka \in \Pi^f_{\mathtt c_{C\ka}}(\mu\ka,\nu\ka) \quad \Longrightarrow \quad \pi \in \pod(\mu,\nu).
\end{equation*}

\end{theorem1}

We start the proof by recalling that, by Proposition \ref{P_equivalence_lifting},
\begin{equation}
\label{E_pioptgraph}
\pi \in \Pi^{\mathrm{opt}}_{\d{\cdot}}(\mu,\nu) \quad \Longleftrightarrow \quad \hat \pi \in \Pi(\hat\mu,\hat\nu),\ \hat{\pi}(\partial^+\Graph\,\psi)=1,
\end{equation}
where $\psi:\R^d\to\R$ is the $\d{\cdot}$-Lipschitz function given by a Kantorovich potential and $\hat\mu$, $\hat\nu$, $\hat\pi$ are the push-forwards of $\mu$, $\nu$, $\pi$ on $\R^{d+1}$ through the map $(\Id\times\psi)$.

By Remark \ref{R_lipgraph}, $\Graph\,\psi \subset \R^{d+1}$ is a \emph{complete $\mathtt c_{\epi\d{\cdot}}$-Lipschitz graph}, according to Definition \ref{D_complete_G}. Then, by Proposition \ref{P_ex_fol}, call $\theta_\psi$ the trivial $\mathtt c_{\C}$-Lipschitz foliation on $\R^{d+1}$ associated to $\Graph\,\psi$.

We now show that Theorem \ref{T_1} follows from \eqref{E_pioptgraph} , thanks to the results of Sections \ref{S_foliations} and \ref{S_disintechnique} and the following theorem.

%Sara
\begin{theorem}
\label{T_cone_graph}
Let $\Graph\,\varphi\subset\R^{d+1}$ be a complete $\mathtt c_{\epi\d{\cdot}}$-Lipschitz graph. Then, the superdifferential partition satisfies the initial forward cone approximation property and the subdifferential partition satisfies the final backward approximation property. 
\end{theorem}

Indeed, recalling Definition \ref{D_coneapprpart}, first notice that if Theorem \ref{T_cone_graph} holds, then by Theorem \ref{T_FC_no_initial}
\begin{equation}
\label{E_hd0}
\mathcal H^d(\mathcal I^+\theta_\varphi)= \mathcal H^d(\mathcal E^-\theta_\varphi) = 0.
\end{equation}
Moreover, by Remark \ref{R_infinslices} and the fact that $\hat{\mathbf D}=\hat{\mathbf D}^+\cap\hat{\mathbf D}^-$ (see \eqref{E_sub_directed_partition_all}), Theorem \ref{T_one_d_slicing_FC} implies that the disintegration of the $d$-dimensional Hausdorff measure on the differential partition $\hat{\mathbf D}$ of $\theta_\varphi$ is regular, namely
has conditional probabilities equivalent to the Hausdorff measures on the locally affine sets on which they are concentrated. 
Therefore, denoting the locally affine partition $\hat{\mathbf D}$ as
\[
\big\{ \hat Z^{k}_{\mathfrak \a},\hat C\ka \big\}_{\overset{k=1,\dots\,d}{\mathfrak a\in \mathfrak A^k}} \subset \mathtt P \bigg( \R^{d+1} \times \bigcup_{k=1}^d \mathcal C(k,\R^{d+1}) \bigg),
\] 
and setting $\{\hat Z^0_\a\}_{\a\in\mathfrak A^0}$ for the $0$-dimensional partition of the fixed points $\mathfrak A^0=\mathcal F\theta_\varphi$,
by \eqref{E_graphpsimu} the sets 
\[
\Big\{ Z^k_\a = \mathtt p_{\R^d} \hat Z^k_\a, C^k_\a = \mathtt p_{\R^d} \hat C^k_\a \Big\}_{\overset{k=0,\dots\,d}{\mathfrak a\in \mathfrak A^k}} \subset \mathtt P \bigg( \R^{d} \times \bigcup_{k=0}^d \mathcal C(k,\R^{d}) \bigg),
\]
define a locally affine directed partition of $\R^{d}$ satisfying $(1)$, $(2)$ and $(3)$. 

Let us now use the fact that $\varphi=\psi$ is a Kantorovich potential for $\Pi^{\mathrm{opt}}_{\d{\cdot}}(\mu,\nu)$ and that $\mu\ll\mathcal L^d$. By \eqref{E_graphpsimu} and \eqref{E_hd0},
\begin{align}
& \mu \ll \LL \quad \Longrightarrow \quad \hat \mu \ll \mathcal H^d \llcorner_{\Graph\,\psi} \quad \Longrightarrow \quad \hat \mu(\mathcal I^+\theta_\psi) = \hat \mu(\mathcal E^-\theta_\psi) = 0, \label{E_mu0}
\end{align}
Then, by Remark \ref{R_4.27}, Proposition \ref{P_disint_fol} applies to the locally affine partition $\big\{ \hat Z^{k}_{\mathfrak \a},\hat C\ka \big\}_{\overset{k=1,\dots\,d}{\mathfrak a\in \mathfrak A^k}}$, giving that any transport plan $\hat\pi$ as in \eqref{E_pioptgraph} satisfies $\hat\pi\in\Pi^f_{\mathtt c_{\hat{\mathbf D}}}(\hat\mu,\hat\nu)$. In particular, by Proposition \ref{P_dispiani}, the disintegration $\displaystyle{\hat\pi = \int \hat\pi\ka\, dm(k,\a)}$ w.r.t. the partition $\{\hat Z\ka \times \R^{d+1}\}_{k,\a}$ satisfies
\[
\hat\pi^k_\a \in \Pi^f_{\mathtt c_{\hat C^k_\a}}(\hat\mu\ka, (\mathtt p_2)_\# \hat\pi\ka),
\]
where $\displaystyle{\hat\mu = \int \hat\mu\ka\,dm(k,\a)}$ is the disintegration w.r.t. the partition $\{\hat Z\ka\}_{k,\a}$. Moreover, by Proposition \ref{P_hat_bf_D_graph}, the partition $\{\hat Z\ka\}_{k,\a}$ satisfies condition \eqref{E_more_than_complet}, which gives 
\[
(\mathtt p_2)_\#\hat\pi\ka \biggl( \hat Z\ka \cup \biggl( \Graph\,\psi \setminus \underset{(k',\a') \not= (k,\a)}{\bigcup} \hat Z^{k'}_{\a'} \biggr) \biggr) = 1.
\]
 Then it is not difficult to see that the directed locally affine partition of $\R^d$ given by $\big\{ Z^{k}_{\mathfrak \a},C\ka \big\}_{\overset{k=0,\dots\,d}{\mathfrak a\in \mathfrak A^k}}$ satisfies also Point $(4)$ of Theorem \ref{T_1}. Finally, if $\nu\ll\LL$, by \eqref{E_hd0} and \eqref{E_graphpsimu} we have also
 \begin{equation}
   \nu \ll \LL \quad \Longrightarrow \quad \hat \nu \ll \mathcal H^d \llcorner_{\Graph\,\psi} \quad \Longrightarrow \quad \hat\nu(\mathcal E^-\theta_\psi) = 0.\label{E_nu0}
 \end{equation}
 Then $\hat\nu(\mathtt p_{\R^{d+1}}(\hat{\mathbf D}))=1$ and Corollary \ref{C_transp_graph} gives, when projected on $\R^d$, the last part of Theorem \ref{T_1}.
%Sara{Ho incorporato anche il vecchio remark 6.4 R_converse_T_1}
\begin{remark}
\label{E_more_nat_form}
Observe that the characterization given by Proposition \ref{P_disint_fol} of the optimal transport plans for the $\mathtt c_{\epi\d{\cdot}}$-Lipschitz set $\Graph\,\psi$ seems more natural than the one given by Theorem \ref{T_1} for their projections on $\R^d$, namely the optimal transport plans for the original convex norm problem. Indeed, in the first case we have a complete (namely, if and only if) geometric characterization of the transport plans by disintegrations into transport plans of finite cone cost w.r.t. their conditional marginals, even in the case in which $\nu$ is not absolutely continuous. This is due to the geometric condition \eqref{E_more_than_complet}, which is satisfied by the partition $\{\hat Z\ka, \hat C\ka\}_{k,\a}$ and not by its projection on $\R^d$.

In particular, there might be decompositions $\{\nu\ka\}$ of $\nu$ which are not obtained by projections of second marginals of disintegrations of $\hat\pi\in \Pi(\hat\mu,\hat\nu),\ \hat{\pi}(\partial^+\Graph\,\psi)=1$ and such that $\Pi^f_{\mathtt c_{\p_{\R^d}\hat{\mathbf D}}}(\mu,\{\nu\ka\}) \not= \emptyset$.
\end{remark}

\begin{proof}[Proof of Theorem \ref{T_cone_graph}]
The proof will be given in two steps.
We prove the initial forward cone approximation property for the superdifferential partition of the forward regular set, being the proof of the final backward cone approximation property for the subdifferential partition analogous.
 
By Definition \ref{D_coneapprpart}, let us consider an initial $1$-dimensional slice
\[
\mathbf Z^{k,+} \cap \mathtt p^{-1}_{V^k} \bigl( (h^-, h^+) \e \bigr)=\underset{k,\a}{\bigcup}\,Z^{k,+}_\a\cap \mathtt p^{-1}_{V^k} \bigl( (h^-, h^+) \e \bigr)
\]
where
\begin{enumerate}
\item $V^k = \langle \e^k_1,\dots,\e^k_k\rangle\in \mathcal G(k,\R^{d+1})$ reference plane of the sheaf set $\mathbf Z^{k,+}$,
\item $\e \in \S^d \cap C(\{\e^k_i\})$,
\item there exists $\eps > 0$ for which
\[
\mathbf Z^{k,+} \cap \mathtt p^{-1}_{V^k} \bigl( ( h^-, h^+ + \eps)\e \bigr)
\]
is still a $1$-dimensional model set. Let $\mathtt d_\e$ be the direction vector field.
\end{enumerate}
{\it Step 1.} Assume $\mathtt d_\e^{h^-}$ is injective. 
  
Then, it is sufficient to prove the forward cone approximation property for the vector field $\mathtt d_\e$ on a fixed perpendicular section, say e.g. $P_{h^-}$. 

First recall that, by the general properties of sheaf sets, i.e. Point $(3)$ of Definition \ref{D_sheaf_set},
\[
\{0\} \cup \R^+ \e \subset C(\{\e^k_i\}) \subset \mathtt p_{V^k}(C^{k,+}_\a), \quad \forall\,\a\in\A^k.
\]
By definition of $1$-dimensional slice, for all $w_{\a,h^-} = (x_{\a,h^-},\varphi(x_{\a,h^-})) \in P_{h^-}$ one has
\begin{equation}
\label{E_tt_d_e_sigma}
\mathtt d_\e(w_{\a,h^-}) = \frac{\sigma^{h^-,h^++\eps}(w_{\a,h^-})-w_{\a,h^-}}{|\sigma^{h^-,h^++\eps}(w_{\a,h^-})-w_{\a,h^-}|},
\end{equation}
being
\[
\sigma^{h^-,h^++\eps}(w_{\a,h^-}) = w_{\a,h^{+}+\eps} = \big( y_{\a,h^{+}+\eps},\varphi(y_{\a,h^{+}+\eps}) \big)
\]
the unique point s.t.
\[
P_{h^++\eps} \cap \big( w_{\a,h^{-}} + C^{k,+}_\a \big) = \{ w_{\a,h^{+}+\eps} \}.
\]
Since now we are dealing with the superdifferential partition, (see Theorems \ref{T_partition_E+-} and \ref{T_partition_E}) for all $w_\a \in Z^{k,+}_\a$
\begin{equation*}
C^{k,+}_{\a} = \R^+ \mathcal D^+{\theta}_{\varphi}(\a,w_{\a}) \quad \text{and} \quad \partial^+\theta_{\varphi}(\a,w_{\a}) \cap \mathcal R^{+,k} \theta_{\varphi} = \big( w_{\a} + C^{k,+}_{\a} \big) \cap Z^{k,+}_\a.
\end{equation*}
Then we conclude that
\[
y_{\a,h^{+}+\eps}=\mathtt p_{\R^d}(w_{\a,h^{+}+\eps})
\]
is the unique point of $\mathtt p_{\R^d}(P_{h^++\eps})$ s.t. 

\[
\varphi(y_{\a,h^{+}+\eps}) - \varphi(x_{\a,h^{-}}) = \big| y_{\a,h^{+}+\eps} - x_{\a,h^{-}} \big|_{D^*},
\]
namely $y_{\a,h^{+}+\eps}$ is the unique maximizer of
\begin{equation}
\label{E_ya_unique}
\varphi(x_{\a,h^-}) = \underset{y \in \mathtt p_{\R^d}(P_{h^++\eps})} \max \Big\{ \varphi(y) - \d{y-x_{\a,h^-}} \Big\}.
\end{equation}

Hence one can construct the finite cone approximations of $\mathtt p_{\R^d} \mathtt d_e$ as in \cite{Car1}, namely discretizing the set $\mathtt p_{\R^d}(P_{h^++\eps})$ and taking the cones given by the differential partition of an optimal potential w.r.t. a strictly convex cost obtained by perturbating the norm cost $\d{\cdot}$ and whose second marginals are Dirac deltas centered at the points of the discretization (see \cite{Car1}). The convergence of the approximations to $\mathtt p_{\R^d}\mathtt d_e$ at a.e. point $x_{\a,h^-}$ as the cost perturbation goes to $0$ and the points of the discretization become dense is given by the uniqueness of the $y_{\a,h^{+}+\eps}\in\mathtt p_{\R^d}(P_{h^++\eps})$ satisfying \eqref{E_ya_unique}.

Lifting the approximating cones with the map $\Id\times\varphi$, one gets finite cones approximations of $\mathtt d_e$ as required.

{\it Step 2.} Let now $\mathtt d_\e^{h^-}$ be possibly multivalued. In order to prove the initial forward cone approximation property, we build as in Step 1 finite cone approximations given by the differential partition of optimal potentials w.r.t. strictly convex approximating costs and second marginals given by Dirac deltas in $\mathtt p_{\R^d}(P_{h^++\eps})$. These will converge to a Borel section $\tilde{\mathtt d}^{h^-}_{+,\e}$ of the direction vector field $\mathtt d_\e^{h^-}$ which by construction satisfies the cone approximation property.
\end{proof}

\begin{remark}
\label{R_part_case_Th7}
The above theorem can also be proved as a particular case of the analysis done in Section \ref{S_cone_approx_folia}: in this case we have a single cone-Lipschitz graph, and the uniqueness role of the linear order is trivial.
\end{remark}

% 
% {\bf Manca il controesempio se $\nu$ non a.c., allora non vale l'opposto}
% 

\section{\texorpdfstring{From $\C^k$-fibrations to linearly ordered $\C^k$-Lipschitz foliations}{From C-fibrations to linearly ordered C-Lipschitz foliations}}
\label{S_cfibr_cfol}

This section is devoted to the proof of Theorem \ref{T_cfibrcfol} stated below, that will be the building block for proving of Theorem \ref{T_subpart_step}. %and Corollary \ref{C_subpart_final}.

Let $\C^k : \R^{d-k} \times \R^k \supset \A \times \R^k \to \mathcal C(k,\R^k)$ be the $\sigma$-compact direction map of a $k$-dimensional fibration $\tilde{\mathbf D}^k$ and $\mathtt c_{\C^k}$ be the associated cost function \eqref{E_cost_fibr}. Let
\[
\tilde \mu = \int \tilde \mu_\a d\tilde m(\a), \qquad \tilde \nu = \int \tilde \nu_\a d\tilde m(\a) 
\]
be probability measures on $\R^d$ %, with $\tilde \mu$ concentrated on the base of the fibration
such that
\begin{equation}
\label{cfibrcfol1}
\Pi^f_{\mathtt c_{\C^k}}(\tilde\mu,\tilde\nu)\neq\emptyset
\end{equation}
and 
\begin{equation}
\label{cfibrcfol2}
\tilde \mu_\a \ll \mathcal H^k \llcorner_{\tilde Z\ka} \quad \text{ for $\tilde m$-a.e. $\a \in \A$.}
\end{equation}

Recall Definition \ref{D_compatible} of $(\mathtt c,\mu,\nu)$-compatible preorder, Definition \ref{D_pimunuconn} of $\Pi^f_{\mathtt c}(\mu,\nu)$-cyclically connected partition and let $\{0,1\}^\N$ be the Polish space of sequences in $\{0,1\}$ endowed with the product topology.

Our main result is the following theorem. Recall that $\omega$ is the first countable ordinal.

\begin{theorem}
\label{T_cfibrcfol} 
If \eqref{cfibrcfol1} and \eqref{cfibrcfol2} hold, then there exists a $(\mathtt c_{\C^k},\tilde \mu,\tilde \nu)$-compatible linear preorder $\bar \preccurlyeq$ with Borel graph on $\A \times \R^k$
\begin{equation}
\label{E_bar_theta_th}
\bar \preccurlyeq = (\bar \theta \times \bar \theta)^{-1}(\trianglelefteq_\omega), \qquad \bar \theta : \A \times \R^{k} \to \A \times \{0,1\}^\N, \ \trianglelefteq_\omega \ \text{linear order,}
\end{equation}
and equivalence classes $\{\bar\theta^{-1}(\a,\t)\}_{\nfrac{\a \in \A}{\t \in \{0,1\}^\N}}$ such that the subcollection of sets $\{\bar Z^k_{\a,\t}\}_{\nfrac{\a \in \A}{\t\in\T^k(\a)}}$ defined by
\begin{equation}
\label{E_subcollection}
\bar Z^k_{\a,\t} = \interr\,\bar\theta^{-1}(\a,\t) \quad \text{and} \quad \tilde\mu_\a(\bar\theta^{-1}(\a,\t))>0
\end{equation}
is $\Pi^f_{\mathtt c_{\C^k}}(\tilde \mu,\tilde \nu)$-cyclically connected.
\end{theorem}

As noticed in Proposition \ref{P_ex_fol}, the equivalence classes of a $\mathtt c_{\C^k}$-compatible linear preorder on $\A\times\R^k$ with $\sigma$-compact graph form a $\mathtt c_{\C^k}$-Lipschitz foliation.
Then, by definition of $(\mathtt c_{\C^k},\tilde \mu,\tilde \nu)$-compatible linear preorder and by disintegration of measures, Theorem \ref{T_cfibrcfol} claims that we can reduce the optimal transportation problem on a $\C^k$-fibration to a family of optimal transportation problems on the level sets of a $\mathtt c_{\C^k}$-Lipschitz foliation, whose $k$-dimensional classes of positive $\tilde\mu_\a$ measure (see the characterization of $\mathtt c_{\C^k}$-Lipschitz foliations given in Proposition \ref{P_fol_char}) satisfy the cyclically connectedness property w.r.t. $\Pi^f_{\mathtt c_{\C^k}}(\tilde\mu,\tilde\nu)$.  

As noticed in the Proposition \ref{P_ex_fol}, since
\[
\mathtt c_{\C^k}(\a,w,\a',w') < +\infty \quad \Rightarrow \quad \a = \a',
\]
the equivalence classes of the preorder $\bar \preccurlyeq$ constructed in Theorem \ref{T_cfibrcfol} will be contained in sections $\{\a\}\times\R^k$. At a first reading, the geometry which lies behind the construction of $\bar\preccurlyeq$ will be clear to the reader even assuming that $\A=\{\a_0\}$ for some point $\a_0$, and thus %or, in other words, assuming 
$\mathtt c_{\C^k}$ is equal to a single convex cone cost. The variable $\a\in\A$ plays in fact the role of a parameter on which the maps used to define the preorder have to depend in a suitably measurable way.

As a preliminary, let us define the sets of $\sigma$-compact \emph{carriages} as follows: for $\tilde \pi \in \Pi^f_{\mathtt c_{\C^k}}(\tilde \mu,\tilde \nu)$ set
\begin{equation}
\label{E_varGamma_def_pi}
\varGamma(\tilde \pi) := \Big\{ \tilde \Gamma \subset (\A\times\R^k) \times (\A\times\R^k) : \tilde \Gamma \subset \{{\mathtt c_{\C^k}} < \infty\},\text{ $\tilde\Gamma$ $\sigma$-compact, }\tilde \pi(\tilde \Gamma) = 1 \Big\}, %\quad\forall\tilde\pi\in\Pi^f_{\mathtt c_{\C^k}}(\tilde\mu,\tilde\nu),
\end{equation}
and define
\begin{equation}
\label{E_varGamma_def}
\varGamma := \bigcup_{\tilde \pi \in \Pi^f_{\mathtt c_{\C^k}}(\tilde \mu,\tilde \nu)} \varGamma(\tilde \pi).
\end{equation}
The section of a carriage $\tilde \Gamma(\a,\a)$ will be also denoted as $\tilde\Gamma(\a)\subset\R^k\times\R^k$.

\subsection{\texorpdfstring{Construction of a $(\mathtt c_{\C^k},\tilde\mu,\tilde\nu)$-compatible linear preorder}{Construction of a (c,m,n)-compatible linear preorder}}
\label{Ss_gamma_order}

 %We denote by $\{0,1\}^\N$ the Polish space of sequence in $\{0,1\}$ endowed with the product topology.
The main result of this section, which is the first step of the proof of Theorem \ref{T_cfibrcfol}, is the following theorem.
\begin{theorem}
\label{T_order_gamma}
For any $\tilde \Gamma \in \varGamma$ there exists a $(\mathtt c_{\C^k},\tilde \mu,\tilde \nu)$-compatible linear preorder $\preccurlyeq_{\tilde\Gamma,\mathtt W^{\tg}}$ with Borel graph
\begin{equation}
\label{E_bar_theta_th_2}
\preccurlyeq_{\tilde\Gamma,\mathtt W^{\tg}} = \big( \theta_{\tilde\Gamma,\mathtt W^{\tg}} \times \theta_{\tilde\Gamma,\mathtt W^{\tg}} \big)^{-1}(\trianglelefteq_\omega), \qquad \theta_{\tilde\Gamma,\mathtt W^{\tg}} : \A \times \R^{d-k} \to \A \times \{0,1\}^\N, \ \trianglelefteq_\omega \ \text{linear order,}
\end{equation}
whose equivalence classes $\{\theta^{-1}_{\tg,\mathtt W^{\tg}}(\a,\t)\}_{\nfrac{\a \in \A}{\t\in\{0,1\}^\N}}$ satisfy
\begin{equation}
\label{E_leb_open}
\mathrm{Leb} \Big( \mathtt p_1(\tg(\a)) \cap \{\theta_{\tg,\mathtt W^{\tg}}(\a,\cdot)=(\a,\t)\} \Big) \quad \text{ is $(\tg(\a),\mathtt c_{\C^k(\a)})$-cyclically connected}.
\end{equation}
\end{theorem}

By the characterization of $\mathtt c_{\C^k}$-Lipschitz foliations given in Proposition \ref{P_fol_char}, \eqref{E_leb_open} must refer to $k$-dimensional equivalence classes. Moreover, by Remark \ref{rem_pre},  $\preccurlyeq _{\tilde\Gamma\cap\cup\theta^{-1}(\a,\t)\times\theta^{-1}(\a,\t)}\subset\preccurlyeq_{\tg,\mathtt W^{\tg}}$ and then $\forall\,x,y\in \mathrm{Leb} \big( \mathtt p_1(\tg(\a)) \cap \{\theta_{\tg,\mathtt W^{\tg}}(\a,\cdot)=(\a,\t)\} \big)$ the $(\tilde\Gamma(\a),\mathtt c_{\C^k(\a)})$-cycle connecting $x$ to $y$ must be contained in $\theta^{-1}(\a,\t)\times\theta^{-1}(\a,\t)$. 

The first step to prove Theorem \ref{T_order_gamma} is to select an $\tilde m$-conegligible set $\tilde{\A}' \subset \A$ and a $\sigma$-compact subset of
\[
\mathtt p_1 \tg \cap \tilde{\A}' \times \R^k %= \bigcup_{\a \in \A'} \tilde Z^k_\a, %\qquad \text{where } \tilde{\mathbf Z} = \mathtt p_{\R^d} \tilde{\mathbf D} = \cup_\a \tilde Z\ka \text{ is the base of the fibration,}
\]
with $\a$-sections countable and dense in $\mathtt p_1\tg(\a)$.

\begin{lemma}
\label{L_z_n_dense_selections}
There exist an $\tilde m$-conegligible $\sigma$-compact set $\tilde{\mathfrak A}' \subset \A\subset\R^{d-k}$ and a countable family $\mathtt W^{\tg}$ of $\sigma$-continuous functions $\mathtt w_n^{\tg}: \tilde{\mathfrak A}' \to \R^k$, $n \in \N$, such that for all $\a \in \tilde{\A}'$
\begin{equation}
\label{E_z_n_dense_sel}
\big\{ \mathtt w_n^{\tg}(\mathfrak a) \big\}_{n \in \N} \subset \mathtt p_1 \tilde \Gamma(\mathfrak a) \subset %\qquad \text{and} \qquad 
\clos\,\{\mathtt w_n^{\tg}(\mathfrak a)\}_{n \in \N}. % \supset \mathtt p_1 \tilde \Gamma(\mathfrak a). %,\quad\forall n \in \N, \mathfrak a \in \tilde{\mathfrak A}'.
\end{equation}
\end{lemma}

\begin{proof}%[Proof of Lemma \ref{L_z_n_dense_selections}]
For shortness we use the notation %define
\begin{equation*}
% \label{E_first_component_tilde_Gamma}
\Lambda := \mathtt p_1 \tilde \Gamma = \Big\{ (\mathfrak a,w) : \exists\, \mathfrak a', w' \text{ s.t. } (\mathfrak a,w,\mathfrak a',w') \in \tilde \Gamma \Big\} \subset \R^{d-k} \times \R^k.
\end{equation*}

{\it Step 1.} Let $\mathfrak Q := \mathtt p_\mathfrak a(\Lambda) \subset \R^{d-k}$ and fix $\eps > 0$. By standard selection theorems (for example, Theorem 5.2.1 of \cite{Sri:courseborel} is sufficient in this setting), there exists $\mathtt w^\eps_0 : \mathfrak Q \mapsto \R^k$ Borel such that $\Graph\,\mathtt w^\eps_0 \subset \Lambda$. By Lusin Theorem (134Yd of \cite{MR2462519}) we obtain an $\tilde m$-conegligible set $\mathfrak Q^\eps_0$ such that $\mathtt w^\eps_0 \llcorner_{\mathfrak Q^\eps_0}$ is $\sigma$-continuous.

Define
\[
\Lambda^\eps_0 := \Lambda \cap (\mathtt p_\mathfrak a)^{-1}(\mathfrak Q^\eps_0), \quad (\Lambda^\eps_1)' := \Lambda^\eps_0 \setminus \big\{ (\mathfrak a,w) : \big| w - \mathtt w^\eps_0(\mathfrak a) \big| < \eps \big\}.
\]
These are clearly Borel sets.

Let $(\mathfrak Q^\eps_1)' := \mathtt p_\mathfrak a((\Lambda^\eps_1)')$, and define $(\mathfrak Q^\eps_0)'' \subset \mathfrak Q^\eps_0 \setminus (\mathfrak Q^\eps_1)'$ as a $\sigma$-compact set with the same $\tilde m$-measure of $\mathfrak Q^\eps_0 \setminus (\mathfrak Q^\eps_1)'$.

{\it Step 2.} If the Borel set $(\Lambda^\eps_n)' \subset \R^{d-k} \times \R^k$ and Souslin set $(\mathfrak Q^\eps_n)' := \mathtt p_\mathfrak a((\Lambda^\eps_n)') \subset \tilde{\A}'$ are given, let $\mathtt w^\eps_n : (\mathfrak Q^\eps_n)' \to \R^{k}$ be a $\varTheta$-measurable selection s.t. $\Graph\,\mathtt w^\eps_n \subset (\Lambda^\eps_n)' $, where $\varTheta$ is the $\sigma$-algebra generated by Souslin sets: its existence is guaranteed by Theorem 5.5.2 of \cite{Sri:courseborel}. As in Step 1, find an $\tilde m$-conegligible set $\mathfrak Q^\eps_n \subset (\mathfrak Q^\eps_n)'$ such that $\mathtt w^\eps_n \llcorner_{\mathfrak Q^\eps_n}$ is $\sigma$-continuous. %Moreover, let $(\mathfrak Q^\eps{n-1})''$ be a $m \llcorner_{\mathfrak Q^\eps_{n-1}}$-conegligible $\sigma$-compact subset of $\mathfrak Q^\eps_{n-1}$ 

Define the Borel sets
\[
\Lambda^\eps_{n+1} := (\Lambda^\eps_n)' \cap (\mathtt p_\mathfrak a)^{-1}(\mathfrak Q^\eps_n), \quad (\Lambda^\eps_{n+1})' := \Lambda^\eps_{n+1} \setminus \big\{ (\mathfrak a,w) : \big| w - \mathtt w^\eps_n(\mathfrak a) \big| < \eps \big\}.
\]
If $(\mathfrak Q^\eps_{n+1})' := \mathtt p_\mathfrak a((\Lambda^\eps_{n+1})')$, let $(\mathfrak Q^\eps_n)'' \subset \mathfrak Q^\eps_n \setminus (\mathfrak Q^\eps_{n+1})'$ be a $\sigma$-compact set with the same measure of $\mathfrak Q^\eps_n \setminus (\mathfrak Q^\eps_{n+1})'$.

Extend also the $\sigma$-compact function $\mathtt w^\eps_n$ to an $\tilde m$-conegligible set by
%Sara
\[
\mathtt w^\eps_n(\mathfrak a) :=
\begin{cases}
\mathtt w^\eps_n(\mathfrak a) & \mathfrak a \in \mathfrak Q^\eps_n, \crcr
\mathtt w^\eps_{m}(\mathfrak a) & \mathfrak a \in (\mathfrak Q^\eps_m)'',\quad m=0,\dots, n-1.
\end{cases}
\]

{\it Step 3.} By repeating the above procedure countably many times, we obtain a countable family of $\sigma$-continuous functions $\mathtt w^\eps_n : \overset{\infty}{\underset{m=0}{\cup}} (\mathfrak Q^\eps_m)'' \to \R^k$, $n \in \N_0$, such that
\[
\Lambda \cap (\mathtt p_\mathfrak a)^{-1} \bigg( \bigcup_{m=0}^\infty (\mathfrak Q^\eps_m)'' \bigg) \subset \Big\{ (\mathfrak a,w) : \dist \big( w, \{ \mathtt w^\eps_n(\mathfrak a)\}_{n \in \N_0} \big\} \big) < \eps \Big\}.
\]
Taking a countable sequence $\eps_i \searrow 0$ as $i \to \infty$, the functions $\{\mathtt w^{\eps_i}_n\}_{i,n \in \N_0}$ satisfy clearly the statement when restricted to an $\tilde m$-conegligible $\sigma$-compact subset $\tilde{\mathfrak A}'$ of $\underset{i \in \N}{\cap} \underset{n \in \N_0}{\cup} (\mathfrak Q^{\eps_i}_n)''$.
\end{proof}

Now we associate to each $(\a,\mathtt w_n^{\tg}(\a))$ the subset of $\{\a\}\times\R^k$ of all the points $(\a,w)$ s.t. $\exists$ an axial path of finite $\mathtt c_{\C^k}$-cost in $\tilde\Gamma$ going from $(\a,\mathtt w_n^{\tilde{\Gamma}}(\a))$ to $(\a,w)$ (see Definition \ref{D_axpath}).

\noindent Define 
\begin{equation}
\label{E_h_n}
H_{\tg,n} := \bigg\{ (\a,w) : \exists\, (\bar w,\bar w') \in \tg \ \text{s.t.} \ \mathtt c_{\tilde{\mathbf C}^k(\a)}(w,\bar w') < \infty \text{ and } (\a,\bar w) \preccurlyeq_{(\tg,\mathtt c_{\C^k})} (\a,\mathtt w^{\tg}_n(\a)) \bigg\},
\end{equation}
where $\preccurlyeq_{(\tg,\mathtt c_{\C^k})}$ is the $(\tg, \mathtt c_{\C^k})$-axial preorder relation defined in \eqref{E_axpreorder}. Notice that 
\begin{equation}
\label{E_ordergamma}
 H_{\tg,n} \cap \p_1\tilde\Gamma=\bigg\{(\a,w):\,(\a,w)\preccurlyeq_{(\tg,\mathtt c_{\C^k})}(\a,\mathtt w^{\tg}_n(\a))\bigg\}.
\end{equation}
Observe that, despite the notation, $H_{\tg,n}$ does not depend only on $\tilde\Gamma$ but also on the sections $\{\mathtt w_n^{\tg}\}_{n\in\N}$ selected in Lemma \ref{L_z_n_dense_selections}. 

\begin{figure}
\centering{\resizebox{12cm}{8cm}{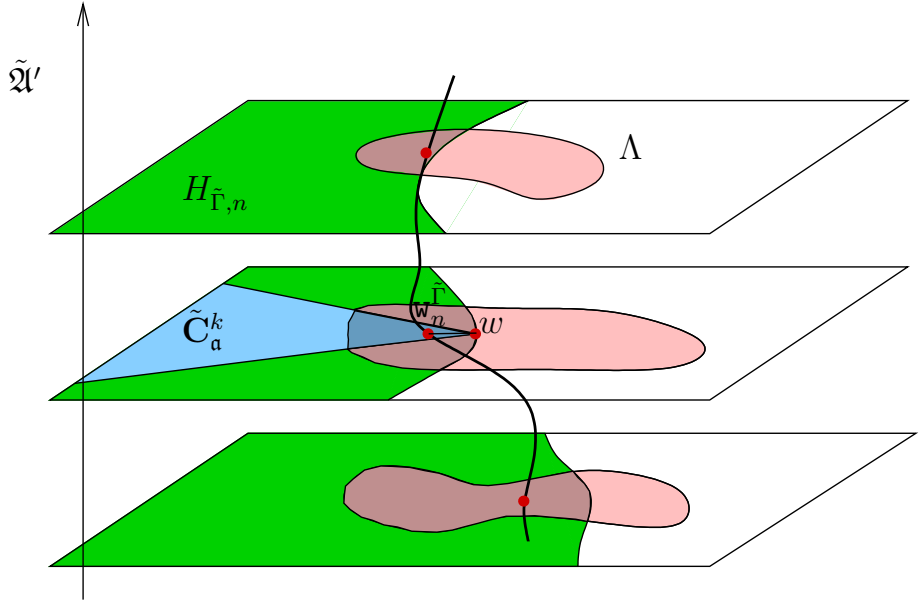}}
\caption{The construction of the set $H_{\tg,n}$.}
\label{Fi_fibration}
\end{figure}

\begin{proposition}
\label{P_Hn_sigma_cpt_compatibility}
The set $H_{\tg,n}$ is $\sigma$-compact in $\tilde{\A}' \times \R^k$ and the set $\tilde{\mathfrak A}:=\bigl\{\a\in\tilde{\A}':\,H_{\tg,n}(\a)\neq\emptyset\bigr\}$ is Borel. Moreover
\begin{equation}
\label{E_conecomp_hn}
w'\in H_{\tg,n}(\a) \quad \Longrightarrow \quad \bigl\{\mathtt c_{\C^k(\a)}(\cdot,w')<+\infty\bigr\}\subset H_{\tg,n}(\a).
\end{equation}
\end{proposition}

\begin{proof}%[Proof of Proposition \ref{P_Hn_sigma_cpt_compatibility}: ]
We prove the $\sigma$-compactness of $H_{\tg,n}$, since \eqref{E_conecomp_hn} is clear from the definition \eqref{E_h_n}. Observing that
\[
\begin{split}
H_{\tg,n} = \Big\{ (\mathfrak a,w) :&~ \exists\, I \in \N, \big\{ (w_i,w'_i) \big\}_{i = 1}^I \subset \tilde \Gamma(\mathfrak a) \text{ s.t. } \crcr
&~ w_1 = \mathtt w_n^{\tg}(\mathfrak a) \ \text{and} \ {\mathtt c}_{\C^k(\a)}(w_{i+1},w'_{i}),\,{\mathtt c}_{\C^k(\a)}(w,w'_I) < \infty \Big\}, \crcr
% =&~ \bigcup_{I \in \N} H^I_{\tg,n}
\end{split}
\]
write
\[
H_{\tg,n} = \bigcup_{I \in \N} H_{\tg,n}^I \qquad \text{where} \qquad H_{\tg,n}^I = \mathtt p_{(\a_{I+1},w_{I+1})}(\tilde H^I_{\tg,n})
\]
% \begin{equation*}
% \begin{split}
% H^I_{\tg,n} :=&~ \Big\{ (\mathfrak a,w)\in\tilde{\A}'\times\R^k : \exists \big\{ (w_i,w'_i) \big\}_{i = 1}^I \subset \tilde \Gamma(\mathfrak a) \text{ with } w_1 = \mathtt w_n^{\tg}(\mathfrak a) \ \wedge \ {\mathtt c}_{\C^k(\a)}(w_{i+1},w'_i),\,{\mathtt c}_{\C^k(\a)}(w,w'_I) < \infty \Big\} \\
% =&\mathtt p_{(\a_{I+1},w_{I+1})}(\tilde H^I_{\tg,n})
% \end{split}
% \end{equation*}
and $\tilde H^I_{\tg,n} \subset (\R^{d-k} \times \R^k)^{2I+1}$ is given by
\begin{align*}
% \label{E_tilde_H_n}
\tilde H^I_{\tg,n} :
=&~ \big\{ w_1 = \mathtt w_n^{\tg}(\mathfrak a_1) \big\} \cap \bigg[ \bigcap_{i=1}^I \big\{ (\mathfrak a_i,w_i,\mathfrak a'_i,w'_i) \in \tilde \Gamma \big\} \bigg] \cap \bigg[ \bigcap_{i=1}^{I} \big\{ {\mathtt c}_{\C^k}(\mathfrak a_{i+1},w_{i+1},\mathfrak a'_{i},w'_{i}) < \infty \big\} \bigg].
\end{align*}
Since $\mathfrak a \mapsto \C^k({\mathfrak a})$ is $\sigma$-continuous, it follows that the set $\{{\mathtt c}_{\C^k} < \infty\}$ is $\sigma$-compact in $(\R^{d-k} \times \R^k) \times (\R^{d-k} \times \R^k)$. Hence, being $\mathtt w_n^{\tg}$ $\sigma$-continuous and $\tilde \Gamma$ $\sigma$-compact, the set $\tilde H^I_{\tg,n}$ is $\sigma$-compact, thus also $H^I_{\tg,n}$, and finally $H_{\tg,n}$ too.
\end{proof}

We are now ready to define the Borel linear preorder of Theorem \ref{T_order_gamma}. If $\mathtt W^{\tg} = \{\mathtt w_n^{\tg}\}_n\in\N$ is the countable family of sections constructed in Lemma \ref{L_z_n_dense_selections}, define the function
\begin{equation*}
% \label{E_theta_Z_definition}
\begin{array}{ccccc}
\theta_{\tg,\mathtt W^{\tg}} &:& \tilde{\mathfrak A}' \times \R^k &\to& \tilde{\mathfrak A}' \times \{0,1\}^\N \crcr
&& (\mathfrak a,w) &\mapsto& \theta_{\tg,\mathtt W^{\tg}}(\mathfrak a,w) := \big( \mathfrak a,\{\chi_{\R^k \setminus H_{\tg,n}(\a)}(w)\}_{n \in \N} \big)
\end{array}
\end{equation*}
Since each component $\mathtt p_i\circ\theta_{\tg,\mathtt W^{\tg}}$ of $\theta_{\tg,\mathtt W^{\tg}}$ is Borel, also $\theta_{\tg,\mathtt W^{\tg}}$ is Borel in the product topology.

On the space $\R^{d-k} \times \{0,1\}^\alpha$, $\alpha$ ordinal number, let us consider the natural linear order given by the lexicographic order. Namely, for $\mathfrak a = (\mathfrak a_1,\dots,\mathfrak a_{d-k}) \in \R^{d-k}$ set
\begin{equation}
\label{E_lexico_on_R_d_k}
\mathfrak a <_{\mathrm{lexi}_{\R^{d-k}}} \mathfrak a' \quad \Longleftrightarrow \quad \exists \,i \in \{1,\dots,d-k\} \text{ s.t. } \forall j < i, \, \mathfrak a_j = \mathfrak a'_j \ \text{and} \ \mathfrak a_i < \mathfrak a'_i,
\end{equation}
and define
\begin{equation}
\label{E_lexico_on_Q_01}
\begin{split}
\big( \mathfrak a,\{s_\beta\}_{\beta < \alpha} \big) \vartriangleleft_\alpha \big( \mathfrak a',\{s'_\beta\}_{\beta < \alpha} \big) \ \Longleftrightarrow &\text{ either } \mathfrak a <_{\mathrm{lexi}_{\R^{d-k}}} \mathfrak a' \text{ or } \crcr
& \ \mathfrak a = \mathfrak a' \ \text{and} \ \exists \bar \beta < \alpha:\, s_\beta = s_\beta' \ \forall\, \beta < \bar \beta,\quad  s_{\bar \beta} = 0 \text{ and } s'_{\bar \beta} = 1.
\end{split}
\end{equation}
Let then  $\trianglelefteq_\omega$ be the lexicographic linear order on $\R^{d-k}\times\{0,1\}^\N$ and define the linear preorder on $\tilde{\A}'\times\R^k$ as
\begin{equation}
\label{E_order_on_fibration_by_Z}
\preccurlyeq_{\tg,\mathtt W^{\tg}} := \big( \theta_{\tg,\mathtt W^{\tg}} \otimes \theta_{\tg,\mathtt W^{\tg}} \big)^{-1}(\trianglelefteq_\omega).
\end{equation}
The induced equivalence relation on $\tilde{\mathfrak A}' \times \R^k$ is given by
\begin{equation*}
\preccurlyeq_{\tg,\mathtt W^{\tg}} \cap \preccurlyeq^{-1}_{\tg,\mathtt W^{\tg}} = \bigl\{ \theta^{-1}_{\tg,\mathtt W^{\tg}}(\a,\t) \bigr\}_{(\a,\t) \in \tilde{\A}' \times \{0,1\}^\N}.
\end{equation*}

\begin{figure}
\centering{\resizebox{14cm}{11cm}{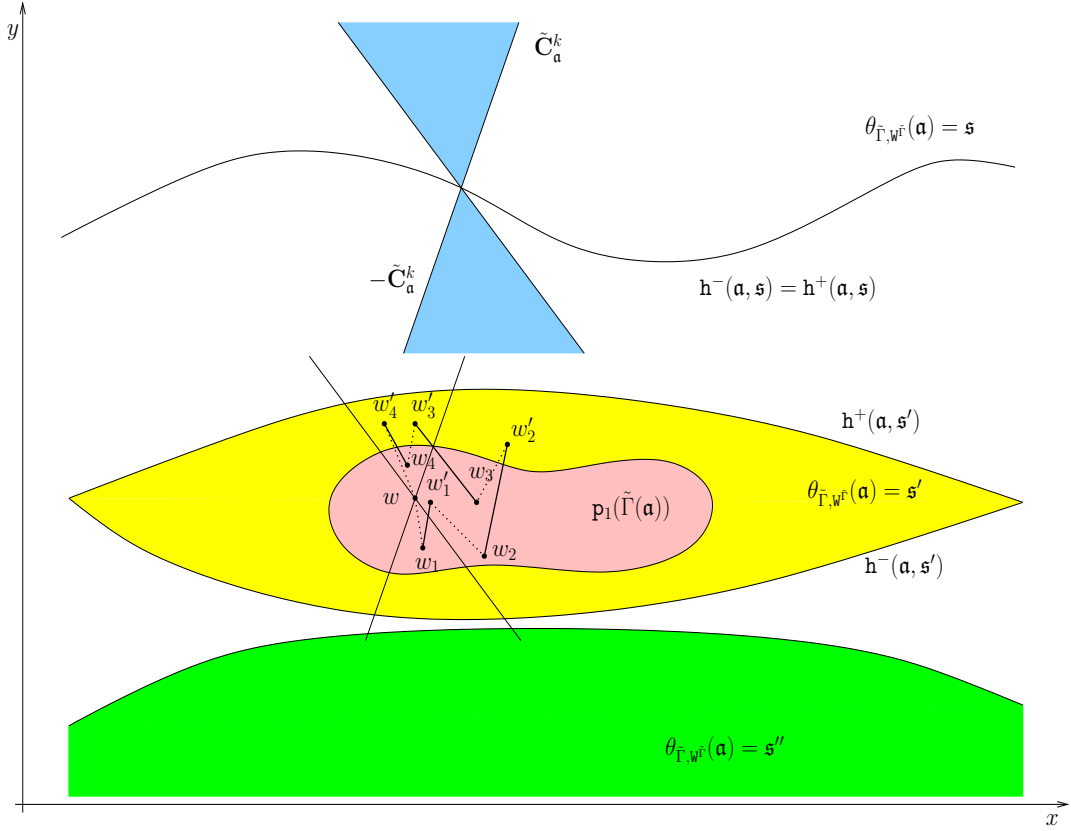}}
\caption{The function $\theta_{\tg,\mathtt W^{\tg}}$ and the cyclical connectedness of the Lebesgue points of $\mathtt p_1(\tilde \Gamma(\a))$.}
\label{Fi_thetafunct}
\end{figure}

\begin{proof}[Proof of Theorem \ref{T_order_gamma}]
The proof is given in two steps.

{\it Step 1.} In this step we prove that the relation $\preccurlyeq_{\tg,\mathtt W^{\tg}}$ defined in \eqref{E_order_on_fibration_by_Z} is a $(\mathtt c_{\C^k},\tilde \mu,\tilde\nu)$-compatible linear preorder with Borel graph. By Remark \ref{rem_pre}, this amounts to prove that $\preccurlyeq_{\tg,\mathtt W^{\tg}}$ is Borel, $\mathtt c_{\C^k}$-compatible and \eqref{E_tA21} holds for the carriage $\tilde\Gamma$.

First of all, $\preccurlyeq_{\tg,\mathtt W^{\tg}}$ is Borel because it is the preimage under $\theta_{\tg,\mathtt W^{\tg}}$, which is a Borel map, of the lexicographic order $\trianglelefteq_\omega$.

Moreover
\begin{equation*}
% \label{E_cone_cpm}
\begin{split}
\mathtt c_{\C^k(\a)}(w,w') < +\infty \quad \overset{\eqref{E_conecomp_hn}}{\Longrightarrow} \quad &\text{if $w\in \R^k \setminus H_{\tg,n}(\a)$, then $w' \in \R^k \setminus H_{\tg,n}(\a)$} \crcr
\quad \Longrightarrow \, \quad &\theta_{\tg,\mathtt W^{\tg}}(\a,w) \trianglelefteq_\omega \theta_{\tg,\mathtt W^{\tg}}(\a,w'),
\end{split}
\end{equation*}
i.e. $\preccurlyeq_{\tg,\mathtt W^{\tg}}$ is $\mathtt c_{\C^k}$-compatible. Formula \eqref{E_tA21} follows directly from \eqref{E_ordergamma}.

{\it Step 2.} Now we prove \eqref{E_leb_open}. Let
\[
w,w' \in \mathrm{Leb} \Big( \mathtt p_1(\tilde \Gamma(\mathfrak a)) \cap \big\{ \theta_{\tg,\mathtt W^{\tg}}(\mathfrak a,\cdot) = (\a, \t) \big\} \Big).
\]
Since $\preccurlyeq_{\tg,\mathtt W^{\tg}}$ is $\mathtt c_{\C^k}$-compatible, by Proposition \ref{P_ex_fol} its equivalence classes form a $\mathtt c_{\C^k}$-Lipschitz foliation and then, from Point \eqref{Prop_3_fol_char} of Proposition \ref{P_fol_char}, there exists $r>0$ such that
\[
B^k(w,r), B^k(w',r) \subset \big\{ \theta_{\tg,\mathtt W^{\tg}}(\mathfrak a,\cdot) = (\a,\t) \big\}.
\]
Hence, by the density of $\{\mathtt w_n^{\tg}(\a)\}_{n\in\N}$ stated in Lemma \ref{L_z_n_dense_selections}, there exist $\mathtt w_{\bar n}^{\tg}$, $\mathtt w_{\bar n'}^{\tg}$ such that
\begin{equation}
\label{E_cycle_contrct}
\theta_{\tg,\mathtt W^{\tg}}(\mathfrak a,\mathtt w_{\bar n}^{\tg}(\mathfrak a)) = \theta_{\tg,\mathtt W^{\tg}}(\mathfrak a,\mathtt w_{\bar n'}^{\tg}(\mathfrak a)) = (\a, \t) \qquad \text{and} \qquad {\mathtt c}_{\C^k(\mathfrak a)}(w,\mathtt w_{\bar n}^{\tg}(\mathfrak a)), {\mathtt c}_{\C^k(\mathfrak a)}(\mathtt w^{\tg}_{\bar n'}(\mathfrak a),w') < \infty.
\end{equation}
The first condition in \eqref{E_cycle_contrct} implies that
\[
\mathtt w_{\bar n}^{\tg}(\mathfrak a) \in H_{\tg,n'}(\mathfrak a), \qquad  \text{i.e.} \qquad \mathtt w_{\bar n}^{\tg}(\mathfrak a) \preccurlyeq_{(\tg,\mathtt c_{\C^k(\a)})} \mathtt w_{\bar n'}^{\tg}(\mathfrak a).
\]
The second condition implies that 
\[
w \preccurlyeq_{(\tg,\mathtt c_{\C^k})} \mathtt w_{\bar n}^{\tg}(\mathfrak a) \qquad \text{and} \qquad \mathtt w_{\bar n'}^{\tg}(\mathfrak a)\preccurlyeq_{(\tg,\mathtt c_{\C^k(\a)})}w'.
\]
Hence, composing the three axial paths, $w\preccurlyeq w$, and exchanging their roles we obtain a $(\tg,\mathtt c_{\C^k(\a)})$-cycle, thus concluding the proof.
\end{proof}
Notice that actually the subset of a $k$-dimensional class $\{\theta_{\tg,\mathtt W^{\tg}}(\a,\cdot)=(\a,\t)\}$ which is contained in the $(\Gamma,\mathtt c_{\C^k(\a)})$-cycle above is
\[
% \R^k\setminus
\begin{split}
\Bigl\{ w \in \mathtt p_1\tg(\a) \cap \{\theta_{\tg,\mathtt W^{\tg}}(\a,\cdot)=(\a,\t)\} :&~ \big( w - \interr \C^k(\a) \big) \cap \mathtt p_1\tg(\a) \cap \{\theta_{\tg,\mathtt W^{\tg}}(\a,\cdot)=(\a,\t)\}\neq \emptyset \crcr
&~ \text{and} \ \big( w + \interr \C^k(\a) \big) \cap \mathtt p_1\tg(\a) \cap \{\theta_{\tg,\mathtt W^{\tg}}(\a,\cdot)=(\a,\t)\}\neq \emptyset \Bigr\}.
\end{split}
\]
In fact, by \eqref{E_z_n_dense_sel} every couple of points $w$, $w'$ in the above set satisfy \eqref{E_cycle_contrct} for some $\mathtt w_{\bar n}^{\tg}$, $\mathtt w_{\bar n'}^{\tg}$.

\subsection{\texorpdfstring{Minimal $(\mathtt c_{\C^k},\tilde\mu,\tilde\nu)$-compatible linear preorder}{Minimal (c,m,n)-compatible linear preorder}}
\label{Ss_minial_order}

Now we apply Theorem \ref{T_minimal_equival} to the class of linear preorders $\{\theta_{\tg,\mathtt W^{\tg}}\}_{\tg,\mathtt W^{\tg}}$ constructed in Theorem \ref{T_order_gamma}, in order to find a Borel $(\mathtt c_{\C^k},\tilde\mu,\tilde\nu)$-compatible linear preorder as in Theorem \ref{T_cfibrcfol}.

Recall the definition of lexicographic order $\trianglelefteq_\alpha$ on $\R^{d-k}\times\{0,1\}^\alpha$, $\alpha$ ordinal number, given in \eqref{E_lexico_on_Q_01}, and recall also the definition of closure under countable intersection of a family of equivalence relations, Definition \ref{D_clos_sigma_E} and Remark \ref{R_min_contain_E}.

\begin{proposition}
\label{P_equiv_coun}
The class of equivalence relations
\begin{align}
\label{E_equivcount}
\Big\{ \preccurlyeq_{\tg,\mathtt W^{\tg}}\cap\preccurlyeq_{\tg,\mathtt W^{\tg}}^{-1}:\, \preccurlyeq_{\tg,\mathtt W^{\tg}} \text{ as in Theorem \ref{T_order_gamma} for some $\tg\in\varGamma$} \Big\}
\end{align}
is closed under countable intersections.
\end{proposition}

%Sara{ho modificato la dimostrazione, conformemente alla nuova notazione. Mi sembra che non sia necessario fissare il piano.}
\begin{proof}
 Let $\{\tg_n\}_{n\in\N}\subset\varGamma$ and $\big\{\preccurlyeq_{\tg_n,\mathtt W^{\tg_n}}=(\theta_{\tg_n,\mathtt W^{\tg_n}}\times \theta_{\tg_n,\mathtt W^{\tg_n}})^{-1}(\trianglelefteq_\omega)\big\}_{n\in\N}$ be a countable family of $(\mathtt c_{\C^k}, \tilde\mu,\tilde\nu)$-compatible Borel linear preorders as in Theorem \ref{T_order_gamma}. Fix $\tg_i\subset \varGamma(\tilde\pi_i)$ for some $i\in\N$. By Definition \ref{D_cmunucomp} of $(\mathtt c_{\C^k}, \tilde\mu,\tilde\nu)$-compatibility,
 \[
  \tilde\pi_i\bigg(\tg_i\underset{n\in\N}{\bigcap}\preccurlyeq_{\tg_n,\mathtt W^{\tg_n}}\cap(\preccurlyeq_{\tg_n,\mathtt W^{\tg_n}})^{-1}\bigg)=1
 \]
 and by Remark \ref{rem_pre}
 \begin{equation}
  \label{E_gammaicomp}
  \preccurlyeq_{\tg_i\underset{n\in\N}{\bigcap}\preccurlyeq_{\tg_n,\mathtt W^{\tg_n}}\cap(\preccurlyeq_{\tg_n,\mathtt W^{\tg_n}})^{-1},\mathtt c_{\C^k}}\subset \underset{n\in\N}{\bigcap}\preccurlyeq_{\tg_n,\mathtt W^{\tg_n}}.
  \end{equation}
 Let then $\bar{\mathtt W}^{\tg_i}$ be a countable family of $\sigma$-compact sections of $\p_1\big(\tg_i\underset{n\in\N}{\bigcap}\preccurlyeq_{\tg_n,\mathtt W^{\tg_n}}\cap(\preccurlyeq_{\tg_n,\mathtt W^{\tg_n}})^{-1}\big)$ as in Lemma \ref{L_z_n_dense_selections}.
 Then, by \eqref{E_gammaicomp}, it follows immediately that the $(\mathtt c_{\C^k}, \tilde\mu,\tilde\nu)$-compatible Borel linear preorder $\preccurlyeq_{\tg_i\underset{n\in\N}{\cap}\preccurlyeq_{\tg_n,\mathtt W^{\tg_n}}\cap(\preccurlyeq_{\tg_n,\mathtt W^{\tg_n}})^{-1},\bar{\mathtt W}^{\tg_i}}$ constructed as in the proof of Theorem \ref{T_order_gamma} satisfies
 \[
  \preccurlyeq_{\tg_i\underset{n\in\N}{\cap}\preccurlyeq_{\tg_n,\mathtt W^{\tg_n}}\cap(\preccurlyeq_{\tg_n,\mathtt W^{\tg_n}})^{-1},\bar{\mathtt W}^{\tg_i}}\subset\underset{n\in\N}{\bigcap}\preccurlyeq_{\tg_n,\mathtt W^{\tg_n}}.
 \]
\end{proof}

We now have all the tools to prove Theorem \ref{T_cfibrcfol}.

\begin{proof}[Proof of Theorem \ref{T_cfibrcfol}]

Let $\bar \preccurlyeq\cap\bar \preccurlyeq^{-1}=\cup_{\a,\t'} \big\{ \bar\theta^{-1}(\a,\t') \big\} \times \big\{ \bar\theta^{-1}(\a,\t') \big\}$ be the minimal equivalence relation in the class \eqref{E_equivcount} w.r.t. the measure $\tilde \mu$, whose existence is guaranteed by Theorem \ref{T_minimal_equival}.

We claim that it satisfies the conclusions of Theorem \ref{T_cfibrcfol}. Thanks to Theorem \ref{T_order_gamma}, we only have to prove \eqref{E_subcollection}. 
Recalling Definitions \ref{D_pimunuconn} and \ref{D_gammaconn}, let $\tilde \pi \in \Pi^f_{\mathtt c_{\C^k}}(\tilde\mu,\tilde\nu)$, $\tg \in \varGamma(\tilde \pi)$. By Remark \ref{rem_pre}, we can consider the carriage
\begin{equation}
\label{E_check_Gamma'}
\check\Gamma := \tilde\Gamma\cap \bigcup_{\a,\t'} \{ \bar\theta^{-1}(\a,\t') \} \times \{ \bar\theta^{-1}(\a,\t') \}\subset\tilde\Gamma
\end{equation}
and prove that the subcollection $\{\bar Z^k_{\a,\t}\}_{\nfrac{\a \in \A}{\t\in\T^k(\a)}}$ of the equivalence classes of $\bar\preccurlyeq$ defined by
\begin{equation}
\label{E_subcollection2}
\bar Z^k_{\a,t} = \interr\,\bar\theta^{-1}(\a,\t) \quad \text{and} \quad \tilde\mu_\a(\bar\theta^{-1}(\a,\t))>0
\end{equation}
is $(\tilde\mu,\check\Gamma,\mathtt c_{\C^k})$-cyclically connected.

Let $\mathtt W^{\check\Gamma}$ be a countable family of $\sigma$-compact sections of $\p_1\check\Gamma$ as in Lemma \ref{L_z_n_dense_selections}. Hence, by \eqref{E_check_Gamma'}, reasoning as in the proof of Proposition \ref{P_equiv_coun}, the equivalence classes of $\preccurlyeq_{\check\Gamma,{\mathtt W}^{\check\Gamma}}$ are contained in those of $\bar \preccurlyeq$. By minimality of $\{\bar \theta^{-1}(\a,\t)\}$ and Corollary \ref{C_constant_for_minimal_equivalence}, there exists a $\tilde \mu$-conegligible $\sigma$-compact set $\check B \subset \R^{d-k} \times \R^k$ and a Borel function
\[
\mathtt s : \R^{d-k} \times \{0,1\}^{\N} \to \R^{d-k} \times \{0,1\}^\N
\]
such that
\[
\theta_{\check \Gamma,{\mathtt W^{\check \Gamma}}} = \mathtt s \circ \bar \theta \qquad \text{ on } \check B.
\] 
The set $\check B$ depends on $\theta_{\check \Gamma,{\mathtt W^{\check\Gamma}}}$. 

In particular, using this result for the equivalence classes of positive $\tilde \mu_\a$-measure of $\bar\theta$, we deduce that there exists a set
\[
\A'' \subset \Big\{ \a : \exists\, \t \ \text{s.t.}\ \tilde \mu_\a \big(\bar \theta^{-1}(\a,\t) \big) > 0 \Big\}
\]
such that
\[
m \Big( \Big\{ \a : \exists\, \t \ \text{s.t.}\ \tilde \mu_\a \big(\bar \theta^{-1}(\a,\t) \big) > 0 \Big\} \setminus \A'' \Big) = 0
\]
and for all $\a \in \A''$, for all $\t$ such that $\tilde \mu_\a(\bar \theta^{-1}(\a,\t)) > 0$, the function $\theta_{\check\Gamma,\mathtt W^{\check\Gamma}}$ is $\tilde\mu_\a$-a.e. constant on $\{\bar \theta^{-1}(\a,\t)\}$.

Hence, using the assumption \eqref{cfibrcfol2} and condition \eqref{E_leb_open} for $\preccurlyeq_{\check\Gamma,{\mathtt W}^{\check\Gamma}}$, the sets $\{\bar \theta^{-1}(\a,\t)\}$ with $\a \in \A''$ and of positive $\tilde \mu_\a$ measure are open (see Proposition \ref{P_fol_char}) and their set of Lebesgue points is of full $\tilde \mu_\a$-measure and $(\check\Gamma,\mathtt c_{\C(\a)})$-cyclically connected.

By Definition \ref{D_gammaconn}, we thus conclude that these sets are $(\tilde \mu,\tg,\mathtt c_{\C})$-cyclically connected, and then applying the same reasoning to any transference plan $\tilde \pi$, we get that the subcollection of sets $\{\bar Z^k_{\a,\t}\}_{\nfrac{\a \in \A}{\t\in\T^k(\a)}}$ defined by
\begin{equation*}
% \label{E_subcollection}
\bar Z^k_{\a,t} = \inter \bar Z^k_{\a,\t}, \quad \text{with} \ \tilde\mu_\a(\bar\theta^{-1}(\a,\t)) > 0,
\end{equation*}
is $\Pi^f_{\mathtt c_{\C^k}}(\tilde \mu,\tilde \nu)$-cyclically connected, thus concluding the proof.
\end{proof}

%Sara{ho tolto il remark}
%\begin{remark}
%\label{R_note_proof_72}
%Note that the key point in the proof of the above theorem is that the minimal equivalence relation $\{\bar\theta^{-1}(\a,\t)\}_{\nfrac{\a\in{\A}}{\t\in\{0,1\}^{\bar \alpha}}}$, $\bar\alpha < \Omega$, given by Proposition \ref{P_equiv_coun}, which a priori is minimal only for a given $\check \pi$, turns out to be a minimal relation for all $\tilde \pi$: this is a consequence of Corollary \ref{C_equiv_pi_munu}, which implies that for any other $\tilde \pi$, $\tilde \Gamma \in \varGamma(\tilde \pi)$ the equivalence classes $\{\bar\theta^{-1}(\a,\t)\}_{\a,\t}$ can be taken as subset of $\{\theta^{-1}_{\tg,\mathtt W^{\tg}}(\a,\t')\}_{\a,\t'}$ by the choice \eqref{E_check_Gamma'}.
%\end{remark}

\begin{remark}
\label{R_why_this_meth}
We observe here that, being the equivalence relation $\bar \preccurlyeq \cap (\bar \preccurlyeq)^{-1}$ constructed in Theorem \ref{T_cfibrcfol} minimal in the family \eqref{E_equivcount}, by Theorem \ref{T_minimal_equival} the equivalence classes cannot be further decomposed by equivalence relations of the form \eqref{E_equivcount}. However, the information on cyclical connectedness of the equivalence classes can be deduced only for the equivalence classes with positive $\tilde \mu_a$-measure, because of the particular choice of the family \eqref{E_equivcount} satisfying Theorem \ref{T_order_gamma}.
\end{remark}

\section{\texorpdfstring{Cone approximation property for linearly ordered $\mathtt c_{\C}$-Lipschitz foliations}{Cone approximation property for linearly ordered C-Lipschitz foliations}}
\label{S_cone_approx_folia}

In this section we prove the following result.

\begin{theorem}
\label{T_coneappr_fol}
Let $\{\theta^{-1}(\a,\t)\}_{\a,\t} \subset \mathtt P(\A \times \R^k)$, $\A \subset \R^{d-k}$, be a $\mathtt c_{\C}$-Lipschitz foliation given by the equivalence classes of a Borel ${\mathtt c_{\C}}$-compatible linear preorder $\preccurlyeq$ on $\R^{d-k}\times\R^k$ as in \eqref{E_bar_theta_th}.
Then
%$\exists\,\A'\subset\A$ $\tilde m$-conegligible s.t. 
for all $\a \in \A$, the subdifferential partition of $\mathcal R^{+}\theta(\a)$ satisfies the initial forward cone approximation property and the subdifferential partition of $\mathcal R^-\theta(\a)$ satisfies the final backward cone approximation property. 
\end{theorem}

In particular, we conclude from Theorem \ref{T_FC_no_initial} that the initial and final points $\mathcal I^+\theta(\a)$, $\mathcal E^-\theta(\a)$ are $\mathcal H^k$-negligible and, by integration w.r.t. $\mathcal H^{d-k}$ on $\A$, the sets $\mathcal I^+\theta$, $\mathcal E^-\theta$ are also $\LL$-negligible. Moreover, by Theorem \ref{T_one_d_slicing_FC}, the disintegration of $\LL$ w.r.t. the differential partition of the regular set $\mathcal R\theta$ is \emph{regular}, i.e. it satisfies \eqref{E_disint_regular} of Definition \ref{D_disint_regular}. 

Since the disintegration of $\mathcal H^k$ on equivalence classes of $\theta$ with positive $\mathcal H^k$-measure is clearly regular, we thus have the following corollary.

\begin{corollary}
\label{C_infinnegl}
If $\{Z^\ell_{\a,\b}\}_{\underset{\a\in\A,\b\in\B}{\ell=1,\dots,k}}$ is the partition of a $\mathtt c_{\C}$-Lipschitz foliation given by the equivalence classes of a Borel ${\mathtt c_{\C}}$-compatible linear preorder obtained as the union of the differential partition and of the classes of positive $\mathcal H^k$-measure, then the disintegration of the Lebesgue measure $\mathcal L^d$ restricted on the $\mathtt c_{\C}$-Lipschitz foliation
\[
\mathcal L^d \llcorner_{\mathbf Z} = \int \upsilon^\ell_{\a,\b} d\eta(\ell,\a,\b), \qquad \mathbf Z = \bigcup_{\ell,\a,\b} Z^\ell_{\a,\b},
\]
satisfies
\[
\upsilon^\ell_{\a,\b} \simeq \mathcal H^\ell \llcorner_{Z^\ell_{\a,\b}}, \quad\text{ for $\eta$-a.e. $(\ell,\a,\b)$}.
\]
\end{corollary}

\begin{remark}
\label{R_contin_quotient}
If the quotient set $\{1,\dots,k\} \times \A \times \B$ is chosen to be a countable union of sets as in \eqref{E_mathfrak_A_k_def}, then the quotient measure
\[
\eta = \sum_{\ell = 1}^k \eta^\ell,\quad\eta^\ell(\{\ell\}\times\A\times\B)=1
\]
satisfies
\[
\eta^\ell \simeq \mathcal H^{d-\ell} \llcorner_{\underset{i\in\N}{\cup} \mathfrak C^\ell_i}
\]
for some $\mathfrak C^\ell_i \subset V^{d-\ell}_i \in \mathcal A(d-\ell,\R^d)$.
\end{remark}

\begin{proof}%[Proof of Theorem \ref{T_coneappr_fol}: ]
In the following we identify $\{\a\}\times\R^k$ with $\R^k$ and omit the variable $\a$ when clear from the context. Unless explicitly stated, for the notions and notations used in the proof we refer to Section \ref{S_disintechnique}. %: in particular we will denote by $w$ the variable on $\R^k$ and by $p$ the variables on a $\ell$-dimensional subspace of $\R^k$. 

Since the proof of the initial forward cone approximation property is the same as the forward cone approximation property up to the Borel selection given by Lemma \ref{L_reg_tt_d_h_pm}, for simplicity we prove the forward cone approximation property.

The proof will be given in three steps, and we will restrict to the case of $\ell < k$, due to the structure of equivalence classes of positive $\mathcal H^k$-measure given in Proposition \ref{P_fol_char} and the existence of at most two degenerate equivalence classes (see Definition \ref{D_non_dege}), which are clearly of positive measure, as observed in Remark \ref{R_nondeg}.

{\it Step 1.} By Definitions \ref{D_coneapprpart} and \ref{D_1_dim_slice_sheaf}, we have to prove the forward cone approximation property for the $1$-dimensional slices $\mathcal D$-cylinder (see Definition \ref{D_mathcal_D_cyl}) of the superdifferential partition of $\mathcal R^{+}\theta(\a)$ ($\ell\in\{1,\dots,k-1\}$) given by %(we omit for notational convenience the subscript $\a$)
\[
\bigl\{ Z^{\ell,+}_{\b}, C^{\ell,+}_{\b} \bigr\}_{\b \in \mathfrak B_{\ell,+}(\a)}, \qquad \mathbf Z^{\ell,+} = \bigcup_{\b \in \mathfrak B_{\ell,+}(\a)} Z^{\ell,+}_{\b} \subset \R^k,
\]
with reference plane $V^\ell = \langle \e^\ell_1,\dots,\e^\ell_\ell \rangle \in \mathcal G(\ell,\R^k)$ and base rectangle $U(\{\e^\ell_i\},\mathrm h^-,\mathrm h^+)$. By \eqref{E_mathfrak_A_k_def}, the set $\B_{\ell,+}(\a)$ is a subset of $(z + (V^\ell)^\perp) \in \mathcal A(k-l,\R^k)$ for some $z \in \interr U(\{\e^\ell_i\},\mathrm h^-,\mathrm h^+)$.

Let us fix a $1$-dimensional slice of $\mathbf Z^{\ell,+}$ with reference configuration $(\e,w + h^-(w,\e)\e,w+h^+(w,\e)\e)$, $w \in \interr U(\{\e^\ell_i\},\mathrm h^-,\mathrm  h^+)$, $\e\in C(\{\e^\ell_i\})$ (see Definition \ref{D_1_sim_slice}), i.e.
\begin{equation}
\label{E_1_dim_superdif}
\mathbf Z^{\ell,+}_\e = \mathbf Z^{\ell,+} \cap \mathtt p^{-1}_{ V^\ell} \Big( w + \big( h^-(w,\e),h^+(w,\e) \big) \e \Big)
\end{equation}
with $\eps>0$ such that the set
\[
\mathbf Z^{\ell,+} \cap \mathtt p^{-1}_{V^\ell} \Big( w + \big( h^-(w,\e)-\eps,h^+(w,\e)+\eps \big) \e \Big)
\]
is still a $1$-dimensional model set. %Notice that 
% \[
% \mathtt p^{-1}_{\langle V^\ell\rangle}(\langle \e\rangle)\in\mathcal G(k-l+1,\R^k).
% \]
Let
\begin{equation}
\label{E_tt_d_e+}
\mathtt d_\e^+ = \mathcal D \cap \mathtt p^{-1}_{V^\ell}\langle \e\rangle
\end{equation}
be its direction vector field as in \eqref{E_mathtt_d_e}.

As in \eqref{E_tt_d_e_sigma}, by definition of $1$-dimensional slice and since the cones of directions are given by the directions of the superdifferential, for all
\[
z_{\b,h^-} \in P_{w+h^-(w,\e)\e} = Z^{\ell,+}_\b \cap \mathtt p^{-1}_{V^\ell}(w+ h^-(w,\e)\e), \quad\b\in\B_{\ell,+}(\a)
\]
we have
\[
\mathtt d_\e^+(z_{\b,h^-}) = \frac{\sigma^{h^-,h^+ + \eps}(z_{\b,h^-}) - z_{\b,h^-}}{\big| \sigma^{h^-,h^++\eps}(z_{\b,h^-}) - z_{\b,h^-} \big|},
\]
where
\[
\sigma^{h^-,h^++\eps}(z_{\b,h^-}) = z_{\b,h^++\eps}
\]
is the unique point of $P_{w + (h^+(w,\e)+\eps) \e}$ satisfying
\begin{equation}
\label{E_w-+}
\begin{split}
\big\{ z_{\b,h^++\eps} \big\} &= \big( z_{\b,h^-} + C^{\ell,+}_\b \big) \cap Z^{\ell,+}_{\b}\cap\mathtt p_{V^\ell}^{-1} \big( w + (h^+(w,\e) + \eps) \e \big) \\
&\overset{\eqref{Point_sudiff_is_partition_in}}{=} \partial^+ \theta(z_{\b,h^-}) \cap \mathtt p_{V^\ell}^{-1} \big( w + (h^+(w,\e) + \eps) \e \big).
\end{split}
\end{equation}

{\it Step 2.}
Let
\[
\bar\mu \simeq \mathcal H^{k-\ell} \llcorner_{P_{w + \mathtt h^-(w,\e) \e}}, \qquad \bar\mu \big( P_{w + \mathtt h^-(w,\e) \e} \big) = 1,
\]
and set $\bar\nu=\sigma^{h^-,h^{+}+\eps}_{\#}\bar\mu$. Then clearly 
\[
\bar\pi := \big( \Id\times\sigma^{h^-,h^{+}+\eps} \big)_{\#} \bar\mu \in \Pi^f_{\mathtt c_{\C(\a)}}(\bar \mu,\bar \nu).
\]

In this step we prove that \eqref{E_w-+} and the fact that $\{\theta^{-1}(\a,\t)\}_{\t \in \T}$ is induced by a Borel $\mathtt c_{\C}$-compatible linear preorder $\preccurlyeq$ imply that
\begin{equation*}
\bar \pi \ \text{\it is the unique transport plan in } \Pi^f_{\mathtt c_{\C(\a)}}(\bar\mu,\bar\nu).
\end{equation*}

First of all observe that, by transversality of $\mathtt p^{-1}_{V^\ell}\langle\e\rangle$ w.r.t. $\C(\a)$, 
\begin{equation}
\label{E_csection}
\Pi^f_{\mathtt c_{\C(\a)}}(\bar\mu,\bar\nu)=\Pi^f_{\mathtt c_{\C(\a)\cap\p_{V^\ell}^{-1}\langle \e\rangle}}(\bar\mu,\bar\nu).
\end{equation}

Then, consider the Borel linear preorder
\[
\bar\preccurlyeq =\, \preccurlyeq \cap \; \bar{\mathbf Z}^{\ell,+}_\e \times \bar{\mathbf Z}^{\ell,+}_\e, \qquad \bar{\mathbf Z}^{\ell,+}_\e := \mathbf Z^{\ell,+} \cap \mathtt p^{-1}_{ V^\ell}\Big( w + \big[ h^-(w,\e),h^+(w,\e) \big] \e \Big). %, \qquad w \in U(\{\e^\ell_i\},\mathrm h^-,\mathrm h^+)
\]
Its equivalence classes are given by the (closed) segments
\[
\bar Z^{\ell,+}_{\b,\e} = Z^{\ell,+}_\b \cap \mathtt p^{-1}_{ V^\ell} \Big( w + \big[ h^-(w,\e),h^+(w,\e) \big] \e \Big), %\qquad w \in U(\{\e^\ell_i\},\mathrm h^-,\mathrm h^+)
\]
from the transversality of $\mathtt p^{-1}_{V^\ell}(w + [ h^-(w,\e),h^+(w,\e) ] \e)$ w.r.t. $Z^{\ell,+}_{\b}$. Moreover, it is clearly $\mathtt c_{\C(\a)\cap\p_{V^\ell}^{-1}\langle \e\rangle}$-compatible and, since a carriage of $\bar\pi$ is given by 
\[
\bar\Gamma = \Graph\,\sigma^{h^-,h^++\eps} \cup \Graph\,\Id \llcorner_{\bar{\mathbf Z}^{\ell,+}_\e},
\]
then
\[
\bar \pi \bigg( \bigcup_{\b} \bar Z^{\ell,+}_{\b,\e} \times \bar Z^{\ell,+}_{\b,\e} \biggr) = 1.
\]
By Remark \ref{rem_pre} and Theorem \ref{T_A2},  we conclude that $\bar \preccurlyeq$ is
a Borel $(\mathtt c_{\C(\a)\cap\p_{V^\ell}^{-1}(\langle \e\rangle)},\bar\mu,\bar\nu)$-compatible linear preorder. In particular, by \eqref{E_csection}, any transport plan $\pi\in\Pi^f_{\mathtt c_{\C(\a)}}(\bar\mu,\bar\nu)$ must satisfy
\[
\pi \bigg( \bigcup_{\b} \bar Z^{\ell,+}_{\b,\e} \times \bar Z^{\ell,+}_{\b,\e} \biggr) = 1,
\]
and then since $\bar\mu$ and $\bar\nu$ are supported only on the sections $P_{w + h^-(w,\e)\e}$ and $P_{w + (h^+(w,\e) + \eps) \e}$, \eqref{E_w-+} implies that $\pi=\bar\pi$.

{\it Step 3.}
In this step we prove the cone approximation property for the $1$-dimensional slice $Z^{\ell,+}_{\e}$. The sequence of approximating finite union of cone vector fields (see Definition \ref{D_cone_vector}) will be given by the transport rays of transport plans which are optimal w.r.t. the secondary cost
\begin{equation}
\label{E_tt_c_C_''}
\mathtt c_{\C(\a)}''(w,w') = 
\begin{cases}
|w-w'| & \text{if $\mathtt c_{\C(\a)}(w,w') < +\infty$}, \\
+\infty & \text{otherwise},
\end{cases}
\end{equation}
and whose first marginal is $\bar\mu$ and second marginal is given by finite sums of Dirac deltas which are weakly converging to $\bar\nu$, as given by the next lemma.

\begin{lemma}
\label{L_appr_meas}
There exists a sequence
\[
\big\{ \bar \nu_n \big\}_{n \in \N} \in \mathcal P \big( P_{w + (h^+(w,\e) + \eps) \e} \big) 
\]
such that
\begin{enumerate}
\item $\bar \nu_n$ is locally finitely atomic, $\bar \nu_n = \underset{i \in \N}{\sum}\alpha_{n_i}\delta_{z'_{n_i}}$,

\item $\Pi^f_{\mathtt c_{\C(\a)}}(\bar\mu,\bar\nu_n) \neq \emptyset$,

\item $\bar \nu_n$ weakly converges to $\bar \nu$ in $\mathcal P(\R^k)$.

\end{enumerate}
\end{lemma}

Once the lemma is assumed to be valid, the proof of the theorem is concluded as follows.

Let $\bar \pi_n \in \Pi^f_{\mathtt c_{\C(\a)}}(\bar \mu,\bar \nu_n)$ be optimal for $\mathtt c''_{\C(\a)}$ and let $\bar \Gamma''_n$ be a $\mathtt c''_{\C(\a)}$-cyclically monotone carriage. A standard argument based on $|\cdot|$-cyclical monotonicity (see e.g. \cite{Car:strictly}) implies the interiors of the segments $[z,z']$, with $[z,z'] \in \bar \Gamma''_n$, do not intersect. Since $\bar \Gamma''_n$ can be decomposed as
\begin{equation}
\label{E_Gamma_n_def}
\bar \Gamma''_n = \bigcup_{i \in \N} B_i \times \{z'_{n_i}\}, \qquad B_i \subset P_{w + h^-(w,\e)\e} \quad \text{Borel and disjoint},
\end{equation}
up to a $\bar\mu$-negligible set, then the approximating sequence of finite unions of cone vector fields (see Definition \ref{D_cone_vector}) is given by 
\[
E_{\mathtt d_i} := \big\{ (1-t) z + t z'_{n_i}: t \in [0,1], z \in B_i \big\}.%, \qquad i = \{1,\dots,I(n)\}.
\]

Let
\[
\sigma^{h^-,h^++\eps}_n : P_{w + h^-(w,\e)\e} \to P_{w + h^+(w,\e)\e}, \qquad \sigma^{h^-,h^++\eps}_n(B_i) = z'_{n_i},
\]
be the Borel function such that $\bar \pi_n = (\Id\times\sigma^{h^-,h^++\eps}_n) _\# \bar \mu$, as it follows from \eqref{E_Gamma_n_def}. The uniqueness result proved in Step 2 implies that the measure $\bar \pi_n$ converges weakly to the measure $\bar \pi$, which equivalently means that
\[
\sigma^{h^-,h^++\eps}_n \to \sigma^{h^-,h^++\eps} \quad \mathcal H^{k-\ell}\text{-a.e.}.
\]
Then, the last condition for the forward cone approximation property required in Definition \ref{D_forw_back} is satisfied and the theorem is proved.
\end{proof}

We are left with the proof of Lemma \ref{L_appr_meas}.

\begin{proof}[Proof of Lemma \ref{L_appr_meas}]
W.l.o.g. assume that $\exists\,C'\in \mathcal C(k-\ell,\R^{k-\ell})$ such that, for all $z_{\b,h^-} \in P_{w + h^-(w,\e)\e}$,
\begin{equation*}
% \label{E_parti_perturb}
\big\{ \mathtt c_{\C(\a)}(z_{\b,h^-},\cdot) < +\infty \big\} \cap \mathtt p_{V^\ell}^{-1} \big( w + (h^+(w,\e) + \eps)\e \big) \supset z_{\b,h^++\eps} + C'
\end{equation*}
and
\[
C' \supset \bigcap_{i=1}^{k-\ell} \big\{ w : w \cdot \mathrm e'_i \geq 0 \big\}
\]
for some fixed system of coordinates $\{\mathrm e'_1,\dots,\mathrm e'_{k-\ell}\}\subset\R^{k-\ell}$
%: in fact, we can assume that $C^{\ell,+}_\b$ are $\ell$-dimensional cones sufficiently close so that there is a common direction $\mathrm e'_{k-\ell}$ pointing in the interior of $\C(\a)$, with a similar partition as in Proposition \ref{P_countable_partition_in_reference_directed_planes}.

Let
\[
\bar Q = \prod_{i=1}^{k-\ell} \bigg[ - \frac{1}{2}, \frac{1}{2} \bigg] \mathrm e'_j, \qquad \bar Q(z',r) : = z' + r \bar Q.
\]
For $n \in \N$, choose $r_n>0$ sufficiently small such that
\[
\bar Q \bigg( z' + \frac{1}{n} \sum_{i = 1}^{k-\ell} \mathrm e'_{i}, 2 r_n \bigg) \subset z ' + C',
\]
for all $z' \in \mathtt p_{V^\ell}^{-1} ( w + (h^+(w,\e) + \eps)\e )$ (clearly any $r_n \leq \frac{1}{n}$ suffices). Let
\[
\big\{ \bar Q(z'_{n_i},r_n) \big\}_{i\in\N}
\]
be a locally finite covering of $\mathtt p_{V^\ell}^{-1}( w + (h^+(w,\e) + \eps)\e )$. Then, define the map
\[
% \begin{array}{ccccc}
\mathtt T_n : \mathtt p_{V^\ell}^{-1} \big( w + (h^+(w,\e) + \eps) \e \big) \to \mathtt p_{V^\ell}^{-1} \big( w + (h^+(w,\e) + \eps) \e \big)
\]
by
\begin{equation}
\label{E_tt_T_n}
\mathtt T_n(z') := z_{n_i}' \qquad \text{where} \ i = \min \bigg\{ j : z'_{n_j} \in \bar Q \bigg( z' + \frac{1}{n} \sum_{i=1}^{k-\ell} \mathrm e'_{i}, 2r_n \bigg) \bigg\}.
\end{equation}

The measures
\[
\bar \nu_n := {\mathtt T_n}_\#\bar \nu, \qquad \bar \pi_n := (\Id\times\mathtt T_n \circ \sigma^{h^-,h^++\eps})_\# \bar\mu.
\]
satisfy
\[
\bar\pi_n\in\Pi^f_{\mathtt c_{\C(\a)}}(\bar\mu,\bar\nu_n) \qquad \text{and} \qquad \bar\nu_n\rightharpoonup\nu\in\mathcal P(\R^k),
\]
proving the lemma.
% Finally
% Now, choose $\bar\pi''_n\in\Pi^{\mathrm{opt}}_{\mathtt c_{\C(\a)}''}(\bar\mu,\bar\nu_n)\subset\Pi^f_{\mathtt c_{\C(\a)}}(\bar\mu,\bar\nu_n)$. It exists since $\mathtt c_{\C(\a)}''$ has $\sigma$-compact graph. Moreover, by $\mathtt c_{\C(\a)}''$-cyclical monotonicity (see \cite{}) it is concentrated on a finite union of cones with vertices at the points $\{w_{n_i}'\}_{i=1}^{I(n)}$. Moreover, the fact that $\bar\nu_n\rightharpoonup\bar\nu$ implies that 
% \[
%  \bar\pi_n''\rightharpoonup\bar\pi''\in\Pi^f_{\mathtt c_{\C(\a)}}(\bar\mu,\bar\nu)
% \]
% and, up to a $\bar\mu$-negligible (thus also $\mathcal H^{k-\ell}$-negligible) set of points in $P_{\mathtt h^-}$, the vector field
% \[
%  v''_\e(w)=\frac{w'-w}{|w'-w|}, \text{ $w'\in\mathtt p_{V^\ell}^{-1}((\mathtt p_{\langle\e\rangle}(w)+h^+(e,\e)+\eps)\e) $ s.t. $(w,w')\in\bar\pi''$}
% \]
% can be pointwise approximated by the cones on which the $\{\bar\pi_n''\}$ are concentrated.
% By Step 1, $\bar\pi''=\bar\pi$ and then $\mathtt d_\e$ satisfies the forward cone approximation property.
% 
\end{proof}

\begin{remark}
\label{R_not_gener_potential}
Since the level sets of the function $\theta$ form a $\mathtt c_{\C(\a)}$-Lipschitz foliation, then from Proposition \ref{P_fol_char} the equivalence classes of negligible $\mathcal H^k$-measure are complete $\mathtt c_{\C(\a)}$-Lipschitz graph. One could then  suspect that the proof of Theorem \ref{T_coneappr_fol} can be deduced from the proof of Theorem \ref{T_cone_graph}.

This would be the case if the disintegration of the $\mathcal H^k$-measure on the level sets of negligible $\mathcal H^k$-measure of the generating $\theta$ were absolutely continuous w.r.t. $\mathcal H^{k-1}$. Here we show that in general this is not the case. In fact, we will construct a Borel function $\Theta : [0,1]^2 \to [0,1]$ whose level sets are subsets of $C^\infty$-functions on $[0,1]$ such that there exists a Cantor set $C$ of positive $\mathcal L^2$-measure on which $\Theta$ is injective. This clearly implies that the disintegration of $\mathcal L^2 \llcorner_C$ w.r.t. $\Theta$ has conditional probabilities made of a single Dirac $\delta$-mass.

Let $\varrho : [0,1] \to [0,1]$ be a strictly increasing $C^\infty$ function such that
\[
\varrho(0) = 0, \ \varrho(1) = 1, \qquad \varrho(x) = 1 - \varrho(1-x), \qquad \frac{d^k\varrho}{dx^k}(0) = \frac{d^k\varrho}{dx^k}(1) = 0 \quad \forall k \in \N. %, \quad \bigg\| \frac{d^k\varrho}{dx^k} \bigg\|_\infty \leq 2^k k!.
\]

Consider the sequence of numbers
\[
c_k = 2^{-2k-1}, \qquad a_0 = 1, \ a_k = \frac{a_{k-1}-c_k}{2} =  3 \cdot 2^{-2-k} + 2^{-2-2k}, \quad b_k = 2^{-2 k^2}.
\]

{\it Step 1.}
If
\[
Q^1 = [0,a_1] \times [0,b_1] = \bigg[ 0,\frac{7}{16} \bigg] \times \bigg[ 0,\frac{1}{4} \bigg]
\]
and having fixed the points
\[
x_0 = 0, \ x_1 = a_1 + c_1 = \frac{9}{16}, \quad y_0 = 0, \ y_1 = b_1 + \frac{b_1}{3} = \frac{1}{3}, \ y_2 = 2 \bigg( b_1 + \frac{b_1}{3} \bigg) = \frac{2}{3},
\]
consider the squares
\[
Q^1_{0,j} = \bigg( x_0,y_j+\frac{1}{24} \bigg) + Q^1, \ Q^1_{1,j} = (x_1,y_j) + Q^1, \qquad j=0,1,2.
\]

We define the level sets of $\Theta$ outside the squares $Q^1_{i,j}$ as follows. First, in the strips
\[
[0,1] \times \bigg[ \frac{7}{24}, \frac{1}{3} \bigg], \quad [0,1] \times \bigg[ \frac{15}{24}, \frac{2}{3} \bigg], \quad [0,1] \times \bigg[ \frac{23}{24}, 1 \bigg]
\]
the level sets are horizontal segments $y = \mathrm{constant}$. In the remaining strips, the construction is completely similar so that we show only the case $[0,1] \times [0,7/24]$.

In the strip $[9/16,1] \times [1/4,7/24]$ the level lines are again $y = \mathrm{constant}$. In the remaining strip $[7/16,9/16] \times [0,7/24]$ we show how to use the function $\varrho$ to connect the points of the vertical segment $\{7/16\} \times [0,1/24]$ to the points of the vertical segment $\{9/16\} \times [0,1/4]$: the construction in the remaining part is symmetrical. In this case, define the level set of $\Theta$ as the curves
\[
\Big\{ y + 5 y \varrho \big( 8 (x - 7/16) \big), x \in [7/16,9/16] \Big\}, \quad y \in [0,1/24].
\]
Due to the regularity of $\varrho$, these curves are $C^\infty$ in the set of definition.

\begin{figure}
\centering{\resizebox{12cm}{12cm}{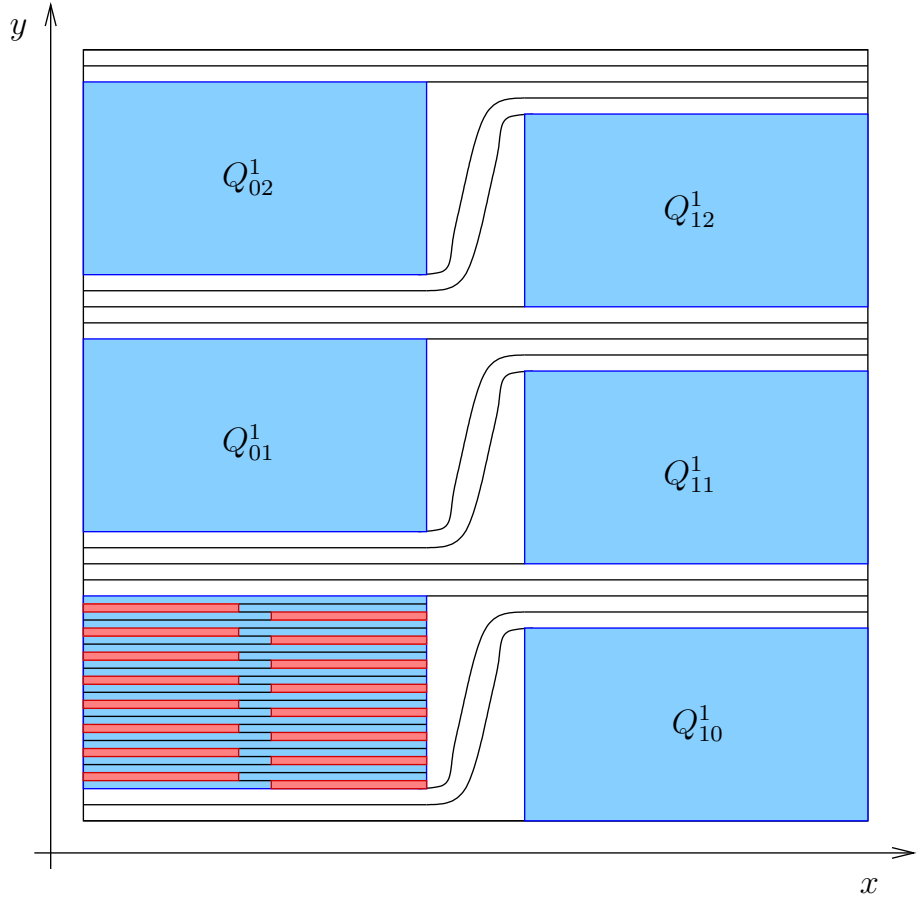}}
\caption{The first steps of the construction done in Remark \ref{R_not_gener_potential}.}
\label{Fi_counterex2}
\end{figure}

{\it Step 2.}
In this step we show how to repeat the above construction in a square of the form $[0,a_k] \times [0,b_k]$, in order to define the level sets of $\Theta$ outside finitely many squares of size $[0,a_{k+1}] \times [0,b_{k+1}]$.

Define the points
\[
x_0 = 0, \ x_1 = a_{k+1} + c_{k+1} = \frac{3}{8} \cdot 2^{-k} + \frac{3}{16} \cdot 2^{-2k}, \quad y_j = j \cdot \bigg( b_{k+1} + \frac{b_{k+1}}{\frac{b_k}{b_{k+1}} - 1} \bigg), \ j=0,\dots,\frac{b_k}{b_{k+1}} - 2.
\]
The new squares which will be used in the next step are given by
\[
Q^k_{0,j} = \bigg( x_0,y_j + \frac{1}{2} \cdot \frac{b_{k+1}}{\frac{b_k}{b_{k+1}} - 1} \bigg) + [0,a_{k+1}] \times [0,b_{k+1}], \quad Q^k_{1,j} = (x_1,y_j) + [0,a_{k+1}] \times [0,b_{k+1}],
\]
with $j=0,\dots,\frac{b_k}{b_{k+1}} - 2$.

As before, in the strips
\[
[0,a_k] \times \bigg[ y_j - \frac{1}{2} \cdot \frac{b_{k+1}}{\frac{b_k}{b_{k+1}} - 1},y_j \bigg], \quad j = 1,\dots,\frac{b_k}{b_{k+1}} - 2,
\]
the level sets of $\Theta$ are straight lines $y = \mathrm{constant}$, and the same for the strips
\[
[0,a_{k+1}] \times \bigg[y_j,y_j + \frac{1}{2} \cdot \frac{b_{k+1}}{\frac{b_k}{b_{k+1}} - 1} \bigg], \quad j = 0,\dots,\frac{b_k}{b_{k+1}} - 2,
\]
\[
[a_{k+1}+c_{k+1},a_k] \times \bigg[y_j - \frac{b_{k+1}}{\frac{b_k}{b_{k+1}} - 1}, y_j - \frac{1}{2} \cdot \frac{(b_{k+1})^2}{\frac{b_k}{b_{k+1}} - 1} \bigg], \quad j = 1,\dots,\frac{b_k}{b_{k+1}} - 1.
\]

Similarly as done in Step 1, we just show how to define the level sets connecting the segments
\[
\{a_{k+1}\} \times \bigg[ 0,\frac{1}{2} \cdot \frac{b_{k+1}}{\frac{b_k}{b_{k+1}} - 1} \bigg] \quad \text{and} \quad \{a_{k+1} + c_{k+1}\} \times [0,b_{k+1}].
\]
One just defines the level sets of $\Theta$ to be
\[
\bigg\{ y +  \bigg( 2 \frac{b_k}{b_{k+1}} - 3 \bigg) y \varrho \bigg( \frac{x - a_{k+1}}{c_{k+1}} \bigg), x \in [a_{k+1},a_{k+1}+c_{k+1}] \bigg\}, \quad  y \in \bigg[ 0,\frac{1}{2} \cdot \frac{b_{k+1}}{\frac{b_k}{b_{k+1}} - 1} \bigg].
\]

{\it Step 3.}
We show that the level sets are $C^\infty$. In fact, by the estimate
\[
\frac{d^\ell}{dx^\ell} \big( b_k \varrho(x/c_k) \big) = \mathcal O(1) \frac{b_k}{c_k^\ell} = \mathcal O(1) 2^{-k^2 + 2 k \ell},
\]
it follows that the curves have a uniform bound in $C^\ell$, for every fixed $\ell$, and thus they belongs to $C^\infty$.

Moreover, it is fairly easy to see that the intersection of each level set with the line $\{1/2\} \times [0,1]$ determines completely the curve, so that the function $\Theta$ can be defined as the $y$-coordinate of this intersection. With a slight variation of this construction one can obtain $\Theta$ to be regular.

{\it Step 4.}
Let $C$ be the compact Cantor set obtained by intersecting all the squares $Q^k_{i,j}$:
\[
C = \bigcap_{k \in \N} \bigcup_{i,j} Q^k_{i,j}.
\]
It is standard to see that %outside the countably many lines which are the boundaries of the squares $Q^k_{i,j}$, 
the function $\Theta$ is single valued on $C$. Moreover, a simple area estimate yields
\[
\mathcal L^2 \bigg( \bigcup_{i,j} Q^k_{i,j} \bigg) = 2^{k} a_{k} \cdot \bigg( 1 - \sum_{\ell = 1}^k b_k \bigg) \geq \frac{3}{4} \cdot \frac{2}{3} = \frac{1}{2},
\]
and this concludes the example.
\end{remark}

\section{Proof of Theorems \ref{T_subpart_final}-\ref{T_Monge_final}.}
\label{S_proff_main_Th}

In this final section we collect the proofs of the remaining main theorems stated in Section \ref{Ss_main_Sud}.

\subsection{Proof of Theorems \ref{T_subpart_step} and \ref{T_final_nu}}
\label{Ss_proof_Th_2}

In this section we prove Theorem \ref{T_subpart_step}, that we recall below.
\begin{theorem2}
\label{T_subpart_step2}
Let $\{Z\ka, C\ka\}_{\nfrac{k = 0,\dots,d}{\mathfrak a \in \mathfrak A^k}}$ be a Lebesgue-regular directed locally affine partition in $\R^d$ and let $\mu$, $\nu$ be probability measures in $\mathcal P(\R^d)$ such that $\mu\ll\LL$ and $\Pi^f_{\mathtt c_{\mathbf D}}(\mu,\nu)\neq\emptyset$.
Then, for all fixed $\check \pi \in \Pi^ f_{\mathtt c_{\mathbf D}}(\mu,\nu)$, there exists a directed locally affine subpartition $\{\check Z^\ell_{\b}, \check C^\ell_{\b}\}_{\nfrac{\ell = 0,\dots,d}{\mathfrak b \in \mathfrak B^\ell}}$ of $\{Z\ka, C\ka\}_{k,\a}$, up to a $\mu$-negligible set $N_{\check \pi}$, such that %with the following properties
\begin{equation*}
\big\{ \check Z^\ell_\b,\check C^\ell_\b \big\}_{\ell,\b} \quad \text{is Lebesgue-regular},
\end{equation*}
and setting $\check \nu^\ell_\b := (\mathtt p_2)_\#\check \pi^\ell_\b$, where $\check \pi^\ell_\b$ is the conditional probability on the partition $\{\check Z^{\ell}_\b \times \R^d\}_{\ell,\b}$, then the sets
\begin{equation}
\label{E_subp_ell_cycl}
\Bigl\{ \check Z^\ell_\b:\,\check Z^\ell_\b\subset Z^\ell_\a\text{ for some $\a\in\A^\ell$}, \ell=1,\dots,d \Bigr\} 
\end{equation}
form a $\Pi^ f_{\mathtt c_{\mathbf D}}(\mu,\{\check \nu^\ell_\b\})$-cyclically connected partition.
\end{theorem2}

\begin{proof}
Let $\{Z\ka, C\ka\}_{\nfrac{k = 0,\dots,d}{\mathfrak a \in \mathfrak A^k}}$ be a Lebesgue-regular directed locally affine partition in $\R^d$ and let $\mu \ll \LL$, $\nu$ be probability measures such that $\Pi^f_{\mathtt c_{\mathbf D}}(\mu,\nu) \neq \emptyset$.

{\it Step 1.}
By Proposition \ref{P_countable_partition_in_reference_directed_planes}, we can restrict the proof of the theorem to a fixed $\ell$-directed sheaf set, which we will denote again by $\{Z^\ell_\a, C^\ell_\a\}_{\a\in\A^\ell}$.

Moreover, by Proposition \ref{P_sub_sheaf_fol} it is enough to prove the existence of subpartitions as in Theorem \ref{T_subpart_step} for the fibration $\{\tilde Z^\ell_\a,\tilde C^\ell_\a\}_{\a\in\A^\ell}$, $\C^\ell(\a)=C^\ell_\a$, given by Proposition \ref{P_map_sheaf_set_into_fibration}. It is indeed clear that $(\mu,\check \Gamma,\mathtt c_{\mathbf D})$-cyclically connected sets, where $\check \pi(\check \Gamma)=1$ are mapped into $(\tilde\mu,\tilde\Gamma,\mathtt c_{\C})$-cyclically connected sets and viceversa, being $\tilde\mu$ obtained through \eqref{E_rxr}. Since the map $\mathtt r$ defined in \eqref{E_r_map} is not a bijection of $\R^d$ into $\A^\ell \times \R^\ell$, then the transport problem on the fibration $\{\tilde Z^\ell_\a,\tilde C^\ell_\a\}_{\a\in\A^\ell}$ depends on the conditional second marginals $\{\check \nu^\ell_\a\}$ of $\check \pi$ w.r.t. the partition $\{Z^\ell_\a\}_{\ell,\a}$. Let $\tilde\nu=\int\check\nu^\ell_\a\,dm(\ell,\a)$.

{\it Step 2.}
Let $\bar \theta$ be the equivalence relation given by Theorem \ref{T_cfibrcfol} for the transport problem $\Pi^f_{\mathtt c_{\C}}(\tilde\mu,\tilde \nu)$. In particular, \eqref{E_subcollection} gives the sets that, when mapped back through the map $\mathbf r$ on the sheaf set, satisfy \eqref{E_subp_ell_cycl} and are $\Pi^f_{\mathtt c_{\mathbf D}}(\mu,\{\check \nu^\ell_\b\})$-cyclically connected.

By Propositions \ref{P_ex_fol} and \ref{P_fol_char}, the remaining sets form a $\mathtt c_{\C}$-foliation into graphs of cone-Lipschitz functions. Let $\{\check Z^\ell_{\a,\b}\}_{\ell'<\ell}$ be the $\mathtt c_{\C}$-differential partition given by Corollary \ref{C_v}. By Theorems \ref{T_coneappr_fol} and \ref{T_FC_no_initial} its complementary $N_{\check \pi}$ is $\mu$-negligible, and by Corollary \ref{C_infinnegl} the partition is Lebesgue-regular. 
\end{proof}

\subsection{Proof of Theorem \ref{T_final_nu}}
\label{Ss_proof_T_final}

The only missing point is to prove that the conditional second marginals $\{\check \nu^\ell_\b\}_{\ell,\b}$ are independent of the particular transference plan $\check \pi \in \Pi^{\mathrm{opt}}_{\d{\cdot}}(\mu,\nu)$ chosen.

From Corollary \ref{C_infinnegl} and Theorem \ref{T_FC_no_initial} it follows that
\[
\nu = \int \check \nu^\ell_\b\,dm(\ell,\b)
\]
is a disintegration, and Corollary \ref{C_transp_graph} implies that $(\mathtt p_2)_\#\check \pi^\ell_\b (\check Z^\ell_\b) = 1$. Hence, from the uniqueness of strongly consistent disintegrations, it follows that
\[
\check \nu^\ell_\b = (\mathtt p_2)_\#\check \pi^\ell_\b \qquad m\text{-a.e.}\ \ell,\b,
\]
yielding the improved version of Conditions (4')-(5') in Theorem \ref{T_final_nu}.

\newpage

\appendix

\section{Minimality of equivalence relations}
\label{A_minimal_equivalence}

Consider a family of Borel equivalence relations on $X$,
\begin{equation}
\label{E_fami_equiv}
\mathbf E = \big\{E_\mathfrak e\subset X\times X, \mathfrak e \in \mathfrak E \big\},
\end{equation}
and let $\mu \in \mathcal P(X)$. By Theorem \ref{T_disint}, we can construct the family of disintegrations
\begin{equation}
\label{E_disi_mu_app}
\mu = \int_{\mathfrak A_\mathfrak e} \mu_{\mathfrak e,\mathfrak a}\, dm_\mathfrak e(\mathfrak a), \quad \mathfrak e \in \mathfrak E,
\end{equation}
where $\A_{\mathfrak e}$, $\mu_{\mathfrak e,\mathfrak a}$, $m_\mathfrak e$ are respectively the quotient space, the conditional probabilities and image measure w.r.t. $E_\mathfrak e\in \mathbf E$.

\begin{definition}
\label{D_clos_sigma_E}
A family of equivalence relations $\mathbf E$ is \emph{closed under countable intersections} if
\[
\{E_{\mathfrak e_i}\}_{i \in \N} \subset \mathbf E \quad \Longrightarrow \quad \exists\,\mathfrak e\in \mathfrak E\text{ s.t. }E_{\mathfrak \e}\subset \bigcap_{i \in \N} E_{\mathfrak e_i}\in \mathbf E.
\]
\end{definition}
In \cite{BiaCar}, a family of equivalence relations is said to be closed under countable intersections if $\{E_{\mathfrak e_i}\}\subset E\quad\Rightarrow\quad \underset{i}{\cap} E_{\mathfrak e_i}\subset E$. However, by direct inspection, also with Definition \ref{D_clos_sigma_E} the following Theorem holds (see Theorem A.11 in \cite{BiaCar}):

\begin{theorem}
\label{T_minimal_equival}
If $\mathbf E$ is closed under countable intersections, there exists $E_{\bar{\mathfrak e}}\in \mathbf{E}$ such that for all $E_\mathfrak e\in \mathbf E$, the following holds:
\begin{enumerate}
\item
\label{item:poifvojfivfi}
if $\mathtt A_\mathfrak e$, $\mathtt A_{\bar{\mathfrak e}}$ are the $\sigma$-subalgebras of the Borel sets of $X$ made of the saturated sets for $E_\mathfrak e$, $E_{\bar{\mathfrak e}}$ respectively, then for all $A \in \mathtt A_\mathfrak e$ there is $A' \in \mathtt A_{\bar{\mathfrak e}}$ s.t.~ $\mu(A \vartriangle A') = 0$;
\item 
\label{item:gfoaifvfffv}
if $m_\mathfrak e$, $m_{\bar{\mathfrak e}}$ are the restrictions of $\mu$ to $\mathtt A_\mathfrak e$, $\mathtt A_{\bar{\mathfrak e}}$ respectively, then $\mathtt A_\mathfrak e$ can be embedded (as measure algebra) in $\mathtt A_{\bar{\mathfrak e}}$ by Point \eqref{item:poifvojfivfi}: let
\begin{equation}
\label{E_disi_m_app}
m_{\bar{\mathfrak e}}  = \int m_{\bar{\mathfrak e},\mathfrak a} dm_\mathfrak e(\mathfrak a)
\end{equation}
be the disintegration of $m_{\bar{\mathfrak e}}$ consistent with the equivalence classes of $\mathtt A_\mathfrak e$ in $\mathtt A_{\bar{\mathfrak e}}$.
\item
\label{item:gfoaivdasvds}
If
\[
\mu = \int \mu_{\mathfrak e,\mathfrak a} dm_\mathfrak e(\mathfrak a), \qquad \mu = \int \mu_{\bar{\mathfrak e},\b} dm_{\bar{\mathfrak e}}(\b)
\]
are the disintegrations consistent with $E_\mathfrak e$, $E_{\bar{\mathfrak e}}$ respectively, then
\begin{equation*}
% \label{E_disinere}
\mu_{\mathfrak e,\mathfrak a} = \int \mu_{\bar{\mathfrak e},\mathfrak b} dm_{\bar{\mathfrak e},\mathfrak a}(\mathfrak b).
\end{equation*}
for $m_\mathfrak e$-a.e.~$\mathfrak a$.
\end{enumerate}
\end{theorem}

\begin{remark}
\label{R_min_contain_E}
From the proof of the above theorem one can observe that it is enough to require that for all $\{\mathfrak e_i\}_{i \in \N} \subset \mathfrak E$ there exists ${ \mathfrak e} \in \mathfrak E$ such that $E_{\mathfrak e}\subset \underset{i}{\cap}\, E_{\mathfrak e_i}$.
\end{remark}

\begin{definition}
\label{D_min_equiv_def}
The equivalence relation $E_{\bar{\mathfrak e}}$ satisfying the assumptions of Theorem \ref{T_minimal_equival} is called the \emph{minimal equivalence relation} in $\mathbf E$ w.r.t. $\mu$.
\end{definition}

In particular, assume that each $E_\mathfrak e$ is given by
\[
E_\mathfrak e = \underset{x'\in X'}{\bigcup}\big\{ \theta_\mathfrak e=x'\big\}\times\{ \theta_\mathfrak e=x'\big\}, \qquad \theta_\mathfrak e : X \to X', \ X' \ \text{Polish}, \ \theta_\mathfrak e \ \text{Borel}.
\]

\begin{corollary}
\label{C_constant_for_minimal_equivalence}
$\forall\,\mathfrak e\in\mathfrak E$, there exists a $\mu$-conegligible set $F \subset X$ such that $\theta_\mathfrak e$ is constant on $F \cap \theta_{\bar{\mathfrak e}}^{-1}(x')$, for all $x' \in X'$.
\end{corollary}

\begin{proof}
Consider the function $\vartheta := (\theta_\mathfrak e,\theta_{\bar{\mathfrak e}})$: by the minimality of $\theta_{\bar{\mathfrak e}}$, it follows that
\[
m_{\bar{\mathfrak e}} = \int m_{\bar{\mathfrak e},(x',x'')} dm_\vartheta(x',x''), \qquad m_\vartheta := \vartheta_\# \mu.
\]
Since $(\mathtt p_2)_\# m_\vartheta = m_{\bar{\mathfrak e}}$, then also
\[
m_\vartheta = \int m_{\vartheta,x''} dm_{\bar{\mathfrak e}}(x''),
\]
and thus
\[
m_{\bar{\mathfrak e}} = \int \bigg[ \int m_{\bar{\mathfrak e},(x',x'')} dm_{\vartheta,t''}(x',x'') \bigg] dm_{\bar{\mathfrak e}}(t'').
\]
This implies that for $m_{\bar{\mathfrak e}}$-a.e. $t''$
\[
\int m_{\bar{\mathfrak e},(x',x'')} dm_{\vartheta,t''}(x',x'') = \delta_{t''},
\]
or equivalently that
\[
m_{\vartheta,x'''} = \delta_{\mathtt x'(x'''),\mathtt x''(x''')}, \quad m_{\bar{\mathfrak e},(\mathtt x'(x'''),\mathtt x''(x'''))} = \delta_{x'''}.
\]
Hence $m_\vartheta$ is concentrated on a graph: by choosing $x''=x'''$, there exists $\mathtt s = \mathtt s(x'')$ Borel such that $m_\vartheta = (\Id,\mathtt s)_\# m_{\bar{\mathfrak e}}$. This is equivalent to say that there exists a $\mu$-conegligible set $F$ such that $\theta_\mathfrak e = \mathtt s \circ \theta_{\bar{\mathfrak e}}$ on $F$.
\end{proof}

%\section{Universally cyclically connected sets}

%Let $X,\,\mu,\,\nu$, $\mathtt c$, $\pi\in\Pi^f_\mathtt c(\mu,\nu)$ and $\Gamma$ be as usual.

%\begin{definition}
% A set $B\subset\mathtt p_1\Gamma$ is $(\Gamma,\mathtt c)$-cyclically connected if
%\[
%\forall\,x_0,x_1\in B, \quad\exists\,\{(x_i',y_i')\}_{i=1}^N\subset\Gamma \quad(\Gamma,\mathtt c)\text{-cycle s.t. $x_i'=x_0, x_j'=x_1$ for some $i,j$.}
%\]
% \end{definition}

%\begin{definition}
% $(\mu,\Gamma,\mathtt c)$-cyclically connected
%\end{definition}

%\begin{definition}
%\label{D_univconn}
% $\Pi^f_\mathtt c(\mu,\nu)$-(or \emph{universally}) cyclically connected
%\end{definition}

%The relevance of Definition \ref{D_univconn} lies in the following theorem. For any $\mathtt c':X\times X\to[0,+\infty)$ which is l.s.c. and bounded, let $\mathtt c_2:X\times X\to[0,+\infty]$ be the \emph{secondary cost}
%\begin{equation}
%\label{E_c2}
% \mathtt c_2(x,y)=\left\{\begin{aligned}
  %                &\mathtt c'(x,y) && &\text{if $(x,y)\in\{\mathtt c<+\infty\}$}\\
  %                &+\infty && &\text{otherwise.}
  %               \end{aligned}\right.
%\end{equation}

%\begin{theorem}
 %Assume there exists a Borel partition $\{Z_\a\}_{\a\in\A}$ into $\Pi^f_\mathtt c(\mu,\nu)$-cyclically connected sets s.t. $\pi\bigl(\underset{\a}{\cup}Z_\a\times Z_\a\bigr)=1$ for all $\pi\in\Pi^f_\mathtt c(\mu,\nu)$. Then, $c_2$-cyclical monotonicity is sufficient for optimality in $\Pi^f_{\mathtt c_2}(\mu,\nu)$, for any secondary cost $\mathtt c_2$ as in \eqref{E_c2}.
%\end{theorem}

\newpage

\section{Notation}
\label{A_appendix_nota}

\begin{tabbing}
\hspace{4cm}\=\kill
$\N$, $\N_0$, $\Q$, $\R$ \> natural numbers, natural numbers with $0$, rational numbers, real numbers
\\
$\Q^+$, $\R^+$ \> positive rational and real numbers
\\
$\R^d$, $\mathbb S^{d-1}$ \> $d$-dimensional real vector space and $(d-1)$-dimensional unit sphere
\\
$|\cdot|$, $\cdot$ \> norm, scalar product in $\R^d$
\\
$B^d(x,r)$ \> open unit ball in $\R^d$ centered in $x$ with radius $r$
\\
$\mathtt P(X)$ \> power set of $X$
\\
$\omega$ \> first countable ordinal number
\\
$\clos\,A$, $\inter\,A$, $\interr A$ \> closure, interior, relative interior of a set $A$
\\
$\partial A$, $\partial_\mathrm{rel} A$ \> topological boundary, relative boundary of a set $A$
\\
$A + A'$ \> vector sum of two sets $A$, $A'$ \eqref{E_vector_sum}
\\
$T A$ \> scalar product of $T \subset \R$ with $A$ \eqref{E_product_scalar_vector_def}
\\
$\underset{i}{\prod}\, X_i$ \> product space of the spaces $X_i$
\\
$\mathtt p_{\bar i} A$, $\mathtt p_{x_{\bar i}} A$, $\mathtt p_{X_{\bar i}} A$ \> projection of $A \subset \underset{i}{\prod}\, X_i$ on the $\bar i$-coordinate
\\
$\Sigma^1_1$ \> class of Souslin sets
\\
$\varTheta$ \> $\sigma$-algebra generated by Souslin sets
\\
$\mathbf f^{-1}$, $A^{-1}$ \> inverse of the multifunction $\mathbf f$ or the set $A \subset X \times Y$ \eqref{E_inverse_multi_function}
\\
$\dom\, \mathbf f$, $\dom\, \mathtt f$ \> domain of the multifunction $\mathbf f$ or the function $\mathtt f$
\\
$\Graph\, \mathtt f$ \> graph of the function $\mathtt f$
\\
$\mathtt g \circ \mathtt f$ \> composition of the functions $\mathtt f$ and $\mathtt g$
\\
$\epi\,\mathtt f$ \> epigraph of the function $\mathtt f$ \eqref{E_epigraph_function}
\\
$\Id$ \> identity map or its graph
\\
$\chi_A(x)$ \> characteristic function of $A$ \eqref{E_char_funct_set_A}
\\
$\ind_A$ \> indicator function \eqref{E_indicator_function_A}
\\
$(X,\mathtt d)$ \> Polish space
\\
$\mathcal G(k,\R^d)$ \> family of linear subspaces of $\R^d$ of dimension $k$
\\
$\mathcal A(k',V)$ \> family of affine subspaces of $V \in \mathcal A(k,\R^d)$ of dimension $k' \leq k$
\\
$\mathtt p_V$ \> orthogonal projection on the affine space $V \in \mathcal A(k,\R^d)$
\\
$\aff\,A$ \> affine space generated by $A$ \eqref{E_affine_span_set}
\\
$\conv\,A$ \> convex envelope of $A \subset \R^d$ %, spherical convex envelope of $A \subset \mathbb S^{d-1}$ \eqref{E_convex_envelope_on_sphere}
\\
$V^\perp$ \> orthogonal space to $V$
\\
$\dim\,A$ \> linear dimension of $\aff\,A$
\\
$\mathcal C(k,\R^d)$ \> family of closed non degenerate cones in $\R^d$ of dimension $k$
\\
$D$ \> compact convex neighborhood of $0$
\\
$|\cdot|_{D^*}$ \> convex norm with unit ball $D$ \eqref{E_norm} %_d_star}
\\
$\mathcal C(k',V)$ \> cones in the affine space $V$ \eqref{E_cones_and_directions_in_subspace}
\\
$\mathring C(\pm r)$ \> cones defined in \eqref{E_epsilon_neigh_of_Cka}, \eqref{E_inverse_neigh_Cone}
\\
$C(\pm r)$ \> closure of the cone $\mathring C(\pm r)$
\\
$\conv_{\mathbb S^{d-1}} A$ \> spherical convex envelope of $A \subset \mathbb S^{d-1}$ \eqref{E_convex_envelope_on_sphere}
% \\
% $\mathcal Z$, $\mathcal Z^k$ \> partitions into sets $Z_\mathfrak a$, $Z^k_\mathfrak a$ 
\\
$\mathcal L^d$ \> Lebesgue measure on $\R^d$
\\
$\mathcal H^k$ \> $k$-dimensional Hausdorff measure
\\
$\varpi_0 \otimes \varpi_1$ \> product of the two measures $\varpi_0$ and $\varpi_1$
\\
$\varpi \llcorner_A$ \> restriction of the measure $\varpi$ to the set $A$
\\
$\mathbf f \llcorner_A$ \> restriction of the function $\mathbf f$ to the set $A$
\\
$\mathrm{Leb}(A)$ \> Lebesgue points of a set $A$ \eqref{E_lebesgue_point_definition}
% \end{tabbing}
% \newpage
% \begin{tabbing}
% \hspace{4cm}\=\kill
\\
$\mathcal B(X)$ \> Borel $\sigma$-algebra on the topological space $X$
\\
$\mathcal P(X)$ \> space of Borel probability measures over $X$ Polish
\\
% $\varpi_0 \otimes \varpi_1$ \> product of the measures $\varpi_0$, $\varpi_1$
% \\
$\varpi_0 \simeq \varpi_1$ \> equivalence of measures \eqref{E_equiv_varpi}
\\
$\mathtt f_\# \varpi$ \> push-forward of the measure $\varpi$ by the function $\mathtt f$ \eqref{E_push_forward}
\\
$\int_\A \mu_\a dm(\a)$ \> disintegration of a measure $\mu$ on a partition $\{Z_\a\}_{\a \in \A}$, Definition \ref{D_dis}
\\
$\int_\A \nu_\a dm(\a)$ \> integration formula \eqref{E_int_measure}
\\
$\int_\A \upsilon_\a d\eta(\a)$ \> disintegration of $\mathcal L^d$ on a partition $\{Z_\a\}_{\a \in \A}$, Remark \ref{R_disint_lebesgue}
\\
$\mathtt c : X \times X \to [0,\infty]$ \> Borel cost function
\\
$\Pi(\mu,\nu)$ \> transference plans with marginals $\mu$, $\nu$ \eqref{E_Pi_mu_nu}
\\
$\Pi^f_{\mathtt c}(\mu,\nu)$ \> transference plans with finite $\mathtt c$ cost \eqref{E_Pi_mu_nu_fin}
\\
$\Pi^{\mathrm{opt}}_{\mathtt c}(\mu,\nu)$ \> $\mathtt c$-optimal transference plans \eqref{E_Pi_mu_nu_optimal}
\\
$\Gamma \subset X \times X$ \> $\sigma$-compact $\mathtt c$-cyclically monotone carriage of $\pi$
\\
$\preccurlyeq_A$ \> preorder on $X$ whose graph is $A$
\\
$\simeq_A$ \> equivalence relation whose graph is $E = \preccurlyeq_A \cap\, (\preccurlyeq_A)^{-1}$
\\
$\preccurlyeq_{(\Gamma,\mathtt c)}$ \> preorder induced by $(\Gamma,\mathtt c)$-axial paths \eqref{E_axpreorder} 
\\
$\partial^+ \varphi$, $\partial^- \varphi$ \> superdifferential, subdifferential of the function $\varphi$, Definition \ref{D_super_sub_diff}
\\
$\mathtt c_{C^k}$ \> convex cone cost \eqref{E_cone_cost}
\\
$\hat \mu$, $\hat \nu$ \> push forward of $\mu$, $\nu$ on the graph of $\psi$ \eqref{E_measure_on_graph}
\\
$\hat \pi$ \> push forward of $\pi$ on the set $\Graph\,\psi \times \Graph\,\psi \subset \R^{d+1} \times \R^{d+1}$ \eqref{E_graph_plan}
\\
$\partial^+ \Graph\, \varphi$, $\partial^- \Graph\, \varphi$ \> super/subdifferential of $\Graph\, \varphi$ \eqref{E_sdiff_varphi}
\\
$Z^k_\mathfrak a$ \> directed locally affine partition of $\R^d$, Definition \ref{D_locaffpart}
\\
$C^k_\mathfrak a$ \> cone of directions for the partition $Z^k_\mathfrak a$, Definition \ref{D_locaffpart}
\\
$\mathbf D$, $\{Z^k_a,C^k_\a\}_{k,\a}$ \> graph of a directed locally affine partition \eqref{E_bD}
\\
$\mathbf Z$, $\mathbf Z^k$ \> base of the partition $Z^k_\mathfrak a$ \eqref{E_mathbf_Z_base_partition}
\\
$\A$ \> quotient space of the directed locally affine partition \eqref{E_disj_unio_A}
\\
$\bigsqcup_k \A^k$ \> disjoint union of the sets $\A^k$ \eqref{E_disj_unio_A}
\\
$\mathcal I(Z^k_\a)$, $\mathcal E(Z^k_\a)$ \> initial, final points of the set $Z^k_\a$, Definition \ref{D_initial_final}
\\
$\mathcal I$, $\mathcal E$ \> initial, final points of a directed locally affine partition \eqref{E_initial_final}
\\
$\mathtt c_{\bD}$ \> cost function associated to a directed locally affine partition \eqref{E_c_bD}
\\
$\int \pi\ka dm(k,\a)$ \> disintegration of the transference plan $\pi$ on $\{Z\ka,C\ka\}_{k,\a}$, Proposition \ref{P_dispiani}
\\
$\Pi^f_{\mathtt c_{\mathbf D}}(\mu,\{\bar\nu^k_\mathfrak a\})$ \> optimal transport plans with finite $\mathtt c_{\mathbf D}$-cost w.r.t. the marginals $\mu$ and $\{\bar\nu\ka\}$, Definition \ref{D_pinu}
\\
$C(\{\mathrm e^k_i\})$ \> cone generated by $\{\mathrm e^k_i\}$ \eqref{E_cone_e_k_i}
\\
$U(\{\mathrm e^k_2\})$ \> unit cube generated by the family of vectors $\{\mathrm e^k_i\}$ \eqref{E_U_square_mathrm_e}
\\
$\mathbf D^k_n$ \> countable decomposition of $\mathbf D$ into sheaf sets, Proposition \ref{P_countable_partition_in_reference_directed_planes}
\\
$\bD^k$ \> sheaf set, Definition \ref{D_sheaf_set}
\\
$\mathfrak A^k \subset \R^{d-k}$ \> quotient space of the sheaf set $\bD^k$ \eqref{E_mathfrak_A_k_def}
\\
$\tilde{\mathbf D}^k$, $\tilde{\mathbf C}^k$ \> $\sigma$-compact fibration, Definition \ref{D_fibration}
\\
$\mathbf r$ \> bi-Lipschitz map of a sheaf set into a fibration \eqref{E_r_map}
\\
$\mathtt i_k$ \> identification map \eqref{E_identif_map}
\\
$\hat \mu$, $\hat \nu$, $\tilde \mu^k_\mathfrak q$, $\tilde \nu^k_\mathfrak q$ \> push forward of $\mu$, $\nu$, $\mu\ka$, $\nu\ka$ by $\mathbf r$ \eqref{E_hat_mu_nu}-\eqref{E_tilde_mu_nu_k_a}
\\
$\{\tilde Z^\ell_{\q,\b},\tilde C^\ell_{\q,\b}\}$ \> subpartition of a directed fibration, Proposition \ref{P_sub_sheaf_fol}
\\
$G$ \> complete $\mathtt c_{\tilde C}$-Lipschitz graph, Definition \ref{D_complete_G}
\\
$\varphi_G$ \> $|\cdot|_{D^*}$-Lipschitz function whose graph is $G$ \eqref{E_phi_graph_G}
\\
$\partial^+ G$, $\partial^- G$ \> super/subdifferential of $G$, Definition \ref{D_tilde_C_diff}
% \end{tabbing}
% \newpage
% \begin{tabbing}
% \hspace{4cm}\=\kill
\\
$C(w,w')$ \> extremal cone of $\tilde C$ satisfying \eqref{E_cww'}
\\
$\mathtt c_{\C}$-Lipschitz foliation \> Definition \ref{D_Lip_fol}
\\
$D(\a)$, $|\cdot|_{D^*(\a)}$ \> convex set and its corresponding norm whose epigraph is $\C(\a)$, Section \ref{Ss_cone_lipschitz}
\\
$\theta : \dom\,\theta \to \mathfrak T$ \> quotient map of a $\mathtt c_{\C}$-Lipschitz foliation, Definition \ref{D_Lip_fol}
\\
$\O(\a)(w,w')$ \> convex set defined in \eqref{E_O_q_ww'}
\\
$\mathtt h^+(\t,x)$, $\mathtt h^-(\t,x)$ \> $|\cdot|_{D^*(\a)}$-Lipschitz functions, Proposition \ref{P_fol_char}
\\
$\partial^+ \theta$, $\partial^- \theta$ \> super/subdifferential of the $\mathtt c_{\C}$-Lipschitz foliation $\{\theta^{-1}(\t)\}_{\t \in \T}$, Definition \ref{D_partial_theta}
\\
$\mathcal T^+\theta$, $\mathcal T^-\theta$ \> forward, backward transport sets \eqref{E_forward_backward_set}
\\
$\mathcal F\theta$ \> set of fixed points \eqref{E_set_of_fixed_points}
\\
$\mathcal D^+ \theta$, $\mathcal D^- \theta$ \> forward/backward direction multifunction \eqref{E_forw_back_dir_theta}
\\
$\mathcal R^{+,\ell} \theta$, $\mathcal R^{-,\ell} \theta$ \> $\ell$-dimensional forward/backward regular transport set \eqref{E_ell_prime_forward_backward_regular_transport_set}
\\
$\mathcal R^+ \theta$, $\mathcal R^- \theta$ \> forward/backward regular transport set \eqref{E_forward_backward_regular_transport_set}
\\
$\mathcal R \theta$ \> regular transport set \eqref{E_regular_transport_set}
\\
$\mathcal N \theta$ \> residual set \eqref{E_residual_set_N}
\\
$\mathtt v^{\pm}$, $\mathtt v$ \> quotient maps defined in \eqref{E_partition_on_T} and \eqref{E_v}
\\
% $\mathfrak B_{\ell,\pm}(\a)$, $\mathfrak B_\ell(\a)$ \> quotient spaces for the maps $\mathtt v^{\pm}$, $\mathtt v$ 
% \\
$\{Z^{\ell,\pm}_{\a,\b},C^{\ell,\pm}_{\a,\b}\}_{\ell,\a,\b}$ \> super/subdifferential directed partition defined in Theorem \ref{T_partition_E+-}
\\
$\{Z^{\ell}_{\a,\b},C^{\ell}_{\a,\b}\}_{\ell,\a,\b}$ \> $\mathtt c_{\C}$-differential locally affine partition defined in Corollary \ref{C_v}
\\
$\hat{\mathbf D}^\pm$, $\hat{\mathbf D}$ \> $\sigma$-compact graphs of the partitions $\{Z^{\ell,\pm}_{\a,\b},C^{\ell,\pm}_{\a,\b}\}$, $\{Z^{\ell}_{\a,\b},C^{\ell}_{\a,\b}\}$ \eqref{E_sub_directed_partition_all}
\\
$\mathcal I^+ \theta$ \> initial points of the superdifferential partition $\{Z^{\ell,+}_{\a,\b},C^{\ell,+}_{\a,\b}\}$ \eqref{E_init_Z+}
\\
$\mathcal E^- \theta$ \> final points of the subdifferential partition $\{Z^{\ell,-}_{\a,\b},C^{\ell,-}_{\a,\b}\}$ \eqref{E_final_Z-}
\\
$\Pi^f_{\mathtt c_{\C},\theta}(\tilde \mu,\tilde \nu)$ \> set of transference plans concentrated on the superdifferential of $\theta$ \eqref{E_542}
\\
$\mathtt d$ \> directional vector field for a $1$-dimensional model set \eqref{E_mathit_v_def}
\\
% $\mathbf Z^1$ \> base of the $1$-dimensional model set $\{Z^1_\a,C^1_\a\}_\a$, Definition \ref{D_modeldire_segm}
% \\
$\mathbf Z^1(\mathtt d,\A,\e,h^-,h^+)$ \> short notation for $1$-dimensional model set $\{Z^1_\a,C^1_\a\}_\a$, Definition \ref{D_modeldire_segm}
\\
$\bar{\mathbf Z}^1$ \> union of the closure of the segments of $\mathbf Z^1$ \eqref{E_clos_Z_1a}
\\
$P_t$, $\mathtt d^t$ \> section of $\mathbf Z^1$, $\mathtt d$ \eqref{E_P_t_section}, \eqref{E_d_t_def}
\\
$\mathtt d^{h^-}$, $\mathtt d^{h^+}$ \> multivalued extensions of $\mathtt d$ to the initial, final points \eqref{E_tt_d_ext}
\\
$\tilde{\mathtt d}^{h^-}_+$, $\tilde{\mathtt d}^{h^+}_-$ \> Borel section of $\mathtt d^{h^-}$, $\mathtt d^{h^+}$, Lemma \ref{L_reg_tt_d_h_pm}
\\
$\tilde{\mathtt d}_+$, $\tilde{\mathtt d}_-$ \> extension of $\tilde{\mathtt d}^{h^-}_+$, $\tilde{\mathtt d}^{h^+}_-$ to the set $\bar Z^1$ \eqref{E_tilde_tt_d_pm}
\\
$\sigma^{s,t}$, $\tilde \sigma^{s,t}_\pm$ \> maps from $P_s$ to $P_t$ defined by $\mathtt d$, $\tilde{\mathtt d}_\pm$ \eqref{E_sigma_s_t}, \eqref{E_tilde_sigma_s_t}
\\
$\mathtt d$, $E_1$, $\bar z \in E_2$ \> cone vector field with base $E_1$ and vertex $\bar z \in E_2 \subset \R^d \setminus E_1$, Definition \ref{D_cone_vector}
\\
$\{\mathtt d_i\}_{i=1}^I$, $E_{\mathtt d_i}$, $\bar z_i$ \> finite union of cone vector fields with bases $E_{\mathtt d_i}$, vertices $\bar z_i$, Definition \ref{D_cone_vector}
\\
$U(\{\e^k_i\},\mathrm h^-,\mathrm h^+)$ \> reference set defined in \eqref{E_pkdimen}
\\
$\mathbf Z^k(\mathcal D,\A^k,\{\e^k_i\},\mathrm h^-,\mathrm h^+)$ \> $k$-dimensional model set $\mathbf Z^k$ with direction map $\mathcal D$, \\
\> reference set $U(\{\e^k_i\},\mathrm h^-,\mathrm h^+)$, Definition \ref{D_k_dim_model}
\\
$(\e,h^-(w,\e),h^+(w,\e))$ \> reference configuration for a $1$-dimensional slice of $\mathbf Z^k$, Definition \ref{D_1_sim_slice}
\\
$\mathtt d_\e$ \> direction map of a 1-dimensional slice \eqref{E_mathtt_d_e}
\\
$\mathcal D^k$ \> direction map of a sheaf set $\mathbf Z^k$ \eqref{E_k_dim_dire_map}
\\
$\mathbf Z^{k,'}$ \> subset of a $k$-dimensional sheaf set \eqref{E_bf_Zk'_def}
\\
$\mathbf Z^k(\mathcal D,\A^{k,'},\{\e^k_i\},\mathrm h^-,\mathrm h^+)$ \> $k$-dimensional $\mathcal D$-cylinder, Definition \ref{D_mathcal_D_cyl}
\\
$\mathbf Z^{k,'}_n$ \> countable family of $k$-dimensional $\mathcal D$-cylinders covering a sheaf set \eqref{E_bf_Z_k'_n}
\\
$\bar \preccurlyeq$ \> linear preorder with minimality properties, Theorem \ref{T_cfibrcfol}
\\
$\bar \theta$ \> Borel function generating $\bar \preccurlyeq$ \eqref{E_bar_theta_th}
\\
$\varGamma(\tilde \pi)$ \> family of $\sigma$-compact carriages of $\tilde \pi$ \eqref{E_varGamma_def_pi}
\\
$\varGamma$ \> family of $\sigma$-compact carriages of transference plans \eqref{E_varGamma_def}
\\
$\{0,1\}^\N$ \> Polish space of sequences with product topology
\\
$\preccurlyeq_{\tilde\Gamma,\mathtt W^{\tg}}$ \> linear preorder constructed for $\tilde \pi$, Theorem \ref{T_order_gamma}
\\
$\theta_{\tilde\Gamma,\mathtt W^{\tg}}$ \> Borel function generating $\preccurlyeq_{\tilde\Gamma,\mathtt W^{\tg}}$ \eqref{E_bar_theta_th_2}
\\
$\mathtt W^{\tg} = \{\mathtt w^{\tg}_n\}_n$ \> countable family of sections, Lemma \ref{L_z_n_dense_selections}
\\
$H_{\tg,n}$ \> sets defined in \eqref{E_h_n}
% \end{tabbing}
% \newpage
% \begin{tabbing}
% \hspace{4cm}\=\kill
\\
$\leq_\mathrm{lexi}$ \> lexicographic ordering on $\R^{d-k}$ \eqref{E_lexico_on_R_d_k}
\\
$\trianglelefteq_\alpha$ \> lexicographic ordering on $\{0,1\}^\alpha$ \eqref{E_lexico_on_Q_01}
\\
$\mathbf Z^{\ell,+}_\e$ \> 1-dimensional slice of a $\mathcal D$-cylinder  $\mathbf Z^{\ell,+}$ \eqref{E_1_dim_superdif}
\\
$\mathtt d^+_\e$ \> vector direction map for the 1-dimensional slice $\mathbf Z^{\ell,+}_\e$ \eqref{E_tt_d_e+}
\end{tabbing}

\newpage

\bibliographystyle{plain}

\end{document}